\documentclass[a4paper,11pt,
 leqno
]{article}
\usepackage{constants}
\usepackage{caption}
\usepackage{subcaption}
\usepackage{comment} 
\usepackage[utf8]{inputenc}
\usepackage[T1]{fontenc}
\usepackage{amsthm}
\usepackage{amssymb}
\usepackage{mathrsfs}
\usepackage{amsmath}
\usepackage{amsfonts}
\usepackage{mathtools}
\usepackage{graphicx}
\usepackage{float}
\usepackage{dsfont}
\usepackage{enumerate}
\usepackage{enumitem}
\usepackage[a4paper]{geometry}
\usepackage{csquotes}
\usepackage{IEEEtrantools}
\usepackage{appendix}
\usepackage{constants}
\usepackage[color=green]{todonotes}
\usepackage{tikz-cd}
\usepackage{enumerate}
\usepackage{mathabx}
\usepackage{hyperref}
\hypersetup{linktocpage,
  colorlinks   = true,
  urlcolor     = blue,
  linkcolor    = blue,
  citecolor   = blue
}

\newtheorem{theorem}{Theorem}[section]
\newtheorem{lemma}[theorem]{Lemma}
\newtheorem{proposition}[theorem]{Proposition}
\newtheorem{claim}[theorem]{Claim}
\newtheorem{Cor}[theorem]{Corollary}
\theoremstyle{definition}

\theoremstyle{remark}

\newtheorem{Rk}[theorem]{Remark}

\numberwithin{equation}{section}
\newcommand{\tend}[2]{\displaystyle\mathop{\longrightarrow}_{#1\rightarrow#2}}

\usepackage{caption}
\title{
\normalsize
\textbf{{
PHASE TRANSITION FOR THE LATE POINTS OF RANDOM WALK%VACANT SET OF\\ RANDOM WALKS AT HIGH INTENSITIES
}}}
\author{}
\date{}
\makeatletter
\newcommand{\Q}{\mathbb{Q}}
\newcommand{\R}{\mathbb{R}}
\newcommand{\Z}{\mathbb{Z}}
\newcommand{\N}{\mathbb{N}}

\newcommand{\F}{\mathcal{F}}

\newcommand{\K}{\mathcal{K}}
\newcommand{\A}{\mathcal{A}}
\newcommand{\W}{\mathcal{W}}
\renewcommand{\L}{\mathcal{L}}
\renewcommand{\P}{\mathbb{P}}
\newcommand{\eps}{\varepsilon}
\newcommand{\1}{\mathds{1}}

\newcommand{\I}{{\cal I}}
\newcommand{\V}{{\cal V}}
\renewcommand{\phi}{\varphi}
\renewcommand{\tilde}{\widetilde}
\renewcommand{\hat}{\widehat}
\renewcommand{\epsilon}{\varepsilon}
\newcommand{\PI}{\mathbb{P}^I}
\newcommand{\tildePI}{\tilde{\mathbb{P}}^I}

\newconstantfamily{c}{symbol=c}

\newcommand{\sast}{\scalebox{0.7}{$*$}}
\newcommand{\mast}{\scalebox{0.8}{$*$}}

\usepackage{color}
\definecolor{Red}{rgb}{1,0,0}
\definecolor{Blue}{rgb}{0,0,1}
\definecolor{Olive}{rgb}{0.41,0.55,0.13}
\definecolor{Yarok}{rgb}{0,0.5,0}
\definecolor{Green}{rgb}{0,1,0}
\definecolor{MGreen}{rgb}{0,0.8,0}
\definecolor{DGreen}{rgb}{0,0.55,0}
\definecolor{Yellow}{rgb}{1,1,0}
\definecolor{Cyan}{rgb}{0,1,1}
\definecolor{Magenta}{rgb}{1,0,1}
\definecolor{Orange}{rgb}{1,.5,0}
\definecolor{Violet}{rgb}{.5,0,.5}
\definecolor{Purple}{rgb}{.75,0,.25}
\definecolor{Brown}{rgb}{.75,.5,.25}
\definecolor{Grey}{rgb}{.7,.7,.7}
\definecolor{Black}{rgb}{0,0,0}

\def\N{\mathbb{N}}
\def\Z{\mathbb{Z}}
\def\R{\mathbb{R}}

\def\F{\mathcal{F}}

\def\B{\mathcal{B}}
\def\BB{Q}
\def\LL{\mathcal{L}}

\renewcommand{\phi}{\varphi}
\renewcommand{\epsilon}{\varepsilon}

\def\K{\mathcal{K}}

\allowdisplaybreaks

\renewcommand{\emptyset}{\varnothing}

\newcommand{\til}{\widetilde}

\newcommand{\E}[1]{\mathbb{E}\!\left[#1\right]}

\newcommand{\Excri}{\mathcal{N}_{\rm{RI}}}
\newcommand{\Excrw}{\mathcal{N}_{\rm{RW}}}

\def\cN{\mathcal{N}}

\def\cA{\mathcal{A}}

\def\P{\mathbb{P}}

\newcommand{\tn}{|\kern-.1em|\kern-0.1em|}

\newcommand{\cp}{\mathrm{cap}}

\newcommand{\cpc}[2]{\mathrm{cap}_{#1}(#2)}

\newcommand\be{\begin{equation}}
\newcommand\ee{\end{equation}}

\def\eps{\varepsilon}

\newcommand{\tcov}{t_{\mathrm{cov}}}

\usepackage{titlesec}
\titleformat{\subsection}[runin]{\normalfont\bfseries}{\thesubsection.}{.5em}{}[.]\titlespacing{\subsection}{0pt}{2ex plus .1ex minus .2ex}{.8em}

\usepackage{titletoc}

\dottedcontents{section}[4em]{}{2.9em}{0.7pc}
\dottedcontents{subsection}[0em]{}{3.3em}{1pc}

\usepackage{multicol}
\usepackage{parcolumns}

\usepackage{tabu}
\usepackage{caption}
\begin{document}
\thispagestyle{empty}
\maketitle
\vspace{0.1cm}
\begin{center}
\vspace{-1.9cm}
Alexis Pr\'evost$^1$, Pierre-Fran\c cois Rodriguez$^2$ and Perla Sousi$^3$

\end{center}
\vspace{0.1cm}
\begin{abstract}
\centering
\begin{minipage}{0.85\textwidth}
Let $X$ be a random walk on the torus of side length $N$ in dimension $d\geq 3$ with uniform starting point, and $\tcov$ be the expected value of its cover time, which is the first time that $X$ has visited every vertex of the torus at least once. 
For $\alpha > 0$, the set $\mathcal{L}^{\alpha}$ of $\alpha$-late points consists of those points not visited by $X$ at time $\alpha \tcov$. We prove the existence of a value $\alpha_{\mast} \in (\frac12,1)$ across which $\mathcal{L}^{\alpha}$ trivialises as follows: for all $\alpha > \alpha_{\mast}$ and $\epsilon\geq N^{-c}$ there exists a coupling of $\mathcal{L}^\alpha$ and two occupation sets $\mathcal{B}^{\alpha_\pm}$ of i.i.d.~Bernoulli fields having the same density as $\L^{\alpha\pm \eps}$, which is asymptotic to $N^{-(\alpha\pm\epsilon)d}$, with the property that the inclusion
$ \mathcal{B}^{\alpha_+} \subseteq \mathcal{L}^{\alpha} \subseteq \mathcal{B}^{\alpha_-}$ holds with high probability as $N \to \infty$. On the contrary, when $\alpha \leq \alpha_*$ there is no such coupling. Corresponding results also hold for the vacant set of random interlacements at high intensities. The transition at $\alpha_{\mast}$ corresponds to the (dis-)appearance of `double-points' (i.e.~neighboring pairs of points) in $\mathcal{L}^\alpha$. We further describe the law of $\L^{\alpha}$ for $\alpha>\frac12$ by adding independent patterns to $\B^{\alpha_{\pm}}$.
In dimensions $d \geq 4$ these are exactly all two-point sets. When $d=3$ one must also include all \textit{connected} three-point sets, but no other.
\end{minipage}
\end{abstract}

\vspace{5.5cm}
\begin{flushleft}

\noindent\rule{5cm}{0.4pt} \hfill September 2023 \\
\bigskip
\begin{multicols}{2}

$^1$Universit\'e de Gen\`eve\\
Section de Math\'ematiques\\
7-9, rue du conseil g\'en\'eral \\
1205 Geneva, Switzerland\\
\url{alexis.prevost@unige.ch}\\[2em]

$^2$Imperial College London\\
 Department of Mathematics\\
 London SW7 2AZ \\
 United Kingdom\\
 \url{p.rodriguez@imperial.ac.uk}

\columnbreak
\thispagestyle{empty}
\bigskip
\medskip

\hfill$^3$University of Cambridge\\
\hfill Faculty of Mathematics \\
\hfill Wilberforce Road\\
\hfill Cambridge CB3 0WA, United Kingdom\\
\hfill\url{p.sousi@statslab.cam.ac.uk}\\[2em]
\end{multicols}
\end{flushleft}

\newpage

\section{Introduction}
\label{sec:intro}
This article studies two models, the random walk on the $d$-dimensional torus $\mathbf{T}=(\Z/N\Z)^d$ of large side length $N$, for $d \geq 3$, at time scales close to the typical time it takes the walk to cover the whole torus, and random interlacements on $\Z^d$ at corresponding intensities. Let $\mathbf{P}$ denote the canonical law of the walk on the graph $\mathbf{T}$ started from uniform distribution, and $X=(X_n)_{n \geq 0}$ be the corresponding (discrete-time) canonical process. It is well-known that the cover time $C_N$ of $X$, which is the first time $X$ has visited every vertex of $\mathbf{T}$ at least once, satisfies 
\begin{equation}\label{e:covertime-scale}
\tcov \stackrel{\text{def.}}{=}\mathbf{E}[C_N]\sim g(0) N^d \log(N^d), \text{ as }N \to \infty, 
\end{equation}
where $g(0)$ denotes the Green's function of the simple random walk on $\Z^d$ at the origin, see Section~\ref{sec:notation} for notation, and $\sim$ means that the ratio of the two quantities tends to one in the given limit. %; in fact, much more is true, see ..... Cover times have been studied intensively, see ...
In view of \eqref{e:covertime-scale}, letting
\begin{equation}
\label{e:vacant-set-RW}
\mathcal{V}_N^u = \mathbf{T} \, \setminus X_{[0,uN^d]}, \text{ for } u > 0,
\end{equation}
where $X_{[0,t]}=\{x \in \mathbf{T}: X_n=x \text{ for some }n \leq t\}$, it is natural to introduce
\begin{equation}
\label{e:u_N}
u_N(\alpha)= \alpha g(0) \log(N^d), \text{ for } \alpha > 0, \, N \geq 1
\end{equation}
(whence $u_N(\alpha)N^d\sim \alpha \tcov$) and consider the vacant set 
\begin{equation}\label{e:L^alpha-RW}
\mathcal{L}^{\alpha} \stackrel{\text{def.}}{=} \mathcal{V}_N^{u_N(\alpha)} \text{ (under $\mathbf{P}$).}
\end{equation} 
The elements of $\mathcal{L}^{\alpha}$ will be referred to as $\alpha$-late points. Note that $\mathcal{L}^{\alpha}$ is decreasing in $\alpha$, and the choices \eqref{e:vacant-set-RW}-\eqref{e:u_N} imply that $\mathcal{L}^{\alpha}$ has density (see~\eqref{eq:Lalphaasymp})
\begin{equation}\label{e:L^alpha-density}
\mathbf{P}(0 \in \mathcal{L}^{\alpha}) \sim N^{-\alpha d} \text{ as $N \to \infty$.}
\end{equation}
The parametrisation in \eqref{e:vacant-set-RW}-\eqref{e:L^alpha-RW} is a matter of convenience; our results do in fact remain true for any choice of `$\alpha$-late time scale' such that \eqref{e:L^alpha-density} holds (for instance, $u = \alpha N^{-d} \tcov$ in \eqref{e:vacant-set-RW}), see Remark \ref{rk:thmsprinkling},\ref{rk:otherparametrization} for this and more; see also Remark~\ref{rk:thmsprinkling},\ref{rk:latecalphapoint} regarding natural (on account of \eqref{e:L^alpha-density}) extensions to \textit{random} timescales such as the first time the vacant set of the walk contains exactly $\lceil N^{(1-\alpha )d} \rceil$ points.

We are interested in global (i.e.~macroscopic) properties of $\mathcal{L}^{\alpha}$ as a subset of $\mathbf{T}$.
One difficulty in addressing questions of this type stems from the long-range correlations inherent to $\mathcal{L}^{\alpha}$.

For the sake of clarity, we focus on \eqref{e:L^alpha-RW} in this introduction. As will turn out, all results presented below allow for either of two generalisations. First, we can deal with late points $\mathcal{L}^{\alpha}_F$ in arbitrary (large) (sub-)regions $F\subseteq \mathbf{T}$ %(whence $\mathcal{L}^{\alpha}$ in \eqref{e:L^alpha-RW} corresponds to the choice $F=\mathbf{T}$), 
at appropriate timescales, ensuring in essence that $\mathcal{L}^{\alpha}_F$ has asymptotic density $|F|^{-\alpha}$, cf.~\eqref{e:L^alpha-density}. Second,
the conclusions of all the results presented in this introduction continue to hold if one replaces \eqref{e:L^alpha-RW} by $\mathcal{L}^{\alpha}= \mathcal{V}^{u_N(\alpha)} \cap ([0,N) \cap \mathbb{Z})^d$ (identifying the vertices of $([0,N) \cap \mathbb{Z})^d$ with those of $\mathbf{T}$), %with $u_N(\alpha)$ as in \eqref{e:u_N} and 
where $\mathcal{V}^u$ refers to the vacant set of random interlacements at level $u$; see Remark~\ref{rk:thmsprinkling},\ref{R:Ri-ext} for details. The set $\mathcal{V}^u$ is characterised by the property that
\begin{equation}\label{e:V^u-def} 
\PI(K\subset\mathcal{V}^u )= \exp\{ - u \text{cap}(K)\}, \text{ for finite $K \subset \Z^d$,}
\end{equation}
and corresponds to a local limit of $\mathcal{V}_N^{u}$ in \eqref{e:vacant-set-RW} as $N\rightarrow\infty$. We refer to Corollary~\ref{cor:coupling} for an explicit coupling between $\V^u$ and $\V_N^u$. This coupling acts as a powerful transfer mechanism. For instance it allows to lift the formula \eqref{e:V^u-def} from $\V^u$ to $\V_N^u$, up to small error, see \eqref{eq:boundonlateRW}. We now present our results.% or \eqref{e:D_N-LB} for concrete applications.

\subsection{Phase transition at \texorpdfstring{$\alpha_*$}{alpha*}} \label{Subsec:PT}
The transition we establish as part of our first main result exhibits a sharp regime of parameters $\alpha$ in which $\mathcal{L}^{\alpha}$ in \eqref{e:L^alpha-RW} completely `trivialises,' i.e.~resembles an i.i.d.~sample of appropriate density. This question has, directly and indirectly, already received considerable attention in the past, as we now briefly review. In view of \eqref{e:L^alpha-density}, the case $\alpha > 1$ is readily dispensed with since $\mathcal{L}^{\alpha}$ is empty with high probability for such values of $\alpha$. In the course of proving Gumbel fluctuations for $C_N$, Belius obtained \cite[(1.4)-(1.5) and Corollary 3.4]{BEL} that the suitably rescaled process of points `around' $\mathcal{L}^1$ converges to a Poisson point process on $(\R/\Z)^d$. Intuitively, this means that the `very late' points, i.e.~the last few vertices to be covered around $\alpha=1$, are roughly independent and uniform. The proof relies on similar results for interlacements \cite{BEL1}, and a coupling of the two objects. We return to this below.

Matters become all the more delicate in the regime $\alpha < 1$, notably because $|\mathcal{L}^{\alpha}|$ is no longer tight in $N$, cf.~\eqref{e:L^alpha-density}.  Let $(\mathcal{B}^{\alpha})_{\alpha \geq 0}$ denote a family of i.i.d.~Bernoulli (site) percolation processes on $\mathbf{T}$ with respective density $\mathbf{P}(0 \in \mathcal{L}^{\alpha})$ and coupled in a monotone fashion in $\alpha$ (e.g.~by means of uniform independent random variables). 
One is naturally led to wonder how $\mathcal{L}= (\mathcal{L}^{\alpha})_{\alpha \geq 0}$ and $\mathcal{B}=(\mathcal{B}^{\alpha})_{\alpha \geq 0}$ relate, if at all. This question was taken up in \cite{JasonPerla}, the main contribution of which can be paraphrased as stating that $\alpha_{**}<1$, where
\begin{equation}
\label{eq:alpha_**}
\alpha_{**} \stackrel{\text{def.}}{=} \inf \big\{ \alpha > 0 : \textstyle \lim_N d_{\rm{TV}}\big(\L^{\beta}, \B^{\beta} \big)= 0 \text{ for all } \beta \geq \alpha \big\},
\end{equation}
%\textcolor{orange}{check if monotone}
where $d_{\rm{TV}}$ denotes the total variation distance; the result that $ d_{\rm{TV}}(\L^{1}, \B^{1}) \to0$ had already been shown prior to this by Prata in \cite[Part~II]{Prata_thesis}.  Within the class of vertex transitive graphs, the recent work~\cite{berestycki2023universality} by Berestycki, Hermon and Teyssier actually gives an optimal  characterization in terms of the diameter for the appearance of Gumbel fluctuations and the uniformity of~$\mathcal{L}^1$.

The method of \cite{JasonPerla} originally gave $\alpha_{**}\leq \alpha_1$ for an explicit value of $\alpha_1=\alpha_1 (d)< 1$ satisfying $\alpha_1 \to 1$ as $d\to \infty$. This deficiency was later removed in~\cite{SamPerla}, where it is shown that 
\begin{equation}\label{e:Perla-Jason-main}
\alpha_{**} \leq \alpha_2 (<\alpha_1<1), 
\end{equation}
with $\alpha_2=\tfrac{3}{4}(d-\tfrac{2}{3})/(d-1)\rightarrow\frac34$ as $d\to \infty,$ see Remark~\ref{rk:thmsprinkling},\ref{rk:removingsprinkling}. To be precise, all afore mentioned results in~\cite{JasonPerla,Prata_thesis,SamPerla} deal with a slightly different set of late points in place of $\mathcal{L}^{\alpha}$, operating at time scales $\alpha t_*$, where $t_*$ is carefully chosen, see e.g.~\cite[(4.3)]{JasonPerla}, and satisfies $t_* = \tcov(1+o(1))$.

A natural barrier for proximity of $\mathcal{L}$ and $\mathcal{B}$ arises as follows. Let \begin{equation}\label{e:double-pts}
D^{\alpha}\stackrel{\text{def.}}{=} \frac12 \sum_{x\sim y} 1\{ x\in \mathcal{L}^{\alpha}, y \in \mathcal{L}^{\alpha}  \},
\end{equation}
where $x\sim y$ denote neighbouring vertices in $\mathbf{T}$, which counts `double-points' in $\mathcal{L}^{\alpha}$, and define
\begin{equation}
\label{eq:alpha_*}
\alpha_*=\alpha_*(d) \stackrel{\text{def.}}{=}\sup\big\{\alpha >0 : \, \textstyle \lim_N \mathbf{E}[D^{\alpha}] =0 \big\},
\end{equation}
(with the convention $\sup \emptyset =0$). The quantity $\mathbf{E}[D^{\alpha}]$ appearing in the limit as $N\rightarrow\infty$ above depends implicitly on $N$ via the choice of the side length for the torus $\mathbf{T}$ on which $\L^{\alpha}$ is defined; this will in fact be the case for all our limits in $N$. One can show, see Lemma~\ref{L:alpha_*} below, that
\begin{equation}
\label{eq:alpha_*-value}
\textstyle \alpha_*=  \frac12\big(1+P_0(\widetilde{H}_0 < \infty)\big) \  \big( \in \big( \textstyle\frac12, 1\big) \big) ,
\end{equation}
where $P_0$ denotes the law of simple random walk on $\Z^d$ with starting point $X_0=0$ and $\widetilde{H}_0=\inf\{ n \geq 1 : X_n =0 \}$. The inequality $\alpha_* >\frac12$ implied by \eqref{eq:alpha_*-value} is important and signals a different qualitative behaviour of $\mathcal{L}^{\alpha}$ and $\mathcal{B}^{\alpha}$. For, the quantity corresponding to $\mathbf{E}[D^{\alpha}] $ with $\mathcal{B}^{\alpha}$ in place of $\mathcal{L}^{\alpha}$ in \eqref{e:double-pts} diverges for all $\alpha > \frac12$. The threshold $\alpha_*$ already appears in \cite{JasonPerla}, where it is shown (for slightly different $\mathcal{L}^{\alpha}$) that $\alpha_{**}\geq \alpha_* (> \frac12)$, which together with \eqref{e:Perla-Jason-main} implies that the threshold $\alpha_{**}$ in \eqref{eq:alpha_**} is non-degenerate.
%\AP{Maybe add that this implies $\alpha_{**}\geq \alpha_*$ by \cite{JasonPerla}?}

\medskip

Our first result establishes a sharp transition for the set $\mathcal{L}^{\alpha}$ in \eqref{e:L^alpha-RW} across the threshold~$\alpha_*$, for a slightly different measure of distance between random sets than \eqref{eq:alpha_**}, allowing the introduction of a small sprinkling parameter $\eps>0$ (the case of $d_{\rm{TV}}$ corresponds to setting $\varepsilon =0$). We note in passing that such a measure of proximity has a long history in the context of problems involving strong correlations, see \cite{BEL, MR3563197, MR2838338, MR2386070} for instance. In the sequel we denote by $\mathcal{Q}_{\alpha, \varepsilon, N}$ the family of all couplings between $\mathcal{L}^{\alpha}$ and $(\mathcal{B}^{\alpha - \varepsilon}, \mathcal{B}^{\alpha + \varepsilon})$, for $\alpha, \varepsilon > 0$ and integer $N \geq 1$.

\begin{theorem}[$d \geq 3$]\label{thm:uncoveredset}  With $\alpha_*$ as in \eqref{eq:alpha_*}, for suitable $\varepsilon_0= \varepsilon_0(\alpha, d) >0$, the following hold:
\begin{enumerate}[label=\roman*)]
\item\label{ite:imain} For all $\alpha > \alpha_*$, $\varepsilon \in(0, \varepsilon_0)$ and $N \geq1$, there exists a coupling $\mathbb{Q} \in \mathcal{Q}_{\alpha, \varepsilon, N} $ such that 
\begin{equation}\label{e:coup-subcrit}
\lim_N \mathbb{Q}\big( \mathcal{B}^{\alpha + \varepsilon} \subset \mathcal{L}^{\alpha}\subset \mathcal{B}^{\alpha - \varepsilon} \big )=1.
\end{equation}
\vspace{-6mm}
\item\label{ite:iimain} If $\alpha=\alpha_*,$ then for all $\varepsilon \in(0, \varepsilon_0)$, with the supremum ranging over $\mathcal{Q}_{\alpha_{\sast}, \varepsilon, N} $ below, 
 \begin{equation}\label{e:coup-crit}
\lim_N\sup_{\mathbb{Q}} \mathbb{Q}\big( \mathcal{B}^{\alpha_{\mast} + \varepsilon} \subset \mathcal{L}^{\alpha_{\mast}}\subset \mathcal{B}^{\alpha_{\mast} - \varepsilon} \big )=e^{-d}.
\end{equation}
\vspace{-4mm}
\item\label{ite:iiimain} For all $\alpha < \alpha_*,$  $\varepsilon \in(0, \varepsilon_0)$ and any coupling $\mathbb{Q} \in \mathcal{Q}_{\alpha, \varepsilon, N}  $ one has that 
\begin{equation}\label{e:coup-supercrit}
\lim\limits_{N} \mathbb{Q}\big( \mathcal{B}^{\alpha + \varepsilon} \subset \mathcal{L}^{\alpha}\subset \mathcal{B}^{\alpha - \varepsilon} \big )=0.
\end{equation}
\end{enumerate}
\end{theorem}

\begin{figure}[ht]
\begin{subfigure}{0.5\textwidth} \centering
\includegraphics[scale=0.22]{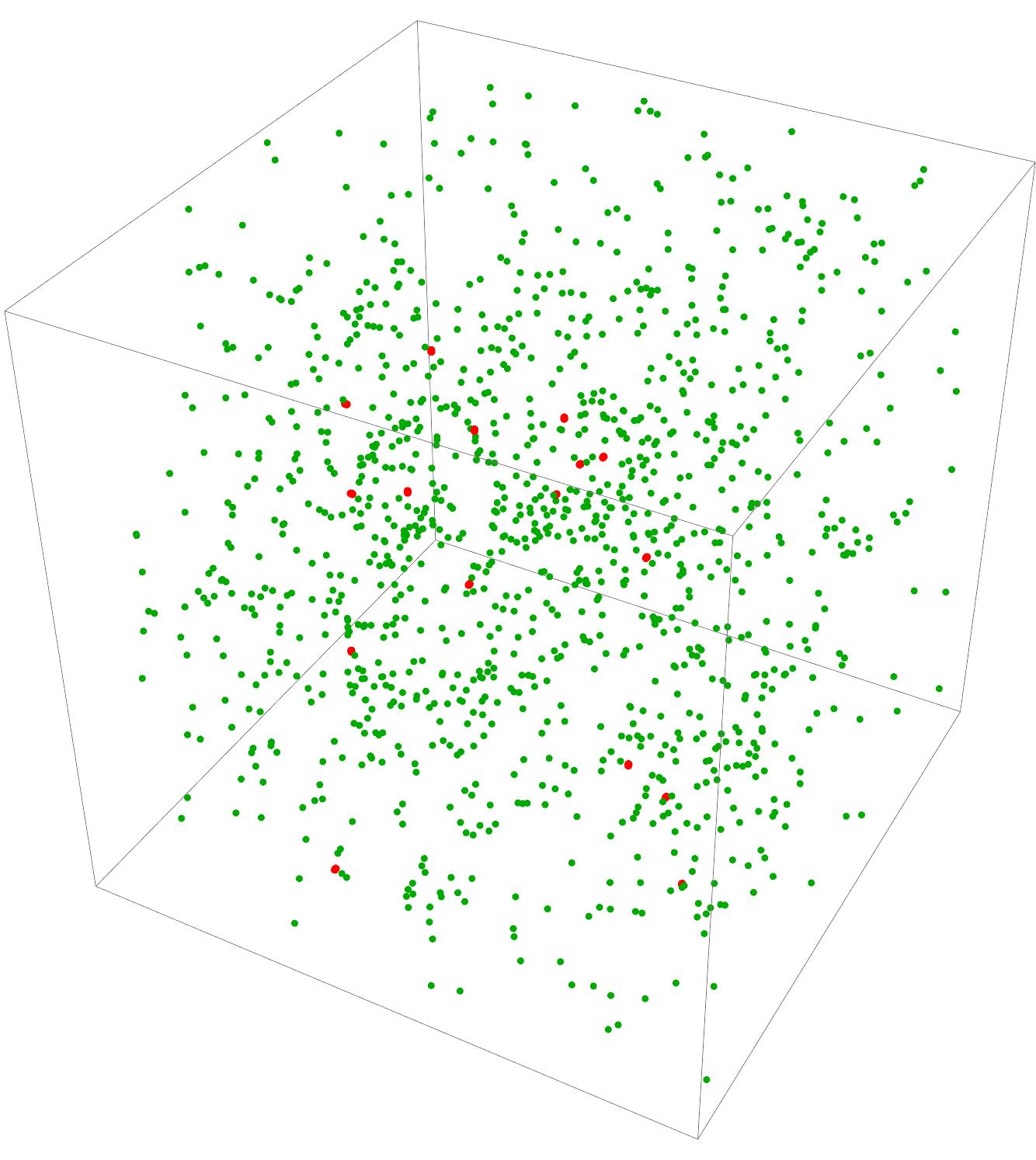}
\caption{Simulation of $\L^{0.6}$ for $N=400$}
\end{subfigure}
\begin{subfigure}{0.5\textwidth} \centering
\includegraphics[scale=0.22]{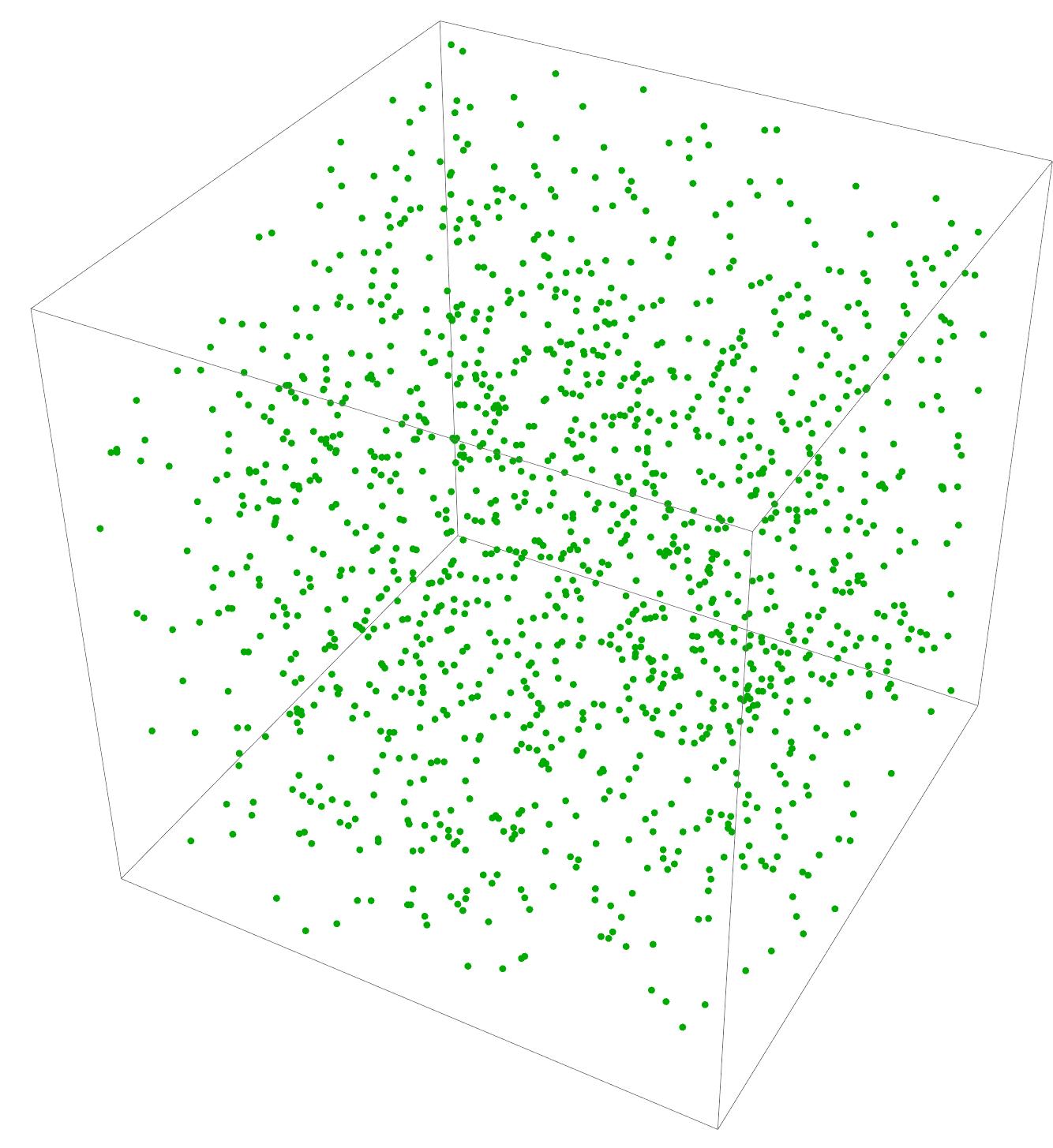}
\caption{Simulation of $\mathcal{B}^{0.6}$ for $N=400$}
\end{subfigure}
\caption{In red, the points of $\L^{\alpha}$ which have a neighbor in $\L^{\alpha}$. In dimension three for $\alpha=0.6$, there are many such ``double points'' in $\L^{\alpha}$, but they disappear at $\alpha_{*}\approx 0.67$. There are no such red points in $\B^{\alpha}$, $\alpha>0.5$.}
\end{figure}

In words, \eqref{e:coup-subcrit}-\eqref{e:coup-supercrit} indicate that $\L^{\alpha}$ is close to an i.i.d.~Bernoulli field, up to a sprinkling parameter $\eps,$ if and only if $\alpha>\alpha_*$. The constant $e^{-d}$ appearing in \eqref{e:coup-crit} corresponds to the probability not to have any neighbours in $\L^{\alpha_{\mast}},$ i.e.~the probability that there does not exist $x\sim y$ in $\mathbf{T}$ such that $\{x,y\}\subset\L^{\alpha_{\mast}},$ see Remark~\ref{rk:thmsprinkling},\ref{rk:constantatcriticality} for details. Theorem~\ref{thm:uncoveredset} is proved in Section~\ref{sec:denouement} and allows for various extensions, see Remark~\ref{rk:thmsprinkling}, which among other things, include an analogue of Theorem~\ref{thm:uncoveredset} for interlacements and give in either case quantitative bounds in $N$ and $\varepsilon$ for the probability of the event in \eqref{e:coup-subcrit}, see \eqref{eq:smallestsprinkling}. In fact the quantitative bounds allow to take $\varepsilon=\varepsilon_N$ polynomially small in $N$ in \eqref{e:coup-subcrit}. Up to a small sprinkling, Theorem \ref{thm:uncoveredset} thus answers Question~1 in Section~7 of \cite{JasonPerla}. Moreover, all afore mentioned results of \cite{JasonPerla,Prata_thesis,SamPerla} (in particular, \eqref{e:Perla-Jason-main}) as well as \cite{BEL1,BEL} can all be easily recovered from Theorem \ref{thm:uncoveredset}, see Remark \ref{rk:thmsprinkling},\ref{rk:removingsprinkling} for details. We return to the challenges in proving Theorem~\ref{thm:uncoveredset} in \S\ref{subsec:loc} below. We further note that the process of covering in dimension $d=2$ is very different, owing to the recurrence of the walk, which causes late points to cluster; see \cite{10.1214/009117905000000387} for more on this.

In order to support the intuition conveyed by \eqref{e:coup-subcrit}-\eqref{e:coup-supercrit} that $\mathcal{L}^{\alpha}$ consists a.s.~of asymptotically independent points if and only if $\alpha>\alpha_*,$ one could instead also examine the convergence of the rescaled point process $\sum_{x\in{\mathcal{L}^{\alpha}}}\delta_{x/N^{\alpha d}}$ to a Poisson point process on $\R^d,$ as done for instance in \cite[Corollary~0.2]{BEL1}, \cite[(1.5)]{BEL}, or \cite[Theorem~1.1]{ChiariniCiprianiSubhra} in the context of the high points of the Gaussian free field.  However, the resulting scaling limit is not fine enough to capture the phase transition of Theorem~\ref{thm:uncoveredset}: the limit is actually Poissonian for \textit{any} $\alpha\in{(0,1]},$ cf.~Remark~\ref{rk:final},\ref{rk:pointprocess}.

In view of Theorem~\ref{thm:uncoveredset}, a natural question is to describe the law of $(\mathcal{L}^{\alpha})_{\alpha>0}$ as a process in $\alpha$. This was first investigated in \cite[Theorem~3]{Prata_thesis} for $\alpha$ very close to $1$ -- more precisely for $\alpha=1+\frac{\beta}{\log N},$ $\beta>0$, at which $|\L^{\alpha}|$ remains tight -- and we will prove a similar result for $\alpha\in{(\alpha_*,1]}$ in \S\ref{subsec:alpha-process}. More precisely, defining
\begin{equation}
\label{eq:bfalpha-intro}
{\alpha}_x = \sup\{\alpha>0:\ x\in{\L^{\alpha}}\},\quad x\in\mathbf{T},
\end{equation}
we show in Theorem~\ref{thm:processusabovealpha*} that
\begin{equation}
\label{eq:alphahatalpha-intro}
\text{$({\alpha}_x-\alpha_{\mast})_{x \in \L^{\alpha_{\sast}}}$ is `close' to a family of i.i.d.~exponential  variables with mean $ d\log N$,}
\end{equation}
where proximity is again measured in terms of a coupling such that both processes differ by at most $\varepsilon$.

\subsection{Localization}\label{subsec:loc}
We now discuss one of the main difficulties permeating this work, which is that of `localizing' the dependence of $\mathcal{L}^{\alpha}$. Attending to it leads to Theorem~\ref{The:shortrangeapprointro-new} below (and its more elaborate version, Theorem~\ref{thm:rwshortrange}), which is of independent interest; see also \cite{10.1214/aop/1019160110} for other localization phenomena for Brownian motion in dimensions $d \geq 3$. To provide some context, we now comment on the proof of Theorem \ref{thm:uncoveredset}, and focus on the case $\alpha>\alpha_*$, %(we refer to the discussion below Theorem \ref{The:alpha>1/2} for the critical case $\alpha=\alpha_*$).
which already highlights the essence of the issue. Following~\cite{SamPerla}, see also \cite{ChiariniCiprianiSubhra} in the context of the Gaussian free field, one may seek to apply the Chen-Stein method, e.g.~in the form presented in \cite[Theorem 3]{Arratia}, directly to $\mathcal{L}^{\alpha}$, which requires individual control on three terms, commonly referred to as $b_1,$ $b_2$ and $b_3$, that each need to be small. The problematic term is $b_3$, which carries the long-range information and leads to limitations of this method, giving rise e.g.~to the threshold $\alpha_2$ in \eqref{e:Perla-Jason-main}. 

To overcome this difficulty, we follow a modified scheme, see Lemma~\ref{lem:corofchenstein} below, which consists of introducing an intermediate (localized) family $\til{\mathcal{L}}= (\til{\mathcal{L}}^{\alpha})_{\alpha \geq 0}$, coupled to ${\mathcal{L}}$ in a way that the two are close (up to sprinkling), and to which \cite{Arratia} applies with $b_3=0$. 
The latter requires a (pointwise) short-range property: $\{x\in{\til{\mathcal{L}}^{\alpha}}\}$ needs to be independent of $\til{\mathcal{L}}^{\alpha}\cap {\BB}(x,R)^c$, for a suitable localization scale $R$. Here $\BB(x,R)$ denotes the box around $x$ with size length $N$ in $\mathbf{T},$ see \S \ref{sec:notation} for a precise definition.  
Our next result asserts that a process $\til{\mathcal{L}}$ with these features can indeed be constructed. In fact, this is not specific to the high intensity
regime \eqref{e:u_N} at all, and best stated in terms of local times. Thus let $({\ell}_{{x},u})_{{x}\in \mathbf{T}}$ denote the local times of $X_{u N^d}$ under~$\mathbf{P}$, cf.~\eqref{eq:deflocaltimes}, so that $\mathcal{V}_N^u=\{ x \in \mathbf{T} :  {\ell}_{{x},u}=0\}$, cf.~\eqref{e:vacant-set-RW}.
The following result is proved in Section~\ref{sec:shortrangeRWRI} (see Theorem~\ref{thm:rwshortrange} for a more general statement) and has an analogue
for random interlacements.

\begin{theorem}[Localization]
\label{The:shortrangeapprointro-new}
For all $N \geq 1$, $R \in [1, \frac{N}{10}]$ and $u_0>0,$ there exists a family $(\til{\ell}_{x,u})_{x\in{\Z^d},u>0}$ such that the following holds. For all $u\in{(0,u_0]}$ and $\varepsilon > 0$, there exists a coupling $\til{\mathbf{P}} $ of ${\ell}_{\cdot,u}$ with $\til{\ell}_{\cdot,u(1\pm \varepsilon)}$ such that:
\begin{align}
&\label{e:fr-intro} \til{\ell}_{x,u(1\pm \varepsilon)} \text{ is independent from } \sigma(\til{\ell}_{y,u(1\pm \varepsilon)} : y \notin \BB(x ,R)), \text{ for all $x \in \mathbf{T}$, and} \\
&\label{e:incl-intro} \til{\mathbf{P}} \big( \, \til{\ell}_{x,u(1-\varepsilon)} \leq {\ell}_{x,u} \leq \til{\ell}_{x,u(1+\varepsilon)}, \ \forall \, x\in \mathbf{T} \big) \geq 1-CN^dR^{2d}\exp\big(-c \epsilon  \sqrt{ u R^{d-2}}\big),
\end{align}
for some $c=c(d)>0$, $C=C(d)< \infty$.
\end{theorem}

The set $\til{\mathcal{L}}^{\alpha}$ alluded to above is then obtained 
by applying Theorem~\ref{The:shortrangeapprointro-new} with $u=u_N(\alpha)$ 
upon defining $\til{\mathcal{L}}^{\alpha}= \{ x \in \mathbf{T} : \til{ {\ell}}_{x,u_N(\alpha)}=0\} $
with $R$ appropriately chosen, e.g.~$R\approx (\eps^{-2}\log(N^d))^{\frac1{d-2}}$, to make the error in \eqref{e:incl-intro} small. From this, \eqref{e:coup-subcrit} eventually follows from \cite[Theorem~3]{Arratia}.

The two localisation features \eqref{e:fr-intro} and \eqref{e:incl-intro} are difficult to engineer simultaneously. The requirements \eqref{e:fr-intro}-\eqref{e:incl-intro} can be seen to correspond to a kind of `approximate spatial Markov property' (at one point) for the local times. Indeed for the related Gaussian free field, a version of Theorem~\ref{The:shortrangeapprointro-new} can be obtained by exploiting the field's Markov property. In a follow-up article, we will exploit this to establish a phase transition for the high points of the field, studied previously in~\cite{ChiariniCiprianiSubhra,SamPerla}, as well as other long-range correlated models.  

The `approximate' spatial Markov property implied by Theorem \ref{The:shortrangeapprointro-new}  is however much harder to obtain, and the proof of Theorem \ref{The:shortrangeapprointro-new} is at the heart of this article. One of our main tools is the soft local times method from \cite{softlocaltimes,AlvesPopov}, with a twist. The method has been introduced to compare random interlacements trajectories on fixed well-separated sets $A_1$ and $A_2,$ to a version of random interlacements independent on $A_1$ and $A_2$, whose definition depend on the choice of $A_1$ and $A_2.$ In order to derive Theorem \ref{The:shortrangeapprointro-new} though, we need to compare random interlacements with a version of interlacements defined on the whole set $\mathbf{T}$ at once, and which is independent on \textit{any} pair of well-separated sets $A_1$ and $A_2$, when $A_1$ is a singleton. We achieve this by introducing an inverse soft local times technique, see Section \ref{sec:softlocaltimes}. 
Actually, our techniques also lead to a new, simpler, more explicit but slightly less general proof of the coupling between the random walk and random interlacements from \cite{MR3563197}, see the discussion below Theorem \ref{thm:rwshortrange} and Corollary~\ref{cor:coupling}.

\subsection{Law of the late points for \texorpdfstring{$\alpha\in{(\frac12,\alpha_*]}$}{alpha in (1/2,alpha*)}}

For $\alpha\leq \alpha_*,$ Theorem \ref{thm:uncoveredset} asserts that $\L^{\alpha}$ and the i.i.d.~set $\B^{\alpha}$ are not close anymore, which, in view of \eqref{e:double-pts}-\eqref{eq:alpha_*}, is due to the emergence of neighboring pairs of points in $\L^{\alpha}$ (but not in $\B^{\alpha}$). But could proximity for $\alpha$ smaller but close to $\alpha_*,$ perhaps be restored by adding such pairs, and, if so, in an independent fashion? What about other `clusters' of late points, e.g.~sets of two points at distance $k \geq 2$, or even finite sets of larger cardinality? Do such `clusters' appear in $\L^{\alpha}$ as $\alpha$ is reduced? Is their occurrence Poissonian? The next result sheds light on these questions. For finite $K \subset \Z^d$, denote by $\mathrm{cap}(K)$ the capacity of $K$ on the transient graph $\Z^d,$ $d\geq3.$ We also extend this definition to any set $K\subset\mathbf{T}$ with $\ell^{\infty}$-diameter $\delta(K),$ (see Section~\ref{sec:notation} for a definition) at most $N-1$ by essentially identifying $\mathbf{T}$ with the cube $Q(0,N)$ via an adapted bijection so that $K\subset Q(0,N-1),$ where $Q(0,R)$ is the box in $\Z^d$ around the origin $0$ of side length $R$ for all $R>0.$ For sets $K\subset\mathbf{T}$ with diameter $N$, we simply take the convention $\mathrm{cap}(K)=\infty$.  We refer to below \eqref{e:e_K} for precise definitions. Now for $K$ either in $\Z^d$ or in $\mathbf{T}$ with diameter at most $N-1,$ let 
\begin{equation}
\label{eq:defalpha*K}
    \alpha_*(K) \stackrel{\text{def.}}{=}(g(0)\mathrm{cap}(K))^{-1}
\end{equation}
%for $K\subset\subset\Z^d,$ or $K\subset\Z_n^d,$ with $\alpha_*(K)>1/2,$ so that for $\alpha>\alpha_*(K),$ $\L^{\alpha}^{F}$ can be well-approximated by independent sets $K'$ with $\alpha_*(K')>\alpha_*(K),$ but that it is not the case anymore for $\alpha<\alpha_*(K).$ 
and note that $\alpha_*(\{x\})=1$ for all $x \in \Z^d$ or $\mathbf{T}$ by \eqref{e:cap-2point}, while in fact $\alpha_*(\{x,y\})=\alpha_*$ as in \eqref{eq:alpha_*}, for $x\sim y$, see Lemma~\ref{L:alpha_*}. As will become clear, the parameter $\alpha_*(K)$ corresponds to the largest value of $\alpha$ at which one can find a translated version of the set $K$ in $\L^{\alpha}$ similarly as in the definition \eqref{eq:alpha_*} of $\alpha_*;$ cf.~\eqref{e:aalpha_*K-equiv} or Lemma~\ref{L:D_LB} for details.

Let $p^{\alpha}(A)=\mathbf{P}\big( \mathcal{L}^{\alpha} \cap \BB(A, R_{\mathbf{T}}) = A\big)$ for all $A\subset\mathbf{T}$ and $\alpha>0,$ where $R_{\mathbf{T}}= \log (|\mathbf{T}|)^{\frac1{d-2}}$, let $(U_A)_{A\subset{\mathbf{T}}}$ be an i.i.d.~family of uniform random variables on $[0,1],$   
and for $K\subset\subset\Z^d$, %with $\pi: \Z^d \to \mathbf{T}$denoting the canonical projection,
define the family ${\mathcal{B}}_{K} =({\mathcal{B}}^{\alpha}_{K})_{\alpha \geq 0}$ of sets 
\begin{equation}
\label{eq:defBFK}
{\mathcal{B}}^{\alpha}_{K}=\bigcup_{\substack {A\subset{\mathbf{T}}:\,\alpha_{\sast}(A)>\alpha_{\sast}(K)\\ U_{A}\leq p^{\alpha}(A)}}A.
\end{equation}
Note that $K\subset\Z^d$ in \eqref{eq:defBFK} (whereas $A\subset\mathbf{T}$). In particular $K$ does not depend on $N$, the side length of the torus, which is important when taking limits as $N \to \infty$, as in \eqref{eq:phasetransitionK} below. Notice also that sets $A\subset\mathbf{T}$ with diameter $N$ are never considered in the definition of $\B^{\alpha}_K$ by the above convention. Actually ${\mathcal{B}}^{\alpha}_{K}$ only depends on $K$ through $\textnormal{cap}(K)$. Moreover, if $K=\{x,y\}$ for $x\sim y$, then $\alpha_*(K)=\alpha_*\geq\alpha_*(A)$ for all $A\subset{\mathbf{T}}$ with $|A|\geq2,$ whence $A$ in \eqref{eq:defBFK} only ranges over singletons and ${\mathcal{B}}^{\alpha}_{K}$ is virtually equal in law to ${\mathcal{B}}^{\alpha}$, %at least when $\alpha > \alpha_*$
as will later be seen in detail. %save for the condition that the $R_{\mathbf{T}}$-neighborhood be free of late points.
This is actually not entirely true, notably when $\alpha\leq \alpha_*$, see \eqref{e:TV-B-Btilde}, but one may pretend for the purposes of this introduction that $({\mathcal{B}}^{\alpha}_{K})_{\alpha \geq 0}\stackrel{\text{law}}{=}({\mathcal{B}}^{\alpha})_{\alpha \geq 0}$ when $K=\{x,y\}$, $x\sim y$. Admitting this, the following result can then essentially be viewed as a generalization of Theorem~\ref{thm:uncoveredset}.

To state it, we introduce one more convenient notation, which allows to capture transitions such as \eqref{e:coup-subcrit}-\eqref{e:coup-supercrit} in a concise (albeit less explicit) way. 
For two decreasing families of random subsets $\mathcal{S}=(\mathcal{S}^{\alpha})_{\alpha \geq 0}$ and $\mathcal{U}=(\mathcal{U}^{\alpha})_{\alpha \geq 0}$ of $\mathbf{T}$, we define for all $\alpha\in{(0,1)}$ and $\eps\in{(0,\alpha)}$ 
\begin{equation}
\label{eq:distancesprinkling}
d_{\eps}(\mathcal{S},\mathcal{U};\alpha)=\inf\bigg\{\delta\in{[0,1]}\,:\begin{array}{l}
\exists\,\text{a coupling $\mathbb{Q}$ between }\mathcal{S}^{\alpha} \text{ and }({\mathcal{U}}^{\alpha-\eps},\,{\mathcal{U}}^{\alpha+\eps})\\ \text{such that }  \mathbb{Q}\big({\mathcal{U}}^{\alpha+\eps}\subset {\mathcal{S}}^{\alpha}\subset {\mathcal{U}}^{\alpha-\eps}\big)\geq 1-\delta.
\end{array}
\bigg\}
\end{equation} 
Note that $\lim\limits_{\eps\rightarrow0}d_{\eps}(\mathcal{L},\mathcal{B};\alpha)=d_{\rm{TV}}(\mathcal{L}^{\alpha},\mathcal{B}^{\alpha})$ by continuity.

\begin{theorem}
\label{cor:phasetransition2-intro}
For $\L^{\alpha}$ as in \eqref{e:L^alpha-RW}, all $K\subset\subset\Z^d$ and all $\alpha\in{(0,1)}$,
	\begin{equation}
	\label{eq:phasetransitionK}
	\lim_{\eps\rightarrow0}\lim_{N\rightarrow\infty}d_\epsilon\big(\mathcal{L},{\B}_{K};\alpha\big)
	\begin{cases}
	=0 &\text{ if } 	\alpha>\alpha_*(K)>1/2\\
	%\in (0,1) &\text{ if } \alpha=\alpha_*(K)\\
	=1 &\text{ if } \alpha<\alpha_*(K)\text{ or }\alpha_*(K)\leq1/2
	\end{cases}
	\end{equation}
and if $\alpha_*(K)>1/2$
\begin{equation}
	\label{eq:phasetransitionK'}
	0< \lim_{\eps\rightarrow0}\liminf_{N\rightarrow\infty}d_\epsilon\big(\mathcal{L},{\B}_{K};\alpha_*(K)\big) \leq \lim_{\eps\rightarrow0}\limsup_{N\rightarrow\infty}d_\epsilon\big(\mathcal{L},{\B}_{K};\alpha_*(K)\big) < 1.
\end{equation}
\end{theorem}

%\textcolor{orange}{change limit for N in middle case to limsup/liminf}
By applying Theorem~\ref{cor:phasetransition2-intro} with $K$ any pair of neighbors, e.g.~$K=K_0= \{0,x\}$, $x \sim 0$, one (essentially, cf.~above) recovers Theorem~\ref{thm:uncoveredset}. Now suppose that $K\subset\subset\Z^d$ is another set with $\alpha_*(K)>1/2,$ not isomorphic to $K_0$ under lattice symmetries; for instance $K= \{0,x\}$, $|x|_{1}=2$. Then \eqref{eq:phasetransitionK} asserts that $\mathcal{L}^{\alpha}$ will be close (as measured by $d_{\varepsilon}$) to ${\mathcal{B}}_{K}$ if and only if $ \alpha > \alpha_*(K)$.
The set ${\mathcal{B}}_{K}^{\alpha}$ comprises independent samples of all allowed `clusters' (namely, sets $A \subset{\mathbf{T}}$ with $\alpha_*(A)> \alpha_*(K)$), at the correct intensity, corresponding to their probability to be seen in $\mathcal{L}^{\alpha}$ and to be isolated, i.e.~with no other late points present in their $R_{\mathbf{T}}$-neighbourhood. This can be regarded as a positive answer to Question 2 in \cite[Section 7]{JasonPerla}.  Note the consistency of the transitions \eqref{eq:phasetransitionK} as $K$ varies: for instance, if $\alpha> \alpha_*(= \alpha_*(K_0))$, then, as it turns out, for every $K \in {\Z^d}$ with $\alpha_*(K)\in{(1/2,1)},$ the sets ${\mathcal{B}}^{\alpha}_{K}$ and ${\mathcal{B}}^{\alpha}_{K_0}$ are virtually indistinguishable as $N\rightarrow\infty,$ and actually also indistinguishable of $\B^{\alpha}.$ In fact \eqref{e:coup-subcrit} will be deduced from the first line of \eqref{eq:phasetransitionK} in precisely this way, cf.~\S\ref{sec:denouement}; see also Remark~\ref{rk:thmsprinkling},\ref{R:direct-Thm1} for alternatives. Note also that the critical case \eqref{eq:phasetransitionK'} is less explicit than \eqref{e:coup-crit}, see Remark~\ref{rk:thmsprinkling},\ref{rk:constantatcriticality} for more on this.

Theorem~\ref{cor:phasetransition2-intro} is mainly interesting when $\alpha_*(K)>1/2$, as it then describes precisely the law of $\L^{\alpha}$ for $\alpha>\alpha_*(K)$ via $\B_K^{\alpha\pm\eps}$. It is easy to see that the constraint $\alpha_*(K) > 1/2$ is saturated for $K=\{x,y\}$ in the limit when $|x-y| \to \infty$, see \eqref{e:cap-2point} below. In particular for all $\alpha>1/2$ there exists $K\subset\subset\Z^d$ with $\alpha>\alpha_*(K)>1/2$, for which $\B_K^{\alpha}$ is a good approximation of $\L^{\alpha}$ by \eqref{eq:phasetransitionK}. On the other hand if $\alpha_*(K)\leq 1/2$, then for all $x,y\in{\mathbf{T}}$ we have $\alpha_*(\{x,y\})>\alpha_*(K)$ and so the probability that $x\in{\B_K^{\alpha}}$, see \eqref{eq:defBFK}, will typically be much larger than the probability that $x\in{\L^{\alpha}}$, which explains the last part of \eqref{eq:phasetransitionK}; cf.~also \eqref{eq:alpha*K<1/2}. An important question is thus to determine which kind of sets $K$ are in
\begin{equation}
\label{eq:defAzd}
\cA_{\mathbf{T}}=\{K\subset \mathbf{T}: K\neq \emptyset, \,\alpha_*(K) \geq \textstyle\frac12\}, % \cp(K)\leq 2/g(0)\},
\end{equation}
%\textcolor{orange}{peut-etre mieux d'ancrer $\cA_{\Z^d}$, i.e.~$0 \in K$}
for these are exactly the sets one needs to consider in Theorem~\ref{cor:phasetransition2-intro} to approximate $\mathcal{L}^{\alpha}$ for $\alpha>\frac12$.  By \eqref{e:cap-2point} below, $\cA_{\mathbf{T}}$ contains at least all singletons and two-point sets. But does $\mathcal{A}_{\mathbf{T}}$ contain larger sets, for instance containing three points, and, if so, which ones? Put differently, as $N\rightarrow\infty,$ does $\{\L^{\alpha}\cap \BB(x,R_{\mathbf{T}}),x\in{\mathbf{T}}\}$ contain sets with three points before containing all possible $2$ points sets with diameter at most $R_{\mathbf{T}}$? Viewing $\Z^2$ as $\Z^2 \times \{ 0\}^{d-2} \subset \Z^d$, let
\begin{equation}
\label{e:3points}
K_1=\{(0,0), (0,1), (0,2)\}, \ K_2=\{(0,0), (0,1), (1,0)\}.
\end{equation}
We denote by $\mathbf{K_1}$ and $\mathbf{K_2}$ the projections of $K_1$ and $K_2$ onto $\mathbf{T}$ (and generally bold sets will always be the projection on $\mathbf{T}$ of sets on $\Z^d$),  which are the only connected sets of cardinality $3,$  up to torus isomorphisms.  We prove in Appendix~\ref{app:B} that
\begin{equation}
\label{eq:enoughtodetermineA1}
\text{ if }\mathbf{K_1},\mathbf{K_2}\notin{\mathcal{A}_{\mathbf{T}}},\text{ then }\mathcal{A}_{\mathbf{T}}\text{ consists only of sets with cardinality at most }2,
\end{equation}
and there are sets $A_i\subset\Z^d,$ $i\in{\{1,\dots,8\}}$ with cardinality at least three such that, denoting by $\mathbf{A_i}$ the projection of $A_i$ on $\mathbf{T}$
\begin{equation}
\label{eq:enoughtodetermineA2}
\begin{gathered}
\text{ if }\mathbf{K_1},\mathbf{K_2}\in{\mathcal{A}_{\mathbf{T}}}\text{ and }\mathbf{A_i}\notin{\mathcal{A}_{\mathbf{T}}}\text{ for all }i\in{\{1,\dots,8\}},\text{ then }\mathcal{A}_{\mathbf{T}}\text{ consists only of sets}
\\\text{with cardinality at most }2\text{ and images of }\mathbf{K_1}\text{ and }\mathbf{K_2}\text{ by torus isomorphisms}.
\end{gathered}
\end{equation}
We then deduce in Theorem~\ref{thm:admissible} that \eqref{eq:enoughtodetermineA1} applies for all $d \geq 4$ and \eqref{eq:enoughtodetermineA2} when $d=3$. The proof of 
Theorem~\ref{thm:admissible} relies on computer-assisted methods partially inspired by \cite[Appendix~B]{MR1174248}, to determine the capacities of the sets $K_1,K_2$ and $A_i,$ $i\in{\{1,\dots,8\}}$ (the latter only when $d=3$) as well as $g(0),$ cf.~\eqref{eq:defalpha*K}, which are computed with precision $10^{-30}$, see Lemma~\ref{lem:computerassisted}.

Together, Theorems~\ref{cor:phasetransition2-intro} and~\ref{thm:admissible} readily yield the following:

\begin{Cor} \label{c:trichotomy}
For small enough $\eta>0$, there exists $D= D(\eta, d)< \infty$ such that, if $\mathcal{D}^{\alpha} \subset \mathbf{T}$, $\alpha = \frac12 + \eta$, is obtained as the union of all sets $K\subset \mathbf{T}$ which are either 
\begin{itemize}[noitemsep]
\item[i)] singletons, 
\item[ii)] pairs of points at $\ell^1$-distance $ \leq D$, %or 
\item[iii)] images of $\mathbf{K_1}$ or $\mathbf{K_2}$ by torus isomorphisms, that is connected sets of cardinality $3$,
\end{itemize}
each sampled independently with probability $p^{\alpha}(K)$, then there exists $\mathbf{Q}= \mathbf{Q}_\eta$ such that if $d=3$
\begin{equation}\label{e:coupling1/2+}
\mathbf{Q}\big(\mathcal{D}^{\frac{1+3\eta}{2}} \subset \mathcal{L}^{\frac12 + \eta} \subset \mathcal{D}^{\frac{1+\eta}{2}}\big) \to 1 \text{ as } N \to \infty
\end{equation}
Moreover if $d=3,$ there is no such $\mathbf{Q}$ if iii) is omitted, whereas if $d \geq 4$, all conclusions remains true upon discarding type $iii)$ from the construction of $\mathcal{D}^{\alpha}$.
\end{Cor}
We conclude by explaining the significance of the value $\alpha=\frac12$ in \eqref{eq:defAzd},  Theorem~\ref{cor:phasetransition2-intro} and Corollary~\ref{c:trichotomy}. 
The parameter $\alpha=\frac12$ has already been identified as the critical parameter for another question in \cite{MP_unif}, and on the torus the main result from \cite{MP_unif} can actually be deduced from Theorem~\ref{cor:phasetransition2-intro}, see Remark~\ref{rk:thmsprinkling},\ref{rk:MP}. For each $\alpha>\frac12,$ one can find a constant $C=C(\alpha)$ such that with high probability, each $x,y\in{\L^{\alpha}}$ verifies either $d(x,y)\leq C$ or $d(x,y)\geq (\log N)^{1/(d-2)}$ (in fact even $d(x,y)\geq N^{2\alpha-1-\eta}$ for some $\eta>0$), see Lemma \ref{lem:wellseparated}. In other words, any value of $\alpha>\frac12$ induces a natural localization scale for $\L^{\alpha}$, which is a union of small sets (with diameter at most $C$) far away from one another. This well-separatedness fails when $\alpha \leq \frac12$, and one cannot straightforwardly deduce an approximative version of the law of $\L^{\alpha}$ in this regime, see Remark \ref{rk:final},\ref{rk:alpha<1/2}. It would of course be interesting to assess whether a result similar to Theorem~\ref{The:alpha>1/2}, which is a consequence of our localization statement Theorem~\ref{The:shortrangeapprointro-new}, valid throughout the phase $\alpha > \frac12$, and one of the driving force behind our main results, still holds for $\alpha\leq \frac12$ or not.

\subsection{Organization of this article} Section \ref{sec:notation} introduces the setup and a minimal amount of useful notation. Section \ref{sec:chen-stein} is centred around the modified Chen-Stein method, which allows for a sprinkling.  Section \ref{sec:softlocaltimes} introduces the soft local time method to couple two different Markov chains from \cite{softlocaltimes}, and explains how to invert it. This method is then applied in Section \ref{sec:shortrangeRWRI} to the excursions of the random walk on $\mathbf{T}$ and random interlacements on $\Z^d$ to prove the localization result, Theorem \ref{The:shortrangeapprointro-new}. Its strengthened version, Theorem~\ref{thm:rwshortrange}, also immediately yields a state-of-the-art coupling between random walk and random interlacements with quantitative coupling error, stated in Corollary~\ref{cor:coupling}. Section~\ref{sec:loc-cons} discusses the consequences of localization for late points. In \S\ref{subsec:L_F}, we introduce the set of late points $\L^{\alpha}_{F}$ to be studied (see \eqref{defL}), which contains $\L^{\alpha}$ defined in \eqref{e:L^alpha-RW} as a special case, and gather its first properties, including precise estimates on the probability for a generic set to be late. In \S\ref{subsec:btilde-compa}, we then compare $\L^{\alpha}_{F}$ with a suitably localized version $\tilde{\mathcal{B}}^{\alpha}_{F}$ defined in \eqref{eq:hatdefBalphaF}. The main result is Theorem~\ref{The:alpha>1/2}.

The pieces are put together in Section \ref{sec:denouement}, where we derive Theorems \ref{thm:uncoveredset} and \ref{cor:phasetransition2-intro}. Remark~\ref{rk:thmsprinkling}, which appears at the end of \S\ref{sec:denouement}, deserves highlighting. It concerns various extensions of these results and also explains how to recover existing ones. Thereafter, \S\ref{sec:extensions} contains some interesting further results which can be derived using our methods. More precisely, \S\ref{subsec:alpha-process} deals with $\L^{\alpha}_{F}$ as a process in $\alpha$ (see~\eqref{eq:alphahatalpha-intro} and Theorem~\ref{thm:processusabovealpha*}), and  \S\ref{subsec:outlook} is an outlook to the regime $\alpha \leq \frac12$ containing a description of the law of large clusters, see Theorem~\ref{The:gen}. Finally, Appendix~\ref{sec:app} contains the proof of two technical ingredients, Lemmas~\ref{lem:concsoftrwri} and~\ref{lem:conc}, relegated from Section \ref{sec:shortrangeRWRI}, which concern certain large deviation estimates for the excursions of random walk and random interlacements. Appendix~\ref{app:B} revolves around the notion of admissible sets. In particular, Theorem~\ref{thm:admissible} identifies the elements of $\A_{\mathbf{T}}$ in \eqref{eq:defAzd}, which leads to Corollary~\ref{c:trichotomy}. The elements
of the proof of Theorem~\ref{thm:admissible} relying on computer assistance are all summarized in Lemma~\ref{lem:computerassisted}. We stress however that -- with the exception of Corollary~\ref{c:trichotomy} -- none of our results make use of Appendix~\ref{app:B}.

We conclude with our convention regarding constants. In the rest of
this article, we denote by $c,$ $c',\dots$  and $C,$ $C',\dots$ positive and finite constants changing from place to place. All constants may depend implicitly on the dimension $d,$ and their dependence on any other quantity will be made explicit.

\medskip

{\bf Acknowledgment:}
AP has been supported by the Engineering and Physical Sciences Research Council (EPSRC) grant EP/R022615/1, Isaac Newton Trust (INT) grant G101121, European Research Council (ERC) starting grant 804166 (SPRS), and the Swiss NSF.

\section{Notation and preliminaries}
\label{sec:notation}

We set up some notation that will be used throughout. We use $\pi: \Z^d \to \mathbf{T}= (\Z/N\Z)^d$ to denote the canonical projection from the infinite $d$-dimensional cubic latice $\Z^d$ to the $d$-dimensional torus $\mathbf{T}$ of size length $N$, for $d \geq 3$, $N \geq1.$ In order to simplify notation, we will often use the notation $\boldsymbol{x}$ to denote the projection $\pi(x)$ of $x\in{\Z^d}.$  We denote by $0$ the origin of $\Z^d,$ and thus call $\boldsymbol{0}:=\pi(0)$ the origin of $\mathbf{T}$.  For each $x\in{\Z^d}$ and $r>0,$ we let  $Q(x,r)=Q_r(x)=x+([-\lfloor(r-1)/2\rfloor,\lceil(r-1)/2\rceil] \cap \Z)^d$ the box of size length $r$ around $x$ in $\Z^d$, and we let $Q(\boldsymbol{x},r)=\pi\big(Q_r({x})\big),$ which only depends on $\boldsymbol{x}$ and $r$.  Our convention for boxes is tuned so that $Q_N(\mathbf{0})=\mathbf{T},$ and $\pi_{|Q_N(0)}$ is a bijection into $\mathbf{T}.$ Note that the definition of $Q_r(x)$ depends on whether $x\in{\Z^d}$ or $x\in{\mathbf{T}}.$ We will also introduce the notation $Q_N$ (see above \eqref{e:u_F}), which will only be used in Sections~\ref{sec:loc-cons}-\ref{sec:extensions} ($Q_N$ can correspond to either $Q_N(\mathbf{0})$ or $Q_N(0)$, depending on the model considered, which will allow for a uniform presentation). We also let $\BB(A,R)$ be the union of the balls $\BB(x,R)$ over all $x\in{A},$ for any $A\subset\Z^d$ or $A\subset \mathbf{T}.$

For a set $A\subset \Z^d$ or $A\subset{\bf T}$ we write $\partial A$ for the internal vertex boundary of $A$, i.e.~$\partial A= \{y\in A: \exists \ x\notin A \text{ adjacent to } y\},$ $|A|$ for the cardinality of $A$ and $\delta(A)$ for the ($\ell^{\infty}$-)diameter of $A,$ that is the smallest $R\in{\N}$ such that $A\subset Q(x_A,R)$ for some vertex $x_A\in{\Z^d}$ or $x_A\in{\mathbf{T}}.$ We use the notation $A\subset\subset\Z^d$ to say that $A$ is a finite subset of $\Z^d.$ Let also $d(x,y)=\delta(\{x,y\})$ for all $x,y\in{\Z^d}$ or $x,y\in{\mathbf{T}},$ which on $\Z^d$ corresponds to the $\ell^{\infty}$-distance between $x$ and $y$. For subsets $K,K'$ either of $\Z^d$ or $\mathbf{T}$ we write $d(K,K')=\inf_{x\in{K},y\in{K'}}d(x,x').$ 

 We write $P_x$ for the canonical law of the discrete-time simple random walk on $\Z^d$ starting at $x \in \mathbb{Z}^d$. We denote by $X=(X_n)_{n \geq 0}$ the corresponding canonical process, and by $\theta_n$, $n \geq 0$, the canonical shifts, so that $X \circ \theta_n = (X_{m+n})_{m \geq 0}$. For $K \subset \Z^d$ we define $H_K= \inf \{ n \geq 0 : X_n \in K \}$ and $T_{K}=H_{\Z^d \setminus K}$ the respective entrance time in $K$ and exit time from $K$, for $K \subset \Z^d$. We write $\widetilde{H}_K=\inf \{ n \geq 1 : X_n \in K \}$ for the hitting time of $K$. Here, we take the convention $\inf\varnothing=+\infty.$ For $n\geq 0$ we also introduce the time of last visit to $K$ (before time $n$), 
\begin{equation}
\label{eq:defLKt}
L_K (n)=\sup\{m\leq n: X_m\in K\}, \, L_K= \lim_{n\to \infty}  L_K(n)
\end{equation}
with the convention $\sup\varnothing=-\infty.$ 
For ${x}\in \mathbf{T}$ and $y\in{\pi^{-1}(\{x\})}$, we set $\mathbf{P}_{{x}}= \pi \circ P_y,$ which is well-defined (i.e.~does not depend on the choice of $y$).  With a slight abuse of notation, we also write $(X_n)_{n \geq 0}$ for the canonical process under $\mathbf{P}_{{x}},$ $x\in{\mathbf{T}}$, which is the simple random walk on $\mathbf{T}$, starting at ${x}$, and extend the definition of $H_K,$ $T_K,$ $\tilde{H}_K,$ $L_K(n)$ and $L_K$ under $\mathbf{P}_x.$ Let also $\mathbf{P}\stackrel{\text{def.}}{=} N^{-d} \sum_{{x} \in \mathbf{T}} \mathbf{P}_{{x}}.$ Under $\mathbf{P}_{{x}} $ (or $\mathbf{P}$), one defines for all $t,u>0$ the (discrete) local times 
\begin{equation}
\label{eq:deflocaltimes}
	{\ell}_{{x}}(t) = \sum_{0 \leq n \leq \lfloor t\rfloor}1\{ X_n={x}\}, \, t\in [0,\infty), \quad {\ell}_{{x},u}= {\ell}_{{x}}(uN^d).
\end{equation}

We denote by $g(x,y)$, $x,y \in \mathbb{Z}^d$ the Green's function of $X$ under $P_x$, which is known to be finite and symmetric. Moreover $g(x,y)=g(x-y,0)\equiv g(x-y)$ so in particular, $g(0)=g(0,0)$. For suitable $f$ on $\mathbb{Z}^d$ (e.g.~with finite support), we write $Gf(x)=\sum_{y\in \Z^d} g(x,y)f(y)$. For finite $K \subset \mathbb{Z}^d$, we denote by $e_K$ the equilibrium measure of $K$,
\begin{equation}\label{e:e_K}
e_K(x)=P_x(\widetilde{H}_K = \infty)1_{K}(x), \text{ for $x \in \mathbb{Z}^d$,} 
\end{equation}
which is finite and supported on $\partial K$, and by $\textnormal{cap}(K)$ its total mass, the capacity of $K$, which is monotone in $K$. We write $\overline{e}_{K}=\frac{e_K}{\text{cap}(K)}$ for the normalized equilibrium measure. If $K\subset \mathbf{T}$ with $\delta(K)<N,$ we also define the capacity of $K$ as follows: if $K\subset Q_{N-1}(\boldsymbol{0})$ and $K'\subset Q_{N-1}(0)(\subset\Z^d)$ is a set such that $\pi(K')=K,$ then we let $\mathrm{cap}(K)=\mathrm{cap}(K').$ For other $K$ with $\delta(K)<N,$ we define $\mathrm{cap}(K)$ by translation invariance. Moreover, we take $\mathrm{cap}(K)=\infty$ for the subsets of $\mathbf{T}$ which are not topologically trivial, i.e.\ when $\delta(K)=N$.  One knows that 
\begin{equation}
\label{eq:lastexit} Ge_K=h_K, \quad \text{ where $h_K(x)=P_x(H_K< \infty)$, $x\in \Z^d$ is the equilibrium~potential of $K$}.
\end{equation}
The following straightforward (strict) monotonicity result for $g$ will be repeatedly used. Let $|\cdot|_1$ denote the $\ell^1$-distance (i.e.~the graph distance) on $\Z^d$.
\begin{lemma} \label{lem:boundongreenfunction} For all integers $n \geq 0$,
\begin{equation}
\label{e:g-monot}\sup_{|x|_1 > n} g(x) < \sup_{|x|_1 = n} g(x).
\end{equation}
\end{lemma}
\begin{proof} Set $B_1(x,n)=\{y\in{\Z^d}: |{y-x}|_1\leq n\}$. First, one classically knows that there exists $c(n)>0$ such that $P_x(H_{B_1(0,n)} = \infty) \geq c(n)$ for all $x \notin B_1(0,n)$. Hence, by the strong Markov property and translation invariance, one obtains, for all $x\in{\Z^d}$ with $|{x}|_1> n$,
\begin{multline*}
	\frac{g(0,x)}{g(0)}\stackrel{\eqref{eq:lastexit}}{=} P_0(H_x<\infty) =E_0\big[1\{H_{B_1(x,n)}<\infty\} P_{X_{H_{B_1(x,n)}}}(H_x<\infty) \big]\\
	 \leq (1-c(n)) \sup_{y\in \partial B_1(x,n)} P_y(H_x<\infty) %= (1-c(n))\sup_{|{y}|_1=n} P_y({H_0<\infty})
	  \stackrel{\eqref{eq:lastexit}}{=} (1-c(n)) \sup_{|y|_1=n} \frac{g(0,y)}{g(0)},
\end{multline*}
from which \eqref{e:g-monot} follows.
\end{proof}

By evaluating \eqref{eq:lastexit} for $x \in K$ and solving the resulting linear system, one can explicitly determine $e_K$ and $\text{cap}(K)$ in terms of $g$. In case $K$ is a singleton or a two-point set, this gives
\begin{align}
&\text{cap}(\{ x,y\})= \frac{2}{g(0)+ g(x-y)}, \text{ for all $x, y \in \Z^d$} \label{e:cap-2point}
\end{align}
(in particular, $\text{cap}(\{ x\})= g(0)^{-1}$). For $K\subseteq U\subset\Z^d$ we further set $e_K^U(x)= P_x(\widetilde{H}_K > T_U)1_{K}(x)$, so that \eqref{e:e_K} corresponds to choosing $U=\Z^d$, and write $\text{cap}_U(K)= \sum_x e_K^U(x)$ for its total mass. We now collect some further facts about capacity, which will be useful in the rest of the article. First, as we now explain, combining \eqref{e:cap-2point} and Lemma~\ref{lem:boundongreenfunction} one obtains that 
\begin{equation}\label{eq:cap-distance2}
\text{cap}(K) > \text{cap}(\{ x,y\}) \text{ for any $K \subset \Z^d$ with $\text{diam}_{\ell^1}(K ) \geq 2$ and } x\sim y.
\end{equation}
Indeed by monotonicity of $K \mapsto \text{cap}(K)$, see \cite[Prop.~2.2.1]{Law91}, it is enough to consider the case $|K|=2$ with two-points at distance $n \geq 2$, from which \eqref{eq:cap-distance2} follows using the formula \eqref{e:cap-2point}, translation invariance, and \eqref{e:g-monot} for $n \geq 2$.
Next, by \cite[Lemma~1.11]{MR3308116} and \cite[Proposition~1.2]{ASSZd}, for any sets $K,K'\subset\subset\Z^d$, one has 
\begin{equation}
\label{eq:capdecoupling}
0\leq\mathrm{cap}({K})+\mathrm{cap}(K')-\mathrm{cap}(K\cup K')\leq c|K'||{K}|d(K,K')^{-(d-2)},
\end{equation}
for some constant $c$ depending only on $d.$ An important application of the right-hand side of \eqref{eq:capdecoupling} is that if $F\subset\subset\Z^d$ and  $K,K'\subset F$ are such that $d(K,K')\geq \log(|F|)^{1/(d-2)}/2,$ then for all $\alpha\in{(0,1]}$
\begin{equation}
\label{eq:decouplatepoints}
|F|^{-\alpha g(0)\mathrm{cap}(K\cup K')}\leq C|F|^{-\alpha g(0)\mathrm{cap}(K)}\cdot|F|^{-\alpha g(0)\mathrm{cap}(K')},
\end{equation} 
for some constant $C$ depending only on $d,$ $|K|$ and $|K'|.$ The following improvement of \eqref{eq:decouplatepoints} will also be used. If, as will commonly occur in practice, $|F| \to \infty$ and $ \frac{d(K,K')}{\log(|F|)^{1/(d-2)}} \to \infty$ whereas $|K|, |K'| \leq C'$, one can further replace the constant $C$ appearing in \eqref{eq:decouplatepoints} by $(1+ o(1))$ (as $|F| \to \infty$). Note that the reverse inequality in \eqref{eq:decouplatepoints} is also true with $C=1$ by \eqref{eq:capdecoupling}. Another interesting application of \eqref{eq:capdecoupling} is the following: for each $r>0,$
\begin{equation}
\label{eq:easy}
	\text{there exists $C=C(r)<\infty$ so that each $K\subset\subset \Z^d$ with $|K|\geq C$ satisfies $\cpc{}{K}\geq r$}.
\end{equation}
Indeed, in order to prove \eqref{eq:easy}, it suffices to prove that if $K$ contains $\lceil r/\mathrm{cap}(\{0\})\rceil+1$ points $x_i$ so that $d(x_i,x_j)$ is large enough for each $i\neq j$, then $\cpc{}{K}\geq r,$ which follows easily from~\eqref{eq:capdecoupling}. Finally note that \eqref{eq:cap-distance2}, \eqref{eq:capdecoupling}, \eqref{eq:decouplatepoints} and \eqref{eq:easy} still hold when $K,K'\subset\mathbf{T}$ as long as the capacities in question are not infinite, which is for instance the case when $\delta(K),\delta(K')<N/2.$

We now briefly introduce the random interlacement process and its associated local times to the extent we need them. We denote by $\omega$ the random interlacements process on $\Z^d$ with law $\PI,$ as defined in \cite{SZNIT_interlacements}, which is a Poisson process of bi-infinite random walk trajectories with positive labels. We will actually never need to describe the full law of $\omega$ here, but only its push-forward $\omega_{B}^u$ to the trajectories with label less than $u$ which hit a ball $B \subset\Z^d,$ and started after their first hitting time of $B,$  for some $u>0,$ that we now describe.  For each finite set $B\subset\Z^d,$ there exists a Poisson process $N_{B}=(N^u_{B})_{u\geq0}$ on $[0,\infty)$ with intensity $\cp(B)$ and an independent i.i.d.\ sequence of random walks $(X^i)_{i\geq1},$ each with law $P_{\overline{e}_{B}},$ such that for all $u>0$, under $\PI,$
\begin{equation}
\label{eq:definterprocess}
    \omega_{B}^u = \sum_{i=1}^{N_{B}^u}\delta_{X^i}.
\end{equation}
This description of $\omega_{B}^u$ via $N_{B}^u$ and $(X^i)_{i\geq1}$ entirely characterises its law. Note that the trajectories $(X^i)_{i\geq1}$ depend on the choice of $B,$ but this dependence does not appear in the notation for simplicity. We write $\I_{B}^u$ for the interlacements set in $B$ and $\ell_{y,u}$ for the associated local time at $y\in{B}$ at level $u$, i.e.
\begin{equation}\label{eq:deflocaltimes-ri}
  \ell_{y,u}= \ell_{y,u}(\omega)=\sum_{i=1}^{N_{B}^u} \sum_{k=0}^{\infty} 1\{X_k^i=y\},
\end{equation}
and
\begin{equation}
\label{eq:interlocal}
\I_{B}^u= \I^u \cap B = \bigcup_{i=1}^{N_{B}^u}X^i[0,\infty)\cap B \quad (= \{ y \in B: \ell_{y,u}>0 \})
\end{equation}
where $X^i[0,\infty)=\{X^i_k,\,k\in\N\}.$ The vacant set at level $u$ is defined as $\V^u= \Z^d \setminus \I^u$ and the formula \eqref{e:V^u-def} follows as $\PI(\V^u \subset B)=\PI(N_B^u=0)$.

\section{Modified Chen-Stein method with sprinkling}
\label{sec:chen-stein}

We now collect a result, of independent interest, which roughly speaking generates a coupling of some (Bernoulli) process $Y$ of interest --possibly highly correlated-- and an independent process $W$ from a given coupling between $Y$ with another process $Z$ in such a way that $(Y,W)$ are `close' whenever $(Y,Z)$ are. Here, `distance' will be quantified by the functional $d_{\varepsilon}$ introduced below, which allows for a sprinkling with parameter $\varepsilon>0$ in the underlying density of the processes. The main result appears in Lemma~\ref{lem:corofchenstein}.  For the applications in this article, $Y$ will be the occupation time field of RW/RI, $Z$ will be a suitable finite-range approximation, in a sense to be made precise, see \eqref{eq:Z}, and $W$ the target i.i.d.~field. We compare the philosophy underlying Lemma~\ref{lem:corofchenstein} with the traditional Chen-Stein method further in Remark~\ref{R:CS}.

In the sequel, given some probability space $(\Omega, \mathcal{A}, \P)$ and a finite set $S$ (typically a finite collection of subsets of $ \Z^d$ or $\Z_N^d$), we call \textit{Bernoulli process on $S$} any random variable $Z=(Z_x)_{x\in S}: \Omega \to \{0,1\}^S$. The process $Z$ will be referred to as \textit{independent} whenever $\{Z_x : x\in S\}$ constitutes an independent family of random variables.
Given $I \subset [0,\infty)$ an interval and two families $Z= (Z^\alpha)_{\alpha \in I}$, $Y= (Y^\alpha)_{\alpha \in I}$ of Bernoulli processes on a countable set $S$ (i.e.~for every $\alpha \in I$, $Y^\alpha,Z^\alpha$ are Bernoulli processes on $S$), we define, for every $\alpha, \varepsilon \geq 0$ such that $\alpha \pm \varepsilon \in I$,
\begin{equation}
\label{eq:defdepsilon}
d_{\varepsilon}(Y,Z;\alpha)=\inf\bigg\{\delta\in{[0,1]}\,:\begin{array}{l}
\exists\text{ a coupling $\hat{\P}$ between $Y^\alpha$ and $(Z^{\alpha-\epsilon}, Z^{\alpha+\epsilon})$} \\
\text{ s.t. } \hat{\P}\big({Z}^{\alpha+\eps}_x\leq {Y}_x^{\alpha}\leq {Z}_x^{\alpha-\eps}\,\forall x\in{S}\big)\geq 1-\delta
\end{array}
\bigg\}.
\end{equation}

Note that $d_{0}(Y,Z;\alpha)= d_{\textnormal{TV}}(Y^{\alpha},Z^{\alpha})$. This is consistent with the notation from \eqref{eq:distancesprinkling} by setting $d_{\varepsilon}(Y,Z;\alpha)=d_{\varepsilon}(\mathcal{S}_Y,\mathcal{S}_Z;\alpha)$, i.e.~identifying $Y= (Y^\alpha)_{\alpha \in I}$ with the corresponding family of occupations sets $\mathcal{S}_{Y}= (\mathcal{S}_{Y}^{\alpha})_{\alpha \in I}$, where $\mathcal{S}_{Y}^{\alpha} =\{x:Y_{x}^{\alpha}=1\}$ and similarly for $Z$. We will use $d_{\varepsilon}$ to measure proximity between $Y$ and $Z$.

The following setup is tailored to our purposes. The process $Y=(Y^{\alpha})_{\alpha \in I}$ is called a \textit{decreasing} family of Bernoulli processes on $S$ if $Y^{\alpha}$, $\alpha \in I$, are defined on a joint probability space and $Y_x^{\beta}\leq Y_x^{\alpha}$ a.s.~for all $x\in{S}$ and $\alpha, \beta \in I$ with $\beta\geq\alpha$. We consider two families $Y=(Y^{\alpha})_{\alpha \in I}$ and $W=(W^{\alpha})_{\alpha \in I}$ with the following properties:
\begin{align}
&\begin{array}{l}\text{$Y$ is a decreasing family of Bernoulli processes on $S$;}\end{array}\label{eq:Y}\\[0.4em]
&\begin{array}{l} 
\text{for each $\alpha \in I$, $W^{\alpha}$ is an independent Bernoulli process on $S$,}\\
\text{decreasing in $\alpha$, and $\P(W^\alpha_x=1)= \P(Y^\alpha_x=1)$ for all $\alpha \in I$, $x \in S$.} \end{array} \label{eq:W}
\end{align}
The goal will typically be to couple $Y$ and $W$ in a manner keeping $d_{\varepsilon}(Y,W;\cdot)$ small.
In practice, we will first approximate $Y$ by a family $Z=(Z^{\alpha})_{\alpha \in I}$ of finite-range processes. Note that our notion of finite range processes, see \eqref{eq:Z}, is slightly different from the usual notion of finite range, and is adapted to our context. The following lemma then allows to `lift' couplings between $(Y,Z)$ controlling the quantity $d_{\varepsilon}(Y,Z;\cdot)$, to couplings with similar properties between $(Y,W)$. %In practice, $Y$ will be given and a coupling $(Y,Z)$ with $Z$ having the properties listed in \eqref{eq:Z} needs to be carefully constructed on a case-by-case basis. Then the following lemma effectively allows to replace $Y$ by $Z$.

\begin{lemma}\label{lem:corofchenstein} Let $Y,W$ satisfy \eqref{eq:Y}-\eqref{eq:W}. If
\begin{equation}
\begin{split}
&\textnormal{$(Z^{\alpha})_{\alpha \in I}$ is a family of Bernoulli processes on $S$ such that}\\[0.4em]
&\begin{array}{l}
%\text{$Z$ is a decreasing family of Bernoulli processes on $K$ and for every $x\in K$, $\alpha \in I$,}\\[0.2em]
\quad \textnormal{(Monotonicity): $Z=(Z^{\alpha})_{\alpha \in I}$ is a decreasing family, and}\\[0.2em]
\quad \textnormal{(Finite range): for every $x\in S$, there exists $\cN_x\subset S$ such that for $\alpha, \beta \in I$, $\alpha \leq \beta$,} \\ 
\qquad\qquad\qquad\quad \  \,\text{$Z^{\beta}_x$ and $Z^{\alpha}_x- Z^{\beta}_x$ are each independent of $\{Z^{\alpha}_y, Z^{\beta}_y : y\notin \cN_x\}$,}
\end{array}\label{eq:Z}
\end{split}
\end{equation}
then for all $\epsilon>0$ and $\alpha\in{I}$ such that $\alpha\pm 3\eps\in{I}$ one has 
		\begin{equation}
\label{eq:finalboundChenStein}
d_{2\eps}(Y,W;\alpha)\leq 400\sup_{\alpha'\in{\{\alpha-2\eps,\alpha\}}}\left\{b_1(\alpha')+b_2(\alpha')+d_{\varepsilon}(Y,Z;\alpha')\cdot|S|^2\right\},
	\end{equation}
	where  
	\begin{align}
	&b_1(\alpha) \stackrel{\textnormal{def.}}{=} \sum_{x\in S} \sum_{y\in \cN_x} \P(Y_x^\alpha=1)\P(Y_y^\alpha=1) \label{eq:b_1},\\
	&b_2(\alpha) \stackrel{\textnormal{def.}}{=} \sum_{x\in S}\sum_{y\in \cN_x\setminus\{x\}} \P(Y_x^\alpha=1, Y_y^\alpha=1).\label{eq:b_2}
	\end{align}
	\end{lemma}
We comment further on the utility of \eqref{eq:finalboundChenStein} in Remark~\ref{R:CS} at the end of the proof.
\begin{proof}
Let $\delta> 0$ and $\P^U$ denote a probability measure carrying a family $U_x$, $x \in S$, of i.i.d.~uniform random variables on $[0,1]$. With a slight abuse of notation, we realize $W$ on the space $\P^U$ along with an auxiliary process $\til{W}$ by setting
\begin{equation}
\label{eq:Wtilde}
{W}^\alpha_x=1\{U_x\leq\P(Y_x^\alpha=1)\}, \quad \til{W}^\alpha_x=1\{U_x\leq\P(Z_x^\alpha=1)\}, \quad x \in S, \, \alpha \in I.
\end{equation}
Clearly, \eqref{eq:W} holds and \eqref{eq:Wtilde} defines a coupling of $(W,\til{W})$. We will first work with $\til{W}$, which matches the one-point densities of $Z$, and return to $W$ towards the end of the proof. Henceforth, fix $\alpha \in I$ and $\varepsilon> 0$ such that $\alpha\pm 3\eps\in{I}$ and let $\gamma(\alpha, \varepsilon)$
refer to the supremum on the right-hand side of \eqref{eq:finalboundChenStein}. Observe that $Z^{\alpha-\varepsilon}-Z^{\alpha+ \varepsilon}$ is a (well-defined) Bernoulli process on $S$ by monotonicity of $Z$ and the same holds true for $\til{W}^{\alpha-\varepsilon}-\til{W}^{\alpha+ \varepsilon}$ on account of \eqref{eq:Wtilde}.
We will first construct a coupling $\widetilde{\mathbb Q}$ between $(Z^{\alpha-\eps},Z^{\alpha+\eps})$ and $(\tilde{W}^{\alpha-\eps},\til{W}^{\alpha+\eps})$ with the property that 
\begin{align}\label{eq:Q}
\widetilde{\mathbb Q}\big(Z^{\alpha-\eps} = \til{W}^{\alpha-\eps},Z^{\alpha+\eps} = \til{W}^{\alpha+\eps}\big)\geq &1 - 384 \gamma(\alpha, \varepsilon) - \delta.
\end{align}
To this end, let $S' = S \times \{-,+\}$ and consider the Bernoulli processes $Z'$ and $\tilde{W}'$ on $S'$ defined as
\begin{equation}
\label{eq:defZ'W'}
Z'_{(x,\sigma)} =
\begin{cases}
Z^{\alpha-\eps}_x-Z^{\alpha+\eps}_x, & \text{ if } \sigma=-\\
Z^{\alpha+\eps}_x, & \text{ if } \sigma=+
\end{cases}, 
\qquad
\tilde{W}'_{(x,\sigma)} =
\begin{cases}
\tilde{W}^{\alpha-\eps}_x-\tilde{W}^{\alpha+\eps}_x, & \text{ if } \sigma=-\\
\tilde{W}^{\alpha+\eps}_x, & \text{ if } \sigma=+
\end{cases}.
\end{equation}
Due to \eqref{eq:Wtilde}, $Z'$ and $\tilde{W}'$ have the same one-dimensional marginals, i.e.~$Z'_{(x,\sigma)}\stackrel{\text{law}}=\tilde{W}_{(x,\sigma)}'$ for all $(x,\sigma) \in S'$. For any point $x' = (x,\sigma) \in S'$, define its neighborhood $\mathcal{N}_{x'}= \mathcal{N}_{x} \times \{\pm \}$, with $\mathcal{N}_x$ as given by \eqref{eq:Z}. We want to bound the total variation distance between $Z'$ and $\til{W}'$. Because $\til{W}'$ is not an independent Bernoulli process, we let $W''$ be an independent Bernoulli process with the same one dimensional marginals as $Z'$. Using the triangle inequality for total variation distance and as we explain below, applying~\cite[Theorem~3]{Arratia} (see also~\cite{Chen, Stein}), one obtains that 
\begin{equation}
\label{eq:CS}
d_{\rm{TV}}(Z', \tilde{W}') \leq d_{\rm{TV}}(Z', {W}'') + d_{\rm{TV}}(W'', \tilde{W}')  \leq 24(b_1' + b_2'),
\end{equation}
where
\begin{equation}
\label{eq:defb1b2}
	b_1' = \sum_{x' \in S'} \sum_{y' \in \cN_{x'}} \P(Z'_{x'}=1)  \P(Z'_{y'}=1),
\quad
	b_2' = \sum_{x' \in S'} \sum_{y' \in \cN_{x'}\setminus \{x'\}}  \P(Z'_{x'}=1, Z'_{y'}=1).
\end{equation}
%and we have abbreviated $p_{x'}= \P[Z'_{x'}=1]$, $p_{x'y'}= \P[Z'_{x'}=1, Z'_{y'}=1]$.
Indeed, the right-hand side of \eqref{eq:CS} a priori includes a third term
\begin{equation}
\label{eq:defb3}
	b_3' = \sum_{x' \in S'} \mathbb E\Big[\Big|\mathbb E\big[Z'_{x'} - \P(Z'_{x'}=1) \, \big| \, Z'_{y'}, y' \notin \cN_{x'}\big] \Big| +  \Big|\mathbb E\big[\til{W}'_{x'} - \P(\til{W}'_{x'}=1) \, \big| \, \til{W}'_{y'}, y' \notin \cN_{x'}\big] \Big| \Big],
\end{equation}
but the finite-range property in \eqref{eq:Z} and the definitions of $Z'_{x'}$, $\til{W}
'_{x'}$ and $\mathcal{N}_{x'}$ imply that $b_3'=0$ (note to this effect that the $\sigma$-algebra generated by $\{Z'_{y'} : y' \notin \cN_{x'}\}$ is the same as that generated by $\{Z^{\alpha-\eps}_y, Z^{\alpha+\eps}_y : y\notin \cN_x\}$). Note also that the $b_1$ and $b_2$ terms we obtain for $d_{\rm{TV}}(W'', \tilde{W}')$ are both smaller than $b_1'$, and hence this justifies the constant $24$ appearing in~\eqref{eq:CS}. The reason we consider the process $\tilde{W}'$ instead of only working with the independent process $W''$ here is that we want the Bernoulli process $W$ we obtain at the end to be decreasing in $\alpha$, see \eqref{eq:W} (note that we would avoid this problem if we used Poisson processes throughout as in \cite[Theorem~2]{Arratia}, instead of Bernoulli processes). 

We now proceed to bound the right-hand side of \eqref{eq:CS} in terms of $\gamma(\alpha, \varepsilon)$. By means of the defining properties of $d_{\varepsilon}$, one sees that for any $\delta' >0$ and $x \in S$,
\begin{equation}
\label{eq:CS9}
\begin{split}
&\P(Z'_{(x,+)}=1)= \P(Z^{\alpha+\eps}_x=1) \leq \P(Y^{\alpha}_x=1) + d_{\varepsilon}(Y,Z;\alpha) + \delta'\\
&\P(Z'_{(x,-)}=1) \leq \P(Z^{\alpha-\eps}_x=1) \leq \P(Y^{\alpha-2\eps}_x=1) + d_{\varepsilon}(Y,Z;\alpha-2\varepsilon) + \delta',
\end{split}
\end{equation}
Thus, letting $\delta' \to 0$, abbreviating $d_{\varepsilon}= d_{\varepsilon}(Y,Z;\alpha) \vee d_{\varepsilon}(Y,Z;\alpha-2\varepsilon) $, observing that $|S'|=2|S|$ and using monotonicity of $Y$ to have $ \P(Y^{\alpha}_x=1) \leq  \P(Y^{\alpha-2\eps}_x=1)$ for all $x\in S$, it follows from \eqref{eq:CS9}, in view of \eqref{eq:defb1b2} and \eqref{eq:b_1}, that
\begin{equation}
\label{eq:CS11}
b_1' \leq 4(b_1(\alpha-2\varepsilon)+ 3d_{\varepsilon} |S|^2).
\end{equation}
Similarly, using that $Z'_{(x,\sigma)} \leq Z_{x}^{\alpha- \varepsilon}$ for all $x \in S $ and $\sigma  \in \{ \pm\}$ by monotonicity in \eqref{eq:Z}, one obtains that whenever $x'= (x,\sigma)$ and $y' =(y,\rho)$ with  $y \neq x$,
$$\P(Z'_{(x,\sigma)}=1, Z'_{(y,\rho)}=1) \leq \P(Z_{x}^{\alpha- \varepsilon}=1, Z_{y}^{\alpha- \varepsilon}=1) \leq  d_{\varepsilon} + \P(Y_{x}^{\alpha- 2\varepsilon}=1, Y_{y}^{\alpha- 2\varepsilon}=1).$$
Observing that $\P(Z'_{(x,\sigma)}=1, Z'_{(x,-\sigma)}=1)=0$ for all $x \in S$ by monotonicity, this yields that
\begin{equation}
\label{eq:CS12}
b_2' \leq 4(b_2(\alpha-2\varepsilon)+ d_{\varepsilon} |S|^2).
\end{equation}
Substituting \eqref{eq:CS11} and \eqref{eq:CS12} into \eqref{eq:CS} readily implies that $d_{\rm{TV}}(Z', \tilde{W}') \leq 384 \gamma(\alpha, \varepsilon)$. The existence of a coupling $\tilde{\mathbb Q}$ having the property \eqref{eq:Q} immediately follows from this (in fact, one could even choose $\delta=0$ by using a maximal coupling but this won't be necessary).

With \eqref{eq:Q} at hand, we now prove \eqref{eq:finalboundChenStein}, which entails finding a coupling of $(W^{\alpha-2\varepsilon},  W^{\alpha+2\varepsilon})$ and $Y^{\alpha}$  with suitable properties. To this end, let $\widehat{\mathbb Q}$ denote a coupling of $Y^{\alpha}$ and $(Z^{\alpha-\eps},Z^{\alpha+\eps})$  satisfying 
\begin{equation}
\label{eq:Qhat}
 \hat{\mathbb Q}\big({Z}^{\alpha+\eps}_x\leq {Y}_x^{\alpha}\leq {Z}_x^{\alpha-\eps}\,\forall x\in{S}\big)\geq 1-d_\varepsilon(Y,Z; \alpha)- \delta,
\end{equation}
which exists by definition of $d_{\varepsilon}$. We proceed to define a measure $\widehat{\P}$ on $(\{0,1\}^S)^3$, with canonical coordinates $(\hat{W}^{-},\hat{Y}, \hat{W}^{+})$ as follows: for all $w^{\pm}, y \in \{0,1\}^S$,
\begin{equation*}
%\label{eq:Phat}
\begin{split}
 &\hat{\mathbb P}\big(\hat{W}^{-}=w^-, \,  \hat{Y}^{}=y, \,  \hat{W}^{+}= w^{+}\big)\stackrel{\text{def.}}{=} \\[0.3em]
 &\ \sum_{\tilde{w}^{+},\tilde{w}^-} \P^U\big({W}^{\alpha\pm2\eps}=w^{\pm} \big| \til{W}^{\alpha\pm\eps}=\tilde{w}^{\pm} \big) \sum_{{z}^{+},z^-} \tilde{\mathbb Q}\big( \til{W}^{\alpha\pm\eps}=\tilde{w}^{\pm}  \big| {Z}^{\alpha\pm\eps}={z}^{\pm}  \big)  \hat{\mathbb Q}\big({Z}^{\alpha \pm\eps}= z^{\pm}, \,  {Y}^{\alpha} =y\big),
 \end{split}
\end{equation*}
where $\{{W}^{\alpha\pm2\eps}=w^{\pm} \}$ is short for $\{ {W}^{\alpha-\eps}=w^{-}, {W}^{\alpha+\eps}=w^{+}\}$ and similarly for other events, and $\tilde{w}^{+},\tilde{w}^-, z^{+}$ and $z^-$ range over all points in $\{0,1\}^S$ such that the corresponding events appearing in each conditioning have non-zero probability. One readily checks using \eqref{eq:Wtilde} and the definition of the couplings~$\til{\mathbb{Q}}$ and~$\widehat{\mathbb{Q}}$ that
$\hat{\mathbb P}$ is a probability measure with marginals $(\hat{W}^{-},\hat{W}^{+})  \stackrel{\text{law}}{=}(W^{\alpha-2\varepsilon}, W^{\alpha+2\varepsilon})$ and $\hat{Y} \stackrel{\text{law}}{=}Y^{\alpha}$, with $Y^{\alpha}$ and $W^{\alpha\pm 2\varepsilon}$ as prescribed by \eqref{eq:Y} and \eqref{eq:W}. Finally it follows from \eqref{eq:Wtilde} combined with the bounds appearing in \eqref{eq:CS9} that ${W}_x^{\alpha-2\epsilon}< \til{W}_x^{\alpha-\epsilon}$ if and only if $ \P(Y^{\alpha-2\eps}_x=1) 
< U_x \leq \P(Z^{\alpha-\eps}_x=1)$, which has probability at most $d_{\varepsilon}$ (similar considerations apply to the event $\til{W}_x^{\alpha+\epsilon}< W_x^{\alpha+2\epsilon}$), whence
\begin{equation*}
\P^U\big(W_x^{\alpha+2\epsilon}\leq \til{W}_x^{\alpha+\epsilon}\leq\til{W}_x^{\alpha-\epsilon}\leq {W}_x^{\alpha-2\epsilon}, \forall \, x\in S\big)\geq 1 - 2|S| d_{\varepsilon},
\end{equation*}
which, together with \eqref{eq:Q} and \eqref{eq:Qhat}, implies that
$$
\hat{\P}\big(\hat{W}^{+}_x\leq \hat{Y}_x\leq \hat{W}_x^{-}\,\forall x\in{S}\big) \geq 1-d_\varepsilon(Y,Z; \alpha)- \delta - 384 \gamma(\alpha, \varepsilon) - \delta -  2 |S| d_{\varepsilon}.
$$
Since $d_\varepsilon(Y,Z; \alpha) \leq  d_{\varepsilon}\leq\gamma(\alpha,\varepsilon)$ and $\delta> 0$ was arbitrary, \eqref{eq:finalboundChenStein} follows.
\end{proof}
	\begin{Rk}\leavevmode\label{R:CS} 
	Of course, the utility of \eqref{eq:finalboundChenStein} as a means to compare $Y$ and $W$ around level $\alpha$ rests in particular on having a good bound on $\sup_{\alpha'\in{\{\alpha-2\eps,\alpha\}}} d_{\varepsilon}(Y,Z;\alpha')$ to begin with, for a process $Z$ satisfying \eqref{eq:Z}, which in the case of the random walk on the torus will be provided by Theorem~\ref{thm:rwshortrange} below. The presence of $d_{\varepsilon}(Y,Z;\alpha')$ acts as a surrogate for a certain quantity ``$b_3$'' (much like $b_3'$ above, cf.~\cite{Arratia}), which would arise when attempting to compare $Y$ and $W$ directly using the Chen-Stein method, as done for instance in \cite{SamPerla} in the present context. The issue with this is that $b_3$ typically turns out to be too large when $Y$ has long-range. In fact, limitations of the method in the presence of long-range correlations are well-known, see for instance the discussion in \cite{Arratia}, Sec.~2 ``Open problem,'' pp.12-13.
	\end{Rk}

\section{Soft local times and inverse soft local times}
\label{sec:softlocaltimes}

In this section we give a brief exposition of the method of soft local times introduced by Popov and Teixeira in \cite{softlocaltimes}. An `inversion' of this technique, introduced below, see \eqref{e:slt-Pext}-\eqref{eq:defetak+1}, leading to Proposition \ref{pro:couplingSLT}, will be used to manufacture couplings in the next section. We defer to Remark~\ref{R:inverse} for a discussion of the benefits of this construction and its interplay with the technique of~\cite{softlocaltimes}.

Consider the measure space $(\Sigma, \mu)$, where $\Sigma$ is a locally compact Polish metric space endowed with its Borel $\sigma$-algebra, carrying a (Radon) measure $\mu$. At its root, the method of \cite{softlocaltimes} is a particular way to sample sequences $Z= (Z_i)_{i\geq 0}$ of $\Sigma$-valued random variables from a Poisson process on $\Sigma \times \R_+$. Although in principle, any sequence $Z$ such that all $Z_i$'s have a density with respect to $\mu$ can be accomodated, cf.~\cite[Section~3]{AlvesPopov}, the following Markovian setup will be enough for our purposes. Note however that all results presented in this section continue to hold at this greater level of generality.

Let $Z= (Z_i)_{i\geq 0}$ be a time-inhomogeneous Markov chain on $\Sigma$. That is,  there exist transition densities $g_i:\Sigma\times \Sigma \to \R_+,$ $i\geq1,$ with respect to $\mu$ (i.e.~the functions $g_i$ are measurable and $\int g_i(x,y) \mu(\mathrm{d}y)=1$ for all $x \in \Sigma$) such that under a probability measure $P$,
\begin{equation}\label{eq:abs-cont}
P(Z_{i+1}\in{\mathrm{d}x}\,|\,Z_i) = g_{i+1}(Z_i,x) \mu(\mathrm{d}x)\text{ for all }i\geq0.
\end{equation}
%and for simplicity, we assume that $Z_0 \in \Sigma$ is deterministic. \\
Under an auxiliary probability $Q,$ let $\eta$ be a Poisson point process on $\Sigma\times \R_+$ with intensity measure $\mu\otimes \mathrm{d}v,$ where $\mathrm{d}v$ denotes Lebesgue measure on $\R_+$. Our assumptions on $(\Sigma, \mu)$ ensure that the construction of $\eta$ falls within the realm of standard theory. We assume that $Q$ carries a random variable having the same law as $Z_0$ under $P$, independent of $\eta$, which we continue to denote by $Z_0$.

 Letting $z_0=Z_0$, $v_0=0$ and $\eta_0= \eta,$ one defines recursively, for $i \geq 0$,
\begin{equation}
\label{eq:softlocalexpo}
\xi_{i+1} = \inf_{(z,v)\in \eta_i} \frac{v}{g_{i+1}(z_i,z)},
\end{equation}
where, in writing e.g.~$(z,v)\in \eta$ we tacitly identify the point measure $\eta$ with its support. Combining Propositions~4.1 and~4.10 of~\cite{softlocaltimes}, it follows that the infimum in \eqref{eq:softlocalexpo} is attained $Q$-a.s.~at a unique pair $(z_{i+1},v_{i+1})$, and defining
\begin{align}\label{eq:softlocalpoisson}
\eta_{i+1}=\sum_{(z,v)\in \eta_i \setminus\{(z_{i+1},v_{i+1})\}} \delta_{(z,v-\xi_{i+1}g_{i+1}(z_i,z))},
\end{align}
the following holds:
\begin{align}
&\label{slt-p1} \xi= (\xi_{i})_{i \geq 1} \text{ are i.i.d.~exponential random variables with parameter $1$;}\\
&\label{slt-p2} \text{for all $i \geq 1$, } (z_0,\dots,z_i) \stackrel{\text{law}}{=}(Z_0,\dots ,Z_i)  \text{ and is independent of $(\xi_1,\dots \xi_i)$;}\\[0.3em]
&\label{slt-p3} \text{for all $i \geq 1$, $\eta_{i}$ is a Poisson process of intensity $\mu\otimes\mathrm{d}v$ independent of $(\xi_{j}, z_j, v_{j})_{0 \leq j\leq i}$}%\\
%&\label{slt4}\\
\end{align}
(with $\xi_0 =0$). We refer to the sequence $z =(z_i)_{i \geq 0}$ thereby constructed %, which has the same law as $Z$ under $P$ on account of \eqref{slt-p2},
as obtained from $(\eta,Z_0,(g_i)_{i \geq 1})$ via soft local times. The associated soft local time process is defined as
\begin{equation}\label{e:slt-def}
G_0(z)=0, \quad G_{i+1}(z) = G_i(z) + \xi_{i+1}g_{i+1}(z_i,z), \, i \geq 0. 
\end{equation}
As the next proposition illustrates, one benefit of this construction is to supply a natural coupling in terms of $\eta$ of two (or more) chains having densities with respect to $\mu$. The coupling allows for a comparison between the ranges of these chains, which is controlled in terms of the scalar fields $G_i(\cdot)$. To wit, let $\til{Z}$ be another Markov chain on $\Sigma$ having transition densities $(\til{g}_i)_{i\geq1}$ with respect to the same measure $\mu$, cf.~\eqref{eq:abs-cont}. %, and $\til{Z}_0 \in \Sigma$ deterministic. 
With hopefully obvious notation, we write $\til{z}_i$ (along with $\til{\xi}_{i}$, $\til{\eta}_i$), $i \geq 0$, under $Q$ (which is tacitly understood to carry a copy of $\til{Z}_0$ independent of $\eta,Z_0$) 
when referring to the chain obtained from $(\eta,\til{Z}_0,(\til{g}_i)_{i \geq 1})$ by soft local times. We denote by $\til{G}_i$, $i \geq 0$  the corresponding soft local times, defined analogously to \eqref{e:slt-def}.

\begin{proposition}\label{pro:corollary}
	The processes $(z_i)_{i\geq0}$, resp.~$(\til{z}_i)_{i\geq0},$ have the same law under $Q$ as $(Z_i)_{i\geq0},$ resp.~$(\til{Z}_i)_{i\geq0},$ and for each  $m,n\geq 1,$ on the event
\begin{equation} \label{eq:SLT-rel-bounds}
	G_m(z)\leq \til{G}_n(z)\text{ for all }z\in{\Sigma},
\end{equation}
one has 
\begin{equation}
\label{eq:zincludedZi}
	\{z_1,\ldots, z_m\} \subseteq \{\til{z}_1,\ldots, \til{z}_n\}.
\end{equation}
\end{proposition}

\begin{proof}
The first part is immediate on account of \eqref{slt-p2}. To see that \eqref{eq:SLT-rel-bounds} implies \eqref{eq:zincludedZi}, observe that, by construction, cf.~\eqref{eq:softlocalexpo}-\eqref{eq:softlocalpoisson} and by definition of $G_i$, see \eqref{e:slt-def}, one has for all, $m ,n \geq 1$,
\begin{align*}
\{z_1,\ldots, z_m\} &= \{ z \in \Sigma : \text{ there exists } (z,v) \in \eta \text{ s.t. } G_m(z) \geq v\}\\
\{\til{z}_1,\ldots, \til{z}_n\} &= \{ z \in \Sigma : \text{ there exists } (z,v) \in \eta \text{ s.t. }\til{G}_n(z) \geq v\}.
\end{align*}
From this, \eqref{eq:SLT-rel-bounds} plainly yields the inclusion \eqref{eq:zincludedZi}. 
\end{proof}

Proposition \ref{pro:corollary} is not entirely adapted to our purpose. In Section \ref{sec:shortrangeRWRI}, see in particular the proof of Proposition \ref{pro:clotheslines}, we will actually need to couple the process $\til{z}$ with the initial chain $Z$ on a suitable extension of $P$. This is conveniently achieved using an inverse soft local time method, that we now explain.

With $P$ referring to the original measure under which the Markov chain $Z=(Z_i)_{i \geq 0}$ is defined, cf.~above \eqref{eq:abs-cont}, let
\begin{align}\label{e:slt-Pext}
\hat{P}: \,&\begin{array}{l}\text{extension of $P$ carrying an independent random variable $\chi=\big((\hat{\xi}_k)_{k\geq1},\hat\eta_0\big),$}\\ \text{where $\hat{\eta}_0$ is a Poisson point process on $\Sigma \times \R_+$ with intensity $\mu\otimes \mathrm{d}v$, and}\\ 
\text{$(\hat{\xi}_k)_{k \geq 1}$ are i.i.d.~exponential variables with mean $1$, independent of $\hat{\eta}_0$.}
\end{array}
\end{align}
Given a realisation of the time inhomogeneous Markov chain $(Z_i)_{0\leq i\leq T}$ up to some deterministic integer time $T< \infty$, we set inductively for $k=0,\ldots, T-1$ (under $\widehat{P}$)
\begin{equation}
\label{eq:defetak+1}
\hat{\eta}_{k+1} = \delta_{(Z_{T-k},\, %\hat{\xi}_{k+1}
\hat{\xi}_{ T-k} g_{T-k}(Z_{T-k-1},Z_{T-k}))}+\sum_{(z,v)\in \hat{\eta}_k} \delta_{(z,\, v + \hat{\xi}_{ T-k}g_{T-k}(Z_{T-k-1},z))}
\end{equation}
and write $\eta=\hat{\eta}_T$. Note that, albeit implicit in our notation, all processes $\hat{\eta}_{k}$, $1\leq k \leq T$, implicitly depend on the choice of $T$.

\begin{lemma}\label{lem:poisson}
Under $\widehat{P}$, for all integers $T \geq 1$, the process $\eta$ is a Poisson process of intensity $\mu\otimes \mathrm{d}v,$ independent of~$Z_0.$ Moreover, the sequence obtained by applying soft local times to $(\eta,Z_{0},(g_i)_{i\geq 1})$ up to time~$T$ is~$(Z_i)_{0\leq i\leq T}$, with corresponding exponential variables $(\hat{\xi}_i)_{1 \leq i \leq T}$.
\end{lemma}

Before proving Lemma \ref{lem:poisson}, we isolate the case $T=1.$

\begin{lemma}\label{lem:poissonprocess}
	Let $\hat{\eta},\xi, Z$ be independent random variables, with $\hat{\eta}$ a Poisson process of intensity $\mu\otimes \mathrm{d}v$, $\xi$ exponentially distributed of parameter $1$ and $Z$ having density $g$ with respect to~$\mu$.  Then 
\[
\eta = \delta_{(Z,\xi g(Z))} + \sum_{(z,v)\in \hat{\eta}} \delta_{(z,v+\xi g(z))} 
\]
is also a Poisson point process with intensity $\mu\otimes \mathrm{d}v$. 
\end{lemma}

\begin{proof}
For $t \in \R$, consider the (measurable) function $f_t:  \Sigma \times \R \to \Sigma \times \R$ with
\begin{equation}\label{eq:f-t}
f_t(z,v)=(z,v + tg(z))
\end{equation}
and note that $f_{-t}=f_t^{-1}$. Given a point measure $\bar{\eta}=  \sum_{\lambda} \delta_{(z_{\lambda}, v_{\lambda})}$ on $\Sigma \times \R$, let $f_t(\bar{\eta})=  \sum_{\lambda} \delta_{f_t(z_{\lambda}, v_{\lambda})}$. Note that $(t,\bar{\eta})\mapsto f_t(\bar{\eta})$ is measurable, and thus $\eta=f_{\xi}(\hat{\eta})$ is also measurable. Consider the Poisson process $\eta_0$ on $\Sigma \times \R_+$ under the  (auxiliary) probability $Q$, cf.~above \eqref{eq:softlocalexpo}, which has intensity $\mu\otimes \mathrm{d}v$. In order to be consistent with the previous setup, fix an arbitrary point $z_0 \in \Sigma$ and define $g_1$ by declaring that $g_1(z_0,z)= g(z)$. Applying \eqref{eq:softlocalexpo}-\eqref{eq:softlocalpoisson} for $i=0$, one finds ${\xi}_1$ and a point $(z_{\lambda_1}, v_{\lambda_1})$, corresponding to the unique minimizer in \eqref{eq:softlocalexpo}, i.e.~with $\xi_1g(z_{\lambda_1})= v_{\lambda_1}$. In view of \eqref{eq:softlocalpoisson} and \eqref{eq:f-t}, one has, for $ \eta_0= \sum_{\lambda} \delta_{(z_{\lambda}, v_{\lambda})}$,
$$
{\eta}_1= \sum_{\lambda \neq \lambda_1} \delta_{f_{{\xi}_1}^{-1}(z_{\lambda}, v_{\lambda})}
$$
which, in particular, yields that
\begin{equation}
\label{e:eta-rewrite}
 \eta_0 = f_{\xi_1}(\eta_1) + \delta_{(z_{\lambda_1}, v_{\lambda_1})}=  f_{\xi_1}(\eta_1) + \delta_{(z_{\lambda_1},\xi_1g(z_{\lambda_1}))}.
\end{equation}
Now, by \eqref{slt-p1} one knows that $\xi_1\stackrel{\text{law}}{=}\xi$, by \eqref{slt-p2} one has that $z_{\lambda_1}\stackrel{\text{law}}{=}Z$ is independent of $\xi_1$ and by \eqref{slt-p3},  $\eta_1$ is a Poisson process of intensity $\mu\otimes \mathrm{d}v$ independent from $\xi_1$ and $z_{\lambda_1}.$ Therefore in view of \eqref{e:eta-rewrite} the point processes $\eta_0$ and $\eta$ have the same law, and so $\eta$ is also a Poisson process of intensity $\mu\otimes \mathrm{d}v.$
\end{proof}

One now easily deduces Lemma \ref{lem:poisson} inductively from Lemma \ref{lem:poissonprocess}.

\begin{proof}[Proof of Lemma \ref{lem:poisson}]
 For $k\in{\{0,\dots,T-1\}},$ assume that $\hat{\eta}_k$ as defined in \eqref{eq:defetak+1} (see also~\eqref{e:slt-Pext} regarding $\hat{\eta}_0$) is a Poisson point process under $\widehat{P}$ independent of $\mathcal{A}_k :=\sigma(Z_i,\hat{\xi}_{i}, 0 \leq i \leq T-k),$ with the convention $\hat{\xi}_0=0$. Note in particular that this is automatically satisfied in case $k=0$ on account of \eqref{e:slt-Pext}.
 Then by Lemma \ref{lem:poissonprocess}, applied with $\hat{\eta}=\hat{\eta}_k,$ $Z=Z_{T-k},$ $\xi=\hat{\xi}_{T-k}$ and $g(z)=g_{T-k}(Z_{T-k-1},z),$ one deduces that, conditionally on  $\mathcal{A}_{k+1},$ $\hat{\eta}_{k+1}$ is a Poisson point process with intensity $\mu\otimes \mathrm{d}v.$ In particular, $\hat{\eta}_{k+1}$ is independent of $\mathcal{A}_{k+1}.$ By induction, we conclude that $\eta=\hat{\eta}_T$ is a Poisson point process with intensity $\mu\otimes \mathrm{d}v$ independent of $\mathcal{A}_{T}=\sigma(Z_0).$

Referring to \eqref{eq:softlocalexpo}-\eqref{eq:softlocalpoisson}, let ${\xi}_i,$ ${\eta}_i,$ $v_i$ and $z_i,$ $0 \leq i \leq k$, be the variables obtained by applying soft local times to $(\eta,Z_0,(g_i))$ up to time $k$ (so in particular $\xi_0= v_0=0$, $z_0=Z_0$ and $\eta_0= \eta (= \hat{\eta}_T)$). Assume that for some $k\in{\{0,\dots,T-1\}}$ we have ${\xi}_i=\hat{\xi}_{i}$, 
${\eta}_i=\hat{\eta}_{T-i},$ $v_i=\hat{\xi}_{i}g_{i}(Z_{i-1},Z_{i})$ and $z_i=Z_i$ for all $0 \leq i \leq k$ (with $v_0=0$). Then 
\[
 \xi_{k+1} \stackrel{\eqref{eq:softlocalexpo}}{=}  \inf_{(z,v)\in {\eta}_{k}} \frac{v}{g_{k+1}(z_k,z)}\stackrel{\substack{\eta_k = \hat{\eta}_{T-k}\\ z_k=Z_k }}{=}\inf_{(z,v)\in \hat{\eta}_{T-k}} \frac{v}{g_{k+1}(Z_k,z)}=\hat\xi_{k+1}
\]
and the infimum is a.s.~uniquely attained at $(Z_{k+1},\hat{\xi}_{k+1}g_{k+1}(Z_k,Z_{k+1}))$ by definition of $\hat{\eta}_{T-k}$ in \eqref{eq:defetak+1}. Moreover one easily checks that ${\eta}_{k+1}=\hat{\eta}_{T-k-1},$ and it follows with a simple induction argument that $z_i=Z_i$ for all $i\leq T.$
\end{proof}

Combining Proposition \ref{pro:corollary} and Lemma \ref{lem:poisson}, one can 
couple $Z= (Z_i)_{i\geq0}$ under the extended measure $\widehat{P}$ defined in \eqref{e:slt-Pext} to any other (inhomogenous) Markov chain having transition densities relative to $\mu$,  with good control on the ranges in terms of appropriately defined (inverse) soft local times, see Remark~\ref{R:inverse} below regarding the terminology. %Recall that a measurable function $h:\Sigma\times \Sigma \to \R_+$ is called transition density with respect to $\mu$ if $\int h(x,y) \mu(\mathrm{d}y)=1$ for all $x \in \Sigma$.
\begin{proposition}
\label{pro:couplingSLT} 
For all $T \geq 1$, $\til{z}\in{\Sigma}$ and any family $(\tilde{g}_i)_{i\geq1}$ of transition densities with respect to $\mu$, one can define under $\hat{P}$ %three
two sequences $\til{Z}= (\til{Z}_i)_{i\geq0}$, %$\xi= (\xi_i)_{i\geq1}$, 
$\til{\xi}= (\til{\xi}_i)_{i\geq1}$ such that,
letting 
\begin{equation}\label{eq:new-SLT}
G_i(z)=\sum_{1\leq k \leq i}\hat{\xi}_k g_k(Z_k,z), \quad \til{G}_i(z)=\sum_{1 \leq k \leq i}\til{\xi}_k\til{g}_k(\til{Z}_k,z)\quad \text{for } i \geq 1, \, z\in{\Sigma}
\end{equation}
(see \eqref{e:slt-Pext} regarding $\hat{\xi}_k$), the following hold:
\begin{itemize}
\item[i)] $\til{Z}$ is a Markov chain with $\til{Z}_0=\til{z}$ and transition densities $(\til{g_i})_{i\geq1}$;
\item[ii)]%$\xi$ (resp.~$\til{\xi}$) 
$\til{\xi}$ are i.i.d.~exponential variables with mean one, independent of $\til{Z}$; %$Z$ (resp.~$\til{Z}$).
\item[iii)]For each $p,m,n\in{\{1,\dots,T\}},$ 
\begin{multline}\label{eq:ifthen}
		\big\{ \til{G}_p(z)\leq G_m(z)\leq \til{G}_n(z), \, \text{ for all }z\in{\Sigma} \big\}\\
\subset \big\{ \{\til{Z}_1,\dots,\til{Z}_p\}\subset\{Z_1,\dots,Z_m\}\subset\{\til{Z}_1,\dots,\til{Z}_n\} \big\}.
\end{multline}
\end{itemize}
\end{proposition}
\begin{proof}
Recall $\eta=\hat{\eta}_T$ from below \eqref{eq:defetak+1} and let $(z_i)_{i}$ and $(\til{z}_i)_{i}$ be the Markov chains obtained by applying soft local times respectively to $(\eta,Z_0, (g_i))$ and $(\eta,\til{z},(\til{g}_i))$. Define $\til{Z}_i=\til{z_i}$, $i \geq0$, and $\til{\xi}$ the corresponding sequence of exponential random variables produced by applying soft local times, cf.~\eqref{eq:softlocalexpo} and \eqref{slt-p1}. In particular, this implies that $(\til{G}_i(\cdot))_{i \geq 0}$ defined in \eqref{eq:new-SLT} is the corresponding soft local time process.

With these choices for $\til{Z}$ and $\til{\xi}$, i) and ii) %in case of $\til{\xi}$ 
follow immediately from the first part of Proposition~\ref{pro:corollary} and by Lemma \ref{lem:poisson}, which guarantees that $\eta$ has the correct law. %The claim ii) for $\xi$ is plain on account of \eqref{e:slt-Pext}. 
 Finally, one notices that, due to Lemma \ref{lem:poisson}, $z_i=Z_i$  for all $0 \leq i\leq T$ and $(G_i(\cdot))_{0 \leq i\leq T}$ as defined in \eqref{eq:new-SLT} is the corresponding soft local time process (up to time $T$). 
From this, iii) follows upon applying Proposition \ref{pro:corollary} twice, swapping the roles of $(g_i)$ and $(\til{g}_i),$ to deduce \eqref{eq:ifthen} from \eqref{eq:SLT-rel-bounds}-\eqref{eq:zincludedZi}.
\end{proof}

\begin{Rk}\leavevmode\label{R:inverse}
\begin{enumerate}[label=\arabic*)]
\item It is natural to refer to $G_i$ defined in \eqref{eq:new-SLT} as an \textit{inverse} (or backwards) soft local time. Unlike its `forward' counterpart \eqref{e:slt-def}, in which the random variables $\xi=(\xi_k)_{k \geq 1}$ emerge as minimizers in \eqref{eq:softlocalexpo},  the random variables $\hat{\xi}=(\hat{\xi}_k)_{k \geq 1}$ involved in \eqref{eq:new-SLT} are given by fiat, see \eqref{e:slt-Pext}. Loosely speaking, this corresponds to the fact that, instead of constructing the chain $Z$ (and the variables $\xi$) from $\eta$, one reconstructs $\eta$ from a given realization of $Z$ (with the help of additional independent randomness, comprising $\hat{\xi}$). The benefit of doing this (and the gist of Proposition~\ref{pro:couplingSLT}) is that, with $\eta$ at hand, one can now apply (forward) soft local times to couple any other chain $\til{Z}$ having transition densities relative to $\mu$ to the original chain $Z$ via $\eta$.
\item\label{R:dependencytildeZ} In Proposition~\ref{pro:couplingSLT}, the random variables $\til{Z}$ and $\til{\xi}$ depend only on the Markov chain $Z$ and on the variable $\chi$ from \eqref{e:slt-Pext}  (as well as the choice of $T,$ of the probability $P$ and of the densities $g$ and $\til{g}$). This will be important in Section~\ref{sec:shortrangeRWRI}, where we will apply Proposition~\ref{pro:couplingSLT} several times simultaneously for varying choices of $Z$ and $\chi.$ These varying choices are coupled together on a common probability space, thus the corresponding varying processes $\tilde{Z}$ and $\tilde{\xi}$ are also naturally defined on the same probability space, see around \eqref{eq:deftildeZji}. 
\end{enumerate}
\end{Rk}

\section{Localization}
\label{sec:shortrangeRWRI}

In this section we prove our main localization result, Theorem~\ref{The:shortrangeapprointro-new}, which will follow from a more general result, Theorem~\ref{thm:rwshortrange} below. This result is of independent interest and is not specific to the ``late'' or ``high-intensity'' regime, to which it will later be applied. Such applications are discussed separately in Section~\ref{sec:loc-cons}. The proof of Theorem~\ref{thm:rwshortrange} is split over \S\ref{sec:constructionexcursions}-\ref{sec:proofthmshortrange}, and involves inverse soft local times, cf.~Proposition~\ref{pro:couplingSLT}. An overview of the proof appears atop of \S\ref{sec:constructionexcursions}.

For the purposes of Theorem \ref{thm:rwshortrange}, which couples processes with range in both $\mathbf{T} (= (\Z/ N\Z)^d)$ and $\Z^d$, it will be important to distinguish clearly between the two.  %, specific to the present sectionand the appendix. 
Recall from \S \ref{sec:notation} that $\pi: \Z^d \to \mathbf{T}$ denotes the canonical projection, that for $x\in{\Z^d}$ we often abbreviate $\boldsymbol{x}=\pi(x),$ that $0$ is the origin of $\Z^d$ whereas $\mathbf{0}$ is the origin of $\mathbf{T},$ and that $Q(x,r)=Q_r(x)$ are boxes around $x$ of side length $r$ either in $\Z^d$ or in $\mathbf{T},$ depending on whether $x\in{\Z^d}$ or $\mathbf{T}$.
 %The restriction $\pi^0$ of $\pi$ to $Q_N$ is a bijection. In the sequel, when using the symbols $x, \boldsymbol{x}$, etc.~within the same expression, it is understood that $x=(\pi^0)^{-1}(\boldsymbol{x})$ (in particular this implies that $x \in Q_N$).

%x+\{y\in{\mathbf{T}}:\,|y|_{\infty}\in{\{-\lfloor(r-1)/2\rfloor,\dots,\lceil(r-1)/2\rceil\}}\}.$ We also define $\pi^0$ as the restriction of $\pi$ to $Q_N,$ which is bijective.

In what follows, a family of (point) processes $(\omega^{(x)})_{x\in{Q_N(0)}}$ is said to have range $R$ in $\Z^d,$ resp.~in~$\mathbf{T},$ if $\omega^{(x)}$ and $\{\omega^{(x')}: x'\in Q_N(0)\setminus Q_R(x)\},$ resp.~$\{\omega^{(x')}: \boldsymbol{x'}\notin Q(\boldsymbol{x},R)\},$ are independent for each $x\in{Q_N(0)}.$ Intuitively, if one identifies $Q_N(0)$ with~$\mathbf{T}$, then $(\omega^{(x)})_{x\in \mathbf{T}}$ has range $R$ in $\mathbf{T}$ if for all $x\in \mathbf{T}$, $\omega^{(x)}$ is independent of $\omega^{(x')}$ for all $x'\notin Q(x,R)$. We refrain from doing such an identification, since it could cause confusion in the next statement, in which the random walk and random interlacements appear jointly. Following is our main localisation result, from which Theorem~\ref{The:shortrangeapprointro-new} will follow as a special case.

\begin{theorem}[Localization]\label{thm:rwshortrange}
For all $\delta\in (0,1),$ there exist $c=c(\delta)>0$ and $C=C(\delta)<\infty$ such that the following holds. For every $N\in\N,$ $R\in [1, \frac{N}{1+\delta}]$ and $u_0>0,$ there exists a probability measure $\til{\mathbf{P}}_{\boldsymbol{0}}$ extending $\mathbf{P}_{\boldsymbol{0}}$, resp.~$\tildePI$, extending $\PI,$ carrying a family of processes $(\omega^{(x)})_{x\in{Q_N(0)}}$ such~that
\begin{align}
&\label{e:loc-law}\text{$\omega^{(x)}$ has law $\PI$ for each $x \in Q_N(0)$,}\\
&\label{e:loc-range} \text{$(\omega^{(x)})_{x\in{Q_N(0)}}$ has range $2(1+\delta)R$ in $\mathbf{T},$ resp.~$\Z^d,$}
\end{align}
 and, writing  $({\ell}^{(x)}_{y,u})_{y\in{\Z^d}, u \geq 0}$ for the field of local times associated to $\omega^{(x)},$ cf.~\eqref{eq:deflocaltimes-ri}, for all $F\subseteq Q_N(0),$ $0< v<u\leq u_0,$  and $\epsilon\in(0,1)$ with $u(1-\eps)>v(1+\eps)$ one has
 	\begin{multline}\label{eq:couplelltilell}
	\til{\mathbf{P}}_{\boldsymbol{0}}\left(
	%\begin{array}{l}
	{\ell}_{y,u(1-\epsilon)}^{(x)}-{\ell}_{y,v(1+\epsilon)}^{(x)}\leq {\ell}_{\boldsymbol{y},u}-{\ell}_{\boldsymbol{y},v} \leq {\ell}_{y,u(1+\epsilon)}^{(x)}-{\ell}_{y, v(1-\epsilon)}^{(x)}, \,
	\forall x\in F, \, y\in{Q(x,R)}
	%\end{array}	
	\right)
	\\[0.3em] \geq 1 - C|F|R^{2d}\lceil uR^{d-2}\rceil\exp\big(-c\cdot \epsilon\cdot \sqrt{v \cdot R^{d-2}} \big),
	\end{multline}
resp.\ for all $u\leq u_0$ and $\eps\in{(0,1)}$,
	\begin{multline}
	\label{eq:couplingshortrangeinter}\tag{\ref*{eq:couplelltilell}'}
	\tildePI\left({\ell}_{y, u(1-\epsilon)}^{(x)}\leq \ell_{y,u}\leq {\ell}_{y, u(1+\epsilon)}^{(x)}, \ \forall \, x\in F, \, y\in{Q(x,R)}\right)
	\geq 1-C|F|R^{2d}\exp\big(-c \epsilon^2 u R^{d-2}\big).
	\end{multline}
\end{theorem}

We refer to Remark~\ref{rk:othershortrange} at the end of this section for various extensions of the above result, reflecting in particular a certain flexibility for the requirement \eqref{e:loc-law}, and alternatives to \eqref{eq:couplelltilell} (see \eqref{eq:couplelltilell'}-\eqref{eq:couplelltilell2}), which do not involve \textit{increments}; these matters are best explained after giving the proof. Applications of Theorem~\ref{thm:rwshortrange} specific to our purposes are postponed to Section~\ref{sec:loc-cons}.

\medskip
Before proceeding any further, let us give the short:

\begin{proof}[Proof of Theorem~\ref{The:shortrangeapprointro-new} (assuming  Theorem~\ref{thm:rwshortrange})]
Define $\til{\mathbf{P}} = N^{-d} \sum_{{x} \in \mathbf{T}} \til{\mathbf{P}}_{{x}}$ where $ \til{\mathbf{P}}_{{x}}$ is obtained from  $\til{\mathbf{P}}_{\boldsymbol{0}}$ through translation by $x\in{\mathbf{T}}$. One then applies Theorem~\ref{thm:rwshortrange} for the choices $\delta=1$, with $2u$ in place of $u$, $v=u$, $F=Q_N(0)$ and $\frac{R}{4}$ in place of $R$. Property \eqref{e:loc-range} is still verified under $\til{\mathbf{P}}$ since the law of $\omega^{(y)}$ under $\til{\mathbf{P}}_x$ does not depend on $x\in{\mathbf{T}}$ for each $y\in{Q_N(0)}$ by \eqref{e:loc-law}. Upon defining $\til{{\ell}}_{\boldsymbol{x},u(1\pm \varepsilon)}= {\ell}_{x,2(1\pm \varepsilon)u}^{(x)}-  {\ell}_{x,(1\mp \varepsilon)u}^{(x)} $ for all $x\in{Q_N(0)},$ and observing that ${\ell}_{\cdot,2u}-{\ell}_{\cdot,u}$ has the same law under $\til{\mathbf{P}} $ as ${\ell}_{\cdot,u}$ under $\mathbf{P}$ due to stationarity, the desired properties \eqref{e:fr-intro} resp.~\eqref{e:incl-intro} follow readily from \eqref{e:loc-range} resp.~\eqref{eq:couplelltilell}.
\end{proof}

As explained at length in \S\ref{subsec:loc}, the main upshot of Theorem~\ref{thm:rwshortrange} is the combined effect of \eqref{e:loc-range} and \eqref{eq:couplelltilell}/\eqref{eq:couplingshortrangeinter}, which yields a  ``close + local approximation'' to the field of interest, the (increments of) local times. By virtue of the additional requirement \eqref{e:loc-law}, Theorem~\ref{thm:rwshortrange} also entails a coupling between random interlacements and random walk, which will prove useful as well; cf.~Lemma~\ref{lem:boundonlateRIRW}. Similar couplings at a mesoscopic scale were first obtained in \cite{MR2386070}, and then improved in \cite{MR2838338,BEL}. At a macroscopic scale, that is, outside of an annulus of size $\delta N$ for some small $\delta>0,$ a similar coupling was proved in \cite[Theorem 4.1]{MR3563197}. In fact, Theorem~\ref{thm:rwshortrange} directly implies the coupling from \cite[Theorem~4.1]{MR3563197} for fixed $\delta>0,$ as we now explain. 

  \begin{Cor}\label{cor:coupling}
    For all $\delta>0,$ there exist $c,C \in (0,\infty)$ depending only on $\delta $ and $d$ such that the following holds. For every $u>0,$ $\eps\in{(0,1)}$ and $N\in{\N}$ there exists a coupling $\til{\mathbf{P}}$ between $({\ell}_{x,u})_{x\in{\mathbf{T}}}$ under $\mathbf{P}$ and $(\ell_{x,u(1\pm\eps)}: x\in Q_N(0))$ under $\PI$ so that
   \begin{align*}
   \til{\mathbf{P}}\left(\ell_{x,u(1-\epsilon)} \leq{ \ell}_{\boldsymbol{x},u}\leq \ell_{x,u(1+\epsilon)},\,  \forall x\in 
   Q_{N(1-\delta)}(0)\right) \geq  1 -CN^{2d}\lceil uN^{d-2}\rceil \exp\big(-c \epsilon \sqrt{u  N^{d-2}} \big).
   \end{align*}
    \end{Cor}
    \begin{proof}
    The desired coupling $\til{\mathbf{P}}$ of $({\ell}_{{x},u})_{{x} \in \mathbf{T}}$, $(\ell_{x,u(1-\eps)})_{x\in Q_N(0)}$ and $(\ell_{x,u(1+\eps)})_{x\in Q_N(0)}$ is obtained as the joint law of $({\ell}_{{x},2u}-{\ell}_{{x},u})_{{x} \in \mathbf{T}}$, $(\ell^{(0)}_{x, 2u(1-\eps/3)}-\ell^{(0)}_{x ,u(1+\eps/3)})_{x \in Q_N(0)}$ and $(\ell^{(0)}_{x, 2u(1+\eps/3)}-\ell^{(0)}_{x, u(1-\eps/3)})_{x\in Q_N(0)}$ under $N^{-d}\sum_{{x}\in{\mathbf{T}}}\til{\mathbf{P}}_{{x}}.$
     Using \eqref{eq:couplelltilell} (replacing $v$ by $u$ and $u$ by $2u$) with $u_0=2u,$ $F=\{0\}$ and $R=\lceil N(1-\delta)\rceil,$ the claim immediately follows from the independence and stationarity of the increments of random interlacements.
    \end{proof}

Note that $\til{\mathbf{P}}$ appearing in Corollary~\ref{cor:coupling} is in fact the same measure as in Theorem~\ref{The:shortrangeapprointro-new}, hence the identical notation. As opposed to \cite[Theorem 4.1]{MR3563197}, the coupling error obtained here is more explicit in $u,\eps$ and $N.$ We believe moreover that our proof is more elementary since it does not involve the coupling appearing in \cite[Theorems 3.1-2]{MR3563197} of general Markov chains via soft local times. As in \cite{MR3563197}, one could probably also take $\delta=\delta(N)\rightarrow0$ slowly enough as $N\rightarrow\infty$ in Theorem~\ref{thm:rwshortrange}. Actually, Corollary~\ref{cor:coupling} is the only reason we chose to prove Theorem~\ref{thm:rwshortrange} for all $\delta>0,$ instead of for all $\delta$ large enough, which would be sufficient for our purposes. The proof of Theorem~\ref{thm:rwshortrange} would be slightly simpler for $\delta$ large, for instance one would not have to use Lemma~\ref{lem:concsoftrwri} below by proceeding similarly as in \cite[Lemma~2.1]{MR3126579} (upon adapting the results from Appendix~\ref{sec:harnack} to obtain Harnack bounds with constants close to $1$ as $\delta\rightarrow\infty$), and the proof of Lemma~\ref{lem:green-killed} would also have been simpler. But Lemma~\ref{lem:concsoftrwri} is actually easy to prove given the tools developed in Appendix~\ref{sec:largedeviationapp}, which are anyway required to prove Lemma~\ref{lem:conc} below. In other words, proving Theorem~\ref{thm:rwshortrange} for all $\delta>0$ does not require a lot of additional work, and provides a more elementary proof of a version of \cite[Theorem 4.1]{MR3563197}.
 
\medskip

In the remainder of this section, we prove Theorem~\ref{thm:rwshortrange}. We shortly  explain the general strategy for the random walk; the case of random interlacements is similar. First, in Section~\ref{sec:constructionexcursions}, for each $B_1\subset B_2\subset B_3 (\subset \Z^d)$, we approximate the excursions $(Z_i)_{i\geq0}$ of the random walk from (the projection of) $\partial B_1$ to $\partial B_2$ before hitting $\partial B_3,$ see~\eqref{eq:defexcrw}, by a process $(\til{Z}_i)_{i\geq0}$ of excursions which is independent of the walk outside of $B_3,$ see Lemma~\ref{lem:independenceforfixedsets}, using the inverse soft local times method from Proposition~\ref{pro:couplingSLT}. Then in Section~\ref{sec:costshortrange} for $r_1<r_2<r_3,$ we put together the excursions $(\til{Z}_i^{(x)})_{i\geq0},$ $x\in{Q_N(0)},$ each corresponding to the choice $B_i=Q({x},r_i)$ for each $i\in{\{1,2,3\}},$ and show that they satisfy a short range property, see Proposition~\ref{pro:clotheslines}. Moreover for each fixed $m\in{\N},$ we estimate the probability that $(\til{Z}_i^{(x)})_{i\leq m}$ is close to the initial excursions $(Z_i^{(x)})_{i\leq m}$ in Lemma~\ref{lem:concsoftrwri}, and show that the number of excursions performed by the random walk or random interlacements at a given time is well concentrated around some deterministic $m$ in Lemma~\ref{lem:conc}. The proof of these two lemmas is given in Appendix~\ref{sec:app}. Finally, the different pieces of the proof are put together in Section~\ref{sec:proofthmshortrange} by defining $\omega^{(x)}$ for each $x\in{Q_N(0)}$ as an interlacement process whose excursions are given by the short-range excursions $(\til{Z}_i^{(x)})_{i\geq0}$ for $r_1=R$, $r_2=R\sqrt{1+\delta}$ and $r_3=R(1+\delta).$

\subsection{Construction of short-range excursions in a fixed set}
\label{sec:constructionexcursions}

We start with a realization $X$ of the random walk on $\mathbf{T}$ under $\mathbf{P}_{\boldsymbol{0}}$, and for  $B_1\subseteq B_2\subseteq B_3\subset Q_{2N}(0)$ with diameter smaller than $N$, we first define a process of excursions $(Z_j)$ in $B_3$ from $B_1$ to $\partial B_2,$ and a corresponding clothesline process $(\zeta_j)$ in~$\partial B_2\times\partial B_3,$  such that for each $j,$ up to projection onto $\mathbf{T}$, $X$ first visits the first coordinate of $\zeta_j,$ then after hitting~$B_1$ follows the excursion $Z_j$ until the last hitting time in $B_2$ before reaching $\partial B_3,$ and reaches $\partial B_3$ in the second coordinate of $\zeta_j,$ similarly as in~\cite[Section~3]{AlvesPopov} and~\cite[Section~4]{MR3563197}. Then, coupling via the inverse soft local time method of Section \ref{sec:softlocaltimes}, conditionally on $(\zeta_j),$ we define excursions $(\til{Z}_j)$ close to $(Z_j)$ and independent of the walk $X$ outside of $B_3,$ see Lemma \ref{lem:independenceforfixedsets}, which will form the basis of the finite-range process $\omega^{(x)}$ from Theorem \ref{thm:rwshortrange} in the random walk case. Finally, we extend this construction to the case of random interlacements.

 Let $B_1\subseteq B_2\subseteq B_3%\subseteq B_4 
 \subset Q_{2N}(0)$ be %four
 three concentric boxes with diameter at most $N-1.$ For a sequence $x=(x_n)_{n \geq 0}$ in $\Z^d,$ we introduce two sequences of successive return and departure times as $R_0(x,B_2,B_3)=0$, $D_0(x,B_2,B_3)= H_{{\partial B_3^c}}$, where $\partial B_3^c=\partial (B_3^c)$ is the exterior boundary of $B_3$, and inductively for integers $k\geq 0$ 
 \begin{equation}
\label{eq:deftilrhorho}
\begin{split}
R_{k+1}(x,B_2,B_3)&= \inf\{n\geq D_{k}(x,B_2,B_3): x_n\in \partial B_2\}, \\
	D_{k+1}(x,B_2,B_3)&=\inf\{n\geq R_{k+1}(x,B_2,B_3): x_n\in \partial B_3^c\},
\end{split} 
\end{equation}
Thus, $D_k \leq R_{k+1} \leq D_{k+1}$ for all $k \geq 0$.
We further define, for $k\geq 1$,
\begin{align*}
	H_k(x,B_1,B_2,B_3)&=\inf\{n\in [R_{k}(x,B_2,B_3),D_{k}(x,B_2,B_3)] : x_n\in \partial B_1\}
\end{align*}
%We follow the standard convention that the infimum of the empty set is $+\infty$.
(we use the convention $\inf\emptyset= + \infty$). Attached to these stopping times are the process $(\zeta_i(x,B_2,B_3))_{i\geq 1}$, where \begin{align}\label{eq:defclothesline}
\zeta_i(x,B_2,B_3) = (x_{R_{i}(x,B_2,B_3)}, x_{D_i(x,B_2,B_3)}) \in  \partial B_2\times \partial B_3^c \ \text{ for } \ i\geq 1,
\end{align} 
and the excursions $(Z_i(x,B_1,B_2,B_3))_{i\geq 1}$ as
\begin{align}\label{eq:defexcrw}
	Z_i(x,B_1,B_2,B_3) = \begin{cases}
 	x_{\left[ H_{i}(x,B_1,B_2,B_3), L_{B_2}(D_i(x,B_2,B_3))   \right]} \quad &\text{ if } \ H_{i}(x,B_1,B_2,B_3)<\infty,\\
 	\Theta &\text{ otherwise,}
 \end{cases}
 \end{align}
where $x_{[s,t]}= \{x_k:\,k\in{[s,t]\cap\N}\},$ the time $L_{B_2}(D_i(x,B_2,B_3))$ is defined similarly as in \eqref{eq:defLKt} but relative to $x$ and $\Theta$ is a cemetery state corresponding to excursions that do not hit $\partial B_1$. In words, $Z_i(x,B_1,B_2,B_3)$ is the part of the $i$-th excursion from~$B_2$ to~$B_3$ from the first time it hits~$\partial B_1$ until the last time it is in $B_2$ before hitting~$\partial B_3^c$. The attentive reader will have noticed that the system of excursions defined by \eqref{eq:defexcrw}, though designed to keep track of $x$ inside $B_1$, actually neglects its first part until first  exiting $B_3$ (which may well intersect $B_1$). % \textcolor{orange}{pas genial mais bon...}

We now introduce $B_4$, a box of side length $N$ satisfying $ B_4 \supset B_3$, which is otherwise arbitrary. With $X$ the random walk in~$\mathbf{T}$ under $\mathbf{P}_{\mathbf{0}}$, the restriction $\pi_{|B_4}: B_4 \to \mathbf{T}$ is a bijection, and we let $\widehat{X}=(\widehat{X}_n)_{n \geq 0}$ with
\begin{equation}\label{e:X-hat}
\widehat{X}_n=(\pi_{|B_4})^{-1}(X_n), \text{ for all integers $n \geq 0.$}
\end{equation} 
%Note that the law of $\widehat{X}$ depends on the choice of $B_4,$ which will be fixed throughout this subsection.
Even though $\widehat{X}$ in \eqref{e:X-hat} depends on the choice of $B_4$, the subsequent construction (in particular the definition of the processes $Z$, $\zeta$ and $\lambda$ below) do not (so long as $B_4 \supset B_3$). We will be interested in the system $Z=(Z_i)_{i \geq 1}$ of excursions (on $\Z^d$)
\begin{equation}\label{e:Z-i-real}
Z_i \stackrel{\text{def.}}{=} Z_i(\widehat{X},B_1,B_2,B_3), \, i \geq 1,
\end{equation}
with corresponding clothesline process $\zeta=(\zeta_i)_{i\geq1}$, where $\zeta_i= \zeta_i(\widehat{X},B_2,B_3)$. For $u>0$, we denote by $\Excrw(X,B_2,B_3,u)$ the total number of excursions $X$ performs across $\pi(B_3\setminus B_2)$ before time $uN^d$ and after time $D_0(\widehat{X},B_2,B_3)$; that is,
\begin{equation}
\label{eq:defnumberexcursionRW}
\Excrw({X},B_2,B_3,u) = \sup\{k\geq 0: R_{k}(\widehat{X},B_2,B_3)< uN^d\}.
\end{equation}
We now adapt this setup to random interlacements. Let~$\omega_{Q_{2N}(0)}=(\omega^u_{Q_{2N}(0)})_{u>0}$ be the restriction of a random interlacements process $\omega$ under $\PI$ to trajectories hitting~$Q_{2N}(0)$. Recalling~\eqref{eq:definterprocess}, $\omega_{Q_{2N}(0)}$ is defined in terms of the Poisson counting process $u  \mapsto N^u= N_{Q_{2N}(0)}^u$ of intensity $\text{cap}(Q_{2N}(0))$ and a family $X^j$, $j \geq 1$, of independent simple random walks having law $P_{\bar{e}_{Q_{2N}(0)}}$ each. Let
\begin{equation}
\label{eq:defTj}
T^j=T(X^j, B_2,B_3) = \sup\{k\geq 0:  R_k(X^j,B_2,B_3)<\infty\}.
\end{equation}
%For each $i\leq T_j,$ we also extend the definitions of the clothesline process ${\zeta_i}(X^j,B_2,B_3),$ see \eqref{eq:defclothesline}, and the excursions $Z_i(X^j,B_1,B_2,B_3),$ see \eqref{eq:defexcrw}, from random walks on $\mathbf{T}$ to random walks on $\Z^d.$ 
Note that $T^j=0$ with positive probability, and that actually $T^j=0$ if and only if  $X^j$ does not visit $B_2$ in view of \eqref{eq:deftilrhorho} and since we assumed $B_3\subset Q_{2N}(0).$ Let us now define the clothesline process $ \lambda = (\lambda_j(\omega,B_2,B_3))_{j\geq1}$ and the excursions process $(W_j(\omega,B_1,B_2,B_3))_{j\geq1}$ of random interlacements by taking for all $j\geq1$ and $1\leq i\leq T^j$ (cf.~\eqref{eq:defclothesline} and~\eqref{eq:defexcrw} for notation)
\begin{equation}\label{eq:defclotheslineri}
\begin{gathered}
\lambda_{i+\sum_{k=1}^{j-1}T^k}(\omega,B_2,B_3)=\zeta_i(X^j,B_2,B_3) \text{ and } \\W_{i+\sum_{k=1}^{j-1}T^k}(\omega,B_1,B_2,B_3)=Z_i(X^j,B_1,B_2,B_3).
\end{gathered}
\end{equation}
In analogy with \eqref{eq:defnumberexcursionRW}, we  write $\Excri(\omega,B_2,B_3,u)$ for the total number of excursions across $B_3\setminus B_2$ performed by the interlacement at level $u$, i.e.
\begin{equation}
\label{eq:defnumberexcursionRI}
\Excri(\omega,B_2,B_3,u) =  \sum_{j=1}^{N^u} T^j,
\end{equation}
with $T^j$ as in \eqref{eq:defTj} and where $N^u$ is a Poisson random variable of parameter $u\cp(Q_{2N}(0)),$ see above \eqref{eq:defTj}. We will frequently omit various arguments, e.g.~the sets $B_1,B_2, B_3$, from the above notation whenever those are clear from the context. 

We now aim to couple the process $Z= (Z_i)_{i \geq 1}$ from \eqref{e:Z-i-real} via inverse soft local times, i.e.~using Proposition~\ref{pro:couplingSLT}, with a process $\til{Z}$ independent from $X$ outside $B_3,$ corresponding to excursions of random interlacements. We start by defining the appropriate state space $\Sigma,$ measure $\mu$ and transition densities $g,$ cf.~\eqref{eq:abs-cont}. Let $\K$ denote the set of finite nearest-neighbor paths in $B_3$ from $\partial B_1$ to $\partial B_2$, i.e.
\begin{equation}
\label{eq:defK}
\K= \left\{z=(z_0,\ldots, z_\ell): z_0\in  \partial B_1, z_j\in B_3\ \forall \ j\leq \ell, z_\ell\in \partial B_2  \right\}.
\end{equation}
Recalling $\Theta$ from \eqref{eq:defexcrw}, which represents a cemetery state corresponding to the excursions from $\partial B_2$ to $\partial B_3^c$ that do not hit~$B_1$, we set $\Sigma= \K \cup \{\Theta\}$ and define for $S\subseteq \Sigma$ %and $X$ a simple random walk in $\mathbf{T}$
\begin{equation}\label{e:mu-ref}
\mu(S) = \sum_{x\in  \partial B_1, y\in \partial B_2} \mathbf{P}_{\boldsymbol{x}} \big(\widehat{X}_{[0,L_{B_2}(T_{B_3})]}\in S \, \big| \, \widehat{X}_{L_{B_2}(T_{ B_3})}=y\big) + 1\{\Theta\in S\},
\end{equation}
with the convention $\mathbf{P}_x(\cdot\,|\,A)=0$ for any event $A$ with $\mathbf{P}_x(A)=0,$ and where $L_{B_2}$ and $T_{B_3}$ ($=H_{\partial B_3^c}$ when starting from $B_3$) are defined as in Section~\ref{sec:notation}, but for $\widehat{X}$ instead of $X$ (since $B_i\subset{\Z^d}$).
For every $(y,w)\in \partial B_2\times \partial B_3^c$ and $z=(z_0,\ldots,z_\ell)\in \K$, abbreviating $H_{K}(\widehat{X})= H_{K}$ we let 
\begin{equation}
\label{eq:defg}
g_{(y,w)}(z) = \mathbf{P}_{\boldsymbol{y}} \big( T_{ B_3} \geq H_{B_1}, \widehat{X}_{H_{ B_1}}=z_0, \widehat{X}_{L_{B_2}(T_{ B_3})}=z_{\ell} \, \big| \, \widehat{X}_{T_{ B_3}}=w\big)
%g_{(y,w)}(z) = \mathbf{P}_{\boldsymbol{y}} \big( T_{ B_3} \geq H_{B_1},  \, \widehat{X}_{[H_{ B_1}, L_{B_2}(T_{ B_3})]}=z \, \big| \, \widehat{X}_{T_{ B_3}}=w\big)
\end{equation}
and also
\begin{equation}
\label{eq:defg2}
g_{(y,w)}(\Theta) =  \mathbf{P}_{\boldsymbol{y}} \big( T_{ B_3}<H_{B_1} \, \big| \, \widehat{X}_{T_{ B_3}}=w \big).
\end{equation}
 It then follows that for all $z \neq \Theta$, 
  \begin{multline*}
	 \mathbf{P}_{\boldsymbol{y}} \big(Z_1=z \, | \, \widehat{X}_{T_{ B_3}}=w\big)  \\= g_{(y,w)}(z)
 \mathbf{P}_{\boldsymbol{y}} \big(Z_1=z \, \big| T_{ B_3} \geq H_{B_1}, \widehat{X}_{H_{ B_1}}=z_0, \widehat{X}_{L_{B_2}(T_{ B_3})}=z_{\ell},\widehat{X}_{T_{ B_3}}=w\big) 
  =	  g_{(y,w)}(z) \mu(\{z\}) ,% \text{ for all $z\in \Sigma$;}
\end{multline*}
where the last equality follows by writing the relevant conditional probability as a ratio and applying the simple(!) Markov property separately to numerator and denominator by summing over all possible values of $L_{B_2}(T_{ B_3})$. One readily finds that the equality $ \mathbf{P}_{\boldsymbol{y}} (Z_1=z \, | \, \widehat{X}_{T_{B_3}}=w) = g_{(y,w)}(z) \mu(\{z\})$ continues to hold for $z= \Theta$.
That is, $g_{(y,w)}$ is the density with respect to $\mu$ of the image on $B_4$ of the random walk path from the first time it hits~$B_1$ until the last time it visits $B_2$ during one excursion which starts from $y$ and ends at $w$, cf.~\eqref{eq:abs-cont}.

By the Markov property of $X$, conditionally on $\zeta=(\zeta_i)_{i\geq1}$ where $\zeta_i = \zeta_i(\widehat{X},B_2,B_3))$, $(Z_i)_{i\geq1}$ is a Markov process with transition densities $(g_{\zeta_i})_{i\geq1}$ with respect to $\mu.$ Here and in the rest of the section, we identify $z\mapsto g_{\zeta_i}(z)$ with the function $(z',z) \in \Sigma \times \Sigma \mapsto g_{\zeta_i}(z),$ since the transition densities of $Z$ only depend on the second variable. This is owed to the fact that, conditionally on $\zeta,$ the random variables $(Z_i)_{i\geq1}$ are independent. With regards to fitting the setup of \eqref{eq:abs-cont}, the latter property also makes specifying $Z_0$ obsolete (for definiteness, the reader may wish to choose some $z\in \Sigma$ and set $Z_0 =z$).

%The function $g,$ the set $\Sigma$ and the measure $\mu$ also depend on the choice of the sets $B_1,B_2$ and $B_3$, but we chose to omit them in order to lighten the notation. 

For $(y,w)\in{\partial B_2\times\partial B_3^c},$ the measure $\mu$ and the function $g_{(y,w)}$ stay the same when replacing in their definition the image $\widehat{X}$ of the random walk under $\mathbf{P}_{\cdot}$  by the random walk $X$ on $\Z^d$ under the probability $P_{\bar{e}_{B_2}}.$ In particular, conditionally on $\lambda=(\lambda_i)_{i\geq1},$ $(W_i)_{i\geq1}$ is a Markov process with transition densities $(g_{\lambda_i})_{i\geq1}$ with respect to $\mu.$  We refer to \cite[Section~3]{AlvesPopov} for a detailed account, with illustrative figures, of  clothesline process, excursions and the resulting density in the case of random interlacements.

With this setup, which fits the framework of Section~\ref{sec:softlocaltimes}, cf.~around \eqref{eq:abs-cont}, we proceed to explain how to use Proposition~\ref{pro:couplingSLT} to approximate the random walk excursions $(Z_i)_{i\geq1}$ in \eqref{e:Z-i-real} by some random interlacements excursions independent of the walk outside of $B_3.$ To this end, assume ${\mathbf{P}}_{\boldsymbol{0}}$ to be suitably extended as to carry, independently of $X$, a family $\chi$ as appearing in \eqref{e:slt-Pext} and an independent clothesline process $\til{\lambda}= (\til{\lambda}_i)_{i\geq1}$ having the same law as $\lambda$ in \eqref{eq:defclotheslineri}. Now consider for each sequence $x=(x_n)_{n\geq0}$ in $\Z^d$ the map
\begin{equation}\label{e:Y-def-outside} 
Y_i(x,B_2,B_3)=\begin{cases}
x_{[0,R_{1}(x,B_2,B_3)]} &\text {if } i = 0,\\
x_{[D_i(x,B_2,B_3),R_{i+1}(x,B_2,B_3)]}, &\text {if } i \geq {1},\end{cases}
\end{equation}
assuming $R_{i+1}(x,B_2,B_3)< \infty$; in case $R_{i+1}(x,B_2,B_3)= \infty$ for some $i \geq 0$ the right endpoint is excluded in the corresponding formula for $Y_i$ in \eqref{e:Y-def-outside}, and by convention $Y_i= \emptyset$ if $D_i(x,B_2,B_3)=\infty$.
Thus, $Y_i(x,B_2,B_3)$ represents the part of $x$ occurring before the clothesline $\zeta_1(x,B_2,B_3)$ or between the clotheslines $\zeta_{i}(x,B_2,B_3)$ and $\zeta_{i+1}(x,B_2,B_3).$ We abbreviate $Y_i=Y_i(\widehat{X},B_2,B_3).$ Importantly, conditionally on $\zeta,$ the processes $(Z_i)_{i\geq1}$ and $(Y_i)_{i\geq0}$ are independent. %, where we fix some deterministic excursion $Z_0\in{\Sigma}$. 
Thus, conditionally on $\zeta,$ $\tilde{\lambda}$ and $(Y_i)_{i\geq0},$ $(Z_i)_{i\geq1}$ is still a Markov chain with transition densities $(g_{\zeta_i})$ with respect to $\mu.$ 
Therefore, applying Proposition~\ref{pro:couplingSLT} conditionally on $(\zeta_i)_{i\geq1},$ $(\til{\lambda}_i)_{i\geq1},$ and $(Y_i)_{i\geq0}, $ we can construct under the extended measure ${\mathbf{P}}_{\mathbf{0}}$ for every integer $T \geq 1$ and initial state $\til{Z}_0=\til{z}$ a Markov chain
\begin{equation}
\label{eq:deftildeZi}
	\big(\til{Z}_i\big)_{i\geq0}=\big(\til{Z}_i(X,\chi,\til{\lambda},T,B_1,B_2,B_3)\big)_{i\geq0}
\end{equation}
with transition densities $\til{g}_i=g_{\til{\lambda}_i},$ $i\geq 1$ and satisfying~\eqref{eq:ifthen}. 
Note that since the transition densities $(z',z)\mapsto g_{\til{\lambda}_i}(z)$ do not depend on the first variable, and since both $(Z_i)_{i\geq1}$ and $(Y_i)_{i\geq0}$ do not depend on $B_4$ (as long as it contains $B_3$), $(\til{Z}_i)_{i \geq 1}$  does not depend on the choice of the initial excursion $\tilde{z},$ which we will henceforth omit, nor on the choice of $B_4.$  %\textcolor{orange}{J'aime pas cette formulation ``does not depend on $B_4$,'' qui ne veut rien dire. Pour le moment une fois que $B_3$ est fixe comme $B_4$ est concentrique et a diametre $N$ il est completemente fixe}
For later reference, we denote by $G^{\rm{RW}}$ (corresponding to $G$ in \eqref{eq:new-SLT}) the soft local times associated to $Z$, that is
\begin{equation}
\label{eq:softlocattimeRW}
G_{m}^{\rm{RW}}(z)=\sum_{k=1}^m\widehat{\xi}_kg_{\zeta_k}(z)\text{ for all }z\in{\Sigma} \text{ and } m \geq 1;
\end{equation}
cf.~\eqref{e:slt-Pext} regarding $\hat{\xi}_k$ (part of $\chi$), which are independent of $Z$. 
In view of \eqref{eq:defg}, $G_{m}^{\rm{RW}}(z)$ actually only depends on its argument through the start- and endpoint of $z$. The clothesline process $\zeta$ corresponding to $Z$ below \eqref{e:Z-i-real} is a Markov chain, as follows readily from the strong Markov property. One can show, see Lemma~\ref{lem:invarexitchain} and \eqref{eq:ident}, that the stationary distribution of $\zeta$ is given by $\bar{e}_2^3(x)P_{x}(X_{T_{B_3}} = y)$ for $(x,y) \in (\partial B_2\times\partial B_3^c)$, where $\bar{e}_2^3= \bar{e}_{B_2}^{B_3}$ denotes the normalized equilibrium measure of $B_2$ relative to $B_3$ (see below \eqref{e:cap-2point} for notation). Abbreviating ${P}_{\mu}= \sum \mu(x){P}_x$ and similarly $\mathbf{P}_{\mu}$, writing $\overline{\mathbf{e}}_{2}^{3}$ for the projection of $\bar{e}_2^3$ onto $\mathbf{T}$ (i.e.~the measure such that $\overline{\mathbf{e}}_{2}^{3}\circ \pi(x) = \overline{e}_{B_{2}}^{B_3}(x)$ for all $x\in{B_4}$), %, which is well-defined when $r_3 < N$), 
 and letting 
\begin{equation}\label{eq:gbar-def}
\bar{g}(z)  =  \mathbf{E}_{\bar{\mathbf{e}}_2^3}[g_{\zeta_1}(z)]\ (=\mathbf{E}_{\bar{\mathbf{e}}_2^3}[g_{\zeta_0}(z)]),
\end{equation} 
(where $\zeta_0$ is declared as in \eqref{eq:defclothesline} but with $i=0$), which amounts to the average of the random
variable $\xi\mapsto g_{\xi}(z)$ of interest under the stationary distribution for the process $\zeta$, it then readily follows that $m\bar{g}(z)$ is the expectation of $G_m^{\rm{RW}}(z)$ starting from stationarity. 

The process $\tilde{Z}$ introduced in \eqref{eq:deftildeZi} by means of Proposition~\ref{pro:couplingSLT} will be the basis of the construction of the short range process ${\omega}^{(x)}$ (and $\ell^{(x)}$) from Theorem~\ref{thm:rwshortrange}, with $x$ denoting the common centre of $B_1$-$B_3$. As one of its central features, which will eventually give rise to the finite-range property \eqref{e:loc-range}, the process $\tilde{Z}$ is independent of $X$ outside of $B_3,$ as argued next.

\begin{lemma}%[under $\widehat{\mathbf{P}}_{\boldsymbol{0}}$]
\label{lem:independenceforfixedsets}
	For all boxes $B_1\subseteq B_2\subseteq B_3 \subset Q_{2N}(0)$ with diameter at most $N-1$ and each $T \geq 1$, the process $\til{Z}= \big(\tilde{Z}_i(X,\chi,\til{\lambda},T,B_1,B_2,B_3)\big)_{i\geq1}$ is independent of $\zeta= (\zeta_i(\widehat{X},B_2,B_3))_{i\geq1}$ and $ Y=(Y_i(\widehat{X},B_2,B_3))_{i\geq0}.$
\end{lemma}
\begin{proof}
By Proposition \ref{pro:couplingSLT}, conditionally on $\til{\lambda},$ $\zeta$ and $Y,$ $\til{Z}$ is a Markov chain with transition densities $(g_{\til{\lambda}_i})_{i\geq1}$ with respect to $\mu,$ and so its conditional law depends only on $\tilde{\lambda}.$ Since $\til{\lambda}$ is independent of $\zeta$ and $Y,$ it thus follows that $\til{Z}$ is independent of $\zeta$ and $Y.$
\end{proof}

Thus, up to controlling their respective soft local times, $(\tilde{Z}_i)_{i\geq1}$ are excursions close to $(Z_i)_{i\geq1}$ by Proposition~\ref{pro:couplingSLT}, but with an extra independence property, similarly as the processes $\I_k^u$ from \cite[Proposition~5.3]{softlocaltimes}. There are however two main differences in our construction: first the soft local times method is used conditionally on the clothesline process $\zeta$ instead of unconditionally as in \cite{softlocaltimes}, following ideas from \cite{AlvesPopov}, and second, using Proposition~\ref{pro:couplingSLT}, the process $(Z_i)_{i\geq1}$ is defined directly in terms of the Markov chain $X$ under $\mathbf{P}_{\boldsymbol{0}}$ (and additional independent randomness) instead of a process on some other probability space having the same law as $(Z_i)_{i\geq1}$, as in \cite{softlocaltimes, AlvesPopov}. These two changes in the method serve the same purpose: we can construct simultaneously the processes $(\til{Z}_i)_{i\geq1}$ for different choices of sets $B_1\subseteq B_2\subseteq B_3$ that are not necessarily disjoint and all have the desired independence properties. Moreover, for each three sets $B_1\subseteq B_2\subseteq B_3$, they are close to the excursion process $(Z_i)_{i\geq1}$ associated to the same initial chain $X.$ We refer to Proposition~\ref{pro:clotheslines} for the exact statement.

We now adapt the previous construction to the case of random interlacements. Suppose that $\PI$ is extended with the same independent processes $(\til{\lambda},\chi)$ as in the case of $
{\mathbf{P}}_{\boldsymbol{0}}$, cf.~above \eqref{e:Y-def-outside} regarding their respective laws. Conditionally on $\lambda$ (see \eqref{eq:defclotheslineri}) and $Y_i(X^j,B_2,B_3)$ for all $ 0 \leq i\leq T^j$ and $j\geq1,$ see \eqref{e:Y-def-outside}, the process $(W_i)_{i\geq1}$ introduced in \eqref{eq:defclotheslineri} is a Markov chain with transition densities $(g_{\lambda_i})$ with respect to $\mu.$ Thus, applying Proposition~\ref{pro:couplingSLT} with any choice of $\til{W}_0$, one obtains for every $T \geq 1$ a Markov chain 
\begin{equation}
\label{eq:deftildeWi}
	\til{W}=\big(\til{W}_i\big)_{i\geq1}=\big(\til{W}_i(\omega,\chi,\til{\lambda},T,B_1,B_2,B_3)\big)_{i\geq1}
\end{equation}
with transition densities $(g_{\til{\lambda}_i})$ with respect to $\mu,$ (independent of $\til{W}_0$). The soft local time associated to $W=(W_i)_{i \geq1}$ is given by
\begin{equation}
\label{eq:GnRI}
G_{m}^{\rm{RI}}(z)=\sum_{k=1}^m \hat{\xi}_kg_{\lambda_k}(z)\text{ for all }z\in{\Sigma} \text{ and } m \geq 1,
\end{equation}
(and similarly for $\til{W}$ with $\til{\xi}$, $\til{\lambda}$ in place of $\hat{\xi}$, $\lambda$). %,
%where $(\hat{\xi}_k)_{k\geq1}$ is an i.i.d.~sequence of exponential random variables with parameter one, independent of $\omega.$ 
By \cite[Lemma 6.1]{MR3563197}, starting from the invariant distribution of $\lambda,$ see \eqref{eq:defclotheslineri}, the expectation of $G_m^{\rm RI}$ is equal to $m\bar{g}(z),$ see \eqref{eq:gbar-def}, i.e.~it is equal to the expectation of $G_m^{\rm RW}$ starting from the invariant distribution of $(\zeta_i)_{i\geq0}$; 
to see this recall $\zeta$ from \eqref{eq:defclothesline} (which can also be defined for $i=0$) and note that the expectation in \eqref{eq:gbar-def} equals $E_{\bar{e}_2^3}[g_{\zeta_0(X,B_2,B_3)}(z)]$. Moreover one proves similarly as in Lemma~\ref{lem:independenceforfixedsets} above that $\til{W}$ is independent of the clothesline process $\lambda$ and the excursions $\{Y_i(X^j,B_2,B_3) : 0\leq i\leq T^j, j\geq1\}$.

\subsection{Simultaneous approximation by short range excursions}
\label{sec:costshortrange}

Towards proving Theorem~\ref{thm:rwshortrange}, we now apply the construction of \S\ref{sec:constructionexcursions} jointly to the sets 
\begin{equation}\label{e:B_i-choice}
B_k=Q(x,r_k), \, k=1,\dots, 4, \, \text{ where } 0<r_1<r_2<r_3< r_4= N
\end{equation}
(with $Q(x,r)=Q_r(x)$ referring to the boxes introduced at the beginning of Section~\ref{sec:notation}) as $x \in Q_N(0)$ varies. In particular, $B_3\subset Q_{2N}(0)$ and has diameter at most $N-1$ as assumed in~\S\ref{sec:constructionexcursions}. The resulting sequences $\til{Z}^{(x)}=(\til{Z}^{(x)}_j)_{j \geq 1} $, see \eqref{eq:deftildeZi} and \eqref{eq:deftildeZji} below, will provide the excursions in terms of which the processes $\omega^{(x)}$, $x \in Q_N(0)$ will later be defined. 
For each $x\in{Q_N(0)}$ and $j\geq 1$, recalling \eqref{eq:defexcrw} for notation, let
\begin{equation}\label{eq:defZji-x}
Z_j^{(x)}(r_1,r_2,r_3)\stackrel{\text{def.}}{=} Z_j(\widehat{X},Q(x,r_1),Q(x,r_2),Q(x,r_3)),
\end{equation}
where $\widehat{X}$ is as in \eqref{e:X-hat} with $B_4=Q(x,N)$,
and write $\zeta^{(x)}= (\zeta^{(x)}_j)_{j\geq 1}$ for the associated clothesline process, i.e.~$\zeta_j^{(x)} = \zeta_j(\widehat{X},Q(x,r_2), Q(x,r_3))$, cf.~\eqref{eq:defclothesline}, and $g_{\zeta^{(x)}}^{(x)}$, the transition densities of $(Z_j^{(x)})_{j\in{\mathbb{N}}}$, cf.~\eqref{eq:defg} and \eqref{eq:defg2}. Let $\til{\mathbf{P}}_{\boldsymbol{0}}$ be an extension of $\mathbf{P}_{\boldsymbol{0}}$ carrying the additional independent processes $\til{\omega}^{(x)},$ $\chi^{(x)},$ $x\in{Q_N(0)},$ having the following distributions. For each $x$, $\til{\omega}^{(x)}$ is an interlacement process, i.e.~it has the same law as $\omega$ above \eqref{eq:definterprocess}. Its induced clothesline process for the choice \eqref{e:B_i-choice} will be denoted by $(\til{\lambda}_j^{(x)})_{j\geq1}$. The process $\chi^{(x)}$ is specified by \eqref{e:slt-Pext} with underlying measure space $(\Sigma,\mu)\equiv (\Sigma^{(x)},\mu^{(x)})$ in \eqref{e:mu-ref} corresponding to \eqref{e:B_i-choice}.

Now, applying Proposition \ref{pro:couplingSLT} in the same manner as above \eqref{eq:deftildeZi} but simultaneously for each $x\in{Q_N(0)}$ yields, for each such $x$  and integers $j,T\geq 1$, the random variables
\begin{equation}
\label{eq:deftildeZji}
\til{Z}_j^{(x)}(T,r_1,r_2,r_3)\stackrel{\text{def.}}{=} \til{Z}_j(X,\chi^{(x)},\til{\lambda}^{(x)},T,Q(x,r_1),Q(x,r_2),Q(x,r_3)).
\end{equation}
As in~\eqref{eq:softlocattimeRW}, the soft local times associated to $Z^{(x)}=(Z_j^{(x)})_{j \geq 1}$ will be denoted by
$G_{m,x}^{\rm{RW}}(z)$, for $m \geq 1$ and $z \in \Sigma^{(x)}$, with inherent i.i.d.~exponential random variables $\hat{\xi}_j^{(x)},$ $j \geq 1$, carried by $\chi^{(x)}$, cf.~\eqref{e:slt-Pext}. The expectation under the stationary distribution of $g_{\zeta_1^{(x)}}^{(x)}(z)$ is written as $\bar{g}^{(x)}(z)$, $z \in \Sigma^{(x)}$, similarly as in \eqref{eq:gbar-def}.

Finally, for each $x\in{Q_N(0)},$ replacing every occurrence of $Z_j,$ $\zeta_j,$ $\til{Z}_j$ and $G_{m}^{\rm{RW}}$ above by $W_j,$ $\lambda_j,$ $\til{W}_j$ and $G_{m}^{\rm{RI}},$ one similarly defines under the extended measure $\tildePI$ (with the same extension $(\til{\omega}^{(x)}, \chi^{(x)})_{x\in{Q_N(0)}}$ as $\til{\mathbf{P}}_{\boldsymbol{0}}$) the processes  $W_j^{(x)},$ $\lambda_j^{(x)},$ $\til{W}_j^{(x)}$ and $G_{m,x}^{\rm{RI}}$ for random interlacements, which correspond to the processes introduced in \eqref{eq:defclotheslineri}, \eqref{eq:deftildeWi} and \eqref{eq:GnRI} for the choices of boxes $B_k$ in \eqref{e:B_i-choice}. As a result, by Proposition \ref{pro:couplingSLT}, for every $x\in{Q_N(0)},$ 
\begin{equation}\label{e:law-Ztilde-x}
 \text{$(\tilde{Z}_j^{(x)})_{j\geq1}$  has the same law conditionally on $\til{\lambda}$ as $({W}_j^{(x)})_{j\geq1}$ conditionally on ${\lambda}$.}
 \end{equation}

In the sequel, mimicking the notion introduced above Theorem~\ref{thm:rwshortrange}, a collection
$\Pi=(\Pi^{(x)}_j)_{x\in{Q_N(0)},j\geq 1}$ of $\Z^d$-valued random paths is said to have range $R$  in $\Z^d$ (resp.~in $\mathbf{T})$ if $(\Pi_j^{(x)})_{j\geq 1}$ is independent of $ (\Pi^{(y)}_j)_{j\geq 1, \, y\in{Q_N(0)\setminus Q(x,R)}}$ (resp.~of $(\Pi^{(y)}_j)_{j\geq 1, \,  \boldsymbol{y}\notin{Q(\boldsymbol{x},R)}}$) for each $x\in{Q_N(0)}.$ Here following the convention from the beginning of \S\ref{sec:notation}, we used the symbol $\boldsymbol{y}=\pi(y),$ and $\boldsymbol{y}\notin{Q(\boldsymbol{x},R)}$ means that $y$ is such that $\boldsymbol{y}\notin{Q(\boldsymbol{x},R)}.$
\begin{proposition}[$N, T \geq 1 , \, 1\leq r_1<r_2<r_3< N$]\label{pro:clotheslines}

\medskip
\noindent\begin{itemize}
\item[i)] The excursion process $\big(\til{Z}^{(x)}_j(T,r_1,r_2,r_3)\big)_{x\in{Q_N(0)},j\geq1},$ resp.\  $\big(\til{W}^{(x)}_j(T,r_1,r_2,r_3)\big)_{x\in{Q_N(0)},j\geq1},$ have range at most $2r_3$ in $\mathbf{T},$ resp.\ in $\Z^d.$
\item[ii)]  For all $F\subset Q_N(0),$ integer $m\geq 1$ and $\eps\in{(0,1)}$ such that $m\epsilon\geq 3$ and $m_{+} \leq T$ where $m_{\pm}= \lceil m(1\pm\epsilon)\rceil$, one has with sup ranging over $x\in{F}$, $z\in{\Sigma^{(x)}}$ and $n\in[m_{-},m_{+}]$ below,
	\begin{equation}
	\label{eq:clotheslines}
	\begin{split}
	&\til{\mathbf{P}}_{\boldsymbol{0}}\big(\{\til{Z}_j^{(x)}: j\leq m_{-}\} \subseteq \{Z_j^{(x)}: j\leq m\}\subseteq \{\til{Z}_j^{(x)}: j\leq m_{+}\},\ \forall \ x\in{F}\big)
	\\ &\qquad \begin{array}{rl}\geq  1 - |F|r_3^{2d}\displaystyle \sup_{x,z,n}  & \hspace{-1ex} \Big\{\til{\mathbf{P}}_{\boldsymbol{0}}\big(|G_{n,x}^{\rm{RW}}(z)-n \bar{g}^{(x)}(z)|\geq \textstyle \frac{\eps}{4} n \bar{g}^{(x)}(z)\big) \\[0.5em]
	& \ \  +  2 \tildePI\big(|G_{n,x}^{\rm{RI}}(z)-n \bar{g}^{(x)}(z)|\geq \textstyle \frac{\eps}{4} n \bar{g}^{(x)}(z)\big) \Big\},\end{array}\\
	&\tildePI\big(\{\til{W}_j^{(x)}: j\leq m_{-}\} \subseteq \{W_j^{(x)}: j\leq m\}\subseteq \{\til{W}_j^{(x)}: j\leq m_{+}\},\ \forall \ x\in{F}\big)
	\\&\qquad \geq 1 - 3|F|r_3^{2d}\sup_{x,z,n}\tildePI\big(|G_{n,x}^{\rm{RI}}(z)-n \bar{g}^{(x)}(z)|\geq \textstyle\textstyle  \frac{\eps}{4} n \bar{g}^{(x)}(z)\big) .
	\end{split}
	\end{equation}
\end{itemize}
\end{proposition}
\begin{proof}
	We first give the proof for the random walk. We start with item $i).$ For every $x\in Q_N(0)$, the family $(Z_j^{(y)},\zeta_j^{(y)})_{j\geq1,\boldsymbol{y}\notin{Q(\boldsymbol{x},2r_3)}}$ only depends on the excursions of $X$ from $Q({y},r_2)$ to $Q({y},r_3)$ for $ y\notin{Q(\boldsymbol{x},2r_3)},$ and is thus measurable with respect to $\big(Y_j(\widehat{X},Q(x,r_2),Q(x,r_3))\big)_{j\geq0}$ defined in \eqref{e:Y-def-outside}, with $\widehat{X}$ as in \eqref{e:X-hat} for $B_4=B(x,N)$. Hence, it is independent of $(\tilde{Z}_j^{(x)})_{j \geq 1}$ by Lemma~\ref{lem:independenceforfixedsets}. 
On account of \eqref{eq:deftildeZi} and Remark~\ref{R:inverse},\ref{R:dependencytildeZ} (the latter implies that $\tilde{Z}_j$ depends on $X$ only through $(Z_j)_{j \geq 1}$ as well $(\zeta_j)_{j\geq1}$ via the densities $(g_{\zeta_j})_{j\geq1}$, the process $(\tilde{Z}_j^{(y)})_{j\geq1}$ depends by construction only on $(Z_j^{(y)},\zeta_j^{(y)})_{j\geq1},$ $\chi^{(y)}$ and $\til{\lambda}^{(y)}$, for all $y\in{Q_N(0)}$. Using independence of $(\chi^{(y)}, \til{\lambda}^{(y)})$ as $y$ varies, it follows overall that $(\tilde{Z}_j^{(x)})_{j\geq 1}$ is independent of $(\tilde{Z}_j^{(y)})_{j\geq1,\boldsymbol{y}\notin{Q(\boldsymbol{x},2r_3)}},$ as claimed.
	
We now show $ii).$ For each $x\in{Q_N(0)},$ let $(\til{G}_{j,x}^{\rm{RI}})_{j\geq 1}$ denote the soft local times corresponding to $(\tilde{G}_j)_{j \geq 1}$ in \eqref{eq:new-SLT} when constructing $(\til{Z}^{(x)}_j)_{j\geq1}$ in \eqref{eq:deftildeZi} by means of Proposition~\ref{pro:couplingSLT}. In view of the choice above \eqref{e:Y-def-outside}  $(\lambda,\hat{\xi})$ and $(\tilde{\lambda},\tilde{\xi})$ have the same law, and thus $\til{G}_{n,x}^{\rm{RI}}(z)$ has the same law as ${G}_{n}^{\rm{RI}}(z)$ in \eqref{eq:GnRI}. Proposition~\ref{pro:couplingSLT} (see \eqref{eq:ifthen}) now gives that for all $m\in \N$ and $\epsilon\in (0,1)$, the event appearing in the first line of \eqref{eq:clotheslines} is implied by
\begin{equation*}
%\begin{gathered}
	\bigcap_{x \in F} \bigcap_{z\in \Sigma^{(x)}}\left\{\til{G}^{\rm{RI}}_{m_{-},x}(z) \leq G_{m,x}^{\rm{RW}}(z) \leq \til{G}^{\rm{RI}}_{m_{+},x}(z)\right\}.
	%\\\subset	
	%\left\{\{\til{Z}_j^{(x)}: j\leq m_{-} \} \subseteq \{Z_j^{(x)}: j\leq m\}\subseteq \{\til{Z}_j^{(x)}: j\leq m_{+}\},\ \forall \ x\in{F}\right\}.
%\end{gathered}
\end{equation*}
Moreover, for each $m\in\N,$ $\epsilon\in (0,1)$ such that $m\epsilon\geq 3$, recalling that $m_{\pm}= \lceil m(1\pm\epsilon)\rceil$, the latter event (for fixed $x\in{F}$ and $z\in{\Sigma^{(x)}}$) is implied by
\begin{multline*}
\left\{|G_{m,x}^{\rm{RW}}(z)-m\bar{g}^{(x)}(z)|\leq \frac{\eps}{4} m\bar{g}^{(x)}(z)\right\}\cap\left\{|\til{G}_{m_{+},x}^{\rm{RI}}(z)-m_{+} \bar{g}^{(x)}(z)|\leq \frac{\eps}{4} m_{+} \bar{g}^{(x)}(z)\right\}
\\\cap\left\{|\til{G}_{m_{-},x}^{\rm{RI}}(z)-m_{-} \bar{g}^{(x)}(z)|\leq \frac{\eps}{4} m_{-} \bar{g}^{(x)}(z)\right\}.
%\\\subset \left\{\til{G}^{\rm{RI}}_{m_{-},x}(z) \leq G_{m,x}^{\rm{RW}}(z) \leq \til{G}^{\rm{RI}}_{m_{+},x}(z)\right\}.
\end{multline*}
The assertion \eqref{eq:clotheslines} now follows by a union bound over $x$ and $z$, upon noting that $g_{\zeta^{(x)}}^{(x)}(z)$ as defined in \eqref{eq:defg} with $B_k$ as in \eqref{e:B_i-choice} only depends on the first and last point of the excursion $z$. The same is thus true of $G_{m,x}^{\rm{RW}}(z),$ $\til{G}^{\rm{RI}}_{m_{\pm},x}(z)$ and  $\bar{g}^{(x)}(z)$, leading to the factor $r_3^{2d}$ in \eqref{eq:clotheslines}. 

The proof for random interlacements is similar, using this time in item $i)$ that the family $(W_j^{(y)},\lambda_j^{(y)})_{j\geq1,y\in{Q(0,N)\setminus Q(x,2r_3)}}$ depends only on the excursions of the trajectories of random interlacements from $Q(y,r_2)$ to $Q(y,r_3)(\subset Q_{2N}(0))$ for $y\in{Q(0,N)\setminus Q(x,2r_3)}$ (if a trajectory does not hit $Q(y,r_2)$ its excursions are just the full trajectory), and is thus measurable with respect to $\{Y_i(X^j,B_2,B_3) : 0\leq i\leq T^j, j\geq1\},$ see above \eqref{eq:defTj} and \eqref{e:Y-def-outside}. \end{proof}

The interlacement processes ${\omega}^{(x)}$ appearing in Theorem \ref{thm:rwshortrange} will be constructed using the short range excursions $(\tilde{Z}_j^{(x)})$ (in case of $\til{\mathbf{P}}_{\boldsymbol{0}}$) or $(\tilde{W}_j^{(x)})$ (in case of $\tildePI$) from Proposition \ref{pro:clotheslines}. Items i) and ii) above thereby roughly correspond to \eqref{e:loc-range} and \eqref{eq:couplelltilell}/\eqref{eq:couplingshortrangeinter}. The latter requires good control on the proximity between the short-range excursion processes $(\til{Z}_j^{(x)})$/$(\til{W}_j^{(x)})$ and the initial excursion processes $(Z_j^{(x)})$/$(W_j^{(x)})$, which is the object of the next two lemmas. 

In view of \eqref{eq:clotheslines}, one central aspect is showing that the relevant soft local times concentrate around their mean. For unconditional soft local times, i.e.~without conditioning on the clothesline process as in the construction of $\tilde{Z}$ in \eqref{eq:deftildeZi}, this was first achieved for random interlacements in \cite[Section 6]{softlocaltimes}, and then for random walk if $\delta$ is large enough in \cite[Lemma 2.1]{MR3126579}. For conditional soft local times as in \eqref{eq:GnRI}, concentration around the mean was proved for interlacements in \cite[Proposition 4.1]{AlvesPopov}. In the following lemma, the proof of which is relegated to Appendix~\ref{sec:app}, we extend this concentration of conditional soft local times to the random walk case, cf.~\eqref{eq:softlocattimeRW}. In doing so we also give a shorter proof of \cite[Proposition 4.1]{AlvesPopov} when the sets $\partial B_2$ and $\partial B_3$ are well-separated, as implied by the parameter $\delta >0$ below.

In view of \eqref{e:B_i-choice} and with a slight abuse of notation, $G_{m}^{\xi}(\cdot)$, $\xi \in \{ {\rm RW, RI}\}$, refers in the sequel to the quantities introduced in \eqref{eq:softlocattimeRW} and \eqref{eq:GnRI} but with the choice $B_k=Q(x,r_k)$ for arbitrary $x \in Q_N(0)$ (implicit below; note however that translation invariance is spoiled under ${\mathbf{P}}_{\boldsymbol{0}}$ so one cannot simply set $x=0$). In particular, $G_{m}^{\xi}(\cdot)$ depends on the parameters $r_k$, $1\leq k \leq 3$. Recall the function $\bar{g}$ from \eqref{eq:gbar-def}.

\begin{lemma}\label{lem:concsoftrwri}
For all $\delta\in{(0,1)},$ there exist $c=c(\delta)$, $C=C(\delta)$, such that for all $N \geq 1$, all $0<r_1<r_2<r_3< N$ with $r_{k+1}\geq (1+\delta)r_{k}$, $k=1,2$, and all $\epsilon\in (0,1),$ $m \geq 1$, $z\in \Sigma$, 
\begin{align*}
{\mathbf{P}}_{\boldsymbol{0}}\left(|G_m^{\rm{RW}}(z) -m\bar{g}(z) |\geq \epsilon m\bar{g}(z)\right) &\leq Cm\exp\big(-c \sqrt{\epsilon^2 m}\big),\\
{\PI}\left(|G_m^{\rm{RI}}(z) - m\bar{g}(z)|\geq \epsilon m\bar{g}(z)\right)&\leq C\exp(-c\epsilon^2 m).
\end{align*}
\end{lemma}

Proposition~\ref{pro:clotheslines} and Lemma~\ref{lem:concsoftrwri} deal with a fixed number of excursions (parametrized by $m$).
For this to be successfully deployed, one needs to show that the actual number of excursions, which is random and given by $\Excrw$/$\Excri$ see \eqref{eq:defnumberexcursionRW}/\eqref{eq:defnumberexcursionRI}, suitably concentrates. Recalling the relevant notation from \S\ref{sec:notation}, see below \eqref{e:cap-2point}, let
\begin{equation}
\label{e:avg-number}
M=M(B_2,B_3)\stackrel{\text{def.}}{=} \text{cap}_{B_3}(B_2);% \frac{N^d}{\mathbf{E}_{\nu}[R_{1}]} \text{ (with $R_1= R_{1}(\widehat{X}, B_2,B_3)$)};
\end{equation}
the quantity $uM$ represents an `asymptotic mean' number of excursions until the terminal time $u N^d$ for the walk; see also \eqref{e:avg-number'} for an alternative formula for $M$ conveying this intuition. 
\begin{lemma}\label{lem:conc}
	For all $\delta\in{(0,1)},$ there exist $c,C \in (0,\infty)$ depending on $\delta$ so that for all $ N \geq1$,  $u>0$, $\epsilon\in (0,1),$ $1\leq r_2< r_3< N$ with $r_3\geq (1+\delta) r_2,$ and $B_2$, $B_3$ as in \eqref{e:B_i-choice},	
	 \begin{align*}
	\mathbf{P}_{\boldsymbol{0}}\big(|\Excrw(B_2,B_3,u) - uM|\geq \epsilon u M \big) &\leq C\lceil uM\rceil\exp\big( -c \sqrt{\epsilon^2uM} \big),\\
 \PI\big(|\Excri(B_2,B_3,u) - uM| \geq \epsilon u M\big) &\leq C\exp\big(-c\epsilon^2u M\big).
	\end{align*}
\end{lemma}

Lemma \ref{lem:conc} is essentially proved in \cite[Proposition 9.1]{MR3563197} for the random walk and in \cite[Proposition~9.3]{MR3563197} for random interlacements, but the bounds obtained therein are not explicit. We prove Lemma \ref{lem:conc} in App.\ \ref{sec:app} using general large deviation results for random walk excursions and random interlacements trajectories, see Propositions \ref{prop:largedeviationrw} and \ref{prop:largedeviationri}, from which Lemma \ref{lem:concsoftrwri} follows as well.

\subsection{Proof of Theorem~\ref{thm:rwshortrange}}
\label{sec:proofthmshortrange}

With Proposition~\ref{pro:clotheslines} and Lemmas~\ref{lem:concsoftrwri} and \ref{lem:conc} at hand, we are now ready to proceed to the:
\begin{proof}[Proof of Theorem~\ref{thm:rwshortrange}]
We focus on the case of the random walk $X$ under $\mathbf{P}_{\boldsymbol{0}}$ and discuss the necessary modifications to accommodate random interlacements at the end of the proof. With $R$ as appearing in the statement of Theorem~\ref{thm:rwshortrange}, let $r_1=R,$ $r_2=r_1(1+\delta')$ and $r_3=r_2(1+\delta'),$ where $\delta'>0$ is such that $(1+\delta')^2=1+\delta.$ In particular, $r_3\leq N$ by assumption on $R$ and this fixes the value of $M=M(Q(x,r_2),Q(x,r_3))$ in \eqref{e:avg-number}, which by a standard capacity estimate satisfies \begin{equation}
\label{e:M-bound}
M\geq \mathrm{cap}(Q(x,r_2))
\geq cr_2^{d-2}\geq c'R^{d-2}.
\end{equation}
Throughout the proof we write $\eps'=\frac\eps 3$ for a given $\varepsilon \in (0,1)$.
We first introduce an approximation for (the increments of) the local times of $X$, which count a fixed number of excursions.
With $Z_j^{(x)}$ as introduced below \eqref{e:B_i-choice} (a (finite) excursion in $\Z^d$), writing $l_j^{(x)}$ for its total length, so that $Z_j^{(x)}=\{ Z_j^{(x)}(n) :  0 \leq n \leq l_j^{(x)}\}$, we set
\begin{equation}\label{e:hat-ell}
\hat{\ell}_{y,[v,u]}^{(x)}=\sum_{j=\lceil vM\rceil }^{\lceil uM\rceil}\sum_{n=0}^{l_j^{(x)}}1\{Z_j^{(x)}(n) =y\}, \  \text{ for all $x\in{Q_N(0)}$,  $y\in{Q_R(x)}$ and $0<v<u$,}
\end{equation}
which counts the total number of visits to $y$ by all excursions with index $j$ between $\lceil vM\rceil$ and $\lceil uM\rceil$. Similarly, let ${\ell}_{{y},[v,u]}= {\ell}_{{y},u}- {\ell}_{{y},v}$, for ${y} \in \mathbf{T}$, which is the relevant quantity appearing in \eqref{eq:couplelltilell}. By \eqref{eq:deflocaltimes} and \eqref{eq:defnumberexcursionRW}, ${\ell}_{{y},[v,u]}$ admits a similar representation as \eqref{e:hat-ell}, but counting excursions with label $j \in [ \Excrw(B_2,B_3,v), \Excrw(B_2,B_3,u)]$ instead, where $B_i=Q(x,r_i)$ as in \eqref{e:B_i-choice}. We now claim that  for all $x\in{Q_N(0)}$, $y\in{Q(x,R)}$ and $0<v<u$,
\begin{multline}
\label{e:inclu-increments}
\big( \big\{ \Excrw(B_2,B_3,v) \geq \lceil v(1-\eps')M \rceil \big\} \\
\cap \big\{ \Excrw(B_2,B_3,u) \leq \lceil u(1+\eps')M \rceil\big\}\big)
\subset \big\{ \ell_{\boldsymbol{y},[v,u]} \leq \hat{\ell}^{(x)}_{y,[v(1-\eps'), u(1+\eps')]} \big\}
\end{multline}
(with $\boldsymbol{y}= \pi(y)$). Indeed, first notice that a possibly non-vanishing contribution to the local time increment can arise from $X_{ \cdot \wedge R_1}$, but only in case where $R_1\geq \lfloor vN^d\rfloor$ which is equivalent to $\Excrw(B_2,B_3,v)=0$ on account of \eqref{eq:defnumberexcursionRW}. This additional contribution to $\ell$ compared to $\hat{\ell}$ is owed to the fact, noted below \eqref{eq:defexcrw}, that the process $Z^{(x)}$ defined in \eqref{eq:defexcrw} neglects the very first excursion of the random walk in $B_1$ before time $R_1$. To obtain~\eqref{e:inclu-increments} one then uses the definition of $\ell$, $\hat{\ell}$ and $\mathcal{N}_{\text{RW}}$ together with the fact that the event in the first line of \eqref{e:inclu-increments} implies $\Excrw(B_2,B_3,v) >0$, which prevents $\ell_{\boldsymbol{y},[v,u]}$ from counting the very first excursion before time~$R_1$.

Combining \eqref{e:inclu-increments}, a similar inclusion yielding a reverse inequality (for which over-counting the first excursion is not an issue), applying a union bound over $x\in F (\subset Q_N(0))$ and using Lemma~\ref{lem:conc}, one thus infers that for all $0< v<u$ and $\epsilon\in (0,1)$ with $u(1-\eps')>v(1+\eps')$,
\begin{multline}\label{eq:couplingforlocaltimes}
	\mathbf{P}_{\boldsymbol{0}}
	\Big( \, \hat{\ell}^{(x)}_{y,[v(1+\eps'), u(1-\eps')]}  \leq \ell_{\boldsymbol{y},[v,u]} \leq \hat{\ell}^{(x)}_{y,[v(1-\eps'), u(1+\eps')]}, \, \text{ for all } x  \in F,y\in{Q(x,R)} \Big)\\
	\geq 1- C\lceil uM\rceil|F|\exp\big(-c \epsilon  \sqrt{v M}\big),
\end{multline}
 for positive constants $c,C$ depending only on~$\delta$. 

To proceed further, we now work under the extended measure $\til{\mathbf{P}}_{\boldsymbol{0}}$ introduced at the beginning of \S\ref{sec:costshortrange}, which will form the basis of the desired coupling. Recall the process $(\til{Z}^{(x)}_j)$ introduced in \eqref{eq:deftildeZji}
 which has range at most~$2r_3$ in $\mathbf{T}$ by Proposition~\ref{pro:clotheslines},i), and choose $T=\lceil 4u_0M\rceil,$ for a (fixed) $u_0>0$ as appearing in the statement of Theorem~\ref{thm:rwshortrange}. Mimicking \eqref{e:hat-ell}, set
\begin{equation}
\label{eq:deftildeL}
	\til{\ell}_{y,[v,u]}^{(x)} =
	\sum_{j=\lceil vM\rceil}^{\lceil uM\rceil}\sum_{n=0}^{\til{l}_j^{(x)}}1\{\til{Z}_j^{(x)}(n) =y\},  \text{ for all }x\in{Q_N(0)},\,y\in{Q_x(R)}\text{ and }0\leq v<u\leq 4u_0,
	\end{equation}
with $\til{l}_j^{(x)}$ denoting the length of $\til{Z}^{(x)}_j$. %, which inherits the finite-range property of $(\til{Z}^{(x)}_j)$. 
Combining Proposition~\ref{pro:clotheslines},ii), applied for the choices $m= \lceil uM\rceil, \lceil vM\rceil$, with Lemma~\ref{lem:concsoftrwri} and a union bound, % we have
%that if $vM\epsilon'$ is large enough, then 
it follows that for all $0< v<u\leq u_0$ and $\epsilon\in(0,1)$ such that $u(1-2\eps')>v(1+2\eps')$,
\begin{multline}\label{eq:proofsuccess}
	\til{\mathbf{P}}_{\boldsymbol{0}}\left(\begin{array}{c}\til{\ell}_{y, [v(1+2\eps'), u(1-2\epsilon')] }^{(x)} \leq\hat{\ell}_{y,[v(1+\eps'), u(1-\epsilon')]}^{(x)}, \\[0.2em] \hat{\ell}_{y,[v(1-\eps'), u(1+\epsilon')]}^{(x)} \leq \til{\ell}_{y, [v(1-2\eps'), u(1+2\epsilon')] }^{(x)}, \\[0.5em] \text{for all }  x  \in F\text{ and }y\in{Q(x,R)} \end{array}\right) \geq 1- Cr_3^{2d}\lceil uM\rceil|F|\exp\big(-c \epsilon\sqrt{v M}\big).
\end{multline}

To complete the proof, we now define a random interlacements process ${\omega}^{(x)}$ (i.e.,~satisfying \eqref{e:loc-law}), which will inherit the finite-range property of $(\til{Z}^{(x)}_j)$ (so as to satisfy \eqref{e:loc-range}) and  whose local times in $Q(x,R)$ are close to $\til{\ell}^{(x)}_{y,[0,\cdot]}$, $ y \in Q(x,R),$ up to sprinkling (thus leading to \eqref{eq:couplelltilell}).

We first construct a family $(X^{(x),\, i})_{ i\geq1, x\in{Q_N(0)}}$ of independent random walks, where $X^{(x),\, i}$ has law $P_{\bar{e}_{ B_2}}$, $B_2=Q(x,r_2)$, for every $i \geq 1$. Importantly, for each $x\in{Q_N(0)}$ the excursions by any of the walks $(X^{(x),\, i})_{i\geq1}$ between $\partial Q(x,r_1)$ and the last exit time of $Q(x,r_2)$ before exiting $Q(x,r_3)$ will be given precisely by $(\til{Z}_j^{(x)})_{j\geq 1},$ and the remaining parts of the random walks will be conditionally independent as $x\in{Q_N(0)}$ varies.

 Recall from above \eqref{eq:deftildeZji} that  $(\til{\omega}^{(x)})_{x\in{\Z^d}}$ is a family of independent random interlacements processes, each with corresponding clothesline process $\til{\lambda}^{(x)}$ associated to the choice $B_k =Q(x,r_k)$, $k=1,\dots 3$, cf.~\eqref{e:B_i-choice}. For $B\subseteq \Z^d$ we denote by $\til{\omega}_{B}^{(x)}$ the restriction of $\til{\omega}^{(x)}$ to forward (unlabeled) trajectories hitting $B,$ and started at their entrance time in $B.$ We call $(\til{X}^{(x),\, i})_{i \geq 1}$, the trajectories thereby obtained from $\til{\omega}^{(x)}_{B_2}$, corresponding to the trajectories in \eqref{eq:definterprocess} when $B=B_2$ and $\omega=\til{\omega}^{(x)}$. Note that by definition, see \eqref{eq:defclotheslineri}, each clothesline $\til{\lambda}_j^{(x)},$ $j\geq1,$ arises from a certain trajectory $\til{X}^{(x),k}.$ As part of the ranges of $(\til{X}^{(x),\, i})_{i \geq 1}$, we now define the sequence $(\til{Y}^{(x)}_j)_{j\geq1}$ as follows. Whenever $\til{\lambda}^{(x)}_j$ and $\til{\lambda}_{j+1}^{(x)}$ correspond to the same trajectory $\til{X}^{(x),\, k}$, we let $\til{Y}^{(x)}_j$ be the excursion starting from $\partial B_3^c$ until first hitting $B_2$ between the last time the clothesline $\til{\lambda}_j^{(x)}$ is visited and the first time the clothesline $\til{\lambda}_{j+1}^{(x)}$ is visited. If the clotheslines $\til{\lambda}^{(x)}_j$ and $\til{\lambda}_{j+1}^{(x)}$ correspond to two different trajectories $\til{X}^{(x),\, k}$ and $\til{X}^{(x),\, k+1}$  of $\til{\omega}^{(x)}_{B_2}$, then $\til{Y}^{(x)}_j$ is defined to be equal to the part of $\til{X}^{(x),\, k}$ after last visiting $\partial B_3^c$.  Now define recursively $\til{V}_0^{(x)}=0$ and $\til{V}_i^{(x)}=\inf\{k> \til{V}_{i-1}^{(x)}:\, \text{range}(\til{Y}_k^{(x)}) \text{ is unbounded}\}.$  Intuitively, $\til{V}_i^{(x)}$ equals the number of excursions from $B_2$ to $\partial B_3^c$ before the $(i+1)$-st walk from $\til{\omega}^{(x)}_{B_2}$ starts, and $\til{V}_{i+1}^{(x)}-\til{V}_{i}^{(x)}-1$ is precisely the number of excursions performed by this walk. Recall from above \eqref{e:mu-ref} that $\Theta$ is the cemetery point of $\Sigma,$ and intuitively corresponds to excursions which do not hit $B_1.$ Lastly, by suitable extension, assume that $\til{\mathbf{P}}_{\boldsymbol{0}}$ carries for each $x \in Q_N(0)$ and $i \geq 1$ independent families $B^{(x), i}=\{ B^{(x), i}_{y,z;k} : y \in \partial B_2, \, z \in \partial B_1\cup\{\Theta\}, k \geq 1\}$ and $\hat{B}^{(x), i}=\{ \hat{B}^{(x), i}_{v,w; k} : v \in \partial B_2, \, w \in \partial B_3^c, k \geq 1 \}$ of independent random variables, whereby $B^{(x), i}_{y,z;k}$ has the same law as $(X_t)_{t\leq H_{B_1}}$ under $P_y(\, \cdot \, | \, H_{B_1}< T_{B_3}, X_{H_{B_1}}=z)$ if $z\in{\partial B_1}$ and the same law as $(X_t)_{t\leq L_{B_2}(T_{B_3})}$ under $P_y(\,\cdot\,|\,H_{B_1}>T_{B_3})$ if $z=\Theta,$ and $\hat{B}^{(x), i}_{v,w;k}$ has the same law as $(X_t)_{t\leq T_{B_3}}$ under $P_v(\, \cdot \, | \, \til{H}_{B_2}> T_{B_3}, X_{T_{B_3}}=w)$.

We can now define the walk $X^{(x),\, i}$ for any $x \in Q_N(0)$ and $i\geq1,$ as follows. We introduce four sequences $(y_k)$, $(z_k)$, $(v_k)$, $(w_k)$ (all implicitly depending on $x$ and $i$), each with $k$ ranging from $1 \leq k \leq \til{V}_i^{(x)}-\til{V}_{i-1}^{(x)}$.
One sets $y_1= \til{X}_{0}^{(x), \, i}$ and for each $1 \leq k < \til{V}_i^{(x)}-\til{V}_{i-1}^{(x)}$,
the vertex $y_{k+1}$ (in $\partial B_2 $) is defined as the endpoint of  $\til{Y}^{(x)}_{\til{V}_{i-1}^{(x)}+k}$. The points $z_k$ and $v_k$
are the start- and endpoints of the excursion $\til{Z}^{(x)}_{\til{V}_{i-1}^{(x)}+k}$ when it is not equal to $\Theta,$ and we take $z_k=\Theta$ and $v_k=y_k$ when $\til{Z}^{(x)}_{\til{V}_{i-1}^{(x)}+k}=\Theta,$  and $w_k$ is the starting point of $\til{Y}^{(x)}_{\til{V}_{i-1}^{(x)}+k}.$
{Now $X^{(x),\, i}$ starts in $y_1$.} Then, for every $k$, the process
 $X^{(x),\, i}$ first follows $B^{(x), i}_{y_k,z_k; k}$, then $\til{Z}^{(x)}_{\til{V}_{i-1}^{(x)}+k}$ if $\til{Z}^{(x)}_{\til{V}_{i-1}^{(x)}+k}\neq\Theta$ (and otherwise stays in place), then $\hat{B}^{(x), i}_{v_k,w_k; k}$ and finally $\til{Y}^{(x)}_{\til{V}_{i-1}^{(x)}+k}$ and the pieces thereby obtained are concatenated as $k\in{\big\{1,\dots,\tilde{V}_i^{(x)}-\tilde{V}_{i-1}^{(x)}\big\}}$ increases to form the sample path of $X^{(x),\, i}$ (note in particular that the last piece is indeed unbounded).

The starting point of $X^{(x),i}$ is the same as $\til{X}^{(x),i},$ and thus has law $\bar{e}_{Q(x,r_2)}$ by the representation \eqref{eq:definterprocess} of random interlacements.  Using \eqref{e:law-Ztilde-x} and a similar calculation as following \eqref{eq:defg2} (in order to witness the correct conditional distributions of the bridges $B^{(x), i}$ and $\hat{B}^{(x), i}$), one concludes that $(X^{(x),\, i})_{i\geq1}$ are i.i.d.~random walks with starting distribution $\bar{e}_{Q(x,r_2)}$ each, as desired. Moreover, by construction,
\begin{equation}\label{e:X-indep}
\text{ $(X^{(x),\, i})_{i\geq1}$ is independent of $\{(X^{(y),i})_{i\geq1}:\,\boldsymbol{y}\notin{Q(\boldsymbol{x},2(1+\delta)R)}\}$,}
\end{equation}
 since $\til{Z}^{(x)}$ has range $2r_3= 2R(1+\delta)$ in $\mathbf{T}$ and $X^{(x),\, i}$ only involves additional randomness which is independent as $x$ varies: namely, $B^{(x), i}$, $\hat{B}^{(x), i}$ and $(\til{Y}^{(x)}_j)_{j\geq1}$ (function of $\til{\omega}^{(x)}$).

To complete the construction of $\omega^{(x)}$, let $(N^{(x),u})_{u\geq0},$ $x\in{Q_N(0)},$ be an i.i.d.~family of Poisson processes with intensity $\mathrm{cap}( B_2)$ and define
\begin{equation*}
{\omega}_{ B_2}^{(x),  u}=\sum_{1\leq i \leq N^{(x),u}}\delta_{X^{(x),\, i}}\text{ for all }u>0.
\end{equation*} 
Then $\big({\omega}_{ B_2}^{(x),  u}\big)_{u>0}$ has the same law as the restriction of $(\omega^u)_{u>0}$ to forward trajectories hitting $ B_2$ after entering $ B_2$ and one completes it independently to obtain an interlacements process $\omega^{(x)}=(\omega^{(x), u})_{u>0}$ at all levels on $\Z^d,$ which has the desired law, see \eqref{e:loc-law}, and satisfies \eqref{e:loc-range} by 
means of \eqref{e:X-indep}.

 It remains to show \eqref{eq:couplelltilell}. Denoting by $(\ell^{(x)}_{y,u})_{u\geq0,y\in{\Z^d}}$ the field of local times associated to $\omega^{(x)}$, recalling \eqref{eq:deftildeL} and noting that the trace of $\omega^{(x)}$ inside $B_1=Q(x,R)$ coincides with that of the excursions $(\tilde{Z}_j^{(x)})_{j \geq 1 }$ which enter it, it then follows
 by Lemma~\ref{lem:conc} and a union bound that for all $0< v<u\leq u_0$ and $\epsilon\in(0,1)$ with $u(1-3\eps')>v(1+3\eps')$,% and $v\eps'M$ is large enough
\begin{equation}\label{eq:proofsuccess2}
\begin{split}
	&\til{\mathbf{P}}_{\boldsymbol{0}}\left(\begin{array}{c}  {\ell}_{y,[v(1+3\eps'), u(1-3\epsilon')]}^{(x)} \leq \til{\ell}_{y, [v(1+2\eps'), u(1-2\epsilon')] }^{(x)}, \\[0.2em] \til{\ell}_{y, [v(1-2\eps'), u(1+2\epsilon')] }^{(x)} \leq {\ell}_{y,[v(1-3\eps'), u(1+3\epsilon')]}^{(x)} , \\[0.5em] \text{for all }  x  \in F\text{ and }y\in{Q(x,R)} \end{array}\right) 
		\geq 1- C|F|\exp\left(-c\cdot \epsilon^2 v M\right).
\end{split}
\end{equation}
Thus, \eqref{eq:couplelltilell} follows by combining \eqref{eq:couplingforlocaltimes}, \eqref{eq:proofsuccess} and \eqref{eq:proofsuccess2} with the lower bound \eqref{e:M-bound} on~$M$.% since $(\ell_y({\lceil un^d\rceil+un^d})-\ell_y({\lceil un^d\rceil}))_{y\in{\mathbf{T}}}$ has the same law as $(\ell_y(un^d))_{y\in{\mathbf{T}}}$ under $\mathbb{P}^W.$	

The proof in the case of random interlacements follows a similar three-step pattern: first one shows using Lemma \ref{lem:conc} for random interlacements that $\ell_{y,u}$ under $\PI$ is well-approximated by a process $\hat{\ell}^{(x)}_{y,u(1\pm\eps')},$ for  $y\in{Q(x,R)}$ and $x\in{F}$, having a fixed excursion count, thus yielding an analogue of \eqref{eq:couplingforlocaltimes}.
This step is somewhat streamlined since there is no subtlety regarding the first excursion, as opposed to $X$.
In the second step, one uses the interlacement parts of Proposition~\ref{pro:clotheslines} and Lemma~\ref{lem:concsoftrwri} to approximate $\hat{\ell}^{(x)}_{y,u(1\pm\eps')}$ by a short-range process $\til{\ell}^{(x)}_{y,u(1\pm2\eps')},$ similarly as in \eqref{eq:proofsuccess}. Finally one reconstructs a short-range family of interlacement processes $({\omega}^{(x)})_{x\in{Q_N(0)}}$  such that their associated local times ${\ell}_{y, u(1\pm3\eps')}^{(x)}$ are good approximations of $\til{\ell}^{(x)}_{y,u(1\pm2\eps')}.$ The second and third of these steps are virtually identical as above upon setting $v=0$.
\end{proof}

\begin{Rk}[Extensions of Theorem~\ref{thm:rwshortrange}]\leavevmode
\label{rk:othershortrange}
\begin{enumerate}[label=\arabic*)]
\item\label{i} (Flexibility with \eqref{e:loc-law}-\eqref{e:loc-range}). One could relax \eqref{e:loc-law}-\eqref{e:loc-range} by requiring instead that $ \ell^{(x)}=({\ell}^{(x)}_{y,u})_{y\in{Q(x,R)}, u \geq 0}$ be \textit{some} field
having a finite-range property, satisfying \eqref{eq:couplelltilell}/\eqref{eq:couplingshortrangeinter} and whose law is translation invariant (that is $(\ell^{(x)}_{y,\cdot})_{y\in{Q(x,R)}}$ has the same law as $(\ell^{(0)}_{y-x,\cdot})_{y\in{Q(x,R)}}$),  see the proof of Theorem~\ref{The:shortrangeapprointro-new} as to why this is necessary. Under these less stringent conditions, one can afford to simply choose ${\ell}^{(x)}_{y,u}= \til{\ell}_{y,[0,u]}^{(x)}$ as in \eqref{eq:deftildeL} and finish the proof with \eqref{eq:proofsuccess} in two steps instead of three; note that the law of $\ell^{(x)}$ is translation invariant by \eqref{e:law-Ztilde-x}. This weaker result is in fact sufficient to
deduce Proposition~\ref{The:shortrangeappro} below, which will be the driving force behind the proof of Theorems~\ref{thm:uncoveredset} and~\ref{cor:phasetransition2-intro} in \S\ref{sec:denouement}  (see also Theorem~\ref{The:alpha>1/2}). Various parts of the coupling also simplify in the process. Indeed, one can define the reference values $M$ in \eqref{e:avg-number} and $\bar{g}(z)$ in \eqref{eq:gbar-def} without identifying the relevant stationary distribution of the clothesline process, see \eqref{e:avg-number'} for $M$ and the proof of Lemma~\ref{lem:concsoftrwri} for $\bar{g}(z),$ thus bypassing the use of exact identities such as \eqref{eq:ident} and \eqref{eq:boundrho1}, see also \cite[Lemma 6.1 and eq.~(9.4)]{MR3563197} which, albeit instructive, are not trivial.

Apart from giving a concrete idea as to what $ \ell^{(x)}$ is, the conditions \eqref{e:loc-law}--\eqref{e:loc-range} present the additional benefit of producing an independent proof of the coupling from \cite{MR3563197} between the random walk and random interlacements, cf.~Corollary~\ref{cor:coupling}, for which $ \ell^{(x)}$ crucially needs to have the correct law.
Moreover, knowing that $ \ell^{(x)}$ are interlacement local times is also essential in the proof of Lemma \ref{lem:boundonlateRIRW} below.

 In a similar vein, one may require as part of  Theorem \ref{thm:rwshortrange} that $ \ell^{(x)}$ be  
the local times associated to a short range family of  random walks on $\mathbf{T}$ (instead of  interlacements). This is essentially a matter of replacing the interlacement clothesline $\til{\lambda}^{(x)}$ in the construction of $\til{Z}^{(x)}$ in \eqref{eq:deftildeZji} by a random walk clothesline $\til{\zeta}^{(x)}$. The proof suffers very minor modifications (mostly trading one of the estimates in either of Lemmas~\ref{lem:concsoftrwri} or~\ref{lem:conc} for the other). In particular, in the context of \eqref{eq:proofsuccess2}, the increments of $\ell^{(x)}$ (now associated to a random walk) will not overcount the first excursion for similar reasons as in \eqref{e:inclu-increments}. Moreover, the law of $\ell^{(x)}$ is still translation invariant when starting the corresponding random walks from the uniform measure on $\mathbf{T}.$
 
 \item \label{rk:increments} (Increments in \eqref{eq:couplelltilell}). The choice of observable ${\ell}_{{y},[v,u]}= {\ell}_{{y},u}- {\ell}_{{y},v}$, for ${y} \in \mathbf{T}$ is a means to avoid potential issues with the very first excursion of $X$, see the discussion leading to \eqref{e:inclu-increments}: the excursion process $Z^{(x)}$ does not count the first excursion of the random walk in $Q(\boldsymbol{x},r_1)$ before time $D_0(X,Q(\boldsymbol{x},r_2),Q(\boldsymbol{x},r_3)) = H_{\partial Q(x,r_3)^c}(X)$, cf.~\eqref{eq:deftilrhorho}, \eqref{eq:defexcrw} and \eqref{eq:defZji-x}, hence this excursion does not appear in \eqref{e:hat-ell} either. Note that this issue does not arise for interlacements since trajectories arrive ``from infinity,'' whence \eqref{eq:couplingshortrangeinter} rather than \eqref{eq:couplelltilell}. As we now explain, instead of the increment ${\ell}_{{y},[v,u]}$ one could consider the field
  $$
 \bar{\ell}_{{y},u}^{(x)} \stackrel{\text{def.}}{=} \sum_{n \geq 0} 1\{ X_n ={y}, \, n \geq H_{\partial Q(\boldsymbol{x}, %\sqrt{1+\delta}
 R')^c}\}, \text{ for } {y} \in Q( \boldsymbol{x}, R ),\, x \in Q(0,N),
 $$
 with $R'= r_3=(1+\delta)R$, and replace \eqref{eq:couplelltilell} by
 \begin{multline}\label{eq:couplelltilell'} %\tag{$\overline{\ref{eq:couplelltilell}}$}
	\til{\mathbf{P}}_{\boldsymbol{0}}\left(
	%\begin{array}{l}
	{\ell}_{y,u(1-\epsilon)}^{(x)}\leq  \bar{\ell}_{\boldsymbol{y},u}^{(x)}  \leq {\ell}_{y,u(1+\epsilon)}^{(x)} \,
	\forall x\in F, \, y\in{Q(x,R)}
	%\end{array}	
	\right) \\\geq 1 - C|F|R^{2d}\lceil uR^{d-2}\rceil\exp\big(-c \epsilon \sqrt{u R^{d-2}} \big).
	\end{multline}
	Observe in particular that $ \bar{\ell}_{\boldsymbol{y},u}^{(x)} ={\ell}_{\boldsymbol{y},u}$ for all $y\in{Q(x,R)}$ under $\mathbf{P}_{\boldsymbol{0}}$ whenever $x \notin Q_{R'}(0)$, for then $X_n\in{Q(\boldsymbol{x},R)}$ implies $n\geq H_{\partial Q(\boldsymbol{x},R')^c}$. Thus \eqref{eq:couplelltilell'} yields a true analogue of  \eqref{eq:couplingshortrangeinter} if one restricts to $x \in F \setminus Q_{R'}(0)$. The proof of \eqref{eq:couplelltilell'} does not require any amendments to the above argument: the restriction on $ n$ inherent to $\bar{\ell}_{\boldsymbol{y},u}^{(x)}$ allows to carry out the proof of Theorem~\ref{thm:rwshortrange} with $v=0$ (and  $\bar{\ell}_{\boldsymbol{y},u}^{(x)}$ in place of ${\ell}_{\boldsymbol{y},[0,u]}$), which in particular does not create issues in \eqref{e:inclu-increments}. Alternatively, one replaces $X$ by $X^{(x)}= X\circ \theta_{H_{\partial Q(\boldsymbol{x}, R')^c}}$ in the definition \eqref{eq:defZji-x} of $Z_j^{(x)}$, which leaves the associated clothesline $\zeta^{(x)}$ unchanged. The field $\bar{\ell}_{\boldsymbol{y},u}^{(x)}$ is measurable in terms of this modified $Z^{(x)}$. The issue with the first excursion disappears in this context in essence because $D_0(X^{(x)},Q(x,r_2),Q(x,r_3))=0,$ cf.~\eqref{eq:deftilrhorho}. 
	
	In fact  \eqref{eq:couplelltilell'} also implies directly an approximation of $\ell_{\boldsymbol{y},u}$ by $\ell_{y,u(1\pm \eps)}^{(x)}+\ell_{\boldsymbol{y},u}-\bar{\ell}_{\boldsymbol{y},u}^{(x)}$ for all $x\in{F}$ and $y\in{Q(x,R)},$ which is also a short-range field under $\til{\mathbf{P}}_{\boldsymbol{0}}$ since it is equal to $\ell_{y,u(1\pm \eps)}^{(x)}$ outside of $Q_{R'}(0).$ However its law is not translation invariant, and thus does not necessarily have range $2(1+\delta)R$  under the probability measure $\til{\mathbf{P}}$ from Theorem~\ref{The:shortrangeapprointro-new}.

\item \label{rk:coup-u}(Coupling in $u$). The following extension of \eqref{eq:couplelltilell'} is tailored to later purposes (see the proof of Theorem~\ref{thm:processusabovealpha*}), but noteworthy in its own right. Let $0< u_1< u_0$. Then, applying \eqref{eq:couplelltilell'} (for $\frac{\eps}{3}$ instead of $\eps$) at levels $u= u_1+k u_1\eps/3$ for each $k\in{\{0,1,\dots,\lceil \frac{3(u_0-u_1)}{\eps u_1}\rceil\}},$ using a union bound and monotonicity of all the relevant fields in $u$, one deduces (as alternative to \eqref{eq:couplelltilell} in the statement of Theorem \ref{thm:rwshortrange}) that
 \begin{multline}\label{eq:couplelltilell2} %\tag{$\overline{\ref{eq:couplelltilell}}$}
	\til{\mathbf{P}}_{\boldsymbol{0}}\left(
	%\begin{array}{l}
	{\ell}_{y,u(1-\epsilon)}^{(x)}\leq  {\ell}_{\boldsymbol{y},u}  \leq {\ell}_{y,u(1+\epsilon)}^{(x)},
	\text{ for all } x\in F \setminus Q_{R'}(0), \, y\in{Q(x,R)}, \, u\in{[u_1,u_0]}
	%\end{array}	
	\right) \\[0.5em] \geq 1 - C|F|R^{2d} \textstyle  \frac{u_0}{\eps u_1} \lceil u_0R^{d-2}\rceil\exp\big(-c \epsilon \sqrt{u_1  R^{d-2}} \big)
	\end{multline}
(with $R'= (1+\delta)R$). Note here that we used that the field $\ell^{(x)}$ from \eqref{eq:couplelltilell'} does not depend on the choice of $u\in{(0,u_0]}$, similarly as in Theorem~\ref{thm:rwshortrange}.   An obvious analogue of \eqref{eq:couplelltilell2} holds for random interlacements, with ${\ell}_{{y},u}$  in place of ${\ell}_{\boldsymbol{y},u}$ and without further restriction on $x \in F$. In closer analogy to \eqref{eq:couplelltilell}, one could also formulate a version of \eqref{eq:couplelltilell2} for increments.
\end{enumerate}
\end{Rk}

\section{Consequences of localization}
\label{sec:loc-cons}

Our main localization result, Theorem~\ref{thm:rwshortrange}, derived in the previous section, has two main applications in the context of late points. First, as asserted in Proposition~\ref{The:shortrangeappro} below, it allows to introduce a (localized) family $\til{\mathcal{L}}= (\til{\mathcal{L}}^{\alpha})_{\alpha \geq 0}$, coupled to ${\mathcal{L}}$ (recall \eqref{e:L^alpha-RW}) in a way that i) the two are close up to sprinkling (see~\eqref{eq:boundonSshortrange} below) and ii) $\til{\mathcal{L}}$ is amenable to Chen-Stein (due to its finite-range property).
Second, as alluded to below \eqref{e:V^u-def}, it allows us by means of Corollary~\ref{cor:coupling} to compute various key quantities of interest related to the random walk with sufficient precision using interlacements, see Lemmas~\ref{lem:boundonlateRIRW} and \ref{L:alpha_*}. Combining these two ingredients with Lemma~\ref{lem:corofchenstein} (the modified Chen-Stein scheme) then leads to Theorem~\ref{The:alpha>1/2}, which is the main result of this section, and will be one of the driving forces behind our main results, proved in \S\ref{sec:denouement}. By exploiting $\til{\mathcal{L}}$ as an intermediate link, Theorem~\ref{The:alpha>1/2} gives quantitative control on the difference between the true set of late points $\L$ and its `Poissonized' version $\tilde{\B}$, comprising a suitable class of \textit{independently} sampled shapes, see \eqref{eq:defAF}. Our arguments hint at a generic phenomenon, which ought to be valid for a variety of models of interest, see Remark~\ref{rk:final},\ref{R:univ}.

\subsection{The set \texorpdfstring{$\mathcal{L}_F^{\alpha}$}{LFalpha} and first properties}\label{subsec:L_F}
We start by introducing a setup that fits all needs. As announced in the introduction, see above \eqref{e:V^u-def}, this includes treating both late-point/high-intensity regimes for random walk/random interlacements, each in a subset $F$ of (but not necessarily equal to) the full torus/box (of side length $N$), at appropriate timescales. Recall from~\S\ref{sec:notation} that $0$ denotes the origin of $\Z^d$ and $\mathbf{0} =\pi(0)$ where $\pi: \Z^d \to \mathbf{T}$ is the canonical projection, and that $Q_R(x)$ is the cube of side length $R$ centred at $x,$ either in $\Z^d$ if $x\in{\Z^d}$ or in $\mathbf{T}$ if $x\in{\mathbf{T}}.$ 

In order to allow for a unified presentation, we introduce the following notation, valid from here on and throughout Sections~\ref{sec:loc-cons}-\ref{sec:extensions}. In writing $\P$ in the sequel, we tacitly allow for either choice $\P\in \{ \P^I, \mathbf{P}\}$, i.e.~all statements made under the measure $\P$ hold for either model (recall that $\P^I$ denotes the canonical law of random interlacements on $\Z^d$ and $\mathbf{P}$ the law of the random walk on $\mathbf{T}$ with uniform starting point). We further define the set $Q_R$ for any $R \leq N$ (where $N$ denotes the side length of $\mathbf{T}$) as
$Q_R=Q_R(0) (\subset \Z^d)$ when $\P=\P^I$ and $Q_R=Q_R(\mathbf{0}) (\subset \mathbf{T})$ when $\P=\mathbf{P}$.

With the above notation, we introduce for finite $F \subset Q_N$ and $\alpha> 0$,
\begin{equation}
\label{e:u_F}
u_F(\alpha)= \alpha g(0) \log(|F|), 
\end{equation}
whence $u_N(\alpha)= u_F(\alpha)|_{F= Q_N}$ in view of \eqref{e:u_N} (when $\mathbb{P}= \mathbf{P}$). The scaling \eqref{e:u_F} is explained in Remark~\ref{R:asymp},\ref{R:density} below. With $\mathcal{V}_N^{\cdot}$ as in \eqref{e:vacant-set-RW} and $\mathcal{V}^{\cdot}$ as below \eqref{eq:interlocal}, we now define under $\mathbb{P}$ the random set $\L_{F}= (\L^{\alpha}_{F})_{\alpha \geq 0}$ for arbitrary finite $F\subset Q_N$ to be %$\L^{\alpha}_{F}\subset F$ given for all $\alpha>0$ and $F\subsetQ_N$ by
\begin{equation}
\label{defL}
\L^{\alpha}_{F}=\begin{cases}
\mathcal{V}_N^{u_F(\alpha)}\cap F, &\text{if $\P= \mathbf{P}$ %(and $F\subset\mathbf{T}$)
}\\
\V^{u_F(\alpha)}\cap F,&\text{if $\P= \mathbb{P}^I$ %(and $F\subset \Z^d$)
}.
\end{cases}
\end{equation}
We simply write $\L^{\alpha}$ when $F=Q_N$, which is consistent with \eqref{e:L^alpha-RW}; the results of the introduction thus deal with the case $\mathbb{P}= \mathbf{P}$ in \eqref{defL} for the specific choice $F=Q_N (=Q_N(\mathbf{0}))$. As will become clear, all of these results can be generalized (with suitable amendments) to the more general framework of \eqref{defL}. 
We start by gathering a few key properties of $\L^{\alpha}_{F}$. %As a consequence of \eqref{e:V^u-def}, one immediately gets that

\begin{lemma}
\label{lem:boundonlateRIRW}
\smallskip
\noindent
\begin{itemize}
\item[{i)}] For all $K\subset F \subset \subset \Z^d$ and $\alpha>0$, 
\begin{equation}
\label{eq:boundonlateRI}
\PI( K\subset \L^{\alpha}_{F})=|F|^{-\alpha g(0)\cp(K)}.
\end{equation}
\item[{ii)}] For all $N \geq1$, $K\subset F \subset  \mathbf{T}$, and all $\beta_0>0,$ the bound
\begin{equation}
\label{eq:boundonlateRW}
\bigg| \, \frac{\mathbf{P} ( K \subset  \L^{\alpha}_{F} )}{|F|^{-\alpha g(0)\cp(K)}}-1 \, \bigg|\leq  C(\beta_0)\frac{\log(N)^{3/2}}{ N^{(d-2)/2}}
\end{equation}
holds whenever $\alpha\in{(0,2]}$ and $\cp(K)\leq \beta_0.$
\end{itemize} 
\end{lemma}

\begin{proof}
The equality \eqref{eq:boundonlateRI} follows directly from \eqref{e:V^u-def}, \eqref{e:u_F} and \eqref{defL}. To deduce \eqref{eq:boundonlateRW} first note that the condition $\mathrm{cap}(K)\leq \beta_0$ implies $|K| \leq C(\beta_0)$ by virtue of \eqref{eq:easy}. Since $\delta(K) \leq N(1-\frac{1}{|K|})$ for any $K \subset \mathbf{T}$, using translation invariance we may therefore assume that $K\subset Q_{N(1-\delta)}(\boldsymbol{0})$ for some $\delta = \delta(\beta_0) > 0$. We then apply Corollary~\ref{cor:coupling} for this $\delta$ with the choice $\eps=\lambda N^{-(d-2)/2}\frac{\log N}{\sqrt{\alpha\log(|F|)}}$ for $\lambda >0$ to find that
\begin{equation*}
\mathbf{P} ( K \subset  \L^{\alpha}_{F} ) \leq \PI\big(K' \subset \L^{\alpha(1-\varepsilon)}_{F} \big)+C \alpha \log(|F|) N^{3d}\exp\big(-c(\delta)\lambda  \log N \big),
\end{equation*}
where $K'\subset  Q_{N(1-\delta)}(0)$ is such that $\pi(K')=K,$ similarly as in the definition of $\mathrm{cap}(K)$ below \eqref{e:e_K}. From this, one of the two bounds implied by \eqref{eq:boundonlateRW} readily follows using \eqref{eq:boundonlateRI} upon taking $\lambda$ large enough in a manner depending on $\beta_0$. The other bound is obtained similarly.
\end{proof}

\begin{Rk}\leavevmode\label{R:asymp}
\begin{enumerate}[label=\arabic*)]
\item \label{R:density} (Asymptotic density of $\L^{\alpha}_{F} $). Applying Lemma~\ref{lem:boundonlateRIRW} for $K=\{0\}$ and using \eqref{e:cap-2point} yields for any $F = F_N \subset Q_N$ with $|F| \to \infty$ as $N \to \infty$ that
\begin{equation}
\label{eq:Lalphaasymp}
\P ( 0 \in \L^{\alpha}_{F})   \sim |F|^{-\alpha} \text{ as }N\rightarrow\infty,
\end{equation}
which accounts for the scaling in \eqref{e:u_N} and \eqref{e:u_F}.

\item\label{rk:boundonprobalate} Throughout \S\ref{subsec:btilde-compa}, the following consequence of \eqref{eq:boundonlateRI} and \eqref{eq:boundonlateRW} will be sufficient, cf.~Remark~\ref{rk:final},\ref{R:univ}: for all $ \beta_0>0$, all $N\geq1,$ $F\subset Q_N$, $K\subset F$ with $\mathrm{cap}(K)\leq \beta_0$, and  $\alpha\in{(0,2]}$,
\begin{equation}
\label{eq:upperboundonprobalate}
\P(K\subset \L^{\alpha}_F)\leq C(\beta_0){|F|^{-\frac{\alpha}{\alpha_{\sast}(K)}}}
\end{equation}
(recall the definition of $\alpha_*(K)$ from \eqref{eq:defalpha*K}). 
%\AP{I created commands ("slash sast" for small, "slash mast" for medium) to make the asterisk look a bit smaller in these kind of displays, since the normal one (*) is oversized. Please use them if you modify/create new displays.}
 In Lemmas~\ref{L:alpha_*} and~\ref{L:D_LB} as well as in  \S\ref{sec:denouement} and \S\ref{sec:extensions} below, we will also use the following lower bound implied by Lemma~\ref{lem:boundonlateRIRW}: under the same assumptions as those of \eqref{eq:upperboundonprobalate}, 
\begin{equation}
\label{eq:lowerboundonprobalate}
\P(K\subset \L^{\alpha}_F)\geq c(\beta_0){|F|^{-\frac{\alpha}{\alpha_{\sast}(K)}}}.
\end{equation}
Actually the constants $C(\beta_0)$ from \eqref{eq:upperboundonprobalate} and $c(\beta_0)$ from \eqref{eq:lowerboundonprobalate} could be replaced by $1+o(1)$ as $|F|\rightarrow\infty,$ but we will not need this fact except in the proof of \eqref{e:coup-crit} to obtain the exact constant $1-e^{-d}.$ In the proof of Lemma~\ref{L:D_LB} below, we will also need the following decoupling formula, which is easily implied by Lemma~\ref{lem:boundonlateRIRW} together with \eqref{eq:decouplatepoints} in the improved form discussed immediately below it (along with its analogue on the torus): for all $\beta_0>0$ and $\alpha>0,$ if $K,K'\subset F$ are such that  $ \frac{d(K,K')}{\log(|F|)^{1/(d-2)}} \to \infty$,
\begin{equation}
\label{eq:decouplingprobalate}
\P( (K \cup K')\subset\L^{\alpha}_F)=(1+o(1)) \P(K\subset\L^{\alpha}_F)\P(K'\subset\L^{\alpha}_F)\text{ as }|F|\rightarrow\infty,
\end{equation}
where $o(1)$ is uniform in $K,K'$ verifying $|K|, |K'| \leq \beta_0$.
\end{enumerate}
\end{Rk}

Before constructing our coupling between $\L$ and Bernoulli random variables, let us collect some interesting consequences of \eqref{eq:upperboundonprobalate}, \eqref{eq:lowerboundonprobalate} and \eqref{eq:decouplingprobalate}, which further elucidate the role of the parameters $\alpha_*$ from \eqref{eq:alpha_*} and $\alpha_*(K)$ from \eqref{eq:defalpha*K}. In view of
\eqref{defL}, the quantity $D^{\alpha}$ introduced in \eqref{e:double-pts} is naturally declared under $\P$ upon summing over all $x\sim y$ with $x,y \in Q_N$. For any set $\mathcal{S}\subset Q_N$ and $K\subset Q_N(0),$ we introduce similarly 
\begin{equation}
\label{eq:DNK}
D_{K}(\mathcal{S})=
\begin{cases}
\sum_{x\in{Q_N}}1\{x+K\subset\mathcal{S}\},&\text{if } \P=\P^I
\\\sum_{x\in{Q_N}}1\{x+\pi(K)\subset\mathcal{S}\},&\text{if } \P= \mathbf{P}\end{cases}
\end{equation}
the number of times a translate of $K$ (or its projection on the torus) by $x \in Q_N$ lies in $\mathcal{S}.$ We will often abbreviate $D_K^{\alpha}=D_K(\L^{\alpha}).$ Note that $D^{\alpha},$ see \eqref{e:double-pts}, is half the sum of $D_{K}^{\alpha}$ over all $K=\{0,x\}$ with $x\sim 0.$ The following result shows that $D_K^{\alpha}$ is small on average if and only if $\alpha>\alpha_*(K)$. %In a similar vein as Remark~\ref{R:asymp},\ref{R:density} above, using Lemma~\ref{lem:boundonlateRIRW} for $K$ a two-point set readily yields the following result. 

\begin{lemma}[Representations of $\alpha_*$] \label{L:alpha_*}
For each $\emptyset \neq K\subset\subset\Z^d,$ with $\alpha_*(K)$ as in \eqref{eq:defalpha*K},
\begin{equation}
\label{e:aalpha_*K-equiv}
\alpha_*(K) = \sup\big\{\alpha >0 : \, \textstyle \lim_N \mathbb{E}[D_{K}^{\alpha}] =0 \big\}.
\end{equation}
In particular for all $x\sim y$,
\begin{equation}
\label{e:aalpha_*-equiv}
\begin{split}
\alpha_* &=\alpha_*(\{x,y\})=1-\frac{1}{2g(0)}
=\frac12\big(1+P_0(\widetilde{H}_0 < \infty)\big).
\end{split}
\end{equation}
\end{lemma}
\begin{proof}
First, observe that \eqref{eq:upperboundonprobalate} and \eqref{eq:lowerboundonprobalate} applied with $F=Q_N$ immediately yields, for all $\alpha > 0$,
\begin{equation}
\label{eq:DNalpha}
 cN^{d(1-\frac{\alpha}{{\alpha}_{\sast}(K)})} \leq \mathbb{E}[D_{K}^{\alpha}] \leq C  N^{d(1-\frac{\alpha}{{\alpha}_{\sast}(K)})},
\end{equation}
from which \eqref{e:aalpha_*K-equiv} follows. The first equality in \eqref{e:aalpha_*-equiv} then follows by rotational invariance of the capacity. Applying the simple Markov property, one obtains, for all $x \sim 0$,
\begin{equation}
\label{eq:poH0vsGreen}
P_0(\widetilde{H}_0 < \infty)= \frac1{2d} \sum_{y \sim 0}P_y(H_0< \infty) \stackrel{\eqref{eq:lastexit}}{=} \frac1{2d} \sum_{y \sim 0}\frac{g(y)}{g(0)}= \frac{g(x)}{g(0)}=1-\frac{1}{g(0)},
\end{equation}
where the two last steps follow by invariance of $P_x$ under lattice rotations and translations. On account of \eqref{eq:alpha_*-value}, \eqref{eq:defalpha*K} and \eqref{e:cap-2point}, this gives the two last equalities in \eqref{e:aalpha_*-equiv}.
\end{proof}

We now collect a lower bound on $D_{K}^{\alpha},$ resp.\ $D^{\alpha},$ in the `supercritical' phase $ \alpha < \alpha_*(K),$ resp.\ $\alpha<\alpha_*,$ which will be useful in due course. For $D^{\alpha},$  a similar but weaker estimate was derived in \cite[p.\ 1040]{JasonPerla}, for a very specific choice of timescale asymptotic to $\alpha \tcov$, see also Remark~\ref{rk:thmsprinkling},\ref{rk:otherparametrization} below. In addition to yielding a stronger bound valid for any set $K$, the proof we present below is considerably simpler, which highlights the strength of Lemma~\ref{lem:boundonlateRIRW}.

\begin{lemma} \label{L:D_LB}For all $K\subset\subset\Z^d,$ $\alpha \in (0,\alpha_*(K)]$ and $\varepsilon_N > 0$ with $\lim_N \varepsilon_N =0$, one has
\begin{equation}
\label{e:D_N-LB}
\liminf_N \P\big(D_{K}^{\alpha} \geq \varepsilon_N N^{d(1-\frac{\alpha}{{\alpha}_{\sast}(K)})}\big)\begin{cases}
= 1&\text{ if }\alpha<\alpha_*(K),\\
>0&\text{ if }\alpha=\alpha_*(K),
\end{cases}
%\lim_N \P\big(D \geq \varepsilon_N N^{d(1-\frac{\alpha}{{\alpha}_\sast})}\big) =1,
\end{equation}
and the same holds true with $\alpha_*$ in place of $\alpha_{*}(K)$  and $D^{\alpha}$ in place of $D_{K}^{\alpha}.$
%and the same holds true with $\P$ in place of $\mathbf{P}$.
\end{lemma}
\begin{proof}
It follows from \eqref{eq:defalpha*K} and \eqref{eq:capdecoupling} that there exists $C=C(K)<\infty$ such that $\alpha_*((K+x)\cup(K+x'))\leq 2\alpha_*(K)/3$ for all $x,x'\in{Q_N}$ with $d(x,x')\geq C.$ Applying \eqref{eq:upperboundonprobalate} (when $d(x,x')\leq C$ or  $C\leq d(x,x')\leq \log(N)^{\frac{2}{d-2}}$) and \eqref{eq:decouplingprobalate} (when $d(x,x')\geq \log(N)^{\frac{2}{d-2}}$), one has for all $\alpha > 0$
\begin{equation}\label{e:D_N-tilde-mom-2}
\begin{split}
\mathbb{E}\big[({D}_{K}^{\alpha})^2 \big] &= \sum_{x ,x'\in{Q_N}} \P\big(K+x,K+x' \subset \mathcal{L}^{\alpha} \big) 
\\&\leq CN^{d-\frac{\alpha d}{\alpha_{\sast}(K)}}+C\log(N^d)^{\frac{2d}{d-2}}N^{d-\frac{3\alpha d}{2\alpha_{\sast}(K)}}+\big(1+o(1)\big) \mathbb{E}[D_{K}^{\alpha} ]^2 
\end{split}
\end{equation}
as $N \to \infty.$ % where in the last inequality we used \eqref{eq:DNalpha} and that $\alpha_*(\{x,x'\})\geq \alpha_*$ for $x\neq x',$ see \eqref{eq:cap-distance2} and \eqref{e:aalpha_*-equiv}.
 If $\alpha < \alpha_*(K)$, then in view of \eqref{eq:DNalpha} 
 the second moment on the right-hand side of \eqref{e:D_N-tilde-mom-2} dominates. Combining \eqref{eq:DNalpha}, \eqref{e:D_N-tilde-mom-2} and a standard second-moment argument, it follows that
$$
\P\big(D_{K}^{\alpha} \geq \varepsilon_N N^{d(1-\frac{\alpha}{{\alpha}_{\sast}(K)})}\big)   \geq 
\P\big({D}_{K}^{\alpha} \geq c\eps_N \mathbb{E}[{D}_{K}^{\alpha}] \big) \geq  \frac{(1-c\eps_N)^2}{1+o(1)}, 
$$
for all $\alpha< \alpha_*(K)$, from which the claim follows since $\eps_N \to 0$. If $\alpha=\alpha_*(K)$ the proof is similar, except \eqref{e:D_N-tilde-mom-2} is now only smaller than $C\mathbb{E}[D_{K}^{\alpha} ]^2$. The statement for $D^{\alpha}$ then follows readily from \eqref{e:aalpha_*-equiv} and \eqref{e:D_N-LB} for $K=\{0,x\},$ $x\sim 0.$
\end{proof}

\subsection{Main approximation result for \texorpdfstring{$\mathcal{L}_F^{\alpha}$}{LFalpha}}\label{subsec:btilde-compa}
We now combine the modified Chen-Stein result (with sprinkling), Lemma~\ref{lem:corofchenstein}, with our main localization result, Theorem~\ref{thm:rwshortrange} and the asymptotic bounds \eqref{eq:upperboundonprobalate} (a consequence of Lemma~\ref{lem:boundonlateRIRW}) to derive our main approximation result for $\mathcal{L}_F^{\alpha}$, see Theorem~\ref{The:alpha>1/2} below. In a nutshell, we first apply Theorem~\ref{thm:rwshortrange} to replace (up to sprinkling) the family $\mathcal{L}_F$ by an approximation $\tilde{\L}_F$ with a certain finite-range property (see Proposition~\ref{The:shortrangeappro} below), to which we then apply Lemma~\ref{lem:corofchenstein}. The above estimate \eqref{eq:upperboundonprobalate} will serve to bound quantities such as $b_1$ and $b_2$ in \eqref{eq:b_1}-\eqref{eq:b_2}.

Theorem~\ref{The:alpha>1/2}, stated below, is of independent interest. In the next section, it will serve as a driving force to derive our main results from \S\ref{sec:intro}.  The approximation result it entails supplies a coupling between $\L_F$ and a suitably defined process $\tilde{\B}$ of i.i.d.~`patterns,' which we now introduce. Let $(U_K)_{K\subset Q_N}$ be an i.i.d.~family of uniform random variables on $[0,1]$ and for $F\subset \Z^d$ or ${\bf T}$ (depending on whether $\mathbb{P}$ equals $\mathbb{P}^I$ or $ \mathbf{P}$) define the set of admissible patterns by 
\begin{equation}
\label{eq:defAF}
\tilde{\A}_{F}=\big\{K\subset F:\,K\neq\varnothing,\,\mathrm{cap}(K)\leq \textstyle \frac{2}{g(0)},\delta(K)\leq R_F\big\}, \text{ where } R_F= \log(|F|)^{\frac1{d-2}}
\end{equation}
(note that compared to \eqref{eq:defAzd} we additionally ask for $\delta(K)\leq R_F$). We return to the choice of $\tilde{\A}_F$ in Remark~\ref{R:A_F} below.
 %\begin{equation}
%\label{eq:defBalphaFK}
%\hat{\B}_{\alpha}^{F,K}=1\{U_K\leq \P(\mathcal{L}_{\alpha}^{F}\cap \BB(K,\log(n))=K)\}\text{ for each }K\subset F\text{ with }\delta(K)\leq \log(n),
%\end{equation}
We also define $p_{F}^{\alpha}(K)=\P(\mathcal{L}^{\alpha}_{F}\cap \BB(K,R_F)=K)$ for each $K\in{\tilde{\A}_F}$ and $\alpha>0,$ and note in passing that $p^{\alpha}(K)$ introduced above \eqref{eq:defBFK} corresponds exactly to $p_{F}^{\alpha}(K)$ for the choice $F=Q_N$ when $\P=\mathbf{P}$. We then define the family $\tilde{\mathcal{B}}_{F}=(\tilde{\mathcal{B}}^{\alpha}_{F})_{\alpha \geq 0}$ as
\begin{equation}
\label{eq:hatdefBalphaF}
\tilde{\mathcal{B}}^{\alpha}_{F}=\bigcup_{K\in{\tilde{\A}_F}:\,  U_K\leq p_F^{\alpha}(K)}K.%=\bigcup_{\substack {K\in{\tilde{\A}_F}: \\ U_K\leq p_F^{\alpha}(K)}}K.
\end{equation}
The following result  gives quantitative control on the proximity of $\mathcal{L}_F$ and $\tilde{\mathcal{B}}_F$ above level $\frac12$, as measured in terms of $d_{\varepsilon}$, cf.~\eqref{eq:distancesprinkling}.
\begin{theorem}[$\L_F$ as in \eqref{defL}]
\label{The:alpha>1/2}
There exist $C,C'<\infty$ such that, for all $\alpha\in{(\frac12,1]},$ $N \geq 1$, $F\subset Q_N$ and $\eps\in{(0,\alpha)}$, one can couple $\L_F^{\alpha}$ with $(U_K)_{K\in{\tilde{\A}_F}}$ so that with probability one minus
\begin{equation}
\label{eq:boundonSgen}
1 \wedge C\log(|F|)^{C'}|F|^{1-2(\alpha-\eps)}\eps^{-\frac{2d}{d-2}},
\end{equation}
one has the inclusions
\begin{equation}
\label{eq:strongercoupling}
%\begin{split}
   %& \Q\big(
   \big\{U_K\leq p_F^{\alpha+\eps}(K)\big\}\subset\big\{\mathcal{L}^{\alpha}_{F}\cap \BB(K,R_F)=K\}\subset\big\{ U_K\leq p_F^{\alpha-\eps}(K)\big\}, \text{  for all $K\in{\tilde{\A}_F}$}.
   %\big)
   % \\&\geq 1-\frac{C\log(|F|)^{C'}}{|F|^{2(\alpha-\eps)-1}\eps^{\frac{2d}{d-2}}}.
   % \end{split}
\end{equation}
Moreover, $d_{\eps}\big(\L_F,\tilde{\B}_{F};\alpha\big)$ is bounded from above by \eqref{eq:boundonSgen}.
\end{theorem}
 The proof of Theorem~\ref{The:alpha>1/2} occupies the remainder of this section. We begin with two preliminary results. The first of these is a consequence of our main localization result, Theorem~\ref{thm:rwshortrange}, applied to $\mathcal{L}_F$ in \eqref{defL}. In what follows, a family of sets $(\tilde{\L}^{\alpha,(x)}_F)_{\alpha\in{(0,2]},x\in{Q_N}}$ is said to be decreasing if $\tilde{\L}^{\alpha,(x)}_F\subset\tilde{\L}^{\beta,(x)}_F$ for all $\alpha\geq \beta$ and $x\in{Q_N}$. %With a slight abuse of notation, we sometimes regard $\mathcal{L}_F^{\alpha}$ for $F \subset Q_N$ as a subset of $Q_N$ in the sequel, by means of the bijection $\pi: Q_N \to \mathbf{T}$ when working under $\mathbf{P}$.

\begin{proposition}[Short-range approximation for $\mathcal{L}_F$] 
\label{The:shortrangeappro} There exist $c,C,C' \in (0,\infty)$ such that the following holds. For all $N \geq 1$, $F \subset Q_N$, $\epsilon\in (0,1)$, $R=\big( \frac{\lambda}{\varepsilon} \big)^{\frac2{d-2}} R_F$ with $\lambda \geq C$, there exists a decreasing family $\tilde{\L}_F=(\tilde{\L}^{\alpha,(x)}_{F})_{\alpha\in{(0,2]},x\in{Q_N}}$ such that 
\begin{enumerate}[label=\roman*)]
\item\label{ite:ishortrangeappro} $\tilde{\L}^{\alpha,(x)}_F\subset \big(\BB(x,R)\cap F\big)$ for each $x\in{Q_N}$ and  $\alpha\in{(0,2]},$\\[-1.2em]
\item\label{ite:iishortrangeappro}  $(\tilde{\L}^{\alpha,(x)}_F)_{\alpha\in{(0,2]}}$ and $(\tilde{\L}^{\alpha,(y)}_F)_{\alpha\in{(0,2]},y \in{\BB(x,3R)^c}}$ are independent for each $x\in{Q_N},$ and \\[-1.2em]
\item\label{ite:iiishortrangeappro}  for each $\alpha\in{(0,2]},$ there exists a coupling $\til{\mathbb{Q}}=\til{\mathbb{Q}}_{\alpha}$ of $\mathcal{L}_F^{\alpha}$ and $\tilde{\mathcal{L}}_F$ such that for all $\eps\in{(0,\alpha)}$,\\
\begin{equation}
\label{eq:boundonSshortrange}
\til{\mathbb{Q}}\big(\tilde{\L}^{\alpha+\eps,(x)}_{F}\subset \big(\L_{F}^{\alpha}\cap \BB(x,R) \big)\subset\tilde{\L}^{\alpha-\eps,(x)}_{F}\text{ for all }x\in{F}\big)\geq 1-{C' \eps^{-C'}|F|^{-c\lambda}}.
\end{equation} 
\end{enumerate}
\end{proposition}
\begin{proof}
We first consider the case $\P=\mathbf{P}$ of the walk.
Consider the processes ${\ell}^{(x)}$ under $\til{\mathbf{P}}=\sum_{{x} \in \mathbf{T}} \til{\mathbf{P}}_{{x}},$ where $\til{\mathbf{P}}_x$ is the translation by $x$ of the probability $\til{\mathbf{P}}_{\boldsymbol{0}}$ from Theorem~\ref{thm:rwshortrange} for $\delta=\frac12$, $R$ as above and $u_0=3u_F$, where $u_F= u_F(\alpha=1)= g(0)\log(|F|)$, cf.~\eqref{e:u_F}. One then defines for each $\alpha\in{(0,4]}$ and $x\in{Q_N(0)},$ 
\begin{equation*}
\tilde{\L}^{\alpha,(\boldsymbol{x})}_F = \pi\big(\big\{y\in{Q(x,R)}:\,\ell^{(x)}_{y,(2+\frac{\alpha}{2})u_F} -\ell^{(x)}_{y,(2-\frac{\alpha}{2})u_F} =0\big\}\big)\cap F.
\end{equation*}
From this, $\ref{ite:ishortrangeappro}$ plainly follows and $\ref{ite:iishortrangeappro}$ is a consequence of \eqref{e:loc-range} since $\ell^{(x)}$ and hence $\tilde{\L}^{\alpha,(x)}_F$ is a function of $\omega^{(x)}$ alone, whose law does not depend on $x$ in view of \eqref{e:loc-law}.
Moreover, by \eqref{eq:couplelltilell} (applied with $\frac{\eps}{6}$ instead of $\eps$) one obtains, for all $\alpha\in{(0,2]}$ $\eps\in{(0,\alpha)},$ $x\in{Q_N}$ and $y\in{Q(x,R)}$
\begin{equation*}
\begin{split}
&\ell^{(x)}_{y,(2+\frac{\alpha-\varepsilon}{2})u_F} -\ell^{(x)}_{y,(2-\frac{\alpha-\varepsilon}{2})u_F} \leq \ell_{\boldsymbol{y},(2+\frac{\alpha}{2})u_F} -  \ell_{\boldsymbol{y},(2-\frac{\alpha}{2})u_F} \leq \ell^{(x)}_{y,(2+\frac{\alpha+\varepsilon}{2})u_F} -\ell^{(x)}_{y,(2-\frac{\alpha+\varepsilon}{2})u_F},
\end{split}
\end{equation*}
with probability at least $1-C|F|^{-c\lambda}\eps^{-C}$ by choice of $R,$ upon taking $\lambda$ large enough. By definition of $\L$ and $\tilde{\L}$, this yields, for all $\alpha\in{(0,2]}$ and $\eps\in{(0,\alpha)}$, that
\begin{equation}
\label{eq:tildeLinclusdifferenceL}
\tilde{\L}^{\alpha+\eps,(x)}_F\subset \big((\L_{F}^{2+\frac{\alpha}{2}}\setminus\L_{F}^{2-\frac{\alpha}{2}})\cap \BB(x,R)\big)\subset\tilde{\L}^{\alpha-\eps,(x)}_F\text{ for all }x\in{F}
\end{equation}
with probability at least $1-C|F|^{-c\lambda}\eps^{-C},$ and $\ref{ite:iiishortrangeappro}$ follows since $\L_{F}^{2+\frac{\alpha}{2}}\setminus\L_{F}^{2-\frac{\alpha}{2}}$ has the same law as $\L_F^{\alpha}$ for each $\alpha>0.$
The proof for $\P=\P^I$ is similar, except that, in view of \eqref{eq:couplingshortrangeinter} (applied with $\frac\eps\alpha$ instead of $\eps$) one simply defines $\tilde{\L}^{\alpha,(x)}_F=\{y\in{\BB(x,R)}:\,\ell_{y, u_F(\alpha)}^{(x)}=0\}$ under $\til{\mathbb{Q}}=\til{\P}.$ \end{proof}

\begin{Rk}
\leavevmode
\label{rk:shortrangeappro}
\begin{enumerate}[label*=\arabic*)]
\item \label{rk:lawtildeL}It follows from Theorem~\ref{thm:rwshortrange} and the independence and stationarity of the increments of random interlacements that for each $F\subset Q_N$ and $x\in{F},$ the process $(\tilde{\L}^{\alpha,(x)}_F)_{\alpha\in{(0,2]}}$ introduced in Proposition~\ref{The:shortrangeappro} has the same law as $\big(\L_{F\cap Q(x,R)}^{\alpha}\big)_{\alpha\in{(0,2]}}$ under $\PI$ when $\P=\P^I$, and the same law as $\big(\pi(\L_{Q_N(0)}^{\alpha})\cap F\cap Q(x,R)\big)_{\alpha\in{(0,2]}}$ under $\PI$ when $\P=\mathbf{P}$. While noteworthy, we will not need this fact in the sequel.
\item For the case $\P=\P^I$, one can afford to take $R=\big( \frac{\lambda}{\varepsilon} \big)^{\frac2{d-2}}$ in Proposition~\ref{The:shortrangeappro}. This can be traced back to the additional square root present in the bound \eqref{eq:couplelltilell} compared to~\eqref{eq:couplingshortrangeinter}. Moreover as can be seen plainly in the above proof, one actually has a coupling $\tilde{\Q}$ uniform in all $\alpha\in{(0,2]}$ in that case, rather than one depending on $\alpha$ as in the random walk case. The reason the coupling $\tilde{\Q}_{\alpha}$ depends on $\alpha$ when $\P=\mathbf{P}$ is that the sets $(\L_{F}^{2+\frac{\alpha}{2}}\setminus\L_{F}^{2-\frac{\alpha}{2}})_{\alpha\in{(0,2)}}$ does not have the same law (as a process in $\alpha$) as $(\L_F^{\alpha})_{\alpha \in (0,2)}$ under~$\mathbf{P}$ (only its one-dimensional $\alpha$-marginals do) since the corresponding random walks have different starting points for different values of $\alpha$. We refer to Remark~\ref{rk:othershortrange},\ref{rk:increments} for a variant of the approximation supplied by Theorem~\ref{thm:rwshortrange} by which this increment problem can be partially circumvented.
\end{enumerate}
\end{Rk}

Next, returning to  $\tilde{\A}_F$ in \eqref{eq:defAF}, we prove a separation property for $\L^{\alpha}_F$. Namely, for suitably large $R$, the set $\L^{\alpha}_F\cap \BB(0,R)$ either belongs to ${\tilde{\A}_F}$ or is empty with high probability.

\begin{lemma}
\label{lem:wellseparated}
For all $N \geq 1$, $F\subset Q_N,$ $\alpha\in{(\frac12,2]}$ and $R\in [R_F, \frac{N}{4})$,
\begin{equation*}
\P\big(\exists\,x\in{F}: \big(\L^{\alpha}_F\cap Q(x,R)\big)\notin ({\tilde{\A}_{F}\cup\{\varnothing\}})\big)\leq {CR^d\log(|F|)^{C}}{|F|^{1-2\alpha}}.
\end{equation*}
%and the same bound remains true under $\mathbf{P}$ (cf.~\eqref{defL}).
%The same inequality holds when replacing $\L^F$ by $\hat{\B}^F,$ see \eqref{eq:defBalphaF}.
\end{lemma}
\begin{proof}
By \eqref{eq:upperboundonprobalate} one has, for all $r \geq 1$, 
\begin{equation*}
\E{\big|\big\{x,y\in{\L^{\alpha}_F}:d(x,y)\in{[r,R]}\big\}\big|}\leq |F|(2R+1)^d\sup_{\substack{%\{x,y\}\subset F\\\,
d(x,y)\geq r}}|F|^{-\alpha g(0)\mathrm{cap}(\{x,y\})}.
\end{equation*}
In particular, by \eqref{eq:decouplatepoints} and since $\mathrm{cap}(\{x\})=g(0)^{-1}$ for all $x\in{F}$ (see \eqref{e:cap-2point}), one obtains with the choice $r=R_F$ by Markov's inequality that
\begin{equation}
\label{eq:firstincluAF}
\P\big(\exists\,x,y\in{\L^{\alpha}_F}:\, R\geq d(x,y)\geq R_F\big)\leq CR^d|F|^{1-2\alpha}.
\end{equation}
The bound \eqref{eq:firstincluAF} takes care of the contribution to the relevant event appearing in Lemma~\ref{lem:wellseparated} when violating the diameter constraint inherent to $\tilde{\A}_{F}$. 
It remains to address the possibility to violate the capacity constraint in \eqref{eq:defAF}. 
To this effect, observe that by  \eqref{eq:easy}, {\eqref{eq:upperboundonprobalate}} and a union bound we have
%Denoting by $\gamma=C(2/g(0))$ the constant from \eqref{eq:easy}, using the latter and \eqref{eq:boundonlateRI}, one finds that
\begin{multline*}
\P\big(\mathrm{cap}\big(\L^{\alpha}_F\cap Q_R \big)\geq \textstyle \frac{2}{g(0)},\delta\big(\L^{\alpha}_F\cap Q_R\big)\leq R_F\big) \\
\leq \P\big(\exists K\subset (\L^{\alpha}_F\cap Q_R) :\,|K|\leq C,\alpha_*(K)\leq \textstyle \frac12,\delta(K)\leq R_F \big) \leq \displaystyle CR^{d}{\log(|F|)^{C'}}{|F|^{-2\alpha}}.
\end{multline*}
Combining this with a union bound over $x \in F$ and \eqref{eq:firstincluAF}, the claim follows.
\end{proof}

\begin{Rk}\label{R:A_F} We now briefly comment on the choice of $\tilde{\A}_F$ in \eqref{eq:defAF}. In view of Lemma \ref{lem:wellseparated}, it was selected so that $\L^{\alpha}_F\cap \BB(x,R)$ belongs to ${\tilde{\A}_F}$ or is empty, with high probability for all $x\in{F}$ when $R\geq R_F$ -- of course, the bound obtained in Lemma~\ref{lem:wellseparated} also implicitly entails an upper bound $R\leq|F|^{(2\alpha-1)/d}\log(|F|)^{-C}$, above which the estimate is useless. Moreover, as implicitly used in the proof, the definition of $\tilde{\A}_F$ also ensures that for all $x\in{F}$ and $R\in{[R_F, \frac N4)}$,
\begin{equation}
\label{eq:cardAF}
\left|\left\{K\subset \BB(x,R):\,K\in \tilde{\A}_F\right\}\right|\leq C|\BB(x,R)\cap F|\log(|F|)^{C'},
\end{equation}
which is an easy consequence of~\eqref{eq:easy} and the bounds on both capacity and diameter in \eqref{eq:defAF}.
\end{Rk}
 Combining Proposition~\ref{The:shortrangeappro} and Lemma~\ref{lem:wellseparated}, we are now ready to give the 

\begin{proof}[Proof of Theorem \ref{The:alpha>1/2}]
All subsequent considerations implicitly hold for all $N \geq 1$ and $F \subset Q_N$. 
We may assume that $\epsilon>|F|^{-C}$ for suitably large~$C$, for otherwise \eqref{eq:boundonSgen} is larger than $1$, and also that $\epsilon$ is small enough so that $\varepsilon^{-\frac2{d-2}} >2$ and $(\frac 12,1]\pm 3\eps\in{(\frac14,2]},$ which can be arranged without loss of generality by monotonicity of $ \varepsilon \mapsto d_{\varepsilon}$, cf.~\eqref{eq:distancesprinkling}.

We now apply Proposition~\ref{The:shortrangeappro} (for $N,F,\eps$ as appearing in Theorem \ref{The:alpha>1/2}) with $R=\big( \frac{\lambda}{\varepsilon} \big)^{\frac2{d-2}} R_F$ and a choice of $\lambda\geq1$ large enough such that the right-hand side of \eqref{eq:boundonSshortrange} is at least $1-C|F|^{-3}$ uniformly in $\alpha \in (0,2]$. In view of the above lower bound on $\varepsilon$, $\lambda$ can be chosen uniformly in $N,F$ and $\eps$. Thus, defining for $K \in \tilde{\A}_F$ and arbitrary (fixed) $x_K \in K$ the random fields
\begin{align*}
&Y_{K}^{\alpha}=1\{\mathcal{L}^{\alpha}_{F}\cap \BB(K,R_F)=K\} \\
&Z_{K}^{\alpha}=1\{\widetilde{\mathcal{L}}_F^{\alpha,(x_K)}\cap \BB(K,R_F)=K\},
\end{align*}
with $\widetilde{\mathcal{L}}$ as supplied by Proposition~\ref{The:shortrangeappro}, \eqref{eq:boundonSshortrange} implies that
\begin{align}\label{eq:rewriteyandz}
	d_{\eps}(Y,Z;\alpha)\leq {C}{|F|^{-3}}\text{ for all }\alpha\in{I} \stackrel{\text{def.}}{=} \textstyle (\frac14,2],
\end{align}
and all $N, F$ and $\varepsilon$ as above.

We now aim to apply Lemma \ref{lem:corofchenstein} for $S=\tilde{\A}_F$ and $I$ as above. First observe that $Y=(Y^{\alpha})_{\alpha \in I}$ is indeed a family of decreasing Bernoulli processes, as required by \eqref{eq:Y}. Thus \eqref{eq:W} defines a process $W=(W^{\alpha})_{\alpha \in I}$ with $W^{\alpha}= (W_K^{\alpha})_{K \in \tilde{\A}_F}$, which has the same law as $(1_{\{U_K\leq p_F^{\alpha}(K)\}})_{K \in \tilde{\A}_F, \alpha \in I}$. In particular, recalling \eqref{eq:hatdefBalphaF}, this means that
\begin{equation}\label{e:b-tilde'}
\tilde{\mathcal{B}}^{\alpha}_{F}\stackrel{\text{law}}{=}\bigcup_{\substack {K\in{\tilde{\A}_F}: W_K^{\alpha}=1}} K.
\end{equation}
Our aim is to control $d_{2\eps}(Y,W;\alpha)$ for $\alpha \in (\frac12,1]$ by means of \eqref{eq:finalboundChenStein}. With regards to the relevant condition~\eqref{eq:Z}, Proposition~\ref{The:shortrangeappro} implies that $Z=(Z^{\alpha})_{\alpha \in I}$ has the requested monotonicity and the finite-range property follows from item $\ref{ite:iishortrangeappro}$ of the same proposition upon choosing the neighborhood $\cN_K$ of $K \in \tilde{\A}_F$ as $\cN_K=\{K'\in{\tilde{\A}_F}:\,K'\cap \BB(K,3R)\neq\varnothing\}.$ Indeed, \
for all $K'\notin \cN_K$ we have that $x_{K'}\notin \BB(x_K,3R)$, and hence $(\tilde{\mathcal{L}}_F^{\alpha,(x_{K'})})_{\alpha\in (0,2],K'\notin{{\mathcal{N}}_K}}$ is independent of $(\tilde{\mathcal{L}}_F^{\alpha,(x_{K})})_{\alpha\in (0,2]}$, and $Z$ inherits this property. Thus, Lemma~\ref{lem:corofchenstein} is in force and the desired bound hinges on suitably estimating $b_1$ and $b_2$ in \eqref{eq:b_1}--\eqref{eq:b_2}. Combining \eqref{eq:upperboundonprobalate} and \eqref{eq:cardAF} gives
\begin{equation}
\label{eq:proofalpha>1/22}
    b_1(\alpha')\leq \frac{C|S|}{|F|^{2\alpha'}}\sup_{K\in{\tilde{\A}_F}}|\mathcal{N}_K|\leq \frac{C\log(|F|)^{C'}}{|F|^{2(\alpha-2\eps)-1}\eps^{\frac{2d}{d-2}}}, \text{ for all $ \textstyle\alpha\in{(\frac12,1]}$ and $\alpha'\geq \alpha-2\eps$.}
\end{equation}
Moreover for all $K, K'\in{\tilde{\A}_F}$ with $K \neq K'$ and such that $Y_K^{\alpha'}=Y_{K'}^{\alpha'}=1$ occurs with positive probability, we must have $d(K',K)\geq R_F/2$ by definition of $Y_{\cdot},$ and thus using \eqref{eq:decouplatepoints} and \eqref{eq:upperboundonprobalate}, one readily shows that $b_2(\alpha')$ verifies the same bound as $b_1(\alpha')$ in \eqref{eq:proofalpha>1/22} over the given range of parameters $\alpha'$. Therefore combining \eqref{eq:finalboundChenStein} with \eqref{eq:cardAF}, \eqref{eq:rewriteyandz} and \eqref{eq:proofalpha>1/22}, one obtains that
\begin{equation*}
	d_{2\eps}(Y,W;\alpha)\leq  \frac{C\log(|F|)^{C'}}{|F|^{2(\alpha-2\eps)-1}\eps^{\frac{2d}{d-2}}}+\frac{C|F|^2\log(|F|)^{C'}}{|F|^3}\leq \frac{C''\log(|F|)^{C'}}{|F|^{2(\alpha-2\eps)-1}\eps^{\frac{2d}{d-2}}},
\end{equation*}
for all $\alpha\in(\frac12,1]$. It readily follows from the definition of $d_{\eps},$ see \eqref{eq:defdepsilon} (and substituting $2\eps$ by $\eps$), that a coupling exists such that \eqref{eq:strongercoupling} occurs except on an event with probability bounded by the expression in~\eqref{eq:boundonSgen}.  Moreover by Lemma~\ref{lem:wellseparated}, since all sets in $\tilde{\A}_F$ have diameter at most $R_F$, one has
\begin{equation}
\label{eq:proofalpha>1/21}
\P\Big(\L^{\alpha}_F\neq\bigcup_{K\in{\tilde{\A}_F}:\,Y_K^{\alpha}=1}K\Big)\leq\P\big(\exists x\in{F}:\,\L^{\alpha}_F\cap \BB(x,2R_F)\notin{\tilde{\A}_F\cup\{\varnothing\}}\big)\leq \frac{C\log(|F|)^{C'}}{|F|^{2\alpha-1}}.
\end{equation}
To conclude, one simply notes comparing \eqref{e:b-tilde'} and \eqref{eq:proofalpha>1/21} that 
$d_{\eps}\big(\L_F,\tilde{\B}_{F};\alpha\big)$ is bounded by the sum of $d_{\eps}(Y,W;\alpha)$ and the probability on the left-hand side of \eqref{eq:proofalpha>1/21}, which is smaller than \eqref{eq:boundonSgen} up to increasing the constants $C,C'$. 
\end{proof}

\section{Denouement}
\label{sec:denouement}

Using the findings of \S\ref{sec:chen-stein} and \S\ref{sec:loc-cons}, with the latter drawing heavily from Theorem~\ref{thm:rwshortrange}, we now prove our main results for the set of late points, Theorems~\ref{thm:uncoveredset} and~\ref{cor:phasetransition2-intro}. Recall $\alpha_*(K)$ from \eqref{eq:defalpha*K}, the measure $\mathbb{P}$ introduced atop \S\ref{subsec:L_F}, and abbreviate $\tilde{\A}_N = \tilde{\A}_{Q_N}$, see \eqref{eq:defAF}, as well as $\tilde{\mathcal B}^{\alpha} = \tilde{\mathcal B}^{\alpha}_{Q_N}$, see \eqref{eq:hatdefBalphaF}. For any set $\mathcal{S}\subset Q_N$ and $K\subset\subset \Z^d$ we introduce, with $R_N=R_{Q_N}= \log (N^d)^{\frac1{d-2}}$ (see \eqref{eq:defAF}), the event
\begin{equation}\label{e:E_N-beta}
E_{K}(\mathcal{S})= \big\{ \exists \, K'\subset\mathcal{S}: \,% \delta(K')\in{(0,R_N]}
 K' \in \mathcal{A}_N  \text{ and }\alpha_*(K')\leq\alpha_*(K) \big\}
\end{equation}
%\AP{Changed this: before $\delta(K')\in{(0,R_N]}$ was replaced by $K'\in{\A_N}.$ We actually need this new version in the proof of Theorem~1.1,ii), and I think it costs nothing in the proof of Lemma~7.1. Please check}
corresponding to the existence of `admissible' sets $K'$ in $\mathcal{S}$ with capacity larger than the capacity of $K.$ Note that $E_K(\mathcal{S})$ depends implicitly on $N$ via the choice of $\mathcal{S}\subset Q_N.$  The constant $\alpha_*(K)$ is chosen so that the following result holds.

\begin{lemma}
\label{lem:condonalphatoobserveK}
For all $\eta\in{(0,1)}$, all sequences $(\alpha_N)$ with $\alpha_N \in(\frac12+\eta,2]$ for all $N \geq 1$, all $K\subset\subset \Z^d$ with $\alpha_*(K)>\frac12,$ abbreviating $E_K^{\alpha}=E_K(\tilde{\B}^{\alpha})$ one has
\begin{equation}
\label{eq:condonalphatoobserveK}
\lim_{N\to\infty}\P\big( E_{K}^{\alpha_N} \big) = 
\begin{cases}
	=0 &\text{ if } \lim_{N}\big(1-\frac{\alpha_N}{\alpha_{\mast}(K)}\big)\log N= -\infty,\\
	%\in (0,1) &\text{ if } \lim_{N}\big(1-\frac{\alpha_N}{\alpha_{\mast}(K)}\big)\log N\in (-\infty,\infty),\\
	=1 &\text{ if } \lim_{N}\big(1-\frac{\alpha_N}{\alpha_{\mast}(K)}\big)\log N=\infty,
\end{cases}
\end{equation}
and 
\begin{equation}
\label{eq:condonalphatoobserveK'}
0< \liminf_{N\to\infty}\P\big( E_{K}^{\alpha_N} \big) \leq \limsup_{N\to\infty}\P\big( E_{K}^{\alpha_N} \big) < 1,\text{ if } \textstyle \lim_{N}\big[\big(1-\frac{\alpha_N}{\alpha_{\mast}(K)}\big)\log N\big]\in (-\infty,\infty).
\end{equation}
\end{lemma}

We refer to Remark~\ref{R:refine-crit} below for a refinement of \eqref{eq:condonalphatoobserveK'}.

\begin{proof} 
Recall that the set $\tilde{\mathcal B}^{\alpha}$ is defined entirely in terms of the family $(U_{K'})_{K'\subset Q_N}$ of i.i.d.~uniform random variables from above \eqref{eq:defAF}. The set $K'$ with the properties postulated by $E_{K}^{\alpha}$ may arise in $\tilde{\mathcal B}^{\alpha}$ for two reasons: either because the uniform variable $U_{K'}$ was triggered, i.e.~it is at most $p_{Q_N}^{\alpha}(K'),$ or because the uniform random variables corresponding to at least two disjoint subsets whose union is included in $K'$ were triggered. Accordingly, let
$\mathcal{U}_{K'}^{\alpha}=\{U_{K'}\leq p_{Q_N}^{\alpha}(K')\}$ and % We focus on the case of interlacements in \eqref{defL} from here on but the proof for the walk is virtually identical. In particular, any result of \S\ref{sec:loc-cons} we use has a version for either model.
consider the events
\begin{align}\label{eq:Falpha}
&F^{\alpha} =\big\{\exists\, K_1,K_2\in \tilde{\A}_N:  \,  K_1 \neq K_2,\delta(K_1\cup K_2)\leq R_{{N}} \text{ and } \mathcal{U}_{K_1}^{\alpha} \cap \mathcal{U}_{K_2}^{\alpha} \text{ occurs} \big\},\\[0.3em]
&G_K^{\alpha} =\big\{\exists \,K' \in{\tilde{\A}_{N}} : \, K' \subset Q_{N}, \, \alpha_*(K')\leq\alpha_*(K) \text{ and }  \mathcal{U}_{K'}^\alpha \text{ occurs}\big\}.\nonumber
\end{align}
With these definitions, one has 
\begin{align}\label{eq:boundingprobbyakn}
\begin{split}
	\P(G_{K}^{\alpha}) \leq \P(E_K^{\alpha}) \leq \P(F^{\alpha}) + \P(G_{K}^{\alpha}),
\end{split}	
\end{align}
for all $\alpha > 0$, $N \geq 1$. By a union bound and using~\eqref{eq:upperboundonprobalate} for singletons and~\eqref{eq:cardAF}, one gets 
\begin{align}
\label{eq:wellseparated}
\P(F^{\alpha_N}) \leq CN^{d} \log (N^d)^{C}\cdot N^{-2\alpha_N d}\rightarrow0 \text{ as } N\to\infty,
		\end{align}
since $\alpha_N>1/2+\eta$ for all $N.$ In view of \eqref{eq:boundingprobbyakn} and \eqref{eq:wellseparated}, it is enough to find the limit of $\P(G_{K}^{\alpha_N}) $ as $N\rightarrow\infty.$ 
To this end, first note that $\mathrm{cap}(\{x,y\})\nearrow 2\mathrm{cap}(\{0\})=2/g(0)$ as $d(x,y)\rightarrow\infty$ by \eqref{e:cap-2point}, and so for each $\alpha'\in (\frac12+\eta,\alpha_*(K)]$ and $x\in{Q_N},$ the number of sets $K'\subset Q_N$ containing $x$ such that $\alpha_*(K')\geq \alpha'$ is bounded uniformly in $N$ and $\alpha'.$ In particular, it follows that $|\{K'\subset Q_{N}:\,\alpha_*(K')=\alpha_*(K)\}|\leq CN^{d}$ and that there exists $\delta=\delta(K)>0$ such that $\alpha_*(K')\leq\alpha_*(K)-\delta$ for each $K'\subset Q_N$ with $\alpha_*(K')<\alpha_*(K).$ By~\eqref{eq:upperboundonprobalate} and~\eqref{eq:cardAF}  it thus follows that 
\begin{equation}
\label{eq:lbAKn}
\begin{split}
\P\big((G_{K}^{\alpha})^c\big)&=\prod_{\substack{K'\subset Q_{N}:\,K'\in{\tilde{\A}_{N}}\\\alpha_*(K')\leq\alpha_*(K)}}\left(1-\P(\L^{\alpha}\cap \BB(K',R_N)=K')\right)
\\&\geq \big(1-{C}{N^{-\frac{d \alpha }{\alpha_{\sast}(K)}}}\big)^{CN^{d}} \big(1-{C}{N^{-\frac{d \alpha }{\alpha_{\sast}(K)-\delta}}}\big)^{CN^{d}\log(N)^C}.
%\geq c\exp\left(-c'N^dn^{d(1-\alpha_N/\alpha_{\sast}(K))}\right).
\end{split}
\end{equation}

We now derive an upper bound on $\P((G_{K}^{\alpha})^c).$ To this end, we first introduce a set $\overline{K} \supseteq K,$ or $\overline{K}\supset \pi(K)$ if $Q_N=Q_N(\boldsymbol{0}),$ as follows. First, applying
\eqref{eq:capdecoupling} one finds $r=r(K)<\infty$ such that, whenever  $x\notin{\BB({K},r)}$, one has $\mathrm{cap}({K}\cup\{x\}) > \mathrm{cap}({K})$. Then, considering all sets of the form $K \cup U$ for $U \subset \BB({K},r)$ one finds for $N$ large enough a set $\overline{K}\subset Q_N$ having this form (or a projection on the torus of a set having this form when $Q_N=Q_N(\boldsymbol{0})$) and such that both $\mathrm{cap}(\overline{K})=\mathrm{cap}(K)$ and $\mathrm{cap}(K')>\mathrm{cap}(K)$ for all $K'\subset Q_N$ with $\overline{K}\subsetneq K'$. It follows from this construction that 
\begin{equation}
\label{eq:capincreases}
\inf_{x\in{\overline{K}^c}}\mathrm{cap}(\overline{K}\cup\{x\})\geq \mathrm{cap}(K)+\delta',
\end{equation}
for some $\delta'=\delta'(K) > 0$. The desired upper bound on $\P\big((G_{K}^{\alpha})^c\big)$ will follow from a lower bound on the probability of the event $\{\L^{\alpha}\cap \BB(\overline{K},R_N)=\overline{K}\}$. Combining the upper bound \eqref{eq:upperboundonprobalate} and \eqref{eq:capincreases}, one obtains that for all $\alpha\in{(0,2]}$,
\begin{multline*}
\P\big(\, \overline{K}\subset\L^{\alpha}\big)-\P\big(\L^{\alpha}\cap \BB(\overline{K},R_N)=\overline K\,\big)\\ \leq \P\big(\exists\, x\in{ \BB(\overline{K},R_N)\setminus \overline{K}}:\,\overline{K}\cup\{x\}\subset\L^{\alpha}\big)
%\\&\ +\P\big(\exists x\in{\BB(\overline{K},\log(n^d)^{\frac1{d-2}})\setminus \BB(\overline{K},C')}:\,\overline{K}\cup\{x\}\subset\L^{\alpha}^{\Z_n^d}\big)
\leq {C\log(N)^{C'}}{N^{-\alpha dg(0)(\mathrm{cap}(K)+\delta')}}.%+\frac{C\log(n^d)^{C'}}{n^{\alpha d g(0)(\mathrm{cap}(K)+1/(2g(0)))}}.
\end{multline*}
If $N$ is large enough, one deduces from this and the lower bound \eqref{eq:lowerboundonprobalate} that there exists $c=c(K)$ such that for all $\alpha\in{(0,2]}$,
\begin{equation}
\label{eq:lowerboundonprobalateKbar}
\P\big(\L^{\alpha}\cap \BB(\overline{K},R_N)=\overline{K}\, \big)\geq cN^{-\frac{\alpha d}{\alpha_{\sast}(K)}}.
\end{equation}
Now observe that the set $K'= \overline{K}+x$ satisfies $\alpha_*(K') = \alpha_*(K)$ by construction, see above \eqref{eq:capincreases}. Thus, by \eqref{eq:lowerboundonprobalateKbar} and translation invariance we have that for all $\alpha\in{(0,2]}$,
\begin{equation}
\label{eq:ubAKn}
\begin{split}
\P\big((G_{K}^{\alpha})^c\big)&\leq \prod_{x\in{Q_N}:\overline{K}+x\subset Q_{N}}\left(1-\P(\L^{\alpha}\cap \BB(\overline{K}+x,R_N)=\overline{K}+x)\right)\leq \big(1-{c}{N^{-\frac{d\alpha}{\alpha_{\sast}(K)}}}\big)^{cN^{d}}.%\leq \exp\left(-cn^{d(1-(\alpha_N\vee\alpha_*(K))/\alpha_*(K))}\right).
\end{split}
\end{equation}
Combining \eqref{eq:lbAKn} and \eqref{eq:ubAKn} with~\eqref{eq:wellseparated} and~\eqref{eq:boundingprobbyakn} for $\alpha=\alpha_N$ readily yields \eqref{eq:condonalphatoobserveK}--\eqref{eq:condonalphatoobserveK'}.
\end{proof}

\begin{Rk} \label{R:refine-crit} We now explain how to refine \eqref{eq:condonalphatoobserveK'}, which is of interest for the purposes of obtaining the exact constant $e^{-d}$ in \eqref{e:coup-crit}. For simplicity, we focus on the case $K=K_0=\{x,y\}$ for some $x \sim y$, whence $\alpha_*(K_0)=\alpha_*$ in view of \eqref{e:aalpha_*-equiv}.  The key is to observe that
\begin{equation}\label{e:ref-crit1}
\big\{K'\subset Q_{N}:\,\alpha_*(K')=\alpha_*\big\} = \big\{ \{z, z'\} \subset Q_{N} : \, z' \sim z \big\};
\end{equation}
indeed, recalling $\alpha_*(\cdot)$ from \eqref{eq:defalpha*K}, it follows immediately with the help of \eqref{e:cap-2point} that $\alpha_*(\{z\})=1$, which is larger than $\alpha_*$ on account of Lemma~\ref{L:alpha_*}, so the set in question in \eqref{L:alpha_*} does not contain singletons. If $K'$ is not a pair of neighbors then $K'$ contains at least two points at $\ell^{1}$-distance $\geq 2$ and it follows that $\alpha_*(K') < \alpha_*$ using \eqref{eq:cap-distance2}. From \eqref{e:ref-crit1}, one deduces in turn that
 \begin{equation}\label{e:ref-crit2}
\big|\big\{K'\subset Q_{N}:\,\alpha_*(K')=\alpha_*\big\} \big| \sim d | Q_{N} |, \text{ as } N \to \infty.
\end{equation}
Note also that as explained in Remark~\ref{R:asymp},\ref{rk:boundonprobalate}, the constants $C(\beta_0)$ and $c(\beta_0)$ in \eqref{eq:upperboundonprobalate} and \eqref{eq:lowerboundonprobalate} for $F=Q_N$ and $\beta_0=2$ can be respectively replaced by $1+o(1)$ and $1-o(1)$ as $N\rightarrow\infty,$ and thus the constant $c$ in \eqref{eq:lowerboundonprobalateKbar} can also be replaced by $1-o(1).$ Now, inspecting the above proof, substituting a suitable upper bound implied by \eqref{e:ref-crit2} into \eqref{eq:lbAKn} and a corresponding lower bound into \eqref{eq:ubAKn}, noting that $\overline{K_0}=K_0$ for the same reasons as those yielding \eqref{e:ref-crit1}, one obtains
\begin{equation}
\label{eq:condonalphatoobserveK''}
 \lim_{N\to\infty}\P\big( E_{K_0}^{\alpha_N} \big) =1-e^{-\gamma d},
 \end{equation}
 if $d \cdot \lim_{N} \big[\big(1-\frac{\alpha_N}{\alpha_{\mast}}\big)\log N\big]= \log \gamma$ for some $\gamma >0 $.
\end{Rk}

We now turn to the proof of our main results. These are formulated entirely within the framework of \eqref{defL}, which subsumes the setup of the introduction (corresponding to $F=Q_N$ for the choice $\mathbb{P}=\mathbf{P}$), thus lending themselves to immediate generalizations, notably to the case of random interlacements,~i.e.~the choice $\mathbb{P}=\mathbb{P}^I$, see Remark~\ref{rk:thmsprinkling},\ref{R:Ri-ext}. With this in mind, we extend the definitions of the sets $\B$ and $\B_K$ from above \eqref{eq:alpha_**} and \eqref{eq:defBFK} to being subsets of $Q_N,$ rather than just subsets of $\mathbf{T}=Q_N(\mathbf{0})$ as in \S \ref{sec:intro} in accordance with the notation introduced above \eqref{e:u_F}, this amounts to simply replacing $\mathbf{T}$ by $Q_N$ in their definition. Combining Theorem~\ref{The:alpha>1/2} and Lemma~\ref{lem:condonalphatoobserveK}, one now readily obtains the following.

\begin{proof}[Proof of Theorem~\ref{cor:phasetransition2-intro}]
Let $\eps_N=\eps_N(\alpha)=N^{-\frac{1}{4}(d-2)(2\alpha-1)}$ for all $N \geq 1.$ It follows from Theorem~\ref{The:alpha>1/2} applied with $F=Q_N$ that for each $\alpha\in{(\frac12,1]}$ and $N \geq 1$, with $\tilde{\B} = \tilde{\B}_{Q_N}$ and $\mathcal{L}= \mathcal{L}_{Q_N}$,
\begin{equation}
\label{eq:couplingatepsn}
 d_{\epsilon_N}\big(\mathcal{L},\tilde{\B};\alpha\big)\leq C\log(N)^{C'}N^{-d(2\alpha-1)/2+2\eps_Nd}\tend{N}{\infty}0.
\end{equation}
Fix a set $K\subset Q_N$ with $\alpha_*(K)> \frac12$. Theorem~\ref{cor:phasetransition2-intro} deals with $\B_K$ introduced in \eqref{eq:defBFK} rather than $\tilde{\B}$ as defined below \eqref{eq:hatdefBalphaF}. We proceed to compare the two sets using Lemmas~\ref{lem:wellseparated} and~\ref{lem:condonalphatoobserveK}. Indeed recalling the event $ E_{K}^{\alpha}=E_K(\tilde{\B}^{\alpha})$ from \eqref{e:E_N-beta}, one has for $N$ large enough that $\{ \tilde{\B}^{\alpha} \neq \B_K^{\alpha}\} \subset E_{K}^{\alpha}$. We used here that the inclusion $\B_K^{\alpha}\subset\tilde{\B}^{\alpha}$ is always satisfied for $N\geq N(K)$ large enough, since the condition $\delta(A)\leq R_N$ in \eqref{eq:defAF} holds for all $A\subset\mathbf{T}$ with $\alpha_*(A)>\alpha_*(K)(>1/2)$ when $N$ is large enough by~\eqref{eq:defalpha*K}, \eqref{e:cap-2point} and \eqref{eq:cap-distance2}. 

Now, first assume that $\alpha\in{(\alpha_*(K),1]}.$ By~\eqref{eq:condonalphatoobserveK} applied to the sequence $\alpha_N = \alpha \pm \varepsilon_N$, which satisfies $\lim_{N}\big(1-\frac{\alpha_N}{\alpha_{\mast}(K)}\big)\log N= -\infty$, one deduces that $\lim_{N}\P(\tilde{\B}^{\alpha\pm\eps_N} \neq \B_K^{\alpha\pm\eps_N})=0$. This implies in turn with \eqref{eq:couplingatepsn} that
$d_{\epsilon_N}\big(\mathcal{L},{\B}_K;\alpha\big)\to 0$ as $N \to \infty$, from which the first line in \eqref{eq:phasetransitionK} directly follows, using that $\varepsilon \mapsto d_{\varepsilon}$ is decreasing. A similar reasoning using the last bound in \eqref{eq:condonalphatoobserveK'} instead, with the choice $\alpha_N = \alpha_*(K) \pm \varepsilon_N$, yields the last bound in \eqref{eq:phasetransitionK'}.

Let us now assume that $\alpha\leq\alpha_*(K).$ % we have 
 %\[
% \lim_{\epsilon\to 0}\lim_{n\to\infty} d_\epsilon(\mathcal{L},\mathcal{B};\alpha)=1.
 %\] 
  To obtain the second line in \eqref{eq:phasetransitionK} as well as the first bound in \eqref{eq:phasetransitionK'}, it suffices to show that for $\eps<\alpha(1-\frac{1}{2\alpha_{\mast}(K)})$ and for any coupling $\mathbb{Q}$ between $\mathcal{L}^{\alpha}$ and~$\B_K^{\alpha-\epsilon}$ one has
   \begin{align}\label{eq:nocoupling}
 \liminf_{N\rightarrow\infty}\mathbb{Q}(\mathcal{L}^\alpha\nsubseteq \mathcal{B}_K^{\alpha-\epsilon})\begin{cases}	=1 &\text{ if } \alpha<\alpha_*(K),
\\>0 &\text{ if } \alpha=\alpha_*(K).
\end{cases}
 \end{align}
Recall $D_{K}^{\alpha} \equiv D_{K}(\mathcal{L}^{\alpha})$ and $ D_{K}(\mathcal{B}_K^{\alpha})$ from \eqref{eq:DNK}; that is, $D_{K}^{\alpha}(\mathcal{B}_K^{\alpha})$ counts the number of times a translated version of $K$ (or its projection on the torus) appears in $\mathcal{B}^{\alpha}_K$. Now fix  $\rho$ such that $1- 2(\alpha-\eps)< \rho <  1- \frac{\alpha}{\alpha_{\mast}(K)}$ (the midpoint for instance), which exists by our assumption on $\eps.$ Then clearly since $\B_K$ is an independent field one has by a similar reasoning as in \eqref{eq:wellseparated} the bound $\mathbb{E}_{\mathbb{Q}}\big[D_{K}(\mathcal{B}_{K}^{\alpha-\eps})\big] \leq  C \log(N)^{C'}N^{d(1- 2(\alpha-\eps))}.$ Hence by Markov's inequality one immediately infers that $\mathbb{Q}\big(D_{K}(\mathcal{B}_K^{\alpha-\varepsilon}) \geq N^{d\rho }\big) \to 0$ as $N \to \infty$. On the other hand, applying Lemma~\ref{L:D_LB} with $\varepsilon_N=N^{d(\rho-1+\frac{\alpha}{\alpha_{\sast}(K)})}$, one finds that  $\liminf_{N}\mathbb{Q}\big(D_{K}(\mathcal{L}^{\alpha}) \geq N^{d\rho }\big)$ is equal to $1$ if $\alpha<\alpha_*(K),$ and is positive if $\alpha=\alpha_*(K)$. Together, these imply \eqref{eq:nocoupling}.

Let us finally assume that $\alpha_*(K)\leq1/2$. For each $x,y\in{Q_N}$, it follows from \eqref{e:cap-2point} that $\alpha_*(\{x,y\})> 1/2\geq \alpha_*(K)$ and so by \eqref{eq:defBFK} $\B_K^{\alpha}$ stochastically dominates $\{x\in{Q_N}:\,\exists\,y\in{Q_N},U_{\{x,y\}}\leq p^{\alpha}(\{x,y\})\}$. 
Let us define independent Bernoulli random variables $X_{x,y}$, $x,y\in{Q_N}$, with parameter $p^{\alpha}(\{x,y\})/2$ so that $X_{x,y}=X_{y,x}=0$  if $U_{\{x,y\}}> p^{\alpha}(\{x,y\})$. If $Y_{x}=1\{\exists\,y\in{Q_N}:\,X_{x,y}=1\}$, then $(Y_x)_{x\in{\Z^d}}$ are i.i.d.\ Bernoulli random variables and $Y_x=1$ implies $x\in{\B_K^{\alpha}}$. Moreover by \eqref{eq:lowerboundonprobalateKbar}
\begin{equation}
\label{eq:alpha*K<1/2}
\P(x\in{\B_K^{\alpha}})\geq \P(Y_x=1)\geq 1-\prod_{y\in{Q_N}}\P(X_{x,y}=0)%= 1-\prod_{y\in{Q_N}}\big(1-p^{\alpha}(\{x,y\})\big)
\geq 1-(1-cN^{-2\alpha d})^{N^d}\geq c'(N^{d-2\alpha d}\wedge 1).
\end{equation}
where in the last inequality we used $1-cN^{-2\alpha d}\leq e^{-cN^{-2\alpha d}}$ and $e^{-x}\leq 1- e^{-1}(x\wedge 1)$ for all $x\geq0$. Therefore $|\B_K^{\alpha}|$ stochastically dominates a binomial random variable with parameters $(N^d,c'(N^{d-2\alpha d}\wedge 1))$, and is thus larger than $c''(N^{2d(1-\alpha)}\wedge N^d)$ with probability going to $1$ as $N\rightarrow\infty$ by Chebyshev's inequality. Moreover $\mathbb{E}[|\L^{\alpha}|]\leq CN^{d(1-\alpha)}$ by \eqref{eq:upperboundonprobalate}, and so by Markov's inequality we have $|\L^{\alpha}|\leq N^{d-\alpha d}\log N$ with probability going to $1$ as $N$ goes to infinity, and we can conclude since $\alpha<1$.
\end{proof}

We postpone further comments for a few lines and first give the:

\begin{proof}[Proof of Theorem~\ref{thm:uncoveredset}] 
Fix two neighbors $x\sim y$, $x,y \in \Z^d$, and let $K_0=\{x,y\}$. As we now explain, $\ref{ite:imain}$ is in essence an application of Theorem~\ref{cor:phasetransition2-intro} for this choice of $K=K_0$.
First recall that
 $\alpha_*=\alpha_*(K_0)$ by Lemma~\ref{e:aalpha_*-equiv}, and that the only sets $K'$ with $\alpha_*(K')>\alpha_*(K_0)$ are singletons on account of \eqref{eq:defalpha*K} and \eqref{eq:cap-distance2}. Therefore, the set $\B_{K_0}^{\alpha}$ in \eqref{eq:defBFK} only has contributions from sets $A$ which are singletons, and recalling $\B^{\alpha}$ from the beginning of \S\ref{Subsec:PT} which we can define via the same uniform random variables $(U_{\{x\}})_{x\in{Q_N}}$ as $\B^{\alpha}_{K_0}$ in \eqref{eq:defBFK} (that we also define under the probability measure $\P$), it follows that for all $\alpha> \alpha_*$,
\begin{equation} \label{e:TV-B-Btilde}
\P\big({\B}^{\alpha}_{K_0}\neq{\B}^{\alpha}\big)\leq \sum_{z\in Q_N} \mathbb{P}\big(z\in \mathcal{L}^\alpha, \BB(z,R_{Q_N})\cap \mathcal{L}^\alpha\neq \{z\}\big)\leq CN^{d(1- \frac{\alpha}{\alpha_{\sast}})},
\end{equation}
where the last inequality relies on~\eqref{eq:upperboundonprobalate}. Combining this with the first line of \eqref{eq:phasetransitionK} immediately yields item $\ref{ite:imain}$ of Theorem~\ref{thm:uncoveredset} for all $\alpha\in{(\alpha_*,1)}$ and $\eps\in{(0,\eps_0)}$, for $\eps_0$ small enough. When $\alpha=1$, the first line of \eqref{eq:phasetransitionK} is still valid, as should be clear from the proof of Theorem~\ref{cor:phasetransition2-intro}, and so item $\ref{ite:imain}$ of Theorem~\ref{thm:uncoveredset} is also fulfilled when $\alpha=1$, as well as when $\alpha>1$ as both $\B^{\alpha}$ and $\L^{\alpha}$ are then empty with high probability by \eqref{eq:Lalphaasymp}. Note that we can actually take $\eps_0(\alpha)=\alpha$ for all $\alpha>\alpha_*$ by monotonicity of $\eps\mapsto d_{\eps}.$

We now turn to the proof of $\ref{ite:iimain}.$ Let $\widetilde{E}(\L^{\alpha_{\mast}})$ refer to the event appearing in the statement of Lemma~\ref{lem:wellseparated} for the choices $F=Q_N$, $R=R_N(=R_{Q_N})$ and $\alpha=\alpha_*$. Thus, if $\widetilde{E}(\L^{\alpha_{\mast}})$ does not occur, for any $x \in Q_N$ one has that $\L^{\alpha_{\mast}}\cap \BB(x,R_N)$ is either empty or an element of $\tilde{\A}_N$.
If in addition $E_{K_0}(\L^{\alpha_{\mast}})$ does not occur, %and that $\L^{\alpha_{\mast}}\cap Q(x,2R_N)\in{\A_N\cup\{\emptyset\}}$ for all $x\in{Q_N},$ 
then assuming that $\L^{\alpha_{\mast}}\cap Q(x,R_N)=K$ for some $x \in Q_N$ and $K\neq\varnothing,$ one has that %$K\in{\A_N},$ and so 
$\alpha_*(K)>\alpha_*(K_0)$ by \eqref{e:E_N-beta}, and so $|K|=1.$ In particular, applying this to $x \in \mathcal{L}^{\alpha_{\mast}}$, one obtains that on the complement of $\widetilde{E}(\L^{\alpha_{\mast}}) \cup E_{K_0}(\L^{\alpha_{\mast}})$, the set $\L^{\alpha_{\mast}}$ is the union of all the $x\in{Q_N}$ such that $\L^{\alpha_{\mast}}\cap \BB(x,R_N)=\{x\}.$ Therefore if $\widetilde{E}(\L^{\alpha_{\mast}}) \cup E_{K_0}(\L^{\alpha_{\mast}})$ does not occur and  \eqref{eq:strongercoupling} at $\alpha=\alpha_*$ is verified for all $K\subset Q_N$ with $|K|=1$, it follows that $\B_{K_0}^{\alpha_{\mast}+\eps}\subset\L^{\alpha_{\mast}}\subset\B_{K_0}^{\alpha_{\mast}-\eps}$. Combining now Theorem~\ref{The:alpha>1/2} %and Lemma~\ref{lem:wellseparated} (with $R=2R_N$)
with the trivial inclusion $\B_{K_0}^{\alpha_{\mast}-\eps}\subset\B^{\alpha_{\mast}-\eps}$ and \eqref{e:TV-B-Btilde} at $\alpha=\alpha_*+\eps$, and using Lemma~\ref{lem:wellseparated} to bound $\P(\widetilde E(\L^{\alpha_{\mast}})) $, one obtains that% that under an appropriate coupling $\mathbb{Q}$
\begin{equation}
\label{eq:upperbounddeps}
d_{\eps}(\L,\B;\alpha_*)\leq \P\big(E_{K_0}(\L^{\alpha_{\mast}})\big)+CN^{-\frac{d\eps}{\alpha_{\sast}}}+C\log(N)^{C'}N^{d-2(\alpha_{\mast}-\eps)d}.
\end{equation}
for some constant $C,C'<\infty$ depending only on $\eps$ and $d.$  An asymptotically matching lower bound is
\begin{equation}
\label{eq:lowerbounddeps}
d_{\eps}(\L,\B;\alpha_*)\geq\P\big(E_{K_0}(\L^{\alpha_{\mast}})\big)-\P\big(E_{K_0}(\B^{\alpha_{\mast}-\eps})\big),
\end{equation} 
which simply follows from the fact that $\L^{\alpha_{\mast}}\subset\B^{\alpha_{\mast}-\eps}$ cannot occur when $\L^{\alpha_{\mast}}$ contains sets $K\in{\tilde{\A}_N}$ and $\alpha_*(K)\leq \alpha_*(K_0)$ but not $\B^{\alpha_{\mast}-\eps}.$ Moreover the event $E_{K_0}(\B^{\alpha_{\mast}-\eps})$ is included in the event $F^{\alpha_{\mast}-\eps}$ from \eqref{eq:Falpha} by definition, and so by \eqref{eq:wellseparated} we have for all $\eps \in (0, \eps_0)$, with $\eps_0(\alpha_*)=\alpha_*-\frac{1}{2},$
\begin{equation}
\label{eq:separationforB}
\begin{split}
&\P\big(E_{K_0}(\mathcal{B}^{\alpha_{\mast}-\varepsilon})\big) \leq C\log(N)^{C'}N^{d-2(\alpha_{\mast}-\eps)d}\tend{N}{\infty}0.
\end{split}
\end{equation}
Combining this with  \eqref{eq:upperbounddeps} and \eqref{eq:lowerbounddeps}, we deduce that for each $\eps\in{(0,\eps_0)},$ $d_{\eps}(\L,\B;\alpha_*)$ is asymptotically equivalent to $\P\big(E_{K_0}(\L^{\alpha_{\mast}})\big)$ as $N\rightarrow\infty.$ %Moreover by \eqref{eq:condonalphatoobserveK''} and the first equality in \eqref{eq:dichotomyforlalpha}, $\P\big(E_{K_0}(\L^{\alpha_{\mast}})\big)$ converges to $1-e^{-d},$ which yields \eqref{e:coup-crit}.
 Using \eqref{eq:condonalphatoobserveK''}, and combining with~\eqref{eq:couplingatepsn} for $\alpha=\alpha_*$ (recall that $\eps_N=N^{-\frac{1}{4}(d-2)(2\alpha_*-1)}$ therein and so $\eps_N\log(N)\rightarrow 0$), we can conclude since
\begin{align}\label{eq:dichotomyforlalpha}
\begin{split}
	\lim_{N\to\infty} \P\big(E_{K_0}(\mathcal{L}^{\alpha_*})\big)=\lim_{N\to\infty} \P\big(E_{K_0}(\tilde{\mathcal{B}}^{\alpha_*\pm\eps_N})\big)=1-e^{-d}.
\end{split}	
\end{align} 
Finally, $\ref{ite:iiimain}$ can be shown %when $\alpha\in{(\frac12,\alpha_*)}$ via a similar reasoning as the one below \eqref{eq:lowerbounddeps}, and when $\alpha\leq\frac12$ (or actually for all $\alpha<\alpha_*$)
in exactly the same way as the case $\alpha<\alpha_*(K)$ in the proof of Theorem~\ref{cor:phasetransition2-intro}, replacing throughout the proof $\alpha_*(K)$ by $\alpha_*,$ $\B_K$ by $\B,$ $D_{K}$ by $D,$ see \eqref{e:double-pts}, and $K$ by $K_0,$ upon taking $\eps_0(\alpha)=\alpha(1-\frac1{2\alpha_{\mast}})$ for $\alpha<\alpha_*$ in Theorem~\ref{thm:uncoveredset}. 
\end{proof}

\begin{Rk}[Extensions of Theorems~\ref{thm:uncoveredset} and~\ref{cor:phasetransition2-intro}] $\quad$\\[-0.8em]
\label{rk:thmsprinkling}
\begin{enumerate}[label*=\arabic*)]
\item\label{R:Ri-ext} Although stated in \S\ref{sec:intro} for $\mathcal{L}^{\alpha}$ as defined in \eqref{e:L^alpha-RW}, the conclusions of Theorems~\ref{thm:uncoveredset} and~\ref{cor:phasetransition2-intro} hold for either of the choices for $\P=\P^I$ or $\P=\mathbf{P}$ above \eqref{e:u_F}, i.e.\ by \eqref{defL} they have an analogue for random interlacements in $Q_N(0)$. Indeed, the proofs of these theorems are actually written so that they are also valid for random interlacements.

\item \label{R:direct-Thm1} The proof of Theorem~\ref{thm:uncoveredset},$\ref{ite:imain}$ given above uses Theorem~\ref{cor:phasetransition2-intro} to first compare $\mathcal{L^{\alpha}}$ to $\B_K$ with $K$ a set of neighbors (which in turn follows via comparison of $\mathcal{L^{\alpha}}$ and $\tilde{\B}$ using Theorem~\ref{The:alpha>1/2} and Lemma~\ref{lem:condonalphatoobserveK} to relate $\tilde{\B}$ and $\B_K$) and then $\B_K$ to $\B$. If $\alpha > \alpha_*$, for the sole purpose of deducing the relevant conclusions $\ref{ite:imain}$ in Theorem~\ref{thm:uncoveredset}, one can actually bypass the intermediate use of $\B_K$ (and $\tilde{\B}$) completely. Indeed, item $\ref{ite:imain}$ can be deduced directly using Lemma~\ref{lem:corofchenstein}, combined with Theorem \ref{The:shortrangeappro} and \eqref{eq:upperboundonprobalate}, similarly as in the proof of Theorem \ref{The:alpha>1/2} itself,  thus yielding that $d_{\eps}\big(\mathcal{L},\B;\alpha\big)\rightarrow0$ if $\alpha>\alpha_*$. Note that the proof of Theorem~\ref{thm:uncoveredset},$\ref{ite:iiimain}$ above also does not require Theorems~\ref{cor:phasetransition2-intro} and~\ref{The:alpha>1/2}, and that if one only wants to prove that the the supremum in Theorem~\ref{thm:uncoveredset},$\ref{ite:iimain}$ is bounded away from $0$ and $1$ uniformly in $N,$ one could bypass the use of $\tilde{\B}$ by proceeding similarly as in the proof of \eqref{eq:phasetransitionK'}. It seems however difficult to obtain the exact constant $e^{-d}$ at criticality in \eqref{e:coup-crit} without using the more general Theorem~\ref{The:alpha>1/2} (or Theorem~\ref{cor:phasetransition2-intro}). In a nutshell, this is because $\tilde{\B}$ provides us with more precise information about $\mathcal{L}^{\alpha_{\mast}},$ see \eqref{eq:dichotomyforlalpha}, than direct moment methods, see Lemma~\ref{L:D_LB}. More generally, proceeding similarly as in \eqref{eq:lowerbounddeps}, \eqref{eq:separationforB} and the first equality in \eqref{eq:dichotomyforlalpha}, one could see Theorems~\ref{thm:uncoveredset}  and~\ref{cor:phasetransition2-intro} when $\alpha>1/2$ as direct consequences of  Theorem~\ref{The:alpha>1/2} and Lemma~\ref{lem:condonalphatoobserveK},  without using the moment methods of Lemma~\ref{L:D_LB} (which is still required for $\alpha\leq1/2$). Lastly we note that if $\alpha_*(K)>1/2$ there is some flexibility in defining $\B_K^{\alpha}$ in \eqref{eq:defBFK} without spoiling the conclusions of Theorem~\ref{cor:phasetransition2-intro}, e.g.~by adding a constraint
on the diameter of $A$ in \eqref{eq:defBFK} of the form $\delta(A)\leq R_N$, similarly as in \eqref{eq:defAF}-\eqref{eq:hatdefBalphaF}. However if $\alpha_*(K)\leq 1/2$, it is not clear if \eqref{eq:phasetransitionK} is still verified when adding such a diameter constraint, see \eqref{eq:alpha*K<1/2}, and it is thus an interesting question if a change in \eqref{eq:defBFK} might be relevant when studying $\L^{\alpha}$ in the phase $\alpha<\frac12$.

 \item\label{rk:constantatcriticality}  As explained in the proof of Theorem~\ref{thm:uncoveredset},\ref{ite:iimain}, the constant $e^{-d}$ in \eqref{e:coup-crit} corresponds to $\P\big(E_{K_0}(\L^{\alpha_{\mast}})^c\big),$ which is essentially the probability that there do not exist two neighbors in $\L^{\alpha_{\mast}};$ we refer to \eqref{eq:condonalphatoobserveK''} and \eqref{eq:dichotomyforlalpha} for how to compute this probability. On the other hand, the limit in the critical case \eqref{eq:phasetransitionK'} of Theorem~\ref{cor:phasetransition2-intro} is not explicit (nor is it clear whether this limit actually exists). This is due to the fact that we only proved \eqref{eq:condonalphatoobserveK''} for $K=\{x,y\},$ $x\sim y,$ and in fact for this choice of $K$ the limit in \eqref{e:coup-crit} is also $e^{-d}$ by a similar reasoning. To obtain the limit in \eqref{e:coup-crit} for other choices of $K$ with $\alpha_*(K)>\frac12,$ one would need to extend \eqref{e:ref-crit2} when replacing $\alpha_*$ by $\alpha_*(K).$ The limit in \eqref{eq:phasetransitionK'} would then be $1-e^{-\beta},$ where $\beta=\beta(K)$ would be the constant multiplying $|Q_N|$ on the right-hand side of this new version of  \eqref{e:ref-crit2}.

\item\label{rk:MP} Theorem~\ref{cor:phasetransition2-intro} focuses on sets $K$ with $\alpha_*(K)>\frac12,$ and we refer to \eqref{eq:proofalpha>1/22} for the main reason why this condition appears.  The pertinence of the first line in \eqref{eq:phasetransitionK} at values $\alpha \in ( \alpha_*(K),  \frac12]$ for any $K$ with $ \alpha_*(K)  \leq \frac12$ is another matter entirely. We refer to \S\ref{subsec:outlook} for some results in this direction. The parameter $\frac12$ is also the critical parameter from \cite{MP_unif}. For the random walk on $\mathbf{T},$  they prove that if $B\subset {\bf T}$ is an independent set chosen uniformly at random, then the total variation distance between $(\L^{\alpha})^c\cap B$ and $B$ goes to $0$ as $N\rightarrow\infty$ if $\alpha>1/2,$ and to $1$ if $\alpha<1/2.$ Actually, this  can be directly deduced from Theorem~\ref{cor:phasetransition2-intro} for the case $\alpha>1/2$ (which is the more difficult case), and also when considering random interlacements instead of the random walk on the torus. %(note that the case $\alpha< 1/2$ is easy, see \cite[Section~6.1]{MP_unif}). 
Indeed, by \cite[Lemma~3.1 and Proposition~3.2]{MP_unif} (applied to $\mathbb{P}(\cdot\,|\,A)$), it is enough to find an event $A$ so that
\begin{equation}
\label{eq:toverifyPeMi}
\mathbb{E}\Big[2^{|\L^{\alpha}\cap\hat{\L}^{\alpha}|}1_A\Big]\tend{N}{\infty}1\quad\text{ and }\quad\mathbb{P}(A)\tend{N}{\infty}1,
\end{equation}
where $\hat{\L}^{\alpha}$ is an independent copy of $\L^{\alpha}.$ Let $\alpha>1/2$ and $x\in{\Z^d}$ be such that $\alpha_*(K)\in{(1/2,\alpha)},$ with $K=\{0,x\},$ which exists since $\alpha_*(\{0,x\})\rightarrow 1/2$ as $x\rightarrow\infty$ by \eqref{e:cap-2point}. Let $A$ be the event that $\{\L^{\alpha}\subset {\B}_K^{\alpha-\eps}\},$ with $\eps<\alpha-\frac12,$ intersected with $\{\hat{\L}^{\alpha}\subset \hat{\B}_K^{\alpha-\eps}\},$ where $\hat{\B}_K^{\alpha-\eps}$ are independent copy of $\B_K^{\alpha-\eps}.$ Then by Theorem~\ref{cor:phasetransition2-intro} we have $\mathbb{P}(A)\rightarrow1$ (up to changing the probability space), and one can easily verify that the left-hand side of \eqref{eq:toverifyPeMi} is satisfied since $\B_K^{\alpha-\eps}\cap \hat{\B}_K^{\alpha-\eps}$ consists of $CN^d$ independent Bernoulli variables each with parameter smaller than $CN^{-2d(\alpha-\eps)},$ and $2(\alpha-\eps)>1.$ 

\item \label{rk:otherparametrization} One could also modify the relevant timescale in the definition \eqref{defL} of $\L^{\alpha},$  and the results of Theorems~\ref{thm:uncoveredset} and~\ref{cor:phasetransition2-intro} remain true as long as \eqref{eq:upperboundonprobalate}-\eqref{eq:lowerboundonprobalate} hold. For instance one could take $\L^{\alpha}$ in \eqref{defL} as the vacant set at time $\alpha \tcov,$ where $\tcov$ is the expected cover time of $Q_N$ for either random walk or random interlacements, or in fact replace $\alpha$ by any sequence $(\alpha_N)$ with $\alpha_N\log(N)\sim \alpha \log(N)$ as $N\rightarrow\infty.$ %see also Remark~\ref{rk:covertime},\ref{rk:latecalphapoint
%} below regarding generalizations to pertinent stopping times.
 Indeed if $\gamma=\lim_N(1-\frac{\alpha_N}{\alpha_{\mast}})\log N$ exists in $[-\infty,+\infty],$ using \eqref{eq:condonalphatoobserveK}, as long as $\alpha_N>\frac12+\eta,$ for any $\eta>0,$ one can use Theorem~\ref{The:alpha>1/2}, \eqref{eq:condonalphatoobserveK}, \eqref{eq:condonalphatoobserveK''}  and \eqref{e:TV-B-Btilde} for $\alpha_N$ instead of $\alpha,$ to show that  for $\eps$ small enough $\lim_Nd_{\eps}\big(\mathcal{L},\B;\alpha_N\big)$ is equal to $0$ if $\gamma=-\infty,$ is equal to $1-e^{-d}$ if $\gamma\in{(-\infty,\infty)},$ and is equal to $1$ if $\gamma=+\infty,$  similarly as in the proof of Theorem~\ref{thm:uncoveredset}. Thus, proceeding similarly for Theorem~\ref{cor:phasetransition2-intro}, our results are robust with respect to small changes of time- reparametrization, contrary to those of \cite{JasonPerla} or \cite{SamPerla} where one had to consider the walk at timescales $\alpha t_*$ for a specific choice of $t_*\sim \tcov.$

\item\label{rk:precisebound} As we now explain, inspection of the proof of \eqref{e:coup-subcrit} and of the first line of \eqref{eq:phasetransitionK} (used in the proof of \eqref{e:coup-subcrit}) reveals that item $\ref{ite:imain}$ of Theorem~\ref{thm:uncoveredset} can be quantified as follows; for all $\alpha\in{(\alpha_*,1]},$ $N \geq 1$, $F\subset Q_N,$ $\eps\in{(0,\alpha)}$, and $\mathcal{L}^{\alpha}_{F}$ as in \eqref{defL}, one has
\begin{equation}
\label{eq:smallestsprinkling}
d_{\eps}\big(\mathcal{L}_{F},\mathcal{U}_F;\alpha\big)\leq{C} |F| \Big( |F|^{-2(\alpha-\eps)}\log(|F|)^{C'}\eps^{-\frac{2d}{(d-2)}}+ |F|^{- \frac{\alpha}{\alpha_{\sast}}} \Big),
\end{equation} 
where $\mathcal{U}_F^{\alpha} \subset F$ and $\{x\in{\mathcal{U}_F^{\alpha}}\}\stackrel{\text{def.}}{=}\{U_x\leq\mathbb{P}(x\in{\mathcal{L}^{\alpha}_{F}})\}$ for all $x\in{F},$ so that in particular, $\mathcal{U}_{Q_N} = \B$. The first term on the right-hand side of \eqref{eq:smallestsprinkling} can be traced back to \eqref{eq:boundonSgen}, which appears in the course of proving \eqref{eq:phasetransitionK}, and the second term to \eqref{e:TV-B-Btilde}. Note also that this second term is also a bound on the first line of \eqref{eq:condonalphatoobserveK} when $K=\{x,y\},$ $x\sim y,$ by \eqref{eq:boundingprobbyakn}, \eqref{eq:wellseparated} and \eqref{eq:lbAKn}. Similarly, for all $\alpha>\alpha_*(K)$ one could prove a bound on the first line of \eqref{eq:phasetransitionK} similar to \eqref{eq:smallestsprinkling} when replacing $\alpha_*$ by $\alpha_*(K)$. The extension to general $F \subset Q_N$ in \eqref{eq:smallestsprinkling} comes for free since all the results of \S\ref{sec:loc-cons} utilized in the proof (namely, Theorem~\ref{The:alpha>1/2} and the bound \eqref{eq:upperboundonprobalate}) hold at this level of generality.

\item  \label{rk:removingsprinkling} In particular, the quantitative bound \eqref{eq:smallestsprinkling} (applied to $F=Q_N$) allows one to choose $\varepsilon= N^{-c}$ for suitable $c=c(\alpha)$ and $\alpha>\alpha_*$, for which 
\begin{equation}
\label{eq:couplingatepsn2}
d_{\eps_N}(\mathcal{L},\mathcal{B};\alpha)\to 0 \text{ as $N \to \infty.$}
\end{equation} 
As we now argue, this yields a non-trivial regime of parameters $\alpha \leq 1$, for which $\eps_N$ is so small that \eqref{eq:couplingatepsn2} can be boosted to $d_{\text{TV}}=d_0$ in place of $d_{\eps_N}$. Let $\eta\in{(0,2\alpha-1)}$. Then in fact \eqref{eq:smallestsprinkling} implies that \eqref{eq:couplingatepsn2} holds for $\eps_N=N^{-c(\alpha)}$ with $c(\alpha)=\frac12(d-2)(2\alpha-1-\eta)$. Moreover, by virtue of Lemma~\ref{lem:boundonlateRIRW} one has that if $\alpha\in{(0,1)}$
\begin{multline}
\label{eq:removesprinkling}
\P(\B^{\alpha-\eps_N}\neq \B^{\alpha+\eps_N})\leq N^d\P(0\in{\L^{\alpha-\eps_N}\setminus\L^{\alpha+\eps_N}})\\
\leq \frac{N^d}{N^{d(\alpha-\eps_N)}}\Big(1-\frac{1}{N^{2d\eps_N}}+CN^{-\frac{d-2}2}\log(N)^{3/2}\Big) \leq C\eps_N N^{d-\alpha d}\log N,
\end{multline}
where in the last inequality we used that $N^{2d\eps_N} \to 1$  and the inequality $N^{-\frac{d-2}2}\log(N)^{3/2}\leq C\eps_N$ valid by our choice of $\eps_N$ and $\alpha.$
Therefore, if $d-\alpha d-\frac12(d-2)(2\alpha-1)<0,$ that is if $\alpha>\frac{3}{4}(d-\frac{2}{3})/(d-1),$ upon choosing $\eta>0$ small enough one deduces that $\P(\B^{\alpha-\eps_N}=\B^{\alpha+\eps_N})\rightarrow1$ as $N\rightarrow\infty,$ which together with  \eqref{eq:couplingatepsn2} and when $\alpha>\alpha_*$ yields that $d_{\text{TV}}(\L^{\alpha},\B^{\alpha})\to 0$.
Interestingly the parameter $\alpha_2=\alpha_*\vee\frac{3}{4}(d-\frac{2}{3})/(d-1)$ thereby emanating coincides with the parameter from \cite{SamPerla}, as one can see by carefully inspecting \cite{SamPerla} (see in particular the term $b_2$ p.10 therein).
In view of~\ref{rk:MP} and~\ref{rk:otherparametrization} above, our findings thus recover the results for the random walk on the torus of \cite{MP_unif,JasonPerla,Prata_thesis,SamPerla} in full (and also extend results such as  \eqref{e:Perla-Jason-main} to an arguably more natural choice of time-parametrization). It is an intriguing question to determine whether $\alpha_2$ and $\alpha_{**},$ see \eqref{eq:alpha_**}, coincide, or whether the (multiple) occurrences of $\alpha_2$ are an artefact of the methods and in reality $\alpha_{**}=\alpha_*$. Note that $\frac{3}{4}(d-\frac{2}{3})/(d-1)>\frac34$ and $\alpha_*,$ which is decreasing in $d$ by  \eqref{e:aalpha_*-equiv} and \cite[Lemma~C.1]{MR1174248}, verifies $\alpha_*<0.68$ by computer-assisted methods, see \eqref{e:aalpha_*-equiv} and Lemma~\ref{lem:computerassisted}. Overall this yields $\alpha_2 =\frac{3}{4}(d-\frac{2}{3})/(d-1)>\alpha_*.$
\item  \label{rk:latecalphapoint} Let us define the $\alpha$-cover time $C^{\alpha}(\mathcal{L})=\inf\big\{\beta>0:|\L^{\beta}|\leq N^{(1-\alpha)d} \big\}$ for $\alpha \in (0,1]$. Recalling ${\B}$ and $\mathcal{B}_K$ from above \eqref{eq:alpha_**} and \eqref{eq:defBFK}, define $C^{\alpha}({\B})$ and $ C^{\alpha}(\mathcal{B}_K)$ similarly but replacing $\mathcal{L}$ by ${\B}$ and $\mathcal{B}_K$, respectively. %Define $(C^{\alpha})^+(\cdot)=C^{\alpha}(\cdot) + (g(0)N^d \log |F|)^{-1}$, without the factor $N^d$ in the case of interlacements. In view of \eqref{eq:defCalphaF-gen}, up to an inconsequential rescaling factor, $(C^{\alpha})^+(\mathcal{L}^F)$ corresponds to the time at which the set of late points consists of exactly $\lfloor |F|^{1-\alpha}$ points
 With $C^{\alpha}= C^{\alpha}(\L)$, one could also show results akin to Theorems~\ref{thm:uncoveredset} when $\alpha\neq\alpha_*$ and to Theorem~\ref{cor:phasetransition2-intro} when $\alpha\neq\alpha_*(K)$, when replacing the set $\L^{\alpha}$ by $\mathcal{L}^{ C^{\alpha}},$  which is the set of late points which contains exactly $N^{d(1-\alpha)}$ points (or $\lceil N^{d(1-\alpha)} \rceil$ in case $N^{d(1-\alpha)}$ is not an integer), and the sets ${\B}_K^{\alpha}$ and $\mathcal{B}^{\alpha}$ by the sets ${\B}_K^{ C^{\alpha}({\B}_K)}$ and $\mathcal{B}^{ C^{\alpha}(\mathcal{B})}.$ %With a similar reasoning as in Remark~\ref{rk:thmsprinkling},\ref{rk:removingsprinkling}, 
We refer to \cite[Theorem~1.2]{JasonPerla} for a similar result without sprinkling when $\alpha$ is close enough to one.
The proof relies on Theorem~\ref{thm:uncoveredset}, which in particular implies for all $\alpha\in{(\alpha_*,1]}$ and $\eps\in{(0,\alpha)}$ the inequality
\begin{multline*}
\P\big(\textstyle \alpha-\frac{\eps}{3}\leq  C^{\alpha}\leq \alpha+\frac{\eps}{3}\big)\geq \P\big( |\mathcal{L}^{\alpha+\eps/3}|<N^{d(1-\alpha)}< |\mathcal{L}^{\alpha-\eps/3}|\big)
\\{\geq}  \P\big(|\B^{\alpha+\eps}|< N^{d(1-\alpha)}< |\B^{\alpha-\eps}|\big)+o(1)\rightarrow 1
%\geq1-\frac{C\log(|F|)^{C'}}{|F|^{2(\alpha-\eps)-1}\eps^{\frac{2d}{d-2}}},
\end{multline*}
as $N\rightarrow\infty$, where the last bound is an easy consequence of concentration bounds for binomial variables and \eqref{eq:Lalphaasymp}, along with similar concentration estimates for $ C^{\alpha}(\mathcal{B})$ (and analogues in the context of Theorem~\ref{cor:phasetransition2-intro} when $\alpha>\alpha_*(K)>1/2$, replacing $\B$ by $\B_K$ throughout). 
\end{enumerate}
\end{Rk}

\section{Extensions}
\label{sec:extensions}

We now discuss two extensions of our main results, one concerning the set $\mathcal{L}^{\alpha}_F$ from \eqref{defL} viewed as a process in $\alpha > 0$, the other regarding a partial description of the behaviour of $\L^{\alpha}_F$ valid in the regime $\alpha\leq \tfrac12$; see Theorems~\ref{thm:processusabovealpha*} and~\ref{The:gen}, respectively, along with the subsequent remarks.

\subsection{The process \texorpdfstring{$ \alpha \mapsto \mathcal{L}^{\alpha}_F$}{alpha->LFalpha}} \label{subsec:alpha-process} 
\label{sec:processus}

Recall the process $(\alpha_x)_{x \in Q_N}$ from \eqref{eq:bfalpha-intro} (see also our convention in \eqref{defL}, by which $(\alpha_x)_{x \in Q_N}$ implicitly refers to either of two choices). Note also that
$\alpha_x \geq \alpha_*$ for any $ x  \in{\L^{\alpha_{\sast}}}$.
\begin{theorem}
\label{thm:processusabovealpha*}
For all $N \geq 1$, there exists a coupling of $({\alpha}_x)_{x\in{Q_N}}$ with a family $(\widehat{\alpha}_x)_{x\in{Q_N}}$ of i.i.d.~exponential random variables of mean $ d\log (N)$ each,
such that for all $\eps>0$,
\begin{equation}
\label{eq:alphahatalpha}
\lim\limits_{N\rightarrow\infty}\P\left(\widehat{\alpha}_x-\eps\leq {\alpha}_x-\alpha_*\leq \widehat{\alpha}_x+\eps\text{ for all }x\in{\L^{\alpha_{\sast}}}\right)=1.
\end{equation}
\end{theorem}

The intuition behind Theorem~\ref{thm:processusabovealpha*} is roughly the following. By similar considerations as in the proof of Lemma~\ref{lem:condonalphatoobserveK}, one argues that, for each $\alpha>\alpha_*,$ all the vertices in $\mathcal{L}^{\alpha}$ are at distance at least $R_N = \log(N)^{\frac1{d-2}}$ from each other with high probability. Applying our localization results of Section~\ref{sec:shortrangeRWRI} at this scale then implies that the hitting time of each late point is roughly independent and distributed as exponential random variable with the above mean.

\begin{proof} We first consider the case $\mathbb{P}=\mathbf{P}$.  Let $t_{\sast}=  u_N(\alpha_{\sast}) N^d$ with $u_N(\alpha)= \alpha g(0) \log(N^d)$ as in \eqref{e:u_N} and define $ \mathcal{F}_{\mast} =\sigma(X_n: 0 \leq n \leq t_{\sast})$. Remark~\ref{rk:othershortrange},\ref{rk:coup-u}  can be applied for the random walk $(X_{t+t_{\sast}}-X_{t_{\sast}})$ under $\mathbf{P}_{\boldsymbol{0}}(\cdot\,|\,\mathcal{F}_{\sast})$ instead of $X$ under $\mathbf{P}_{\boldsymbol{0}}$ since it has law $\mathbf{P}_{\boldsymbol{0}}$ by Markov's property, and we denote by $(\ell_{y,u}^{(x)})_{y\in{\Z^d},x\in{Q_N}, u \in [u_1,u_0]}$ the associated short-range field of local times when $F=Q_N,$ $R=\log(N)^{\frac 2{d-2}},$ $\delta=1,$ $\eps/3$ instead of $\eps$, $u_1=u_N(\eps/2)$ and $u_0=u_N(2)$, defined on some extended probability space $\til{\mathbf P}_{\boldsymbol{0}}.$ Let 
\begin{equation}\label{eq:well-sep-K}
\mathcal{K}=\big\{x\in{\mathcal{L}^{\alpha_{\sast}}}:\,\mathcal{L}^{\alpha_{\sast}}\cap Q(x,2R)=\{x\},\ell_{x-X_{t_{\sast}}, u_N(2)}^{(x-X_{t_{\sast}})}>0\big\}
\end{equation}
denote the set of points visited according to the short-range field by the terminal time $u_N(2)N^d$. Conditionally on $\mathcal{F}_{\mast},$ we then define 
\begin{equation*}
\widehat{\alpha}_x=\sup\big\{\alpha\in{[0,2]}:\,\ell_{x-X_{t_{\sast}}, u_N(\alpha)}^{(x-X_{t_{\sast}})}=0\big\}\text{ for each }x\in{\K};
\end{equation*}
similarly, for each $x\in{Q_N\setminus\K}$ such that $\mathcal{L}^{\alpha_{\sast}}\cap Q(x,2R)=\{x\},$ we define independently $\widehat{\alpha}_x$ as $2$ plus an Exp($d\log(N)$)-distributed random variable, and for each $x\in{Q_N\setminus\K}$ such that $\mathcal{L}^{\alpha_{\sast}}\cap Q(x,2R)\neq\{x\},$ define independently $\widehat{\alpha}_x$ as some Exp($d\log(N)$)-distributed random variable. In view of \eqref{e:loc-law}, \eqref{eq:boundonlateRI}, and the  memorylessness property of the exponential random variable, one checks that $\widehat{\alpha}_x$ is Exp($d\log(N)$)-distributed for each $x\in{Q_N}.$ Moreover, it follows from the short-range property of $\ell^{(x)}$, see \eqref{e:loc-range}, that conditionally on $\mathcal{F}_{\mast}$, $\widehat{\alpha}_x$ is independent of $\sigma(\widehat{\alpha}_{y},y\in{\mathcal{K}\setminus\{x\}})$ for each $x\in{K}$, since $d(x,y)\geq 2R$ for all $x\neq y\in{\mathcal{K}}.$ Therefore $(\widehat{\alpha}_x)_{x\in{Q_N}}$ is an i.i.d.\ family of Exp($d\log(N)$)-distributed random variables.

It thus remains to prove \eqref{eq:alphahatalpha}. First \eqref{eq:couplelltilell2} for our choice of parameters implies that a.s.\
\begin{equation}
\label{eq:proofprocalpha1}
	\til{\mathbf{P}}_{\boldsymbol{0}}\left(\widehat{\alpha}_x-\eps\leq {\alpha}_x-\alpha_*\leq \widehat{\alpha}_x+\eps\text{ for all }x\in{\K\cap Q(X_{t_{\sast}},2R)^c}\,\big|\,\mathcal{F}_{{\mast}}\right)\tend{n}{\infty}1.
\end{equation}
Indeed for each $x\in{\mathcal{K}\cap Q(X_{t_{\sast}},2R)^c},$ if ${\alpha}_x-\alpha_*\geq \eps/2$ then the event $\widehat{\alpha}_x-\eps\leq {\alpha}_x-\alpha_*\leq \widehat{\alpha}_x+\eps$ is directly implied by the event in \eqref{eq:couplelltilell2} for $u_1= u_N(\eps/2)$ as above and the inequalities $\widehat{\alpha}_x/(1-\eps/3)\leq \widehat{\alpha}_x+\eps $ (for $\eps$ small enough) as well as $\widehat{\alpha}_x/(1+\eps/3)\geq \widehat{\alpha}_x-\eps.$  If ${\alpha}_x-\alpha_*\leq \eps/2$ the inequalities $\widehat{\alpha}_x-\eps\leq 0\leq{\alpha}_x-\alpha_*$ are similarly implied by \eqref{eq:couplelltilell2} for $u=u_N(\eps/2),$ and the inequality ${\alpha}_x-\alpha_*\leq \widehat{\alpha}_x+\eps$ is trivial, which concludes the proof of \eqref{eq:proofprocalpha1}. Moreover
\begin{multline}
\label{eq:proofprocalpha2}
\mathbf{P}_{\boldsymbol{0}}\big(\mathcal{L}^{\alpha_{\sast}}\cap Q(X_{t_{\sast}},2R)\neq\varnothing\big)\leq \mathbf{P}_{\boldsymbol{0}}\big(\mathcal{L}^{\frac12}\cap Q(X_{t_{\sast}},2R)\neq\varnothing\big)
\\\leq \sup_{x\in\mathbf{T}}\mathbf{P}_{\boldsymbol{0}}\big(\mathcal{L}^{\frac12}\cap Q(x,2R)\neq\varnothing\big)+\exp(-cN^{d-2}) \leq \frac{CR^d}{N^{d/4}},
\end{multline}
where in the second inequality we applied the Markov property at time $t_{\frac12}= u_N(\frac12)N^d$, observed that $t_{\sast} - t_{\frac12} \geq cN^d$ since $\alpha_{\sast}> \frac12$ (see \eqref{e:aalpha_*-equiv}) and applied a classical bound on the mixing time of $X,$ see for instance \cite[Theorem~5.6]{LPW}, to deduce that $X_{t_{\sast} - t_{1/2}}$ conditionally on $\mathcal{L}^{\frac12}$ is suitably close to being uniformly distributed on a sub-lattice; the last inequality then follows by \eqref{eq:boundonlateRW} combined with a union bound. Let us now consider vertices $x\in{\mathcal{L}^{\alpha_{\sast}}\setminus \K}.$ By \eqref{eq:boundonlateRI}, \eqref{eq:boundonlateRW} and a union bound and Markov's inequality, we have for large enough $N$
\begin{multline}
\label{eq:proofprocalpha3}
\til{\mathbf{P}}_{\boldsymbol{0}}\left(\widehat{\alpha}_x\leq \eps,{\alpha}_x-\alpha_*\leq \eps\text{ for all }x\in{\mathcal{L}^{\alpha_{\sast}}\setminus\K}\right)
\\\geq \til{\mathbf{P}}_{\boldsymbol{0}}(\widehat{\alpha}_x\leq 2\text{ for all }x\in{Q_N})
-3e^{-\eps d\log(N)}\mathbf{E}_0\left[\left|\left\{x\in{\mathcal{L}^{\alpha_{\sast}}}:\mathcal{L}^{\alpha_{\sast}}\cap Q(x,2R)\neq\{x\}\right\}\right|\right].
\end{multline}
Moreover, by \eqref{e:cap-2point}, \eqref{eq:cap-distance2}, \eqref{eq:upperboundonprobalate} and \eqref{e:aalpha_*-equiv} we have 
\begin{multline}
\label{eq:proofprocalpha4}
\mathbf{E}_0\left[\left|\left\{x\in{\mathcal{L}^{\alpha_{\sast}}}:\mathcal{L}^{\alpha_{\sast}}\cap Q(x,2R)\neq\{x\}\right\}\right|\right]\\
\leq \mathbf{E}_0\left[\left|\left\{x,y\in{\mathcal{L}^{\alpha_{\sast}}}:d(x,y)\in{[1,2R]}\right\}\right|\right] \leq CN^dR^dN^{-d}\leq C\log(N)^{\frac{2d}{d-2}}
\end{multline}
since  $R=\log(N)^{\frac2{d-2}},$ and by \eqref{eq:upperboundonprobalate} again and a union bound we know that
\begin{equation}
\label{eq:proofprocablpha5}
\til{\mathbf{P}}_{\boldsymbol{0}}(\widehat{\alpha}_x\leq 2\text{ for all }x\in{Q_N})\geq 1-CN^{-d}.
\end{equation}
 The claim now readily follows by combining \eqref{eq:proofprocalpha1}, \eqref{eq:proofprocalpha2}, \eqref{eq:proofprocalpha3}, \eqref{eq:proofprocalpha4} and \eqref{eq:proofprocablpha5}. Finally when $\P=\PI$ the proof is similar except that one conditions on $\F_{\mast}=\sigma(\omega^{u_N(\alpha_*)}_{Q_N})$ instead, see \eqref{eq:definterprocess}, and uses a version of \eqref{eq:couplingshortrangeinter} for the process consisting of the trajectories of interlacements above level $u_N(\alpha_*)$, which has the same law conditionally on $\F_{\mast}$ as an interlacement process since the increments are stationary and independent.
\end{proof}

\begin{Rk} \begin{enumerate}[label*=\arabic*)]
\item In much the same way as in Remarks~\ref{rk:thmsprinkling},\ref{R:Ri-ext} and~\ref{rk:precisebound}, Theorem~\ref{thm:processusabovealpha*} implicitly applies to both random walk on $\mathbf{T}$ and random interlacements in the box $Q_N$. It further naturally generalises to $F\subset Q_N$, i.e.~to the process ${\alpha}^F_x\stackrel{\text{def.}}{=} \sup\{\alpha>0:\ x\in{\L^{\alpha}_F}\}$, $x\in{F}$ (so that $\alpha_x=\alpha^{Q_N}_x$, see \eqref{eq:bfalpha-intro}). The conclusions of Theorem~\ref{thm:processusabovealpha*} remain true upon replacing the reference process $\widehat{{\alpha}}$ by i.i.d.~exponential random variables with mean $\log|F|$. Note however that the proof of Theorem~\ref{thm:processusabovealpha*} relies heavily on the Markov property of the random walk, or the independence of the increments of interlacements, and thus might be harder to generalize to other models than our other results, see  Remark~\ref{rk:final},\ref{R:univ}.

\item It might be at first surprising that the proof of Theorem~\ref{thm:processusabovealpha*}, giving a description of $\L^{\alpha}$ as a \textit{process} in $\alpha$ for $\alpha>\alpha_*$, does not rely on the Chen-Stein method from Section~\ref{sec:chen-stein}, contrary to the proof of Theorem~\ref{thm:uncoveredset}, which gives a description of $\L^{\alpha}$ at \textit{fixed} $\alpha>\alpha_*$. The reason is that Theorem~\ref{thm:processusabovealpha*} only describes the law of the hitting times of the points in $\L^{\alpha}$, $\alpha>\alpha_*$, and not their position on the lattice. In particular, one cannot deduce Theorem~\ref{thm:uncoveredset} from Theorem~\ref{thm:processusabovealpha*}. In the proof of Theorem~\ref{thm:processusabovealpha*}, the use of Chen-Stein is essentially bypassed by our localization result, Theorem~\ref{The:shortrangeapprointro-new}, applied to the random walk after time $u_N(\alpha_*)$, which gives a short-range field $\tilde{\ell}$ independent of $\L^{\alpha_{\sast}}$. The short-range property of $\tilde{\ell}$ then manifests itself as independence property on the set $\K$ of well-separated points in $\L^{\alpha_{\sast}}$, which is almost equal to $\L^{\alpha_{\sast}}$, see \eqref{eq:proofprocalpha3} and \eqref{eq:proofprocalpha4}. However, if $\tilde{\ell}$ is now the process from Theorem~\ref{The:shortrangeapprointro-new} applied to the random walk after time $0$, then it is not independent of $\L^{\alpha_{\sast}}$, and so it is not clear at all that the process $\tilde{\ell}_{x,u}$, $x\in{\K}$, is independent, hence our use of the Chen-Stein method to overcome this issue in the proof of Theorem~\ref{thm:uncoveredset}.

\item One can also generalize the description of $\mathcal{L}_{\alpha}$ as a processus in $\alpha$ for $\alpha\geq\alpha_*$ from Theorem~\ref{thm:processusabovealpha*} to a description of  $\mathcal{L}_{\alpha}$ as a processus in $\alpha$ for $\alpha>1/2$ as follows. For simplicity we focus on the case $\mathbb{P}=\mathbf{P}$, cf.~\eqref{defL}. Informally, for each $\eta>0,$ on an event $E$ occurring with high probability, the set $\mathcal{L}^{\frac12+\eta}$ is a union of islands drawn from $\tilde{\A}_{\mathbf{T}}$ in \eqref{eq:defAF}, each at distance at least $\log(N)^{\frac2{d-2}}$ (say) from one another. Then $({\alpha}_x)_{x\in Q_N}$ behaves almost independently on each island as the hitting time of this island by interlacements. 

We now formulate this precisely. For $R>0$, we say that $\mathcal{K}$ is an $R$-well-separated partition of $S\subset Q_N$ if $\mathcal{K}$ is a partition of $S$ such that $\delta(K)< R$ for all $K\in{\mathcal{K}}$ and $d(K,K')\geq R$ for all $K\neq K'\in{\mathcal{K}}.$ Note that there is at most one $R$-well separated partition of $S.$ An example is the set $\mathcal{K}$ from \eqref{eq:well-sep-K}, which forms an $R$-well-separated partition of $S=\mathcal{L}^{\alpha_{\sast}}$ into singletons $K=\{x\}$, $x \in \mathcal{L}^{\alpha_{\sast}}$, with high probability as $N \to \infty$, as shown above. 

In a similar vein, let now $R= \log(N)^{\frac2{d-2}}$ and fix $\eta \in (0,\frac12)$. Define $E=E(\mathcal{L}^{\frac12+\eta})$ the event that $\mathcal{L}^{\frac12+\eta}$ has a $4R$-well-separated partition $\mathcal{K}$, where $t_{\eta}=  u_N(\frac12 + \eta) N^d$. Then similarly as in Theorem~\ref{thm:processusabovealpha*} one can define on a suitable extension of $\P$  an independent family $(\widehat{\alpha}^{K})_{K\in{\mathcal{K}}},$ such that $\widehat{\alpha}^{K}$ has the same law as $(\alpha_x)_{x\in K}$ under $\PI$ (corresponding to~\eqref{eq:bfalpha-intro} for interlacements) for every $K \in \mathcal{K}$ and in addition, for all $\eps >0$,
\begin{equation}
\label{eq:thm:processusabove1/2}
\P\big(E,\widehat{\alpha}^{K}_x-\eps\leq {\alpha}_x-\textstyle\frac12-\eta\leq \widehat{\alpha}^{K}_x+\eps\text{ for all }x\in{K}\text{ and }K\in{\mathcal{K}}\big) \to 1 \text{ as } N \to \infty.
\end{equation}
In essence, \eqref{eq:thm:processusabove1/2} asserts that, with high probability, the set $\mathcal{L}_{\frac12+\eta}$ consists of `islands' $K$ (corresponding to the elements of $\mathcal{K}$) which are far away from one another and such that the law of $\alpha_{|K}-\frac12-\eta$ is close, up to sprinkling, to the hitting level of each island by independent random interlacements. The proof of \eqref{eq:thm:processusabove1/2} follows similar lines as that of \eqref{eq:alphahatalpha} and relies on our localization result \eqref{eq:couplelltilell2}. Note that the law of $\alpha_{|K}$ for random interlacements, i.e.\ the law of $\widehat{\alpha}^K$ in \eqref{eq:thm:processusabove1/2}, can be explicitly described as follows: first wait a time Exp($d\mathrm{cap}(K)g(0)\log N$), at which a first trajectory in the random interlacements hits $K.$ This trajectory has law $P_{\bar{e}_K}$ and visits a subset $K'$ of $K.$ One can then let $K_1=K\setminus K',$ and similarly obtain a set $K_2$ by repeating the previous procedure but with $K_1$ instead of $K.$  Iterating this procedure until $K_n$ is empty, the law of ${\alpha}_x,$ $x\in{K},$ is then the same as the law of the first time at which $x\notin{K_n},$ $x\in{K}.$
\item One can readily deduce Theorem~\ref{thm:processusabovealpha*} from \eqref{eq:thm:processusabove1/2}. To this end, one takes $\eta=\alpha_*-1/2$ and defines $\widehat{\alpha}_x=\widehat{\alpha}_x^{\{x\}}$ for each $x\in{Q_N}$ such that $\{x\}\in{\mathcal{K}},$ on the event $E(\mathcal{L}^{\alpha_{\sast}}),$ and for each other vertex $x\in{Q_N}$ samples $\widehat{\alpha}_x$ as independent Exp($d\log(N)$) random variable. One can control $\mathbb{E}\left[\left|\left\{x\in{\mathcal{L}_{\alpha_*}}:\{x\}\notin{\mathcal{K}}\right\}\right|\right]$ in effectively the same way as \eqref{eq:proofprocalpha4}, and conclude as in \eqref{eq:proofprocalpha3}. Using the explicit description of the law of 
$\widehat{\alpha}^{K}$ above, one checks that $(\widehat{\alpha}_x)_{x\in{Q_N}}$ are indeed i.i.d.~Exp($d\log(N)$)-distributed. 
\item It would also be interesting to prove a version of \eqref{eq:alphahatalpha} without sprinkling, at least for $\alpha$ close enough to $1$, that is to show that with high probability the hitting times of $x$, $x\in{\L^{\alpha}}$, are close in total variation to i.i.d.~exponentials with mean $d\log(N)$ for $\alpha$ large enough. This does not seem to follow easily from Theorem~\ref{thm:processusabovealpha*}, as the method from Remark~\ref{rk:thmsprinkling},\ref{rk:removingsprinkling} only shows that $\B^{\alpha-\eps_n}=\B^{\alpha+\eps_n}$ for an adapted choice of $\eps_n\rightarrow0$ at a fixed level $\alpha$, and not as a process in $\alpha$.
\end{enumerate}
\end{Rk}

\subsection{Outlook: the regime \texorpdfstring{$\alpha \leq \frac12$}{alpha<1/2}} \label{subsec:outlook}

Let us finish this section with a partial description of the behaviour of the late points $\L^{\alpha}_{F}$ for $\alpha\leq 1/2.$ For each $F,K\subset Q_N$ with $\delta(K)\in{(0,R_F]}$ let $S_{F,K}=\{x\in{Q_N}:\,x+K\subset F\}$, recall from \eqref{eq:defAF} that $R_F= \log(|F|)^{\frac1{d-2}}$ and leet
\begin{equation}
\label{eq:defLFK}
    \L^{\alpha}_{F,K}=\big\{x\in{S_{F,K}}:\,\L^{\alpha}_F\cap Q(x+K,R_F)=x+K\},
\end{equation}
be the set of $x\in{S_{F,K}}$ such that $x+K$ is exactly the set of $\alpha$-late points in $F$ in a small neighborhood around $x+K,$ and we take $\L^{\alpha}_{F,K}=\varnothing$ if $\delta(K)\geq R_F.$ Correspondingly, we also define 
\begin{equation*}
{\B}^{\alpha}_{F,K}=\big\{x\in S_{F,K}: \ U_{x+K}\leq  \P\big(\mathcal{L}^{\alpha}_{F}\cap Q(x+K,R_F)=x+K\big)\big\}.
\end{equation*}
%Let us finally define for all $K\subset\Z^d$
%\begin{equation*}
%c_0^K(\alpha,b)=\frac{2\alpha}{\alpha_*(K)}-1-\frac{2bd}{d-2}\text{ for all }\alpha,b>0.
%\end{equation*}

\begin{theorem}
\label{The:gen}
Fix $\beta_0\in{(0,\infty)}.$ There exists $C=C(\beta_0)<\infty,$ such that for all $N\in\N,$ $F,K\subset Q_N$ with $\mathrm{cap}(K)\leq \beta_0,$ $\delta(K)\in{(0,R_F]},$ $\alpha\in{(\frac{\alpha_{\mast}(K)}{2},1]}$ and $\eps\in{(0, \frac\alpha2)}$,
\begin{equation}
\label{eq:gen}
    d_{\eps}\big(\L_{F,K},{\B}_{F,K};\alpha\big)\leq {CR_F^{d}} {|F|^{-(\frac{2(\alpha-2 \eps)}{\alpha_{\sast}(K)}-1)}\eps^{-\frac{2d}{d-2}}}.
\end{equation}
\end{theorem}
We refer to Remark~\ref{rk:final} below for further comments on the above theorem.
\begin{proof}
Consider $\alpha,F,K,\eps$ as in the statement of Theorem~\ref{The:gen}. By translation invariance we may assume that $0\in{K}.$ We may also assume that $\eps\geq |F|^{-\frac{(d-2)(2\alpha-1)}{2d}},$ since otherwise the right-hand side of \eqref{eq:gen} is always larger than $1$ (up to taking $C\geq1$ therein). Consider the field $(\tilde{\L}_F^{\alpha',(x)})_{\alpha'\in{(0,2]}, x\in Q_N}$  from Theorem \ref{The:shortrangeappro}, where $\lambda\geq2$ is a large enough constant chosen so that \eqref{eq:boundonSshortrange} with $R=\big( \frac{\lambda}{\varepsilon} \big)^{\frac2{d-2}} R_F$ is larger than $1-C/|F|^3,$ uniformly in $\eps$ as before.  We now define 
\begin{align*}
	\tilde{\L}_{F,K}^{\alpha'} = \big\{ x\in S_{F,K}: \ \tilde{\L}^{\alpha',(x)}_F\cap Q(x+K,R_F)=x+K\big\}
\end{align*}
and aim to apply Lemma \ref{lem:corofchenstein} with the choices $S=S_{F,K},$ $I=[0,2],$ $Y_{x}^{\alpha'}=1\{x\in{\L^{\alpha'}_{F,K}}\},$  $Z_{x}^{\alpha'}=1\{x\in{\tilde{\L}^{\alpha'}_{F,K}}\}$  and  $\cN_x=S_{F,K}\cap Q(x,3R),$ $x\in{S_{F,K}}$. Assumption~\eqref{eq:Z} is verified by our choice of $\tilde{\L}$  in Theorem \ref{The:shortrangeappro}. Moreover, by \eqref{eq:boundonSshortrange} and since $x+K\subset Q(x,R)$ for each $x\in{S_{F,K}}$ under our assumptions we have that
\begin{align}\label{eq:proofgen1}
	d_{\eps}(Y,Z;\alpha')\leq {C}{|F|^{-3}}\text{ for all }\alpha'\in{I}.
\end{align} %It follows from Theorem \ref{The:shortrangeappro} that 
  We thus only need to bound the constants $b_1$ and $b_2$ from Lemma \ref{lem:corofchenstein}. By \eqref{eq:upperboundonprobalate}, we have for all $\alpha\in{(\frac{\alpha_{\sast}(K)}{2},1]}$ and $\alpha'\geq \alpha-2\eps$,
\begin{equation}
\label{eq:proofgen2}
    b_1(\alpha')\leq {C|S_{F,K}|} \cdot \big(\sup_{x\in{S_{F,K}}}|\mathcal{N}_x| \big) \cdot {|F|^{-\frac{2\alpha'}{\alpha_{\sast}(K)}}}\leq R_F^d {|F|^{1- \frac{2 (\alpha-2\eps)}{\alpha_{\sast}(K)}}\eps^{-\frac{2d}{d-2}}},
\end{equation}
for some constant $C=C(\beta_0)<\infty.$ Moreover for $x\in{S_{F,K}}$ and $y\in{\mathcal{N}_x\setminus\{x\}}$ with $Y_x^{\alpha'}=Y_y^{\alpha'}=1$ we have $d(x,y)\geq R_F/2,$ and so by \eqref{eq:decouplatepoints}, \eqref{eq:easy} and \eqref{eq:upperboundonprobalate} we readily see that $b_2(\alpha')$ satisfies a bound similar to $b_1(\alpha')$  in \eqref{eq:proofgen2}.
 We can now conclude by combining~\eqref{eq:finalboundChenStein} with~\eqref{eq:proofgen1} and~\eqref{eq:proofgen2}.
\end{proof}

\begin{Rk}
\label{rk:final}
\begin{enumerate}[label*=\arabic*)]
\item Theorem \ref{The:gen} indicates that for each $K\subset\subset\Z^d,$ when considering only the subsets of the late points $\L_F^{\alpha}$ which look locally exactly a translation of $K,$  or its projection on the torus, these sets can be well-approximated up to a sprinkling by independent translations of $K,$ as long as $\alpha>\frac{\alpha_{\sast}(K)}{2}.$ For $K=\{0\},$ since $\alpha_*(\{0\})=1,$ this corresponds to an approximation of the isolated vertices of $\L_F^{\alpha}$ by independent vertices as long as $\alpha>\frac12,$ which is essentially contained in Theorem \ref{cor:phasetransition2-intro}. But when $|K| \geq 2$, Theorem~\ref{The:gen} describes the behaviour of sets which are translations of $K$ for some $\alpha\leq\frac12$ as well. %It is an interesting open question to describe the full law of the set $\L^F_{\alpha}$ for $\alpha\leq1/2$ as $|F|\rightarrow\infty,$ and the above discussion indicates that the main obstacle will come from isolated vertices, at least for $\alpha$ close to $1/2,$ see also Remark \ref{rk:parameter1/2}. 

\item Actually, Theorem \ref{The:gen} is mainly interesting when $\alpha\leq \alpha_*(K).$ Indeed, for $\alpha>\alpha_*(K),$ using a first-moment bound and \eqref{eq:upperboundonprobalate}, one easily sees that ${\B}^{\alpha}_{F,K}$ and $\L^{\alpha}_{F,K}$ are both empty with high probability as $|F|\rightarrow\infty,$ so $\L^{\alpha}_{F,K}$ is trivially well-approximated by ${\B}^{\alpha}_{F,K}.$ In particular, for any sets $K\subset\Z^d$ with $|K|\geq2,$ we thus have that both $\L^{\alpha}_{F,K}$ and ${\B}^{\alpha}_{F,K}$ are empty with high probability for all $\alpha>\alpha_*,$ a fact which is already implicit in Theorem~\ref{thm:uncoveredset}.

\item\label{rk:pointprocess} Another result which remains true in the regime $\alpha\leq 1/2$ is the convergence of the empirical process associated to $\mathcal{L}^{\alpha}$ to a Poisson point process on $[0,1]^d.$ More precisely, for each $\alpha\in{(0,1)}$ the point process $\sum_{x\in{\mathcal{L}^{\alpha}}}\delta_{x/N^{\alpha}}$ converges  in law to a point process on $\R^d$ with intensity the Lebesgue measure. This can be proved using \eqref{eq:Lalphaasymp} in exactly the same way as in  \cite[Corollary~0.2]{BEL1} for random interlacements and as in \cite[Corollary~3.4]{BEL} for the random walk.

\item \label{rk:alpha<1/2} It is an interesting open question to obtain a description for the asymptotic law of the full set $\L^{\alpha}_F$ for $\alpha\leq \frac12,$ and not only of $\L^{\alpha}_{F,K}$ for large enough $K$ as in Theorem~\ref{The:gen}. The main obstacle in order to do so is the lack of clustering for $\alpha\leq \frac12.$ Indeed, Theorem~\ref{The:alpha>1/2}, see also Lemma \ref{lem:wellseparated} and \eqref{eq:separationforB}, indicates that for $\alpha>\frac12,$ $\L^{\alpha}$ consists with high probability of `islands' with capacity smaller than $\frac2{g(0)},$ each with diameter smaller than $R_N$ and at distance at least $N^{2\alpha-1-\eta}$ for any $\eta>0$ from one another. Adapting Lemma \ref{lem:wellseparated}, one could even show that for each $\alpha\in{(\frac12,1]},$ these islands have diameter at most $C=C(\alpha,\eta)$ with high probability, and are thus asymptotically independent, as highlighted in Theorem \ref{The:alpha>1/2}. %Note that if $K_p\subset\Z^d,$ $p\in\N,$ is a sequence of finite sets with $\delta(K_p)\rightarrow\infty,$ the asymptotic capacity of $K_p$ is at least $2/g(0),$ which corresponds to the limit as $p\rightarrow\infty$ of $\mathrm{cap}(\{0,p\}).$ Since $1/(g(0)\alpha)< 2/g(0),$ 

However, when $\alpha\leq\frac12$ a reasoning similar to the proof of \eqref{eq:firstincluAF} shows that for each $p\in\N$ the average number of points in $\L^{\alpha}$ at distance at least $p$ from one another, but less than $\log(N^d),$ diverges to infinity as $N\rightarrow\infty.$ In other words, $\L^{\alpha}$ cannot be decomposed in bounded islands at infinite asymptotic distance from one another, which is the main conceptual obstacle in extending Theorem~\ref{The:alpha>1/2} to $\alpha\leq\frac12.$ Nevertheless, for large enough sets $K,$ $\L^{\alpha}_{K,Q_N}$ still consists of bounded islands at infinite asymptotic distance from one another, which are thus independent as highlighted in Theorem~\ref{The:gen}.

\item\label{R:univ} With future applications in mind, let us briefly explain which properties of random interlacements and random walk are used to obtain all main results from Sections~\ref{sec:loc-cons}-\ref{sec:denouement}, including Theorem~\ref{The:gen}. First Theorem~\ref{The:alpha>1/2} only uses the bound \eqref{eq:upperboundonprobalate} and Proposition~\ref{The:shortrangeappro}. From this, one can also obtain the case $\alpha>\alpha_*$ from Theorem~\ref{thm:uncoveredset}, the case $\alpha>\alpha_*(K) > \frac12 $ from Theorem~\ref{cor:phasetransition2-intro}, the last bound in \eqref{eq:phasetransitionK'}, as well as Theorem~\ref{The:gen}. One additionally needs the lower bound \eqref{eq:lowerboundonprobalate} and the decoupling \eqref{eq:decouplingprobalate} to obtain Lemma~\ref{L:D_LB} and its consequences, namely the case $\alpha<\alpha_*$ from Theorem~\ref{thm:uncoveredset}, the case $\alpha<\alpha_*(K)$  or $\alpha_*(K)\leq \frac12$ from Theorem~\ref{cor:phasetransition2-intro} and the first bound in \eqref{eq:phasetransitionK'}. Except for the critical case \eqref{eq:phasetransitionK'}, one could afford weaker versions of the bounds \eqref{eq:upperboundonprobalate} and  \eqref{eq:lowerboundonprobalate} with some additional subpolynomial term, that is only the polynomial order of $\P(0\in{\L^{\alpha}})$ is important, i.e.~the limit of $ \frac {\log \P(0\in{\L^{\alpha}})}{\log N}$ as $N \to \infty$. To obtain the precise asymptotic in the critical case $\alpha=\alpha_*,$ see \eqref{e:coup-crit}, one needs the  asymptotics of $\P(0\in{\L^{\alpha}})$ as $N \to \infty$, by which the constants $C(\beta_0)$ and $c(\beta_0)$ from \eqref{eq:upperboundonprobalate}-\eqref{eq:lowerboundonprobalate} are replaced by $1+o(1)$.  

 This hints at a universal phenomenon, valid for essentially any model satisfying estimates like \eqref{eq:upperboundonprobalate}, \eqref{eq:lowerboundonprobalate} (possibly up to some subpolynomial factor except at criticality) and \eqref{eq:decouplingprobalate}, and allowing for a `finite-range' approximation with properties akin to $\ref{ite:ishortrangeappro}$-$\ref{ite:iiishortrangeappro}$ of Proposition~\ref{The:shortrangeappro}. We hope to return to this elsewhere. Another possible extension is to consider other graphs than the $d$-dimensional torus under suitable hypotheses (e.g.\ polynomial decay of the Green function and polynomial volume growth as in \cite{drewitz2018geometry}), for which our method should be stable, see for instance Remark~\ref{R:robust}. In a related direction, we refer to \cite{berestycki2023universality} for recent work characterizing the `universality class' of Gumbel fluctuations for cover times. 
 \end{enumerate}
\end{Rk}

\appendix

\section{Appendix: proofs of Lemmas~\ref{lem:concsoftrwri} and \ref{lem:conc}}
\label{sec:app}
\renewcommand*{\thetheorem}{A.\arabic{theorem}}
\renewcommand{\theequation}{A.\arabic{equation}}

In this appendix, we prove Lemmas~\ref{lem:concsoftrwri} and \ref{lem:conc} using some large deviations estimates for excursions of random walks or random interlacements, see Propositions~\ref{prop:largedeviationrw} and \ref{prop:largedeviationri}. In order to finish the proof of Lemma~\ref{lem:concsoftrwri}, we are also going to need some Harnack-type estimate to show that the function $g_{\zeta}(z),$ see \eqref{eq:defg} and \eqref{eq:defg2}, does not depend, up to constants, on the choice $\zeta \in \partial B_2\times \partial B_3^c$.

\subsection{Harnack-type estimates}\label{sec:harnack}
The following results are tailored to our purposes. Throughout this section we only deal with the process $X$ under $P_x$, but the results immediately transfer to the walk on $\mathbf{T}$ as long as the events in question are measurable in terms of $X_{\cdot \wedge T_{Q}}$ with $Q=Q(x,r)$ for some $r< N$ under $P_x$ (this typically means $r_3< N$ below). We refer to Section~\ref{sec:notation} for notation. A function $f:\Z^d \to \R$ is called harmonic in $K \subset \Z^d$ if $f(x)=E_x[f(X_1)]$ for all $x \in K$ (which only requires knowing $f$ in the $1$-neighborhood of $K$). By \cite[Theorem~1.7.2]{Law91}, one knows that for all $\delta \in (0,1)$, $r \geq 1$ and $f$ non-negative and harmonic in $Q(0,{r(1+\delta)})$,  
\begin{align}\label{e:Harnack-0}
{f(x)}\leq  C(\delta) {f(y)}, \text{ for all }x,y\in Q(0,r)
\end{align}
(note that the ref.~\cite{Law91} states \eqref{e:Harnack-0} for Euclidean balls but \eqref{e:Harnack-0} can be deduced from it via a straightforward chaining argument). In the sequel we abbreviate $B_k= Q(0,r_k)$, for $k=1,2,3$ with $1 \leq r_1< r_2 < r_3$; similarly as in \eqref{e:B_i-choice}. Let $g_K$ denote the Green's function killed on the set $K \subset \Z^d$, so that $g_{\emptyset}(x,y)= g(x,y)$, cf.~above \eqref{e:e_K}. We start with a control which involves killing in nearby $\ell^{\infty}$-boxes. Note that the following result is completely standard for large $\delta$ (larger than $C \in (1,\infty)$) but the case of small $\delta$ requires some care.

\begin{lemma}\label{lem:green-killed}
For all $\delta\in(0,1)$, $r_{3}\geq r_1(1+\delta)^2$, $K \subset B_1$ a box (possibly $K=\emptyset$) and $x,y \in Q(0, {r_3}/{(1+ \frac{\delta}2})) \setminus  Q(K, \frac\delta2 r_1)$, one has
\begin{equation}\label{eq:green-killed}
c(\delta)|x-y|^{2-d} \leq g_{K\cup B_3^c}(x,y) \leq C|x-y|^{2-d}.
\end{equation}
Moreover, with $r_1, r_3$ as above and for all $x\in Q(0, (1+ \delta)r_1) \setminus Q(0, (1+ \frac\delta2)r_1) $,
\begin{equation}
\label{eq:ctehittingproba}
c(\delta)\leq P_x(H_{B_1}<T_{B_3})\leq P_x(H_{B_1}<\infty)\leq 1-c(\delta).
\end{equation}
\end{lemma}
\begin{proof}
First we observe that the first bound in \eqref{eq:ctehittingproba} is an easy consequence of \eqref{eq:green-killed}. Indeed, by a last-exit decomposition similar to \eqref{eq:lastexit}, one obtains that for all $x\in Q(0, (1+ \delta)r_1) $,
\begin{equation}
\label{eq:proofctehittingproba}
P_x(H_{B_1}<T_{B_3})\geq \inf_{x' \in B_{1}} g_{B_3^c}(x,x')  \text{cap}_{B_3}(B_{1}) \geq c(\delta),
\end{equation}
where the last step uses \eqref{eq:green-killed} for $K=\emptyset$ along with the capacity estimate $\text{cap}_{B_3}(B_1) \geq \text{cap}(B_{1}) \geq c r_2^{d-2}.$ In order to prove the last bound in \eqref{eq:ctehittingproba}, first notice that $P_x(H_{B_1} \geq T_{Q(0,\lambda r_1)})\geq c(\lambda,\delta)$ for any $\lambda>0 $ by projecting onto a coordinate and using a Gambler's ruin estimate. Moreover if $\lambda$ is chosen large enough, it follows from \eqref{eq:lastexit} similarly as in \eqref{eq:proofctehittingproba} that $P_y(H_{B_1}<\infty)\leq \frac12$ for all $y\in{Q(0,\lambda r_1)^c}.$ The upper bound follows by the Markov property.

We now prove \eqref{eq:green-killed}. For any $K \subset \Z^d$ one has $g_K \leq g$ and the upper bound in \eqref{eq:green-killed} follows immediately from standard estimates on the Green kernel, see for instance \cite[Theorem 1.5.4]{Law91}. We now show the lower bound, and note that by monotonicity of $B_3\mapsto g_{K\cup B_3^c}(x,y)$ and symmetry in $x$ and $y$ we may simply assume that $|y|_{\infty}\geq |x|_{\infty}$ and $y\in{\partial Q(0,r_3/(1+\frac{\delta}2))}$. Noting that by the Markov property $g_{K\cup B_3^c}(x,y)\geq P_x(H_{B_1}=\infty)\inf_{z\in{\partial Q(0,r_3/(1+\frac{\delta}2))}}g_{K\cup B_3^c}(z,y)$ and using the last bound in \eqref{eq:ctehittingproba}, we may further assume that $x\in{\partial Q(0,r_3/(1+\frac{\delta}2))}.$ Let us first fix $\lambda=\lambda(d) \in (1,\infty)$ large enough such that for all $x \in \Z^d$, 
\begin{equation}
\label{eq:lambda-choice}
\sup_{|z|_{\infty} \geq \lambda |x|_{\infty}} g(z) \leq \frac12 g(x)
\end{equation}
(recall that $g(x)=g(0,x)$); the bound \eqref{eq:lambda-choice} is obtained again using e.g.~\cite[Theorem 1.5.4]{Law91}. We now distinguish two cases. Suppose first that $x,y\in{\partial Q(0,r_3/(1+\frac{\delta}2))}$ and $|x-y|_{\infty} \leq \frac{\delta r}{10\lambda},$ where we abbreviate $r=r_3.$ Then applying the strong Markov property at time $H_{K\cup B_3^c}$, it follows that
\begin{equation*}
 g_{K\cup B_3^c}(x,y) = g(x,y) - E_x\big[ g(X_{H_{K\cup B_3^c}},y)\big] \stackrel{\eqref{eq:lambda-choice}}{\geq} \frac12 g(x,y),
\end{equation*}
where in the last step, we used that $|X_{H_{K\cup B_3^c}}-y|_{\infty} \geq \frac12\delta r \geq \lambda |x-y|_\infty$. Along with the standard bounds on $g$, this completes the verification of \eqref{eq:green-killed} in that case. 

Now suppose that $|x-y|_{\infty} \geq \frac{\delta r}{10\lambda}$. Then since $x,y\in{\partial Q(0,r/(1+\frac{\delta}2))}$ the boxes $B_y= \BB(y, \frac{\delta r}{100\lambda})$ and $B_x=\BB(x, \frac{\delta r}{100\lambda})$ can be joined using a chain of $C(\delta)$ many boxes $B_i$, each having radius $\frac{\delta r}{100\lambda}$, in such a manner that $i)$ any two consecutive boxes overlap (i.e.~$B_i \cap B_{i+1} \neq \emptyset$) %contains a box of radius $\frac{\delta r}{200\lambda}$, 
and $ii)$ if $\widetilde{B}_i\supset B_i$ refers to the concentric box having radius $\frac{\delta r}{10}$, then 
$\tilde{B}_i$ does not intersect $B_1\cup B_3^c$. It follows that for all $ x\in B_i$,
\begin{equation}\label{eq:box-to-box-GF}
P_x(H_{B_{i+1}}< H_{K\cup B_3^c}) \stackrel{{ii)}}{\geq} P_x(H_{B_{i+1}}< T_{\widetilde{B}_i}) % \geq \inf_{x' \in B_{i+1}} g_{\widetilde{B}_i^c}(x,x')  \text{cap}_{\tilde{B}_i^c}(B_{i+1}) 
\geq c,
\end{equation}
%where the penultimate bound is due to a last-exit decomposition similar to \eqref{eq:lastexit} and the last step 
 where the last bound uses monotonicity and \eqref{eq:ctehittingproba}, which is in force due to the first inequality in \eqref{eq:proofctehittingproba} and the lower bound on the killed Green's function at `short' distances already obtained, see also $i)$ and the choice of radius for $\widetilde{B}_i.$ Iterating \eqref{eq:box-to-box-GF} using the Markov property yields
\begin{equation*}
 g_{K\cup B_3^c}(x,y) \geq P_{x}\big( H_y < H_{K\cup B_3^c}\big) \geq c \inf_{y' \in B_y} P_{y'}\big( H_y < H_{K\cup B_3^c}\big) \geq c'(\delta) r^{2-d} \geq c''(\delta)|x-y|^{2-d},
\end{equation*}
where the penultimate step follows by bounding $P_{y'}(H_y < H_{K\cup B_3^c})\geq {g_{K\cup B_3^c}(y',y)}/g(0)$ and using the lower bound already derived, and the last step because $|x-y| \geq c(\delta)r $ by assumption. 
\end{proof}

We are now ready to prove that the function $g_{\zeta}(\Theta),$ see \eqref{eq:defg2}, is of constant order for suitable choice of the radii $r_k$ for $B_k$.

\begin{lemma}\label{lem:separation}
For all $\delta\in(0,1)$, $r_{k+1}\geq r_k(1+\delta)$, $k=1,2$, and all $y\in \partial B_2$ and $w\in\partial B_3$,
	\[
	P_y\big(T_{B_3}<H_{{B_1}} \, \big|\, X_{T_{B_3}}=w\big) \geq c(\delta).
\]
\end{lemma}
\begin{proof} By a last-exit decomposition in $B_2$, one finds that
\begin{equation}
\label{eq:A-1.1}
P_y\big(T_{B_3}<H_{{B_1}} , \,X_{T_{B_3}}=w\big) = \sum_{z\in \partial B_2} g_{B_1 \cup B_3^c} (y,z) P_z( \widetilde{H}_{B_2} > T_{B_3}, \, X_{T_{B_3}}=w).
\end{equation}
Using the lower bound in \eqref{eq:green-killed} with the choice $K=B_1$ and the upper bound with $K=\emptyset$, it follows that $g_{B_1 \cup B_3^c} (y,z) \geq c(\delta) g_{ B_3^c} (y,z) $ for all $y,z\in \partial B_2$. Substituting above, it follows that the right-hand side of \eqref{eq:A-1.1} is bounded from below by
$$
c(\delta)  \sum_{z\in \partial B_2} g_{ B_3^c} (y,z) P_z( \widetilde{H}_{B_2} > T_{B_3}, \, X_{T_{B_3}}=w) = c(\delta) P_y(X_{T_{B_3}}=w),
$$
where the last equality follows again by last-exit decomposition.
\end{proof}

Next, we prove that the function $g_{\zeta}(z),$ $z\in{\mathcal{K}},$ see \eqref{eq:defg}, does not depend on the choice of $\zeta\in{\partial B_2\times\partial B_3^c},$ up to constants. Recall from \eqref{eq:defLKt} that $L_{B_2}(T_{B_3})$ denotes the time of last visit to $B_2$ prior to exiting $B_3$.

\begin{lemma}\label{lem:indepofy}
For all $\delta\in (0,1)$, $r_{k+1}\geq r_k(1+\delta)$, $k=1,2$, and $x,y,z\in \partial B_2$, $v\in \partial B_1$, $w\in \partial B_3^c$,\begin{align}\label{eq:g_xi_stable}
c(\delta)\leq  \frac{P_y\big(X_{H_{B_1}\wedge T_{B_3}}=v, X_{L_{B_2}(T_{B_3})}=z \, \big| \, X_{T_{ B_3}}=w\big)
}{P_x\big(X_{H_{B_1}\wedge T_{B_3}}=v,X_{L_{B_2}(T_{B_3})}=z\big)} \leq C(\delta).
\end{align}
\end{lemma}

We first isolate the following:

\begin{claim}\label{cl:harnack}
	For all $\delta \in (0,1)$, $r_3\geq (1+\delta) r_2$,
	all $y,z\in \partial B_2$ and $w\in \partial B_3^c$, 
	\[
c(\delta)\cdot P_y\big(X_{T_{B_3}}=w\big)
 \leq P_z\big(X_{T_{B_3}}=w \, \big| \, T_{B_3}<\til{H}_{B_2} \big)\leq C(\delta) \cdot P_y\big(X_{T_{B_3}}=w\big).
	\]
\end{claim}

\begin{proof}
Let $V=\partial Q(0,r_2(1+\delta')),$ where $(1+\delta')^2=1+\delta.$ Since $V$ separates $B_2$ from $B_3^c$, applying the strong Markov property at time $H_{V}$, one obtains that 
		\begin{multline}\label{eq:firstequ}
		P_z\big(X_{T_{B_3}}=w, \,  T_{B_3}<\til{H}_{B_2} \big)
			 = \sum_{z'\in V} P_z\big(X_{T_{B_3}}=w, \,  T_{B_3}<\til{H}_{B_2}, \, X_{H_{V}}=z' \big)\\
		=\sum_{z'\in V} P_z \big(H_{V}<\til{H}_{B_2}, X_{H_{V}}=z' \big) P_{z'} \big(T_{B_3}<H_{B_2} \, \big| \, X_{T_{B_3}}=w \big) P_{z'}\big(X_{T_{B_3}}= w \big).
	\end{multline}
	Using Lemma~\ref{lem:separation} we know that the middle term in the second line is at least $c(\delta')$ (and at most $1$). 
Since $z' \mapsto P_{z'}(X_{T_{B_3}}= w)$ is harmonic in $B_3$, by \eqref{e:Harnack-0} one obtains that the last term is bounded from above and below by $P_{y}\big(X_{T_{B_3}}= w \big)$, up to constants depending only on $\delta$. Finally one knows by \eqref{eq:ctehittingproba} that $c(\delta') \leq P_{z'} (T_{B_3}<H_{B_2}) (\leq 1) .$ 
Substituting all of this into~\eqref{eq:firstequ} yields that the left-hand side of \eqref{eq:firstequ} is bounded up to constants from above and below by
\[
P_{y}\big(X_{T_{B_3}}= w \big) \sum_{z'\in V} P_z \big(H_{V}<\til{H}_{B_2}, X_{H_{V}}=z' \big) P_{z'}\big(T_{B_3}<H_{B_2} \big)= P_{y}\big(X_{T_{B_3}}= w \big)  \cdot P_{z}\big(T_{B_3}<\til{H}_{B_2} \big),
\]
which is the claim.
\end{proof}

It remains to give the

\begin{proof}[Proof of Lemma~\ref{lem:indepofy}]
We first reduce the task to the case $x=y$, by applying \eqref{e:Harnack-0} and a chaining argument  to the function $x \mapsto P_x(X_{H_{B_1}\wedge T_{B_3}}=v,X_{L_{B_2}(T_{B_3})}=z)$, which is harmonic in $B_3 \setminus B_1$. Harmonicity holds crucially because $v \in \partial B_1$, which forces the walk to visit $B_1$ prior to time $L_{B_2}(T_{B_3})$, whence $L_{B_2}(T_{B_3}) \geq 1$ under $P_x$. %The chaining argument involves boxes of side length $\frac{\delta r}{10}$ say (first for $r \geq C(\delta)$ and then adapting constants) connecting $x$ and $y$ while avoiding $B_1 \cup B_3^c$. 
This allows to effectively replace the starting point $x$ by $y$ in the denominator appearing in \eqref{eq:g_xi_stable}.

We now show \eqref{eq:g_xi_stable} for $x=y$, and write 
\begin{multline}\label{eq:eqfrac}
P_y\big(X_{H_{B_1}\wedge T_{B_3}}=v, X_{L_{B_2}(T_{B_3})}=z \, \big| \, X_{T_{ B_3}}=w\big)
\\
= \frac{P_y\big(X_{T_{ B_3}}=w \, \big| \, X_{H_{B_1}\wedge T_{B_3}}=v, X_{L_{B_2}(T_{B_3})}=z\big)}{P_y\big(X_{T_{ B_3}}=w \big)} P_y\big(X_{H_{B_1}\wedge T_{B_3}}=v, X_{L_{B_2}(T_{B_3})}=z \big)
\end{multline}
By a last-exit decomposition in $\partial B_2$,
\begin{align*}
	P_y\big(X_{T_{ B_3}}=w \, \big| \, X_{H_{B_1}\wedge T_{B_3}}=v, X_{L_{B_2}(T_{B_3})}=z\big) = P_z\big(X_{T_{ B_3}}=w \, \big| \,  T_{B_3}< \widetilde{H}_{B_2}\big).
\end{align*}
Inserting this into \eqref{eq:eqfrac} and using Claim~\ref{cl:harnack} completes the proof.
\end{proof}

\subsection{Large deviation estimate for excursions}
\label{sec:largedeviationapp}

In this section, we prove Lemma \ref{lem:concsoftrwri} on the concentration of the soft local times, as well as Lemma~\ref{lem:conc} on the concentration of the number of excursions, both for random walk and random interlacements. As in \S\ref{sec:harnack}, we assume that $B_k = Q(0, r_k)$ for $k=1,\dots 3$ with $r_1 < r_2 < r_3<N$.

We start by collecting some preliminary large-deviation estimates, which concern the random walk on the torus $\mathbf{T}$ of side length $N \geq 1$ (in dimension $ d \geq 3$). Recall the definition of the successive return times $R_k= R_k (\widehat{X},B_2,B_3)$ from \eqref{eq:deftilrhorho} and \eqref{e:X-hat} (well-defined when $r_3 < N$).
In the sequel, we denote by $\nu$ the stationary measure of $(\widehat{X}_{R_k})_{k \geq 1}$, which is supported on $\partial B_2$. To avoid clumsy notation, we identify $\nu$ with its projection on the torus (which is the invariant distribution of $({X}_{R_k})_{k \geq 1}$), and abbreviate $X_{[s,t]} = (X_n)_{s \leq n \leq t}$ in the sequel.

\begin{lemma}\label{lem:invarexitchain}
For all $\delta>0$ and $N >r_3\geq (1+\delta)r_2 \geq 1$, the Markov chain $(Y_k)_{k\geq 0}$ with $Y_k=X_{[R_k, R_{k+1}] }$ has invariant distribution $\mathbf{P}_{\nu}(X_{[0,  R_{1}]}\in{\cdot})$. Moreover, 
\begin{equation}
\label{eq:boundmixing}
d_{\textnormal{TV}}\big( \LL( (Y_{ iM})_{i=1}^K ), \,  (\mathbf{P}_{\nu}(X_{[0,  R_{1}]}\in{\cdot}))^{\otimes K} \big)\leq C K e^{-c(\delta)M},
\end{equation}
for all $M, K\in{\N}=\{1,2,\dots \}$, with $\mathcal{L}$ denoting the joint law of $(Y_{ iM})_{i=1}^K$ under $\mathbf{P}_{{x}}$, ${x} \in \mathbf{T}$. % and 
\end{lemma}

\begin{proof}
By definition, $\nu$ is the invariant distribution of $(X_{R_k})_{k\geq1}$ and since $(Y_n)_{n\geq1}$ are independent random variables conditionally on $(X_{R_k})_{k\geq1},$ one readily concludes that the invariant distribution of $(Y_k)_{k\geq1}$ is $\mathbf{P}_{\nu}(X_{[0,  R_{1}]}\in{\cdot})$.  A claim on the total variation similar to \eqref{eq:boundmixing} but concerning $(X_{R_k})_{k\geq0}$ instead of  $Y$ holds by~\cite[Lemma 2.2]{JasonPerla}, and \eqref{eq:boundmixing} then follows easily. Note that \cite[Lemma 2.2]{JasonPerla} is stated for the exit chain, i.e.~$(X_{D_k(\widehat{X},B_1,B_2)})_{k \geq 0}$, but the proof for the entrance points is identical. There it is further assumed that $r_3\geq 10\sqrt{d}r_2,$ but this is owed to the fact that excursions from square boxes to round boxes are considered. If instead one considers excursions from square boxes to square boxes as in the present case, the assumption $r_3 \geq (1+\delta) r_2$ for some $\delta>0$ is sufficient.
\end{proof}

\begin{Rk}[Identifying $\nu$] \label{R:nu-equ} 
Recalling that $\overline{\mathbf{e}}_{2}^{3}$ denotes the projection of the equilibrium measure $\overline{e}_{B_{2}}^{B_3}$ onto $\mathbf{T}$, see above \eqref{eq:gbar-def} and below \eqref{e:cap-2point} for notation, it follows from \cite[Lemma 6.1]{MR3563197} that
\begin{equation}\label{eq:ident}
\nu= \overline{\mathbf{e}}_{2}^{3}
\end{equation}
for $r_3 < N$. 
Moreover, by \cite[eq.~(9.4)]{MR3563197}, one has the exact(!) formula
\begin{equation}
\label{eq:boundrho1}
 \mathbf{E}_{\overline{\mathbf{e}}_{2}^{3}}[R_1(X,B_2,B_3)]  = \big(\cpc{B_3}{B_2}\big)^{-1} N^d.
\end{equation}
The identities \eqref{eq:ident} and \eqref{eq:boundrho1} are needed to precisely match both the number of excursions and the soft local times between random walk and random interlacements, but are otherwise unnecessary; cf.~Remark~\ref{rk:othershortrange},\ref{i}.
\end{Rk}

Let $W_{2,3}$ denote the set of nearest-neighbors paths in $\mathbf{T}$ starting in $\partial B_2,$ hitting $\partial B_3^c,$ and then ending the next time $\partial B_2$ is hit. Thus $W_{2,3}$ represents the state space of the excursion process $(Y_k)_{k \geq0}$ appearing in Lemma~\ref{lem:invarexitchain}. We now prove the following large deviations estimate for these excursions, from which Lemmas~\ref{lem:concsoftrwri} and \ref{lem:conc} for the random walk will later follow. 

The following setup is tailored to our purposes. We consider $(Z_i)_{i \geq 0}$ an i.i.d.~sequence of random variables with values in a measurable space $(A,\mathcal{A})$ and independent of $X$ (under $\mathbf{P}_{{x}}$). % or $\mathbb{P}$ depending on context; cf.~Section~\ref{sec:notation}). 
For a measurable function $F:W_{2,3}\times A\rightarrow[0,\infty)$ we then introduce the random variables $V_i=F(X_{[R_{i}, R_{i+1}]},Z_i)$ for $i \geq 0 $ and for $i_0 \geq 0$ and $m\in{\N}$ the average
\begin{equation}\label{e:shifted-avg}
\overline{V}_m= \overline{V}_{i_0,m} =\frac 1m\sum_{i=0}^{m-1}V_{i_0 + i},
\end{equation}
\begin{proposition}
\label{prop:largedeviationrw}
For all $i_0\geq0,$ $\delta>0$, $N > r_3\geq (1+\delta)r_2 \geq 1$, the following holds. If for $\theta\geq 1$,
\begin{equation}
\label{eq:assumptionV1}
\begin{split}
&%V_{i_0} \in L^1(\mathbf{P}_{{x}}) \text{ for some } {x}\in{\mathbf{T}}, \ 
 \sup_{{x}\in{\mathbf{T}}}
\mathbf{E}_{{x}}[V_{i_0}]\leq \theta 
\inf_{{x}\in{\mathbf{T}}}
\mathbf{E}_{{x}}[V_{i_0}] < \infty \text{ and }  \\
&\mathbf{E}_{x}[V_{i_0}^k]\leq k!\theta^k\mathbf{E}_{x}[V_{i_0}]^k\text{ for all } {x}\in{\mathbf{T}}, \,   k\in{\N}, %i \geq i_0,
\end{split}
\end{equation}
then there exist $C=C(\theta,\delta)<\infty$ and $c=c(\theta,\delta)>0$ such that for all $m\in{\N}$ and $\eta\in{(0,1)}$, 
\begin{equation*}
\sup_{{x}\in{\mathbf{T}}}
 \mathbf{P}_{{x}}\big(\, |\overline{V}_m-\mathbf{E}_{\nu}[{V}_{i_0}]|>\eta \mathbf{E}_{\nu}[{V}_{i_0}]\big)\leq Cm\exp\big\{ -c\sqrt{\eta^2m}\big\}.
\end{equation*}
\end{proposition}
%We note that since $N \geq 1$ is arbitrary, the previous conclusion holds equally well under $P_0$ without upper bound constraint on $r_3$.
\begin{proof}
Let $W=(W_i)_{i\geq0}$ be i.i.d.~centered random variables each having the law of ${V}_{i_0}-\mathbf{E}_{\nu}[{V}_{i_0}]$ under $\mathbf{P}_{\nu}$ and $P^W$ denote their joint law. Then,  using \eqref{eq:boundmixing} and e.g.~the characterization of $d_{\textrm{TV}}$ in terms of couplings, one obtains for all $M, K \in \mathbb{N}$ and $1\leq i \leq K$ the bound
	 \begin{align}\label{eq:couplevandw}
	d_{\textnormal{TV}}\big( \LL\big((V_{iM}-\mathbf{E}_{\nu}[{V}_{i_0}])_{i=1}^{K} \big), \,  \LL\big((W_i)_{i=1}^{K}\big) \big) \leq CK e^{-c(\delta)M},
	\end{align} 
with $\mathcal{L}$ governing the $V_{\cdot}$'s on the left-hand side referring to their joint law under $\mathbf{P}_{{x}}$ for any ${x} \in \mathbf{T}$. Assuming $\frac{m}{M}\geq2$ we have $K\stackrel{\text{def.}}{=}\lfloor m/M\rfloor-1\geq 1,$ and by the triangle inequality and a union bound, 
	\begin{equation}
	\label{eq:1stboundlargedeviationrw}
	\mathbf{P}_{{x}}\big(|\overline{V}_m-\mathbf{E}_{\nu}[{V}_{i_0}]|>\eta \mathbf{E}_{\nu}[{V}_{i_0}]\big)\leq a_1 + a_2 + a_3,
	\end{equation}
where
\begin{align*}
&a_1=\mathbf{P}_{{x}} \bigg( \exists\,p\in{\{0,\dots,M-1\}}:\,\Big|\sum_{i=1}^{K}(V_{i_0 + p+iM}-\mathbf{E}_{\nu}[{V}_{i_0}])\Big|>\frac{\eta m\mathbf{E}_{\nu}[{V}_{i_0}]}{2M} \bigg),
	\\&a_2= \mathbf{P}_{{x}} \bigg(\sum_{i=0}^{M-1}\big|V_{i_0 + i}-\mathbf{E}_{\nu}[{V}_{i_0}]\big|>\frac{\eta m\mathbf{E}_{\nu}[{V}_{i_0}]}{4}\bigg), \\
	&a_3= \mathbf{P}_{{x}} \bigg(\sum_{i=M+KM}^{m-1}\big|V_{i_0 +i}-\mathbf{E}_{\nu}[{V}_{i_0}]\big|>\frac{\eta m\mathbf{E}_{\nu}[{V}_{i_0}]}{4}\bigg).
	\end{align*}
	Applying~\eqref{eq:couplevandw}, the strong Markov property at time $R_{k}$ for $i_0\leq k< i_0+M$, a union bound and letting $t=\frac{\eta m}{2M}\mathbf{E}_{\nu}[{V}_{i_0}],$ $a_1$ is bounded by 
	\begin{align}
		\label{eq:2ndboundlargedeviationrw}
		M\sup_{{x}\in{\mathbf{T}}}\mathbf{P}_{{x}}\bigg(\Big|\sum_{i=1}^{K}V_{iM}-\mathbf{E}_{\nu}[{V}_{i_0}]\Big|>t \bigg) \leq M P^W \bigg(\Big|\sum_{i=1}^{K}W_{i}\Big|>t \bigg) + CMKe^{-cM}.
	\end{align}
Under \eqref{eq:assumptionV1}, it follows from the Bernstein inequality, see for instance \cite[Corollary 2.11]{MR3185193} for the version we use here, that for some constant $c>0,$
\begin{equation}
\label{eq:Bernstein}
P^W \bigg(\Big|\sum_{i=1}^{K}W_{i}\Big|>t \bigg)\leq2\exp\bigg\{-\frac{ct^2}{K\theta^4\mathbf{E}_{\nu}[{V}_{i_0}]^2+\theta^2\mathbf{E}_{\nu}[{V}_{i_0}]t}\bigg\};
\end{equation}
here we are implicitly using that the controls on higher moments appearing in the second line of \eqref{eq:assumptionV1} hold with $\theta^2$ in place of $\theta$ and $\mathbf{E}_{\nu}[\,\cdot \,]$ in place of $\mathbf{E}_{x}[\,\cdot \,]$ everywhere. Recalling $t$ and that $K \leq  \frac{m}{M}$, one sees that the right-hand side of \eqref{eq:Bernstein} is bounded by $\exp\{ -c(\theta) \frac{\eta^2m}{M}\}$, and together with \eqref{eq:2ndboundlargedeviationrw} this yields that 
\begin{equation}
	\label{eq:3rdboundlargedeviationrw}
a_1 \leq M e^{-c(\theta) \frac{\eta^2m}{M}} +  CMKe^{-cM}.
\end{equation}
Next, we bound $a_2$ for $\eta m\geq C(\theta) M$. To this effect, first note that, combining the assumptions in \eqref{eq:assumptionV1} and the strong Markov property, one readily obtains that $\theta^{-1} \mathbf{E}_{\nu}[{V}_{i_0}] \leq \mathbf{E}_{{x}}[{V}_{i_0}] \leq \theta \mathbf{E}_{\nu}[{V}_{i_0}]$. Feeding this into $a_2$, a union bound and the strong Markov property (applied at time $R_i$) then give that if $\eta m\geq C(\theta) M$,
	\begin{equation}
		\label{eq:4thboundlargedeviationrw}
\begin{split}
		a_2
		&\leq M\sup_{{x} \in \mathbf{T}}\mathbf{P}_{{x}}\bigg(\big|V_{i_0}-\mathbf{E}_{{x}}[{V}_{i_0}]\big|\geq \frac{c\eta m\mathbf{E}_{{x}}[{V}_{i_0}]}{4M} \bigg)
		\leq M\exp\left(-\frac{c'(\theta)\eta m}{M}\right),
\end{split}
	\end{equation}
where the second inequality is obtained by Bernstein's inequality similarly as in \eqref{eq:3rdboundlargedeviationrw} for $K=1.$ Similarly, since $M+KM\geq m-M+1$ by definition of $K,$ we have that if $\eta m\geq C(\theta)M$, 
	\begin{align}
		\label{eq:5thboundlargedeviationrw}
		a_3 \leq M e^{-\frac{c(\theta)\eta m}{M}}.
	\end{align}
	Choosing $M=\lceil \sqrt{\eta^2 m}\rceil,$ noting that $\eta m\geq C(\theta)M$ and $m/M\geq2$ hold if $\eta m \geq C'(\theta)$, which is no loss of generality, we conclude by combining \eqref{eq:1stboundlargedeviationrw}, \eqref{eq:3rdboundlargedeviationrw}, \eqref{eq:4thboundlargedeviationrw} and \eqref{eq:5thboundlargedeviationrw}.
\end{proof}

We now have all the tools to give the proof of Lemma~\ref{lem:concsoftrwri} in the random walk case.

\begin{proof}[Proof of Lemma~\ref{lem:concsoftrwri} (Random walk case)] One applies Proposition \ref{prop:largedeviationrw} with the following choices. Let $\eta=\eps$. Recalling $G^{\rm{RW}}$ from \eqref{eq:softlocattimeRW}, one takes $Z_i=\widehat{\xi}_i$, $V_i=\widehat{\xi}_ig_{\zeta_i}(z),$ and notes that $\zeta_i$ is $X_{[R_{i},R_{i+1}]}$-measurable, cf.~\eqref{eq:deftilrhorho}-\eqref{eq:defclothesline}. Then with  $\overline{V}_m= \overline{V}_{1,m}$, i.e.~$i_0=1$ in \eqref{e:shifted-avg}, one has $m\overline{V}_m=G_m^{\rm RW}(z),$ see \eqref{eq:softlocattimeRW}, and by \eqref{eq:ident}, it follows that $\mathbf{E}_{\nu}[{V}_1]=\bar{g}(z),$ see \eqref{eq:gbar-def}. Moreover $\mathbf{E}_{{x}}[\xi_1^j]=j!$ since $\xi_1$ is an exponential random variable with parameter one, and one readily deduces that the assumption \eqref{eq:assumptionV1} holds for some $\theta=\theta(\delta)<\infty$ by Lemmas~\ref{lem:separation} and~\ref{lem:indepofy}, see also \eqref{eq:defg} and \eqref{eq:defg2}, if $r_3\geq (1+\delta)r_2$ and $r_2\geq (1+\delta)r_1$. The claim follows.
\end{proof}

The proof of Lemma~\ref{lem:conc} for the random walk involves another application of Proposition~\ref{prop:largedeviationrw}. Verifying the relevant condition \eqref{eq:assumptionV1} in that case will rely on the following result.

\begin{lemma}\label{L:hit-torusbox}
	For any $\delta > 0$ and $r\leq N$, abbreviating $Q= {Q}(\mathbf{0},r)$, one has
\begin{align}
\sup_{{x}\in \mathbf{T}} \mathbf{E}_{{x}}\big[H_{Q} \big] &\leq C \cdot N^2  \Big(\frac{N}{r}\Big)^{d-2}, \label{eq:hit-torusbox-ub}\\
\inf_{{x}\in \mathbf{T} \setminus {Q}(\mathbf{0}, r(1+\delta)) }  \mathbf{E}_{{x}}\big[H_{{Q}} \big] &\geq c(\delta) \cdot N^2  \Big(\frac{N}{r}\Big)^{d-2} \label{eq:hit-torusbox-lb}.
\end{align}
\end{lemma}

\begin{proof} We will often use the classical fact that for all $\delta \in (0,1)$, $x \in Q(\mathbf{0},r(1-\delta))$ and $r< N$,
\begin{equation}\label{eq:exit-time}
c(\delta) r^2 \leq \mathbf{E}_{{x}}[T_{{Q}}] \leq Cr^2,
\end{equation}
which follows e.g.~by observing that $\mathbf{E}_{{x}}[T_{{Q}}]= {E}_{\tilde{x}}[T_{{Q}({0},r)}] = \sum_{y \in Q(0,r)} g_{\Z^d \setminus {Q}({0},r)}(\tilde{x},y)$, where $\tilde{x} \in Q(0,r) \subset \Z^d$ is such that $\pi(\tilde{x})=x$, and performing the sum using \eqref{eq:green-killed} (the upper bound in \eqref{eq:green-killed} remains valid without restriction on $x$ and $y$).

We now first assume that $r\geq \frac{N}{100}$. In this case \eqref{eq:hit-torusbox-lb} is immediate since $\mathbf{E}_{{x}}[H_{Q} ] \geq \mathbf{E}_{{x}}[T_{{Q}({x},\frac{\delta}{2} r)} ]  \geq c(\delta) N^2$ by \eqref{eq:exit-time} and assumption on $r$. As to  \eqref{eq:hit-torusbox-ub}, writing the expected value in terms of its tail probabilities, one readily obtains for any $\lambda \geq 1$ that
\begin{equation}\label{eq:hit-torus1}
\mathbf{E}_{{x}}[H_{{Q}} ] \leq \lambda N^2 \big( 1 + \sum_{k \geq1}\mathbf{P}_{{x}}[H_{{Q}} > k \lambda N^2 ] \big).
\end{equation}
One then argues that for all $x \in \mathbf{T}$, with $\tilde{x}$ as above and $\widetilde{Q}=Q(\tilde{x}, \lambda^{1/4} N)$,
\begin{equation*}
\mathbf{P}_{{x}}(H_{{Q}} \leq \lambda N^2) = P_{\tilde{x}}(H_{\pi^{-1}(Q)} \leq \lambda N^2)  \geq P_{\tilde{x}}(H_{\pi^{-1}(Q)} \leq T_{\widetilde{Q}}) - c\lambda^{-1/2} 
 \geq c_0
\end{equation*}
upon choosing $\lambda$ large enough; here the first lower bound follows from \eqref{eq:exit-time} and a first-moment estimate and the second one simply by observing that $\pi^{-1}(Q) \cap \widetilde{Q}$ always contains at least one translate of $Q(0,r)$ in its bulk, i.e.~at distance at most $N$ from $\tilde{x},$ and so since $r \geq cN$ it follows from \eqref{eq:ctehittingproba} and monotonicity that the hitting probability $P_{\tilde{x}}[H_{\pi^{-1}(Q)} \leq T_{\widetilde{Q}}]$ admits a uniform lower bound. Feeding the resulting estimate into \eqref{eq:hit-torus1} and applying the Markov property yields that the sum on the right-hand side is bounded by $\sum_{k \geq1} (1-c_0)^k< \infty$, and \eqref{eq:hit-torusbox-ub} follows.

%\PF{recheck boundaries for torus sets e.g.~$\partial \BB(0,n)$ could be empty}
Assume now that $r\leq \frac{N}{100}$. Consider the set $V= \partial Q(\boldsymbol{0},\frac{N}{20})$. As we now explain, it is enough to argue that 
\begin{equation}\label{eq:hit-torus2}
c \cdot N^2  \Big(\frac{N}{r}\Big)^{d-2} \leq \mathbf{E}_{{y}}\big[H_{Q} \big] \leq C \cdot N^2  \Big(\frac{N}{r}\Big)^{d-2}, \quad y \in V;
\end{equation}
indeed, once \eqref{eq:hit-torus2} is shown, the bound \eqref{eq:hit-torusbox-ub} immediately follows by applying the strong Markov property at time $H_V$, by which $\mathbf{E}_{{x}}[H_Q] \leq \mathbf{E}_{{x}}[H_V] + \sup_{y \in V} \mathbf{E}_{{y}}[H_{Q} ]$, using \eqref{eq:hit-torus2} to bound the second term  and \eqref{eq:hit-torusbox-ub} in the case already treated to deduce that $\sup_{x\in \mathbf{T}}\mathbf{E}_{{x}}[H_V] \leq CN^2$. To obtain \eqref{eq:hit-torusbox-lb}, one writes instead $\mathbf{E}_{{x}}[H_{Q} ] \geq \mathbf{P}_{{x}}(H_{Q} > H_V)  \inf_{y \in V} \mathbf{E}_{{y}}[H_{Q} ]$. The desired lower bound now follows from \eqref{eq:hit-torus2} and since by  \eqref{eq:ctehittingproba} $$\mathbf{P}_{{x}}(H_{Q} > H_V) \geq \inf_{z \in \Z^d \setminus Q(0,r(1+\delta))} P_z[H_{Q(0,r)}= \infty] \geq c(\delta).$$
%\PF{this is a bit annoying to show because of $\delta$; one option is by projecting onto one coordinate and using the gambler's ruin estimate.}

It thus remains to show \eqref{eq:hit-torus2}, under the assumption $r\leq \frac{N}{100}$. Throughout the rest of this proof we abbreviate $R_k=R_k({X},Q(\boldsymbol{0},\frac{N}{10})^c\cup Q,V^c)$ and $D_k=D_k({X},Q(\boldsymbol{0},\frac{N}{10})^c\cup Q,V^c),$ see \eqref{eq:deftilrhorho} whose definition can easily be extended from boxes in $\Z^d$ to general sets in $\mathbf{T}$, the system of successive stopping times corresponding to the excursions from $Q(\boldsymbol{0},\frac{N}{10})^c\cup Q$ to $V.$  %$R_0= 0$ and for $k \geq 1$, $\mathbf{D}_{k+1}= R_{k} + (H_{Q} \wedge T_{Q (\boldsymbol{0},\frac{N}{10})}) \circ \theta_{R_{k}}$ and $R_{k+1}= D_{k+1} + H_V \circ \theta_{D_{k+1}}$. 
Let
\begin{equation}\label{eq:hit-torus3}
K=  \min \{ k \geq 1: X_{R_k} \in Q\},
\end{equation}
which counts the number of excursions of type $X_{[D_k, R_{k+1}]}$ until the first one that visits $Q$. For all $y \in V$ one obtains by a reasoning similar to \eqref{eq:proofctehittingproba}, using the assumption on $r$, that $\mathbf{P}_{{y}}(X_{R_1} \in Q)$ is comparable with $(\frac{r}{N})^{d-2}$, which together with the Markov property, is readily seen to imply that $K$ stochastically dominates/is stochastically dominated by geometric random variables with corresponding parameters. In particular, it follows that for all $y \in V$,
\begin{equation}\label{eq:hit-torus7}
c \cdot  \Big(\frac{N}{r}\Big)^{d-2} \leq \mathbf{E}_{{y}}[K] \leq C \cdot  \Big(\frac{N}{r}\Big)^{d-2} ,
\end{equation}
which is all we will use in the sequel. Now, by definition of $R_k$ and $D_k$, for any $y \in V$, one has that $\mathbf{P}_y$-a.s.~$H_Q= R_{K}$, hence
\begin{equation}\label{eq:hit-torus4}
%\mathbf{E}_{{y}}\Big[ \sum_{1\leq k< K} (R_{k+1}-R_k) \Big] \leq  
\mathbf{E}_{{y}}\big[H_{Q} \big]  = \mathbf{E}_{{y}}\Big[ \sum_{1\leq k \leq K} (R_{k}-R_{k-1})\Big] %-  \mathbf{E}_{{y}}[R_{K+1} -D_K] .
\end{equation}
Owing to the strong Markov property, with $\mathcal{F}_k = \sigma(X_{\cdot \wedge R_k} )$, since $\{ K \geq k\} \in \mathcal{F}_{k-1}$, one has that
\begin{multline}\label{eq:hit-torus5}
\mathbf{E}_{{y}}\Big[ \sum_{1\leq k \leq K} (R_{k}-R_{k-1}) \Big] = \sum_{k \geq 1} \mathbf{E}_{{y}}\big[ (R_{k}-R_{k-1}) 1\{ K \geq k\}\big]\\=  \sum_{k \geq 1} \mathbf{E}_{{y}}\big[\mathbf{E}_{{y}}[ (R_{k}-R_{k-1}) | \mathcal{F}_{k-1} ] 1\{ K \geq k\}\big]
\end{multline}
and for any $y \in V$ and $k\geq1$
$$
\mathbf{E}_{{y}}[ (R_{k}-R_{k-1}) | \mathcal{F}_{k-1} ] =\mathbf{E}_y\big[\mathbf{E}_{{X_{R_k}}}[R_1]\big] \leq \sup_{v \in V}\mathbf{E}_{v}[ T_{Q (\boldsymbol{0},\frac{N}{10})}] + \sup_{w \in \mathbf{T}} \mathbf{E}_{w}[H_{V}] \leq CN^2, 
$$
which follows on account of \eqref{eq:exit-time} and \eqref{eq:hit-torusbox-ub} for the choice $r=\frac{N}{20}$ (already treated). One also has a corresponding deterministic lower bound of the same order by \eqref{eq:exit-time}, since by the Markov property, $R_{k}-R_{k-1}$ under $\mathbf{P}_y(X_{D_{k-1}}={y'}\,|\,\mathcal{F}_{k-1}),$ $y'\in{V},$ is stochastically dominated by $T_{Q(y', \frac{N}{50})}$, using the fact that $r\leq \frac{N}{100}$. Feeding the above deterministic upper/lower bound on the conditional expectation into \eqref{eq:hit-torus5} and using \eqref{eq:hit-torus7} to bound  the resulting $\mathbf{E}_{{y}}[K]$, one deduces \eqref{eq:hit-torus2} from \eqref{eq:hit-torus4}. 
\end{proof}

With Lemma~\ref{L:hit-torusbox} at hand, we proceed with the:

\begin{proof}[Proof of Lemma~\ref{lem:conc} (Random walk case)] 

We aim to apply Proposition~\ref{prop:largedeviationrw} with the choices $i_0=0$, $V_i=R_{i+1}-R_{i},$ so that $V_{i_0}=R_1$ and $m\overline{V}_m=R_{m}$ for any $m \geq 1$ in view of \eqref{e:shifted-avg}. Now pick $\eta=1-\frac1{1+\eps}$ and $m=\lceil(1+\eps)uM\rceil$ with 
\begin{equation} \label{e:avg-number'}
M= \frac{N^d}{\mathbf{E}_{\nu}[V_{0}]},
\end{equation} 
which equals the value defined by \eqref{e:avg-number} on account of \eqref{eq:ident}-\eqref{eq:boundrho1}. Note in passing that \eqref{e:avg-number'} is very intuitive (more so than its pendant \eqref{e:avg-number}): $uM$ with $M$ as in \eqref{e:avg-number'} is the total time $uN^d$ for the walk, divided by the `average' time $\mathbf{E}_{\nu}[V_{0}]$ consumed by an excursion, whence $M$ counts the `average' number of excursions.
With the above choices,
\begin{equation*}
	\mathbf{P}_{\boldsymbol{0}}\big(\Excrw(B_2,B_3,u)\geq (1+\epsilon)uM\big)
	\stackrel{\eqref{eq:defnumberexcursionRW}}{\leq} \mathbf{P}_{\boldsymbol{0}}\big( R_{\lceil(1+\epsilon)uM\rceil}\leq uN^d\big)\leq\mathbf{P}_{\boldsymbol{0}}\big( \, \overline{V}_m\leq (1-\eta)\mathbf{E}_{\nu}[{V}_0]\big),
\end{equation*}
with $B_k$ as in \eqref{e:B_i-choice}.
 Similarly, taking $m'=\lfloor(1-\eps)uM\rfloor+1$ and $\eta'=-1+\frac1{1-\eps/2},$ one obtains that if $\eps uM\geq2$
\begin{align*}
	\mathbf{P}_{\boldsymbol{0}}\big(\Excrw(B_2,B_3,u)\leq (1-\epsilon)uM\big)\leq\mathbf{P}_{\boldsymbol{0}}\big(\, \overline{V}_{m'}\geq (1+\eta')\mathbf{E}_{\nu}[{V}_0]\big).
\end{align*}
Since by monotonicity, we may assume that $\eps uM\geq2$ and $\eps\leq1/2$, the claim immediately follows by means of Proposition~\ref{prop:largedeviationrw}, provided we show that \eqref{eq:assumptionV1} holds, which we proceed to do with the help of Lemma~\ref{L:hit-torusbox}.

Recalling the definition of $R_1 (= V_0)$ from \eqref{eq:deftilrhorho}, identifying $B_k$ with its projection onto $\mathbf{T}$, one has that $R_1 \geq T_{B_3}$ holds $\mathbf{P}_{{y}}$-a.s.~for any $y \in \mathbf{T}$. Thus, applying the strong Markov property at time $T_{B_3}$, and combining \eqref{eq:hit-torusbox-ub} and the exit time estimate $ \mathbf{E}_y[T_{B_3}] \leq Cr_3^2$, valid for all $x \in \mathbf{T}$ (cf.~\eqref{eq:exit-time}), one sees that
\begin{equation}\label{eq:RW-1stmom-V1}
\sup_{{y}\in{\mathbf{T}}}
\mathbf{E}_{{y}}[V_0] \leq  \sup_{y\in B_3} \mathbf{E}_y[T_{B_3}]+ \sup_{{z}\in{\mathbf{T}}}
\mathbf{E}_{{z}}[H_{B_2}]  \leq C \Big(r_3^2 + \frac{N^d}{r_2^{d-2}}\Big) \leq C' \frac{N^d}{r_2^{d-2}}.% \stackrel{\eqref{eq:boundrho1}}{\leq}\theta\mathbf{E}_{\nu}[V_1],
\end{equation} 
In particular, this implies that $V_0 \in L^1(\mathbf{P}_{{y}})$ for any $y \in \mathbf{T}$, as required by \eqref{eq:assumptionV1}. On the other hand, combining the Markov property and \eqref{eq:hit-torusbox-lb}, it follows that
\begin{equation}\label{eq:RW-1stmom-V2}
\inf_{{y}\in{\mathbf{T}}}
\mathbf{E}_{{y}}[V_0] \geq \inf_{z \notin B_3} \mathbf{E}_{{z}}[H_{B_2}] \geq c(\delta)  \frac{N^d}{r_2^{d-2}}.
\end{equation}
Combining \eqref{eq:RW-1stmom-V1} and \eqref{eq:RW-1stmom-V2}, the condition in the first line of \eqref{eq:assumptionV1} immediately follows, for all suitably large $\theta \geq C(\delta)$.
%where the last step also used that $V_1=R_1$ under $\mathbf{P}_{\nu}$ and that $\mathrm{cap}_{B_3}(B_2)\leq C(\delta)r_2^{d-2}$ if $r_3\geq (1+\delta)r_2$, which follows by \cite[Lemma 5.1]{MR3563197} \textcolor{orange}{check this}. 
Regarding higher moments, noting that by the Markov property $V_0 -T_{B_3}$ has the same law as $H_{B_2}$ starting from some random point of $B_3^c,$ and applying the bound $(a+b)^k \leq 2^{k}(a^k+ b^k)$ valid for all $a,b \geq 0$,  yields that for all $x \in \mathbf{T}$, and $ k \geq 2$,
\begin{equation}\label{eq:RW-higher-mom-V1}
\mathbf{E}_{x}[V_0^k] \leq 2^{k} \big(\mathbf{E}_{x}[T_{B_3}^k] + \sup_{y \notin B_3} \mathbf{E}_{y}[H_{B_2}^k] \big).
\end{equation}
Applying a similar argument as e.g.~in \cite[(2.21)]{MR3936156}, one obtains that for all $x \in \mathbf{T}$,
\begin{equation}\label{eq:RW-higher-mom-V2}
 \mathbf{E}_{x}[T_{B_3}^k]  \leq k!\sup_{{y}\in{\mathbf{T}}}
\mathbf{E}_{{y}}[T_{B_3}]^k \stackrel{ \eqref{eq:exit-time}}{\leq}  k! C^k r_3^{2k} \stackrel{\eqref{eq:RW-1stmom-V2}}{\leq} k! C'(\delta)^k \inf_{{y}\in{\mathbf{T}}}
\mathbf{E}_{{y}}[V_0]^k. 
\end{equation}
Similarly, using the fact that $V_0 \geq H_{B_2}$ holds $\mathbf{P}_{{y}}$-a.s.~for any $y \in \mathbf{T}$, one finds that for all $y \notin B_3$,
\begin{equation}\label{eq:RW-higher-mom-V3}
 \mathbf{E}_{y}[H_{B_2}^k]  \leq k!\sup_{{x}\in{\mathbf{T}}}
\mathbf{E}_{{x}}[H_{B_2}]^k \leq k!\sup_{{x}\in{\mathbf{T}}}\mathbf{E}_{{x}}[V_0]^k \leq k! C(\delta)^k\inf_{{x}\in{\mathbf{T}}}\mathbf{E}_{{x}}[V_0]^k ,
\end{equation}
where the last step follows from the first moment comparison in \eqref{eq:assumptionV1} already established. Feeding \eqref{eq:RW-higher-mom-V2} and \eqref{eq:RW-higher-mom-V3} into \eqref{eq:RW-higher-mom-V1} completes the verification of \eqref{eq:assumptionV1}, for suitably large choice of $\theta=\theta(\delta) \in (1,\infty)$, and with it the proof.
\end{proof}

We now turn to the proofs of Lemmas~\ref{lem:concsoftrwri} and \ref{lem:conc} for random interlacements. The starting point is the following large deviation estimate, similar to Proposition \ref{prop:largedeviationrw} above, but simpler. For finite $B\subset\Z^d$ let $W_B$ denote the set of infinite nearest-neighbor paths in $\Z^d$ starting in $\partial B$ escaping all finite sets in finite time.  Recalling \eqref{eq:definterprocess} with $B=Q_r(0)$, $r \geq 1,$ $(X^j)_{j\geq 1}$ denotes in the sequel the random walks on $\Z^d$ corresponding to the restriction of the interlacement process to $Q_r(0),$ which are i.i.d.~with law $P_{\overline{e}_{r}}$, where $\overline{e}_{r} \equiv \overline{e}_{Q_r(0)}$, and at level $u$ the number $N_r^u\equiv N_{Q_r(0)}^u$ of trajectories hitting $Q_r(0)$ is a Poisson random variable with parameter $u\mathrm{cap}(Q_r)$, independent of $(X^j)_{j\geq 1}$. Similarly as above
\eqref{e:shifted-avg}, we consider an $(A,\mathcal{A})$-valued sequence $(Z^j)_{j \geq1}$, which we assume to be declared under $\mathbb{P}^I$ and independent of $((X^j)_{j\geq 1}, N_r^u)$, and study for measurable $F:W_{Q_r}\times A\rightarrow[0,\infty),$ $u>0$ and $r \geq 1$ the averages
\begin{equation}\label{eq:avg-RI}
\overline{V}_u= \frac1{u\mathrm{cap}(Q_r)}\sum_{j=1}^{N_r^u}V_j, \quad V_{j}=F(X^j,Z^{j}).
\end{equation}
(with $\overline{V}_u=0$ by convention whenever $N_r^u=0$). %Note that $\mathbb E^I[V_1]= E_{e_r}[F(X,Z^{1})]$ where $Z^1$ is sampled independently of $X$. With a slight abuse of notation, we denote the latter by 
\begin{proposition}
\label{prop:largedeviationri}
For all $r, \theta \geq 1,$ there exist $c=c(\theta), C=C(\theta) \in (0,\infty)$ such that, if
\begin{equation}
\label{eq:assumptionV2}
\mathbb{E}^I[V_1 ]< \infty \text{ and } \mathbb{E}^I\big[V_1^k\big]\leq k!\theta^k\mathbb{E}^I\left[V_1\right]^k, \text{ for all }k \geq 2,
\end{equation}
then for all $u>0$ and $\eta\in{(0,1)}$, one has 
%\begin{equation}
%\label{eq:largedeviationRI1}
%\mathbb{P}^I\big(|\overline{V}_m-m\mathbb{E}^I[{V}_1]|>\eta m\mathbb{E}^I[{V}_1]\big)\leq C\exp\Big(-c\eta^2m\Big)
%\end{equation}
%and for all $u>0$
\begin{equation}
\label{eq:largedeviationRI}
\mathbb{P}^I\big(|\overline{V}_{u}-\mathbb{E}^I[{V}_1]|>\eta \mathbb{E}^I[{V}_1]\big)\leq C\exp\big\{-c\eta^2u\mathrm{cap}(Q_r)\big\}.
\end{equation}
\end{proposition}
\begin{proof}
 Under \eqref{eq:assumptionV1}, it follows from Bernstein's inequality, see for instance \cite[Corollary 2.11]{MR3185193} for the version we use here, that for all $\delta\in{(0,1)}$ and integers $m \geq 1$ (see around \eqref{eq:Bernstein} for a similar argument)
\begin{equation}
\label{eq:1ststeplargedeviationri}
\mathbb{P}^I\Big(\Big|\frac1m\sum_{j=1}^m{V}_{j}-\mathbb{E}^I[{V}_1]\Big|>\delta \mathbb{E}^I[{V}_1]\Big) %\leq 2\exp\left\{-\frac{c(m\eta\mathbb{E}^I[{V}_1])^2}{m\theta^2\mathbb{E}^I[{V}_1]^2+\theta m\eta\mathbb{E}^I[{V}_1]^2}\right\}
\leq 2\exp\big\{-c(\theta)\delta^2m\big\},
\end{equation}
Using a Chernoff bound for Poisson random variables, see for instance \cite[p.21-23]{MR3185193} combined with the inequality $(1+\eps)\log(1+\eps)-\eps\geq \eps^2/4$ for all $\eps\in{(-1,1)},$ one has for all $u>0$ and $\eps\in{(0,1)}$
\begin{equation*}
\P^I(|N_r^u-u\mathrm{cap}(Q_r)|>\eps u\mathrm{cap}(Q_r))\leq \exp\big\{-c\eps^2u\mathrm{cap}(Q_r)\big\}.
\end{equation*}
Combining this with \eqref{eq:1ststeplargedeviationri}, applied with $(\eps,\delta,m)$ chosen either as $(\eta/2,\eta/(2+\eta),\lfloor u(1+\eta/2)\mathrm{cap}(Q_r)\rfloor)$ or $(\eta/2,\eta/(2-\eta),\lceil u(1-\eta/2)\mathrm{cap}(Q_r)\rceil),$ and assuming that $\eta u\mathrm{cap}(Q_r)\geq 2$ and $\eta<1/4$ which is no loss of generality, \eqref{eq:largedeviationRI} readily follows by means of a suitable union bound.
\end{proof}

\begin{Rk}
Similarly as in Remark~\ref{R:nu-equ} one can derive an exact formula for a key quantity associated to random interlacements. First, defining $T=T(X, B_2,B_3)$ as in \eqref{eq:defTj}, by \cite[eq.~(6.9)-(6.11)]{MR3563197} we have
\begin{equation}
\label{eq:expectedgzeta}
{E}_{\overline{e}_{B_2}}\Big[\sum_{i=0}^{T}1\{X_{R_i(X,B_2,B_3)}=x\}\Big]= \frac{{e}_{B_2}^{B_3}(x)}{\mathrm{cap}(B_2)}\text{ for all }x\in{B_2}.
\end{equation}
Note that compared to \cite{MR3563197} we started the sum at $i=0$ instead of $i=1,$ which is due to the fact that we started the definition of $R_i$ from $R_0=0$, see~\eqref{eq:deftilrhorho}, while in the paragraph above~(4.8) in~\cite{MR3563197} it starts from $R_1=0$. In particular in view of \eqref{eq:ident}, the left-hand side of \eqref{eq:expectedgzeta} is crucially proportional (up to projection on the torus) to the invariant distribution of the stopping times~$R_i$ for the random walk on the torus. Note that we only use these exact formulas to prove that the means in Lemmas~\ref{lem:concsoftrwri} and~\ref{lem:conc} are the same for random interlacements and the random walk, and thus that the process $\omega^{(x)}$ is an interlacement process in \eqref{e:loc-law} under $\til{\mathbf{P}}_0.$ In particular, if one is only interested in the proof of all our results for random interlacements, these formulas are never required. 
\end{Rk}

Similarly as for the random walk, we are now ready to prove Lemmas~\ref{lem:concsoftrwri} and~\ref{lem:conc} for random interlacements (starting with the latter), using Proposition \ref{prop:largedeviationri} instead of Proposition \ref{prop:largedeviationrw}.

\begin{proof}[Proof of Lemma~\ref{lem:conc} (Random interlacements case)] 
For $r=r_2$ take $V_j=T^j+1$ in \eqref{eq:avg-RI}, see~\eqref{eq:defTj},  where $T^j=T^j(B_2,B_3)$ is defined as in \eqref{eq:defnumberexcursionRI} but for the walk $X^j$ now starting in~$B_2,$ and $\eta=\eps.$  With these choices $\overline{V}_{u}= (u\mathrm{cap}(Q_r))^{-1}\mathcal{N}_{\rm{RI}}(\omega,B_2,B_3,u)$. Note that we consider here $T^j+1$ instead of $T^j$ since  the walks $X^j$ in \eqref{eq:defnumberexcursionRI} were started outside of $B_3,$ and thus the times $R_i,$ see \eqref{eq:deftilrhorho}, are shifted by $1$ compared to the corresponding walk started in $B_2.$  It moreover follows from \eqref{eq:ctehittingproba} and monotonicity that the random variable $T^1$ is dominated by a geometric random variable of parameter $p$ for some $p=p(\delta)>0,$ which readily implies that condition \eqref{eq:assumptionV2} is verified for some $\theta=\theta(\delta)$.  Moreover by the lower bound in  \eqref{eq:ctehittingproba}, we have $\cpc{B_3}{B_2}\leq C(\delta)\cpc{}{B_2}$. Since $\mathbb{E}^I[{V}_1]={E}_{\overline{e}_{B_2}}[T(X,B_2,B_3)]+1$ %by \cite[eq.~(9.24)]{MR3563197}, with $T(X,B_2,B_3)$ defined as in \eqref{eq:defTj} but with $X$ in place of $X^j$, one has   
%\begin{equation*}
%\mathbb{E}^I[{V}_1]={E}_{\overline{e}_{B_2}}[T(X,B_2,B_3)]= \frac{\cpc{B_3}{B_2}}{\cpc{}{B_2}},
%\end{equation*}
the claim follows by an application of \eqref{eq:largedeviationRI} and summing \eqref{eq:expectedgzeta} over $x\in{B_2}.$
\end{proof}

\begin{proof}[Proof of Lemma~\ref{lem:concsoftrwri} (Random interlacements case)] 
%By translation invariance, we can assume w.l.o.g.\ that $x=0.$ 
Take $r=r_2$ and denote again by $T^j=T(X^j,B_2,B_3)$ the total number of excursions that the walk $X^j$ performs across the annulus $B_3\setminus B_2,$  see \eqref{eq:defTj}, and let $(\zeta_i(X^j,B_2,B_3))_{i\geq 0}$ be the clothesline process associated to $X^j,$ see \eqref{eq:defclothesline} (whose definition can clearly be extended to the case $i=0$), where $(X^j)_{j\geq1}$ is now the set of walks in the interlacements process $\omega_{B_2}$ that hit $B_2$, started after first hitting $B_2.$ Let  $ \widehat{\xi}_i^j $ refers to $\widehat{\xi}_k$ as appearing in \eqref{eq:GnRI} for the unique choice of $k$ such that $k= i + \sum_{1 \leq n < j} (T^n+1),$ which is an independent and i.i.d.~sequence of exponential random variables with mean~1. Define for some fixed $z\in{\Sigma}$
\[V_j=\sum_{0 \leq i \leq T^j} \widehat{\xi}_i^jg_{\zeta_i(X^j,B_2,B_3)}(z).\]
As in the random walk case, by assumption on the $r_k$'s  and Lemmas~\ref{lem:separation} and~\ref{lem:indepofy}, one has that $g_\zeta(z)\leq C(\delta)g_{\zeta'}(z)$ for all $\zeta,\zeta' \in \partial B_2\times \partial B_3^c$ and $z\in \Sigma.$ One can then deduce that assumption \eqref{eq:assumptionV2} holds similarly as in the proof of Lemma~\ref{lem:conc} for random interlacements, for suitable choice of $\theta=\theta(\delta) \in (1, \infty).$ Moreover, it follows from \eqref{eq:defclothesline}, \eqref{eq:gbar-def}, \eqref{eq:expectedgzeta} and a small calculation involving the Markov property that
\begin{equation*}
	\mathbb{E}^I[V_1] = \bar{g}(z) \cdot \frac{\cpc{B_3}{B_2}}{\cpc{}{B_2}}.
\end{equation*}
Finally, for all $u_1<u_2$ we have by the definitions of $\mathcal{N}_{\rm RI}$ in \eqref{eq:defnumberexcursionRI} and $G_{\cdot}^{\rm RI}$ in \eqref{eq:GnRI} (note that similarly as before the clotheslines in these definition are shifted by $1$ since the walk therein is started outside of $B_3$ instead of inside $B_2$)
\[\mathcal{N}_{\rm RI}(\omega,B_2,B_3,u_1)\leq m \leq \mathcal{N}_{\rm RI}(\omega,B_2,B_3,u_2)\quad\Longrightarrow\quad  u_1\overline{V}_{u_1}\leq\frac{G_{m}^{\rm RI}(z)}{\mathrm{cap}(B_2)}\leq u_2\overline{V}_{u_2}.\]
Taking $u_1=m/(M(1+\eps/3))$ and $u_2=m/(M(1-\eps/3))$ we conclude by combining \eqref{eq:largedeviationRI} for $u=u_1,u_2$ and $\eta=\eps/3,$ with the concentration of $\Excri$ supplied by Lemma~\ref{lem:conc}, for $u=u_1,u_2$ and $\eps/3$ in place of $\eps,$ as well as the inequality $\cpc{B_3}{B_2}\leq C\cpc{}{B_2}.$ 
\end{proof}

\begin{Rk}[Extensions]\label{R:robust}
As with all results of Section~\ref{sec:shortrangeRWRI}, both Lemmas~\ref{lem:concsoftrwri} and~\ref{lem:conc} involve square boxes. In particular, this means for instance that within the setup of Theorem~\ref{thm:rwshortrange} or Corollary~\ref{cor:coupling}, no $\ell^2$-smoothing of boxes as used e.g.~in \cite{MR3563197} is necessary. %Indeed, no issues with regards to regularity of entrance or exit laws arises, cf.~\S\ref{sec:harnack}, as long as one applies Harnack's inequality with care. 
This degree of flexibility is relevant for applications to more general classes of graphs (for which a meaningful notion of smoothing is often not even clear) under suitable hypotheses (e.g.\  polynomial volume growth, polynomial decay of the Green function and the validity of an elliptic Harnack inequality, cf.~\S\ref{sec:harnack}), to which the above arguments can likely be extended.
\end{Rk}

\section{Appendix: admissible sets}
\label{app:B}
\renewcommand*{\thetheorem}{B.\arabic{theorem}}
\renewcommand{\theequation}{B.\arabic{equation}}

As explained below Theorem~\ref{cor:phasetransition2-intro}, an important question is to determine which sets are admissible, i.e.~belong to $\cA_{\mathbf{T}}$ in \eqref{eq:defAzd}, for these are precisely the `patterns' that can be seen as part of $\mathcal{L}^{\alpha}$ for some $\alpha > \frac12$. We classify these sets in Theorem~\ref{thm:admissible} below. Once this is established, one readily deduces Corollary~\ref{c:trichotomy} from Theorem~\ref{cor:phasetransition2-intro}; the short proof appears at the end of this appendix.

Let
\begin{equation}
\label{eq:defAzd-real}
\cA_{\Z^d}=\{K\subset \Z^d: K\neq \emptyset,  \textstyle \cp(K)\leq \frac 2{g(0)}\}
\end{equation}
so that, in view of \eqref{eq:defalpha*K}, 
\eqref{eq:defAzd}, \eqref{eq:easy} and our definition of the capacity for subsets of $\mathbf{T}$, see below \eqref{e:e_K}, the family $\cA_{\mathbf{T}}$ corresponds precisely to projections onto $\mathbf{T}$ of sets belonging to $\cA_{\Z^d}$ 
when $N \geq C$.

\begin{theorem}[Admissible sets]
\label{thm:admissible}
		\begin{align*}
			&	\cA_{\Z^d} =\{K\subset \Z^d: |K|\leq 2\}, \text{ for all $d \geq 4$, and}\\
	&\mathcal{A}_{\Z^3} = \{K\subset \Z^d: |K|\leq 2 \text{ or } |K| =3  \text{ and } K \text{ is connected}\}. 
	\end{align*}
\end{theorem}

The following lemma will be used to reduce to the cases $d\in{\{3,4\}}$ when proving Theorem~\ref{thm:admissible}. In the sequel we add subscripts $\Z^d$ to various quantities such as $\text{cap}(\cdot)$ or $g(\cdot)$ to highlight their dependence on the underlying graph.

\begin{lemma}
\label{pro:increasingdimension}
	For all $d \geq 3$ and finite $K\subset \Z^{d}$, the function $$d' \in \{ d,d+1,\dots\} \mapsto \frac{\cp_{\Z^{d'}}(K\times \{0\}^{d'-d})}{\cp_{\Z^{d'}}(\{0\})}$$ is non-decreasing. 
\end{lemma}
\begin{proof}
To stress its dependence on dimension, we write ${P}^{\Z^d}_x$ for the canonical law of simple random walk on $\Z^d$ starting from $x\in \Z^d$. For all $x,y\in K$ writing $x'=(x,0,\ldots,0)$ and  $y'=(y,0,\ldots,0)\in \Z^{d'}$ one has 
\begin{align*}
	\frac{g_{\Z^d}(x,y)}{g_{\Z^d}(0,0)} = P^{\Z^d}_x(H_y<\infty) = P_{x'}^{\Z^{d'}} (H_{\{y\}\times {\Z^{d'-d}}}<\infty) \geq P^{\Z^{d'}}_{x'}(H_{\{y\}\times \{0\}^{d'-d}}<\infty) 
	= \frac{g_{\Z^{d'}}(x', y')}{g_{\Z^{d'}}(0,0)}.
\end{align*}
The claim follows using a well-known variational characterisation of the capacity, see \cite[(1.61)]{MR2932978}, whose proof easily extend to infinite transient graphs.
\end{proof}

Our next result will be used as a further reduction step, by which fully determining $\mathcal{A}_{\Z^d}$ in \eqref{eq:defAzd} will effectively boil down to computing the capacities of a small number (at most ten) of sets. % in order to fully determine $\mathcal{A}_{\Z^d}.$
 To simplify notation we will from now for each $d'\leq d$ identify $\Z^{d'}$ with $\Z^{d'}\times \{0\}^{d-d'}\subset\Z^d$. Recall the sets $K_1=\{(0,0), (0,1), (0,2)\}$ and $K_2=\{(0,0), (0,1), (1,0)\}$ from \eqref{e:3points}, viewed as subsets of $\Z^2\times\{0\}^{d-2}$ for $d\geq3$ according to our convention, and which correspond to all connected sets with cardinality three. Further, let
\begin{equation}
\label{eq:defAi}
\begin{aligned}
	A_1&= \{ (0,0,0), (0,2,0), (0,1,1)\},&
	A_2&= \{ (0,0,0), (0,2,0), (0,3,0)\},\\
	A_3 &=  \{ (0,0,0), (1,1,0), (0,3,0)\}, & 
	A_4 &=  \{ (0,0,0), (0,2,0), (1,2,0)\}, \\
	A_5 &=  \{ (0,0,0), (1,1,0), (1,2,0)\}, & 
	A_6 &=  \{ (0,0,0), (0,2,0), (1,1,1)\}, \\
	A_7 &=  \{ (0,0,0), (1,1,0), (1,1,1)\}, &
	A_8 &=  \{ (0,0,0), (0,0,1), (0,1,0), (0,1,1)\}.
\end{aligned}
\end{equation}
The following result mirrors Theorem~\ref{thm:admissible}. Its first part \eqref{pro:highd} will be enough to treat the cases $d \geq 4$; the more refined \eqref{pro:d=3} will be used to deal with the case $d=3$. In what follows, $K$ and $K'$ are called isomorphic if $K$ can be obtained from $K'$ by lattice symmetries.

\begin{proposition}[$d \geq 3$, $K\subseteq \Z^d$, $|K|\geq 3$]\label{pro:finite-red}
	\begin{align}
	 \inf_K\, \cp(K)&\geq \min_{i=1,2}\cp(K_i), \label{pro:highd}\\
	 \widetilde{\inf_K} \,  \cp(K) &\geq \min_{1\leq i\leq 8} \cp(A_i), \label{pro:d=3}
	\end{align}
	where $ \widetilde{\inf}_K $ refers to a restricted infimum over sets $K$  not isomorphic to $K_1$ or $K_2$.
\end{proposition}
\begin{proof} We start by making the following observation, which will be used extensively throughout the proof. Using again the variational characterization \cite[(1.61)]{MR2932978} of $\text{cap}(\cdot)$, one sees that,
\begin{equation}
\label{cap-monot-cond}
\text{\parbox{14cm}{if $K,K' \subset \Z^d$ are such that there exists a bijection $\phi:K\rightarrow K'$ with $g(x,y) \leq g(\phi(x), \phi(y))$ for all $x,y \in K$, then $\text{cap}(K) \geq \text{cap}(K')$.}}
\end{equation}
We proceed to show \eqref{pro:highd}. By monotonicity of $K \mapsto \text{cap}(K)$, it is sufficient to prove the claim for a set $K$ with $|K|=3$. Write $K=\{x,y,z\}$ and without loss of generality suppose that $|{x-z}|_1\geq 2$. Then by Lemma~\ref{lem:boundongreenfunction} and translational and rotational invariance we have
\begin{align*}
	g(x,z)\leq \sup_{v: |v|_1=2} g(0,v) = (g((0,0),(0,2)) \vee  g((0,0),(1,1)).
\end{align*}
	Similarly, 
	\begin{align*}
		(g(x,y) \vee g(y,z) )\leq \sup_{v: |v|_1=1} g(0,v) =g((0,0),(0,1))=g((0,1),(0,2)) = g((0,1), (1,1)).
	\end{align*}
	So if $g((0,0),(0,2))\geq g((0,0),(1,1))$, then \eqref{cap-monot-cond} applies with $K'= \{(0,0), (0,1), (0,2)\}= K_1$, $\phi(x)=(0,0)$, $\phi(y)=(0,1)$ and $\phi(z)=(0,2)$ and yields that $
		\cp(K)\geq \cp(K_1)$. If instead $g((0,0),(0,2))\leq g((0,0),(1,1))$, then \eqref{cap-monot-cond} applies similarly with $K' = \{(0,0), (0,1), (1,1) \} = K_2$ to give 
$\cp(K)\geq  \cp(K_2).$ Overall, \eqref{pro:highd} follows.

We now show \eqref{pro:d=3}, and first assume to this effect that $|K|=3$ and $K$ is not isomorphic to $K_1$ or $K_2$. Writing $K=\{x,y,z\}$, one notices since $K$ is not isomorphic to $K_1$ or $K_2$, then unless $K$ is isomorphic to $A_1$, in which case $\text{cap}(K)$ is evidently bounded from below by the right-hand side of \eqref{pro:d=3}, the set $K$ must contain two points at $\ell^1$-distance at least three, say $x$ and $z$. Without loss of generality suppose that $|x-y|_1\geq 2$ (otherwise $|y-z|_1 \geq 2$ by the triangle inequality). Then by  Lemma~\ref{lem:boundongreenfunction},
\begin{align*}
	g(x,z)\leq \sup_{v: |v|_1=3} g(0,v),\qquad
	g(x,y)\leq \sup_{v: |v|_1=2} g(0,v),\qquad
	g(y,z)\leq \sup_{v: |v|_1=1} g(0,v).
\end{align*}
	We then consider three different cases depending on which $v\in \{(0,3,0), (1,2,0), (1,1,1)\}$ achieves $\sup_{v: |v|_1=3} g(0,v)$. Then for each possible $v$ we consider two different cases depending on which $v'\in{\{(0,2,0),(1,1,0)\}}$ achieves $\sup_{v': |v'|_1=2} g(0,v')$. Note also that $\sup_{v'': |v''|_1=1} g(0,v'')$ must be achieved at $g((0,0,0),(0,0,1))$ by symmetry. If for instance the previous suprema are achieved at $v=(0,3,0)$ and $v'=(0,2,0),$ \eqref{cap-monot-cond} implies $\mathrm{cap}(K)\geq \mathrm{cap}(A_{2}).$ Considering all five other possible cases gives us that if $K$ is not isomorphic to $K_1$ or $K_2,$ then
	\[
	\cp(K)\geq \min_{1  \leq i\leq 7} \cp(A_i),
	\]
which completes the proof in the case $|K|=3$. Suppose next that $|K|\geq 4$. It suffices to prove that assuming $\cp(K)<\min_{i\leq 7} \cp(A_i)$, then $\cp(K)=\cp(A_8)$. The assumption that $\cp(K)<\min_{i\leq 7} \cp(A_i)$ implies that all $K'\subseteq K$ with $|K'|=3$ must satisfy $\cp(K')<\min_{i\leq 7}\cp(A_i)$. From the proof above for $|K|=3$ we can now deduce that any such $K'$ must be isomorphic to either $K_1$ or $K_2.$ By an elementary geometric argument, it then follows that the only possible shape for $K$ for which this is possible is $K=A_8$ modulo isomoprhisms. Hence $\cp(K)=\cp(A_8)$, which finishes the proof.
\end{proof}

Combining Lemma~\ref{pro:increasingdimension} and Proposition~\ref{pro:finite-red} with the next result, we will soon see that in order to identify the set $\mathcal{A}_{\Z^d}$ for all $d\geq3,$ it will be enough to compute the Green's function $g(0)$ and the capacities of the sets $K_1$ and $K_2$ from \eqref{e:3points} in dimensions three and four, as well as the capacities of the sets $A_i,$ $i\in{\{1,\dots,8\}}$ from \eqref{eq:defAi} in dimension three. The following lemma gathers these numerical computations, and, in doing so, also isolates the parts of the argument which rely on computer-assisted methods. Note that we express our numerical results with an absolute error of $10^{-30}$ as this might be useful in the future, but we will actually only need a precision $10^{-3}$.

\begin{lemma}
\label{lem:computerassisted}
With an absolute error of at most $10^{-30},$ one has when $d=3$ 
\begin{equation*}
\begin{split}
g_{\Z^3}(0)&=1.516386059151978018156012159681\\
\cp_{\Z^3}(K_1)\vee\cp_{\Z^3}(K_2)&=1.271113197748638670916474203095\\
\min_{1\leq i\leq 8}\cp_{\Z^3}(A_i)&=1.335471948363948449723770501931
\end{split}
\end{equation*}
and when $d=4$
\begin{equation*}
\begin{split}
g_{\Z^4}(0)&=1.239467121848481712678697664859\\
\cp_{\Z^4}(K_1)\wedge\cp_{\Z^4}(K_2)&=1.849398784221098051683201012328.
\end{split}
\end{equation*}
\end{lemma}

Before explaining how the values in Lemma~\ref{lem:computerassisted} are obtained, let us conclude the proof of Theorem~\ref{thm:admissible}.
\begin{proof}[Proof of Theorem~\ref{thm:admissible}] We first observe that the inclusion 
\begin{equation}\label{eq:A-easy}
\cA_{\Z^d} \supset\{K\subset \Z^d: 0< |K|\leq 2\}
\end{equation}
holds true for all $d \geq 3$ on account of \eqref{e:cap-2point}. We now proceed in increasing order of difficulty, and start with the case $d=4$.  By Lemma~\ref{lem:computerassisted} we have that
$\cp_{\Z^4}(K_1)\wedge\cp_{\Z^4}(K_2)\geq 1.84>1.62\geq {2}/{g_{\Z^4}(0)}$. Using~\eqref{pro:highd}, we deduce that there are no sets $K\in{\mathcal{A}_{\Z^4}}$ with $|K|\geq3.$ It follows that the inclusion in \eqref{eq:A-easy} is in fact an equality.

Next when $d\geq5$, one deduces from Lemma~\ref{pro:increasingdimension} and the previous case that $g_{\Z^d}(0)(\cp_{\Z^d}(K_1)\wedge\cp_{\Z^d}(K_2))\geq g_{\Z^4}(0)(\cp_{\Z^4}(K_1)\wedge\cp_{\Z^4}(K_2))>{2}$. Using \eqref{pro:highd}, we can conclude similarly as before.

Finally, assume that $d=3.$ Then by Lemma~\ref{lem:computerassisted} one has 
\begin{equation*}
\cp_{\Z^3}(K_1)\vee\cp_{\Z^3}(K_2)\leq 1.28<1.31\leq \frac{2}{g_{\Z^3}(0)}\text{ and }\min_{1\leq i\leq 8}\cp_{\Z^3}(A_i)\geq 1.33>1.32\geq \frac{2}{g_{\Z^3}(0)}.
\end{equation*}
Using \eqref{pro:d=3}, we deduce that the only sets $K\in{\mathcal{A}_{\Z^3}}$ with $|K|\geq3$ are isomorphic to $K_1$ and $K_2$. Together with \eqref{eq:A-easy}, the claim follows. 
\end{proof}

We now explain in detail how Lemma~\ref{lem:computerassisted} is obtained. The first step in our algorithm consists in computing the Green's function $g(x,y)$ for any $x,y$ belonging to the sets for which we want to compute the capacity, which will be enough in view of \eqref{eq:equiviagreen} below. We follow the strategy developed in \cite[Appendix~B]{MR1174248}. Let us provide some details for the reader's convenience. The main idea is to use the formula 
\begin{equation}
\label{eq:formulaforGreen}
g(x)=\int_{-\infty}^{\infty}F_x(u)\mathrm{d}u,\text{ where }F_x(u)=de^u\prod_{k=1}^d\exp(-e^u)I_{|x_k|}(e^u)\text{ for all }x\in{\Z^d},
\end{equation}
and $I_k(t)$ denotes the modified Bessel function of the first kind with parameter $k$ at time $t.$ The formula \eqref{eq:formulaforGreen} is a simple consequence of \cite[(2.10)]{MR88110} and the substitution $u\mapsto de^u.$ In \cite{MR1174248} the integral in \eqref{eq:formulaforGreen} is approximated by a finite sum in five steps, which we now summarize.
\begin{enumerate}[label*=\arabic*)]
 \item Replace the integral in \eqref{eq:formulaforGreen} by the Riemann sum $h\sum_{m=-\infty}^{\infty}F_x(mh)$ for some small $h>0$ to be chosen later. We denote the absolute error made in this step by Error1$(h,d)$, which corresponds to \cite[(B.64)]{MR1174248} for the choice $s=\arctan(2\pi/h).$
 \item Remove $h\sum_{m=-\infty}^{-(M+1)}F_x(mh)$ in the previous Riemann sum, for some large $M$ to be chosen later. We denote the absolute error made in this step by Error2$(h,d,M)$, cf.~\cite[(B.68)]{MR1174248}. Note in particular that Error1 and Error2 are uniform in $x\in{\Z^d}$.
 \item Replace $h\sum_{m=M+1}^{\infty}F_x(mh)$ in the previous Riemann sum by
\begin{equation*}
\frac{d}{(2\pi)^{\frac{d}2}}\frac{h\exp\big(-(M+1)(\frac{d}2-1)h\big)}{1-\exp\big(-(\frac{d}2-1)h\big)}.
\end{equation*}
We denote the absolute error made in this step for any  $|x|_{\infty} \leq N$ by Error3$(h,d,M,N)$, which corresponds to \cite[(B.72)]{MR1174248}. Note that in \cite[(B.72)]{MR1174248} it is assumed that $N\leq 54$ and $Mh\geq 45$, as will be the case for us in \eqref{eq:choiceparameters}.
\item For some large $T>0$ and $J\in{\N}$ with $T\leq J/2$ to be chosen later, replace the occurrence of $e^{-t}I_k(t)$  (part of $F_x$) in the remaining finite sum  by $T(t,k,J)$ for each $t\leq T$, where 
\begin{equation*}
T(t,k,J)=e^{-t}\left(\frac{t}{2}\right)^k\sum_{j=0}^J\frac{(t^2/4)^j}{j!(j+k)!}.
\end{equation*} 
The relative error made in this step is Error4$(T,J)$, which is uniform in $k$; it corresponds to \cite[(B.40)]{MR1174248}.
\item For some $\tilde{J}\in{\N}$ to be chosen later, replace $e^{-t}I_k(t)$ by $A(t,k,\tilde{J})$ for each $t> T$ and $k\leq N,$ where 
\begin{equation*}
A(t,k,\tilde{J})=\frac{1}{\sqrt{2\pi t}}\sum_{j=0}^{\tilde{J}}\frac{(-1)^j(k,j)}{(2t)^j},\text{ where }(k,j)=\frac{1}{4^jj!}\prod_{i=1}^{j}(4k^2-(2i-1)^2).
\end{equation*} 
We denote the relative error made in this step by Error5$(T,\tilde{J},N)$, which corresponds to \cite[(B.7) and (B.46)-(B.47)]{MR1174248}.
\end{enumerate}

Combining all these steps one can approximate $g_{\Z^d}(x)$ by
\begin{equation*}
\begin{split}
\tilde{g}(x,d,h,M,T,J,\tilde{J})\stackrel{\text{def.}}{=}&\sum_{m=-M}^Mdhe^{hm}\prod_{k=1}^d\left(T(e^{mh},|x_k|,J)\1_{e^{mh}\leq T}+A(e^{mh},|x_k|,\tilde{J})\1_{e^{mh}>T}\right)
\\&+\frac{d}{(2\pi)^{\frac{d}2}}\frac{h\exp\big(-(M+1)(\frac{d}2-1)h\big)}{1-\exp\big(-(\frac{d}2-1)h\big)}.
\end{split}
\end{equation*}
The function $\tilde{g}$ consists of finite sums and products of usual functions, and can thus be approximated using a computer with high precision. Such computations were performed in high dimensions in \cite[Section~5]{MR3719063} using a Mathematica notebook called ``SRW.nb'' available at \cite{FitznerWebsite}. We modified this notebook to include as well the computation of the capacity, and a version called ``Cap.nb'' is available at \cite{code}. One still needs to choose the parameters $h,M,T,J$ and $\tilde{J}.$ Since $g_{\Z^d}(0)\leq 2,$ we have for all $x\in{\Z^d}$ with $|x|_{\infty}\leq N$
\begin{equation}
\label{eq:error}
\begin{split}
|g(x)-\tilde{g}(x,d,h,M,T,J,\tilde{J})|\leq& \text{Error1}(h,d)+\text{Error2}(h,d,M)+\text{Error3}(h,d,M,N)
\\&+2\big((1+\text{Error4}(T,J)\vee \text{Error5}(T,\tilde{J},N))^d-1\big).
\end{split}
\end{equation}
One can find the exact formulas for these errors in the file ``Errors.nb'', also available at \cite{code}, where they are also computed. To make these errors small, one typically needs to choose $h$ small, and $hM,$  $T,$  $J/T$ and $\tilde{J}$ large. Choosing
\begin{equation}
\label{eq:choiceparameters}
N=3,\quad d\in{\{3,4\}},\quad h=\frac{76}{630},\quad M=630,\quad T=80,\quad J=139\text{ and }\tilde{J}=30
\end{equation}
we obtain that the total error in \eqref{eq:error} is at most $10^{-32}.$ We then compute the values of $\tilde{g}$ for these parameters and any $x$ appearing in the sets from Lemma~\ref{lem:computerassisted}, which are stored in the file ``SRWIntegralsData.nb''.

Let us now briefly explain how the capacities can be deduced from these Green's functions. For each $K\subset\Z^d,$ denoting by $G_K$ the matrix $(g(x,y))_{x,y\in{K}},$ by ${e_K}$ the vector $(e_K(x))_{x\in{K}}$ and by $\mathbf{1}$ the vector of size $|K|$ with all coordinates equal to one, we have by \eqref{eq:lastexit}
\begin{equation}
\label{eq:equiviagreen}
G_K{e_K}=\mathbf{1}.
\end{equation}
Once $G_K$ is known, one can thus use Mathematica again to solve the linear system \eqref{eq:equiviagreen}, which yields the equilibrium measure and the capacity after summation. Note that we do not know the values of $G_K$, but only their approximate values $\tilde{G}_K$ with $|\tilde{G}_K(x,y)-G_K(x,y)|\leq 10^{-32}$ for each $x,y\in{K}$. Denoting by $\tilde{e}_K$ the computed solution of $\tilde{G}_K\tilde{e}_K=1$, we have $e_K(x)-\tilde{e}_K(x)=G_K^{-1}(\tilde{G}_K-G_K)\tilde{e}_K(x)$. Using (1.38) and Proposition~1.11 in \cite{MR2932978}, one can moreover easily show that $\sum_{x,y\in{K}}|G_K^{-1}(x,y)|\leq 2|K|$, and for our choices of $K$ one easily deduces that the values we obtain for the capacity have an error of at most $5\cdot 10^{-31}$. This solution to the system \eqref{eq:equiviagreen} is also implemented in our Mathematica notebook ``Cap.nb'', and the results are stored in the file ``SRWCapacityData.nb''. Running this program, which should take under fifteen minutes on modern computers, finishes the proof of Lemma~\ref{lem:computerassisted}.

\begin{Rk} In order to check the consistency of our results, let us also mention that the value of $g_{\Z^3}(0)$ can alternatively also be computed using the formula from \cite[(2.11)]{MR88110}. This formula corresponds to three times the integral $I_3$ from \cite{MR1257}, whose approximate value can for instance found at \url{http://oeis.org/A091672}. The difference between the value obtained by this method and the value obtained by our Mathematica notebook is $2.8\cdot10^{-33},$ which is consistent with our error of $10^{-32}.$ Note that this error seems to mainly come from the term Error2 in \eqref{eq:error}.
\end{Rk}

We conclude this appendix with the short 

\begin{proof}[Proof of Corollary~\ref{c:trichotomy}] Let $\eta>0$ be small enough so that $\frac12 + \eta < \alpha_{\mast}(K_1) \wedge \alpha_{\mast}(K_2)$ when $d=3$ and $\frac12 + \eta < \alpha_{\mast}$ when $d=4$ (with $\alpha_{\mast}$ as in \eqref{eq:alpha_*}). Now fix any two-point set $K$ such that $\frac12<\alpha_{\mast}(K)< \frac{1+ \eta}{2}$ and that there exists $D<\infty$ verifying for any $x,y\in{\mathbb{T}}$ that $\alpha_{\mast}(K)<\alpha_{\mast}(\{x,y\})$ if and only if $|x-y|_1\leq D$. Such a set can always be found since $\alpha_{\mast}(\{x,y\}) \searrow \frac12$ as $|x-y|_1 \to \infty$, and since $(x,y)\mapsto\alpha_{\mast}(\{x,y\})$ is decreasing in $|x-y|_1$ by \eqref{lem:boundongreenfunction} and \eqref{e:cap-2point}.

Consider now first the case $d=3$. Applying Theorem~\ref{cor:phasetransition2-intro} (specifically, the first line of \eqref{eq:phasetransitionK}) to this choice of $K$ and with $\alpha = \frac12 + \eta (> \alpha_{\mast}(K))$ yields a coupling $\mathbf{Q}$ with the property \eqref{e:coupling1/2+} upon identifying $\mathcal{D}^{\cdot}$ with $\mathcal{B}_K^{\cdot}$. In view of \eqref{eq:defBFK}, \eqref{eq:defAzd}, the condition on $\eta$ and by Theorem~\ref{thm:admissible}, the patterns sampled independently as part of $\mathcal{D}^{\alpha}$ for $\alpha > \frac{1+ \eta}{2}$ are precisely of the form $i)$-$iii)$ in Corollary~\ref{c:trichotomy}.% with the choice $D= \sup \{ \delta(A) : \alpha_{\mast}(A) > \alpha_{\mast}(K)  \}$. Note that $D< \infty$ since the supremum in question only ranges over finitely many sets of cardinality two and $K_i$, $i=1,2.$ 

Let now $\mathcal{D}^{\alpha}_-$ be obtained from $\mathcal{D}^{\alpha}$ by removing all sets $A$ included in \eqref{eq:defBFK} corresponding to images of $K_i$ by torus isomorphisms for any $i\in \{1,2\}$. 
By Lemma~\ref{L:D_LB} applied with $K=K_i$ one knows that $\lim_N \mathbf{Q}(D_{K_i}(\mathcal{L}^{\frac12 + \eta}) \geq N^{c(\eta)})=1$ for some $c(\eta)>0$. On the other hand, by a similar calculation as below \eqref{eq:nocoupling} one sees that $\lim_N \mathbf{Q}(D_{K_i}(\mathcal{D}_-^{\frac{1 + \eta}{2}}) \geq N^{c(\eta)})=0$. It follows that the limit in \eqref{e:coupling1/2+} vanishes if one replaces $\mathcal{D}^{\cdot}$ by $\mathcal{D}^{\cdot}_-$.

Finally, if $d\geq4$, the above coupling $\mathbf{Q}$ satisfies \eqref{e:coupling1/2+} but does not include patterns of type $iii)$ on account of Theorem~\ref{thm:admissible}. However, including images  $A$ by torus isomorphisms of $K_i$ for $i=1,2$ independently with probability $p^{\alpha}(A)$ for $\alpha>\frac12$ makes no difference. Indeed by Markov's inequality, the probability to sample such $A$ in $Q_N$ is bounded by $C N^d \mathbb{P}(K_1 \subset \mathcal{L}^{\alpha}) \to 0$ as $N \to \infty$ on account of \eqref{eq:upperboundonprobalate} since $\frac{\alpha}{\alpha_{\sast}(K_1)} >1$ for $\alpha > \frac12$ when $d \geq 4$. 
\end{proof}

\bibliography{biblio-new}
\bibliographystyle{abbrv}

\end{document}